\documentclass[12pt]{amsart}
\usepackage{amsfonts}
\usepackage{amsmath}
\usepackage{amsxtra}
\usepackage{amssymb,latexsym}
\usepackage[mathcal]{eucal}
\usepackage{amscd}

 \input xy  \xyoption{all}

 \usepackage[pdftex,bookmarks,colorlinks,breaklinks]{hyperref}

\newtheorem{theo}{{\bfseries Theorem}}[section]
\newtheorem{prop}[theo]{{\bfseries Proposition}}
\newtheorem{lem}[theo]{{\bfseries Lemma}}
\newtheorem{cor}[theo]{{\bfseries Corollary}}

\theoremstyle{definition}
\newtheorem{ex}[theo]{{\bfseries Example}}
\newtheorem{remark}[theo]{{\bfseries Remark}}
\newtheorem{df}[theo]{{\bfseries Definition}}

\newtheorem{Qu}[theo]{{\bfseries Question}}
\newtheorem{Qs}[theo]{{\bfseries Questions}}

\def \ul {\underline}
\def \ol {\overline}
\def \N {\mathbb N}

\def \Z {\mathbb Z}
\def \R {\mathbb R}
\def \A {\mathcal A}
\def \B {\mathcal B}

\def \D {\mathcal D}
\def \E {\mathcal E}
\def \F {\mathcal F}

\def \L {\mathcal L}

\def \M {\mathcal M}
\def \NN {\mathcal N}

\def \P {\mathcal P}
\def \Q {\mathcal Q}

\def \T {\mathcal T}

\def \mm {\mathbf m}
\def \nn {\mathbf n}

\def \ww {\mathbf w}
\def \rr {\mathbf r}
\def \ss {\mathbf s}
\def \00 {\mathbf 0}
\def \a {\alpha }
\def \b {\beta}
\def \ep {\epsilon}
\def \g {\gamma}
\def \d {\delta}

\def \o {\omega}
\def \t {\tau}
\def \r {\rho}
\def \z {\zeta}
\def \Lra {\Longrightarrow}

\def \proof { {\bfseries Proof:} }
\def \LAB {\L \A \B}

\def \Nuc {{\rm{Nuc}}}

\usepackage{amssymb,latexsym}
\usepackage[mathcal]{eucal}
\usepackage{amscd}
\usepackage{amsmath}
\usepackage{graphicx}
\usepackage{amsfonts}
\usepackage{amscd}

\numberwithin{equation}{section}

\begin{document}

\title{ WAP Systems and Labeled Subshifts}
\vspace{1cm}


\author{Ethan Akin and Eli Glasner}
\address{Mathematics Department,
 The City College, 137 Street and Convent Avenue,
 New York City, NY 10031, USA}
\email{ethanakin@earthlink.net}

\address{Department of Mathematics,
Tel-Aviv University, Ramat Aviv, Israel}
\email{glasner@math.tau.ac.il}

 \date{August  18, 2016}




\keywords{WAP, HAE, LE dynamical systems, space of
labels, expanding functions,
enveloping semigroup, adherence semigroup, subshifts, countable subshifts, symbolic dynamics,
null, tame}


\thanks{{\em 2010 Mathematical Subject Classification} 37Bxx, 37B10, 54H20,
54H15}

\maketitle

\tableofcontents

\section*{Introduction}\label{intro}


The main object of this work is to present a powerful method of construction of
subshifts which we use chiefly to construct WAP systems with various properties.
Among many other applications of
these so called labeled subshifts,
we obtain examples of null as well as non-null WAP subshifts,
WAP subshifts of arbitrary countable (Birkhoff) height,
and completely scrambled WAP systems of arbitrary countable height.
We also construct LE but not HAE subshifts, and recurrent non-tame subshifts.
Of course all of these notions, with some or all of which the reader may not be familiar,
will be defined and illustrated in due course.

The notion of weakly almost periodic (WAP) functions on a
locally compact abelian
group $G$
was introduced by Eberlein \cite{Eb}, generalizing Bohr's notion of almost periodic (AP) functions.
As the theory of AP functions was eventually reduced to the study of the largest topological
group compactification of $G$, so the theory of WAP functions can be reduced
to the study of the largest semi-topological semigroup compactification of $G$.
Following Eberlein's work there evolved a general theory of WAP functions on a general
topological group $G$, or even more generally, on various type of semigroups.
From the very beginning it was realized that a dual approach, via topological
dynamics, is a
very fruitful tool as well as an end in itself. Thus in the more recent
literature on the subject one is usually concerned with WAP dynamical systems $(X,G)$.
These are defined as continuous actions of the group $G$ on a compact Hausdorff space $X$
such that, for every $f \in C(X)$, the weak closure of the orbit $\{f \circ g : g \in G \}$
is
weakly
compact.
The turning point toward this view point is the paper of Ellis and Nerurkar \cite{EN},
which used the famous double limit criterion of
Grothendieck to reformulate the definition
of WAP dynamical systems as those $(X,G)$ whose enveloping
semigroup $E(X,G)$ consists of continuous maps (and is thus a semi-topological semigroup).

In the last two decades the theory of WAP dynamical systems was put into
the broader context of hereditarily almost equicontinuous (HAE) and tame dynamical systems.
The starting point for this direction was the proof, in the work \cite{AAB} of
Akin Auslander and Berg, that WAP systems are HAE. For later development
along these lines see e.g. \cite{GM}.

Most of the extensive literature on the subject of WAP functions and WAP dynamical systems
has a very abstract flavor.
The research in these works is mostly concerned
with related questions in harmonic analysis, Banach space theory,
and the topology and the algebraic structure of the universal WAP semigroup compactification.
Very few papers deal with presentations and constructions of concrete WAP
dynamical systems. As a few exceptions let us point out the works of
Katznelson-Weiss \cite{KW}, Akin-Auslander-Berg \cite{AAB1},
Downarowicz \cite{D}
and Glasner-Weiss \cite[Example 1, page 349]{GW-LE}.
Even in these few works the attention is usually directed toward examples of
{\bf recurrent} WAP topologically transitive systems. These are the (usually metric) WAP
dynamical systems
which admit a recurrent transitive point.

A point $x$ in a metric dynamical system $(X,G)$ is a point of equicontinuity
if  for every $\epsilon >0$ there is a $\delta >0$ such that $d(gx,gx') < \epsilon$
for every $x' \in B_\delta(x)$ and every $g \in G$. The system is called
{\em almost equicontinuous} (AE) if it has a dense (necessarily $G_\delta$)
subset of equicontinuity points. It is {\em hereditarily almost
equicontinuous} (HAE) if every subsystem (i.e. non-empty closed invariant subset)
is AE.

As our work deals almost exclusively with cascades (i.e. $\Z$-dynamical systems), in
the sequel we will consider dynamical systems of the form $(X,T)$ where $T : X \to X$
is a
homeomorphism. A large and important class of cascades
is the class of
symbolic systems or subshifts. We will deal only with
subsystems of the
{\em Bernoulli} dynamical system
$(\{0,1\}^\Z, S)$, where $S$ is the shift transformation defined by
\begin{equation}\label{shift}
(Sx)_n = x _{n+1}\ ( x \in \{0,1\}^\Z, \ n \in \Z).
\end{equation}
We will call such dynamical systems {\em subshifts}.
It was first observed in \cite{GM} that a subshift is HAE iff it is countable (see
Proposition
\ref{propHAE} below). In particular it follows that WAP subshifts are countable.
Since a dynamical system which admits a recurrent non-periodic point is necessarily uncountable,
it follows that in a WAP subshift
the only recurrent points are the periodic points.
These considerations immediately raise the question which countable subshifts are WAP,
and how rich is this class ?
This question was
the starting point of our investigation.

As we proceeded with our study of that problem we were able to construct several simple
examples of both WAP and non WAP topologically transitive countable
subshifts, but
particular constructions of WAP subshifts turned out to be quite complicated. After many trials
we finally discovered the beautiful world of {\em labels}. We begin with $FIN(\N)$, the additive semigroup
of nonnegative integer-valued functions with finite support defined on $\N$, the set of positive
integers. A label is a subset $\M$ of $FIN(\N)$ which is hereditary in the sense
that $ \00 \leq \mm_1 \leq \mm$ and $\mm \in \M$ imply $\mm_1 \in \M$. The space $\LAB$ of labels
has a natural compact metric space structure.

For an odd positive integer $b = 2e + 1$ every integer $t$ has a unique symmetric base $b$ expansion
using
the function $k = k_b$ defined by
$k_b(i) = b^{i - 1}$ for $i \in \N$.
$ t = \Sigma_{i = 1}^{\infty} \ \ep_i k(i)$ with $|\ep_i| \leq e$ and with $\ep_i = 0$ for all but finitely
many indices $i$. We use $b \geq 5$ and extend $k$ to define $k : \Z \to \Z$ such that
$k(-n) = -k(n)$ and $k(n + 1) > 3 \Sigma_{i = 0}^n \ k(i)$ for $n \geq 0$. Together with $k$ we use
an infinite partition $\{ D_{\ell} : \ell \in \N \}$ of $\N$
 by infinite subsets. The set $IP(k) \subset \Z$
 \footnote{
$IP(k)$ is an example of a {\em symmetric IP set}; for more information on IP sets and their connections to dynamical systems see e.g.  \cite{FW-79}, \cite{F},
\cite{FK-85}, \cite{GW-06} and the recent
paper \cite{AG-16}.}  
consists of the sums of finite subsets of the image of $k$.
That is, it is the set of $t$ such that $\ep_i = \pm 1$ whenever $\ep_i \not= 0$.
Each
$t \in IP(k)$  has a unique expansion $t = k(j_1) + \dots + k(j_r)$ with
$|j_1| > \dots >|j_r|$. The \emph{length vector} $\rr(t) \in FIN(\N)$ counts the number
of
 occurrences of each member of the
 partition in the expansion. That is, $\rr(t)_{\ell} = \# \{ i : j_i \in D_{\ell} \}$. For a label $\M$
 the set $A[\M]$ is the set of $t \in IP(k)$ such that $\rr(t) \in \M$. For example, $\emptyset$ and $ 0 = \{ \00 \}$
 are labels with $A[\emptyset] = \emptyset$ and $A[0] = \{ 0 \}$.

Once the base $b$ and the partition $\{ D_{\ell} \}$ are fixed, there is a canonical injective, continuous map
from the space of labels $\LAB$ into $\{0,1\}^\Z, \ \M \mapsto x[\M]$ where $x[\M]$ is the
characteristic function of the set $A[\M]$. Thus
 to a label $\M$ there is assigned a subshift $X(\M)$, the orbit closure of
$x[\M]$ under $S$.

We show in Theorem \ref{towtheoBanach01}
that the set $IP(k)$ of all expanding times has upper Banach density zero.
This, in turn, implies that for every label $\M$ the corresponding subshift
$(X(\M),S)$ is uniquely ergodic with the point measure at $e = x[\emptyset]$ the unique invariant
probability measure. It follows that each such system has zero topological entropy.

The space $\LAB$ is naturally equipped with an action of the discrete semigroup
$FIN(\N)$,
$$(\rr, \M) \mapsto  \M - \rr = \{ \ww \in FIN(\N) : \ww + \rr \in \M \}.$$
We denote the compact orbit closure of a label $\M$ under this action
by $\Theta(\M)$.

The key lemma which connects the two actions (the $FIN(\N)$ action on labels and
the shift $S$ on subshifts)
is Lemma \ref{towlem11}.  Let
$\{ t^i  \}$ be any sequence of in $IP(k)$  such
that
the sequence of smallest terms $\{ |j_r(t^i)|\}$
tends to
infinity and let $\{ \rr(t^i) \} $ be the corresponding sequence of length
vectors. Then for any sequence of labels
$\{\M^i\}$, the sequences $\{ S^{t^i}(x[ \M^i ]) \}$ and $\{ x[ \M^i
- \rr(t^i) ] \}$ are asymptotic in $\{ 0, 1 \}^{\Z}$.

We show that for a $FIN(\N)$-recurrent label the corresponding $x[\M]$ is an $S$-recurrent
point.  At the other extreme we have the {\em labels of finite type}.
For such a label $\M$, $e = x[\emptyset]$ is the only recurrent point in
$X(\M)$.
These labels are particularly amenable to our analysis, which leads to a complete
picture of the resulting subshift.
In fact for a label $\M$ of finite type
$$
X(\M) = \{S^k x[\NN] : k \in \Z, \ \NN \in \Theta(\M)\} =
\bigcup_{k \in \Z} S^k
x[\Theta(\M)].
$$

We show (see Corollary \ref{lattice1-1}) that, given $\M \in \LAB $,
the maps $\Phi(\cdot)$, which for a closed $S$-invariant subset $Y$ of $X(\M)$ is defined by
$$
Y \mapsto \Phi(Y) =  \{ \NN \in \Theta(\M)  : x[\NN] \in Y \},
$$
and $X(\cdot)$, which for a closed $FIN(\N)$-invariant $\Psi \subset \Theta(\M)$ is defined by
$$
\Psi \mapsto X(\Psi)  = {\text{the subshift generated by}} \ \{ x[\NN] : \NN \in \Psi \},
$$
have the following properties:
\begin{enumerate}
\item
The map $\Psi  \mapsto X(\Psi)$ is one-to-one from the collection of closed $FIN(\N)$ invariant
subsets of $\Theta(\M)$ into the collection of closed $S$-invariant subsets of $X(\M)$.
\item
If $\M$ is of finite type then this map is surjective, i.e.
every closed $S$-invariant subsets $Y$ of $X(\M)$
is of the form $X(\Psi)$ for some closed $FIN(\N)$ invariant $\Psi \subset \Theta(\M)$,
with
$$
\Phi(X(\Psi)) = \Psi \ \text{ and} \  Y = \Phi(X(\Psi)).
$$
\end{enumerate}

Thus, for a finite type label $\M$, the lattice of subsystems of the dynamical system
$(\Theta(\M), FIN(\N))$  fully describes the lattice of subsystems of the
dynamical system $(X(\M), S)$.

2
Two useful subcollections of the collection of finite type labels are the classes of the
{\em finitary labels} and of the {\em simple labels}. For each label $\M$
in either one of these special classes,
the corresponding subshift $X(\M)$ is a countable WAP system whose enveloping semigroup
structure is encoded in the structure of the label $\M$. This fact enables us to produce
WAP subshifts with various dynamical properties by tinkering with their labels.

The recurrent labels are far less transparent
and for these labels the image
$x[\Theta(\M)]$,
which in this case is a Cantor set, forms
only a meagre subset of the subshift $X(\M)$. Nonetheless it seems that this image forms a kind
of nucleus which encapsulates the dynamical properties of $X(\M)$.

The table of contents will now give the reader a rough notion of the
structure of our work. In the first section (section \ref{sec,WAP}) we deal with abstract
WAP systems and their enveloping semigroups and present
simple examples of WAP and non-WAP systems. 
Among other considerations it is shown that topologically
transitive WAP systems are coalescent and that
a general WAP system is E-coalescent. 
For an arbitrary
separable metric system the hierarchies $z_{NW}$ and $z_{LIM}$ of
non-wandering and $\alpha \cup \omega$ limiting procedures are studied.
Both lead, by transfinite induction, to the Birkhoff center of the
dynamical system. We call the ordinal at which the limiting
$\alpha \cup \omega$ transfinite sequence stabilizes, the {\em
height} of the system. 
In the final subsection we
describe some general
constructions like  the discrete suspension,
and the spin construction.

The space of labels is introduced and studied in section \ref{sec,labels}.
The associated subshifts are introduced and studied in section \ref{sec,exp}.
Section \ref{ss,finitarysimple} deals with WAP labels and their corresponding subshifts.
Finally, in
sections \ref{sec,dyn-prop} and \ref{sec,scrambled} these tools are applied to obtain many interesting and subtle
constructions of subshifts. Let us mention just a few.
On the
finite type
side we obtain examples of null as well as non-null WAP subshifts,
Example \ref{ex11}
(answering
a question of
Downarowicz);
WAP subshifts of arbitrary (countable) height,
Theorem \ref{towtheo28};
topologically transitive subshifts which are LE but not HAE,
Example \ref{ex,unc} and Remark \ref{re,LEnotHAE}
(these seem to be the first such examples);
and completely scrambled
WAP systems (although not subshifts) of arbitrary countable
height, Example \ref{ex11a}
(answering a question which is left open in Huang and Ye's work
\cite{HY}).
On the recurrent side we construct various examples of non-tame subshifts.
Of course many questions are left open, especially when labels which are not
of finite type are considered, and we present some of these throughout the work
at the relevant places.

We thank Benjy Weiss for several helpful suggestions which greatly improved this work.
We also thank the two anonymous referees for a careful reading of the manuscript and for
their useful comments and corrections.

\vspace{1cm}

\section{WAP systems}\label{sec,WAP}
\vspace{.5cm}


\subsection{Transitivity,  Recurrence and Enveloping Semigroups}\label{Tr-E}
$\qquad$

\vspace{.5cm}

The type of dynamical system of greatest interest for us is the
\emph{cascade}:
a pair  $(X,T)$ with $T$ is a homeomorphism on a nonempty compact Hausdorff space $X$, usually a metric space.

 We follow some of the notation of \cite{A} concerning relations on
 a space. In particular, we will use the
the \emph{orbit relation}
 $$
 O_T \ = \ \{ \ (x,T^n(x)) \ : \ x \in X, n \in \Z \ \}
 $$ and the associated
  limit relation:

  $\ R_T = \o T \cup \a T$,

 where
 $$ \o T \ = \ \{ \ (x,x') \ : \ x \in X, \ x' = \lim_{i \to \infty} T^{n_i} x \ \text{with}\
 n_i \nearrow \infty \ \},  $$
 and
  $$ \a T \ = \ \{ \ (x,x') \ : \ x \in X, \ x' = \lim_{i \to \infty}
 T^{-n_i} x \ \text{with}\  n_i  \nearrow \infty \ \}.  $$

$R_T$ is a pointwise closed relation (each $R_T(x)$ is closed) but
not usually a closed relation
(i.e. $R_T$ is usually not closed in $X \times X$).

We can regard the cascade $(X,T)$ as an action of the group $\Z$ on $X$ by $(t,x) \mapsto T^t(x)$.
We will need certain results for more general actions.

Let $\Gamma$ be a discrete, countable, commutative
 monoid ( = a semigroup with an identity element $1$). Let $\Gamma' = \Gamma \setminus \{1 \}$\label{gammaprime}.

A $\Gamma$-dynamical system
is a pair $(X, \Gamma)$\label{gammasys} where $X$ is a nonempty, compact Hausdorff space and $\Gamma$
acts on $X$ via a homomorphism of $\Gamma$ into the semigroup
$C(X,X)$ of continuous maps from $X$ to itself,
mapping
$1$
to the identity map, $id_X$. We will write $\Gamma \cdot x = \{ g x : g \in \Gamma \}$ for the $\Gamma$
orbit of a point $x \in X$.

We will call the action \emph{point transitive} when it admits a transitive point,
i.e. a point $x^*$ such that $\Gamma \cdot x^*$ is dense in $X$. $(X,\Gamma)$ is called \emph{minimal} when every
point of
$X$
is a transitive point.

A subset $X_0$ is \emph{invariant} if  $x \in X_0$ implies $\Gamma \cdot x \subset X_0$. The closure of an
invariant set is invariant since the action is continuous.

Given two sets $A, B \subset X$ we let $N(A,B) = \{ g \in \Gamma : g A \cap B \not= \emptyset \}$.
We call $x \in X$ a \emph{recurrent point} if $x \in \ol{\Gamma' \cdot x}$ or, equivalently, if
$N(\{ x \},U) \cap \Gamma' \not= \emptyset$ for every neighborhood $U$ of $x$.
An open set $A$ (or more generally a set with nonempty interior) is 
called \emph{non-wandering} if
$N(A,A) \cap \Gamma' \not= \emptyset$.
If $\Gamma$ is a group
then $A$ is called \emph{wandering} if $\{ g(A) : g \in \Gamma \}$ is a pairwise disjoint collection indexed
by $\Gamma$ and a set is either wandering or non-wandering.
A point $x$ is
\emph{non-wandering} if
 every neighborhood $U$ of $x$ is non-wandering.
It is easy to check that the set of non-wandering points is closed and contains the set of recurrent points.
We will call the system \emph{central}  if every point is non-wandering or, equivalently, if
$N(U,U) \cap \Gamma' \not= \emptyset$ for  for every nonempty open  $U \subset X$.

We will call the system \emph{transitive}  if $N(U,V) \not= \emptyset$ for every pair of nonempty open  $U, V \subset X$.

\begin{prop}\label{transprop} Let $(X, \Gamma)$ be a $\Gamma$ dynamical system with $X$ metrizable.
\begin{enumerate}

\item[(a)] If the system is central then the set of recurrent points is a dense $G_{\delta}$ subset of $X$.

\item[(b)] If the system is transitive then
the set
of transitive points is a dense $G_{\delta}$ subset of $X$ and so the system
is point transitive.

\item[(c)] If $\Gamma$ is a group and the system is point transitive, then it is transitive
and the set of transitive
points is invariant.

\end{enumerate}
\end{prop}

\proof These are just easy versions of the results for cascades with $\Gamma = \Z$ and so we will just sketch the
proofs. Let $\B$ be a countable base for $X$. For $A \subset X$ let $(\Gamma')^{-1}(A) = \bigcup_{g \in \Gamma'} \{ g^{-1}(A) \}$
with an analogous definition for $(\Gamma)^{-1}(A) = (\Gamma')^{-1}(A) \cup A$.

(a): Let $\A$ be a finite cover of $X$ by elements of $\B$.

$$Recur = \bigcap_{\A} \ \bigcup_{U \in \A}  U \cap (\Gamma')^{-1}(U)$$
is the set of recurrent points and it is the countable intersection of dense open sets when $(X,\Gamma)$ is central.

(b): The set of transitive points is $\bigcap_{U \in \B} (\Gamma)^{-1}(U)$ and this is the countable intersection of
dense open sets when $(X,\Gamma)$ is transitive.

(c): If
$x$
is a transitive point and $U, V \subset X$ are nonempty open sets then there exist $g_1, g_2$ such that
$g_1 x \in U, g_2 x \in V$. So $g_2 g_1^{-1} \in N(U,V)$ which is defined since  $\Gamma$ is a group.

For any monoid action $y \in \Gamma \cdot x$ implies $\Gamma \cdot y \subset \Gamma \cdot x$ with equality when $\Gamma$
is a group. Hence, if $y$ is a transitive point then $x$ is and the converse is true when $\Gamma$ is a group.

$\Box$ \vspace{.5cm}

If $X_0 \subset X$ is nonempty, closed and invariant,  then
$\Gamma$ acts on $X_0$ by restriction, and we call $(X_0,\Gamma)$ a \emph{subsystem} of $(X,\Gamma)$.
In particular, the orbit closure $ \ol{\Gamma \cdot x }$ is a closed, invariant set for any $x \in X$. By the
compactness and the usual Zorn's Lemma argument, any nonempty, closed and invariant subset contains a
nonempty, closed and invariant subset $M$ which is minimal with respect to inclusion. This is equivalent to
the condition that the subsystem $(M,\Gamma)$ is minimal in the previously mentioned sense, i.e. every
point $x \in M$ is a transitive point for $(M,\Gamma)$. Since the intersection of closed, invariant sets is closed and
invariant, it follows that distinct minimal subsets are disjoint.

A not necessarily closed subset $X_0$ is \emph{orbit-closed} if  $x \in X_0$ implies $\ol{\Gamma \cdot x} \subset X_0$.
An orbit-closed set is invariant and a closed, invariant set is orbit-closed.

In particular, a cascade $(X,T)$ is transitive if for every two
non-empty open sets $U, V $ in $X$ there is an $n \in \Z$ with $T^{-n}U \cap V \not=\emptyset$.
By Proposition \ref{transprop}, if $X$ is metrizable, then transitivity is equivalent to point
transitivity and implies that $X_{tr}$, the set of transitive points, is a dense $G_\delta$ subset of $X$.

A space is Polish if it is separable and admits a complete metric, e.g. a compact metric space.
Since a Polish space is separable the set of isolated points is finite or countably infinite.
A nonempty Polish space without isolated points is a union of
Cantor subsets, see, e.g. \cite[Proposition 2.3]{A-02}.
Since a  $G_{\delta}$ subset of a Polish space is Polish, any nonempty $G_{\delta}$ subset $A$ of
a Polish space either contains
points isolated in $A$ or contains a Cantor Set. In particular, in a countable Polish space the isolated points are dense.
Note that the set of rationals in $\R$ is not Polish.

The action of $\Gamma$ on $X$ is \emph{faithful} if $g_1 x = g_2 x$ for all $x \in X$ implies $g_1 = g_2$, i.e.
the homomorphism from $\Gamma$ to $C(X,X)$ is injective.
If $\breve{\Gamma}$ is the image of $\Gamma$ in
 $C(X,X)$ then $\breve{\Gamma}$ is a countable, abelian submonoid of $C(X,X)$ which acts faithfully
on $X$.
We call the action of $\Gamma$ on $X$ \emph{weakly faithful} if $gx = x$ for all $x \in X$ implies $g = 1$, i.e. the
only element of $\Gamma$ which acts as the identity is $1$.

We let $\Gamma_u$ denote the \emph{group of units}\label{gammau} in $\Gamma$, i.e. the set of elements $g \in \Gamma$ such that
there exists a -necessarily unique- inverse $\bar g$ such that $g \bar g = \bar g g = 1$.

\begin{prop}\label{counttransa} Let $(X, \Gamma)$ be a point transitive $\Gamma$ dynamical system. Assume that
$X$ is not perfect, i.e. it contains isolated points.

(a) Assume that the action is weakly faithful. If some isolated point $x^*$ of $X$ is a transitive point, then the
set of transitive points is $\Gamma_u x^*$ and these are all isolated points. None of the transitive points
is recurrent.

(b) If $\Gamma$ is a group then the set of transitive points is
the countable, dense open set of
isolated points. Moreover, this set consists of a single $\Gamma$ orbit.
\end{prop}

\proof (a) If $x^*$ is an isolated transitive point and $\bar x$ is another transitive point then there exists $g^* \in \Gamma$
such that $g^* \bar x = x^*$ since the orbit of $\bar x$ meets the clopen set $\{ x^* \}$.
$(g^*)^{-1}(x^*) = \{ y \in X : g^* y = x^* \}$ is a
clopen set containing $\bar x$. Because $x^*$ is a transitive point, there exist $\bar g \in \Gamma$ such that
$\bar g x^* \in (g^*)^{-1}(x^*)$ and so $g^* \bar g(x^*) = x^*$. Because $\Gamma$ is abelian, this implies
that $g^* \bar g$ acts as the identity on $\Gamma x^*$ which is dense in $X$. Because the action is continuous and
weakly faithful, it follows that $g^* \bar g = 1$. Commutativity implies that $g^*$ and $\bar g$ are inverses in $\Gamma$
and so lie in $\Gamma_u$. Thus, $\bar x \in \Gamma_u x^*$.

On the other hand, if $g \in \Gamma_u$ and $x = g x^*$ then since $g$ is invertible, $x^* \in \Gamma x$ and so
$\Gamma x^* $ is contained in - actually it equals - $\Gamma x$ and so the latter is dense. That is, $x \in X_{tr}$.

Since the elements of $\Gamma_u$ act as homeomorphisms and $x^*$ is isolated, it follows that every element of
$\Gamma_u x^*$ is isolated.

Finally, if $g x^* = x^*$ then as above $g = 1$. Hence, $\Gamma' x^*$ is disjoint from the clopen set $\{ x^* \}$.
Hence, $x^*$ is not recurrent.

(b) By replacing $\Gamma$ by $\breve{\Gamma}$ we may assume that the action is faithful. Observe that the $\Gamma$ orbit
is the same as the $\breve{\Gamma}$ orbit for every point of $X$.   If
$x^*$ is any isolated  point and $\bar x$ is a transitive point then there exists $g^* \in \Gamma$
such that $g^* \bar x = x^*$. Because $\Gamma$ is a group, it acts via homeomorphisms and so $\bar x$ is an isolated point
as well as a transitive point. Since $\Gamma$ is a group, $\Gamma_u = \Gamma$ and so applying (a) we obtain that
the orbit $\Gamma \bar x$ is the set of transitive points all of which are isolated. There are countably many because
$\Gamma$ is countable and the set is dense because $\bar x$ is a transitive point.

$\Box$ \vspace{.5cm}

\begin{cor}\label{counttrans} If $(X, \Gamma)$ is a transitive $\Gamma$ dynamical system
with $X$ countable and $\Gamma$ a group, then
$X$ is metrizable and
the set of transitive points is
the dense open set of
isolated points. Moreover, this set consists of a single $\Gamma$ orbit.\end{cor}

\proof  We first observe that a countable compact space has a countable base and so is metrizable.
For each pair $(x,y)$ of distinct points of $X$, choose $U_{(x,y)}$ an open set
containing $x$ and with $y \in X \setminus \ol{U_{(x,y)}}$. Since $X$ is countable $\{ U_{(x,y)} \}$ is a countable collection
of open sets. Using compactness it is easy to check that the finite intersections of such sets form a basis for $X$. Hence,
$X$ is compact and metrizable. Because it is countable,
no open subset contains a Cantor set and so
the isolated points are dense.
Furthermore, since the action is transitive, it is point transitive by Proposition \ref{transprop} (b).
The result then follows from Proposition \ref{counttransa} (b).

$\Box$ \vspace{.5cm}

\begin{ex}\label{oddaction} Let $\Gamma$ be the one-point compactification of $\N$ with the multiplication
$t \cdot s = \min(s,t)$. Thus, $\Gamma$ is a countable, compact, abelian topological monoid (i.e. the multiplication
is jointly continuous). The point $\infty$ at infinity is the identity element and is the one non-isolated point.
The action of $\Gamma$ on itself by multiplication is  faithful and $\infty$ is the unique transitive point. Thus, the
system is point transitive but not transitive.
Contrast this with Proposition \ref{transprop}.
The set $\{ 1 \}$ is the unique minimal subset of $\Gamma$. \end{ex}
\vspace{.5cm}

The system $(X,\Gamma)$ is called {\em weakly mixing}
when the diagonal action of $\Gamma$ on $X \times X$ is transitive.

\vspace{.5cm}

For a cascade $(X,T)$ we recall the definitions of $\epsilon$-chains and chain transitivity.
Given $\epsilon > 0$
an
\emph{$\epsilon$-chain from $x$ to $y$} is a finite sequence
$x = x_0, x_1, \dots ,x_n= y$ such that
$n > 0$ and $d(T(x_i),x_{i+1}) < \epsilon$ for $i = 0, \dots, n - 1$.
The system $(X,T)$ is \emph{chain transitive} if
for any pair of points $x, y \in X$
and any $\epsilon  > 0$ there is an $\epsilon$-chain going from $x$ to
$y$.
An \emph{asymptotic chain} is an infinite sequence $ \{ x_i : i \in \Z_+ \}$ or $ \{ x_i : i \in \Z \}$  such that
$\lim_{|i| \to \infty} \ d(T(x_i),x_{i+1}) \ = \ 0$. It is a \emph{dense asymptotic chain}
if for every $N \in \N$
$\{ x_i : i \geq N \}$ is dense in $X$ (and in the $\Z$ case $\{ x_i : i \leq -N \}$ is dense in $X$ as well).
If $(X,T)$ is chain transitive and $x \in X$ then
there exists a dense asymptotic chain $\{ x_i : i \in \Z \}$ with $x = x_0$.
\vspace{.5cm}

The following construction is due to Takens (see, e. g. \cite[Chapter 4, Exercise 29]{A}
).

\begin{ex}\label{chainconstruct} If $(X,T)$ is a chain transitive metric system and
$\{ x^i : i \in \Z \}$ is a dense asymptotic chain then let
\begin{equation}\label{con01}
z^i \quad = \quad \begin{cases} (x^i,(2i + 1)^{-1}) \qquad \mbox{for} \ i \geq 0, \\
(x^i,(2|i|)^{-1}) \qquad \mbox{for} \ i < 0. \end{cases}
\end{equation}
Let $X^* = X \times \{ 0 \} \cup \{ z^i : i \in \Z \}, x^* = z^0$.
Extend $T = T \times id_{0}$ on $X \times \{ 0 \}$ identified with $X$, by
$T(z^i) = z^{i+1}$ for $i \in \Z$. Then $(X^*,T)$ is a topologically transitive metrizable
system with transitive point $x^*$ and $X^* = X \cup O_T(x^*)$ and $X = \o T(x^*)$.
\end{ex}
\vspace{.5cm}

Now we return to $\Gamma$ actions.

The {\em enveloping semigroup\/} $E=E(X,\Gamma)$ of the dynamical system
$(X,\Gamma)$ is defined as the closure in
$X^X$ (with its compact, usually non-metrizable,
pointwise convergence topology) of the image of $\Gamma$ in $C(X,X)$
considered as a subset of $X^X$. Since $\Gamma$ is a monoid, $E(X,\Gamma)$ is a monoid with
the identity $id_X$.

It follows directly from the definitions that,
under composition of maps, $E$ forms a compact
semigroup in which the operations
\begin{equation*}\label{df,E}
p \mapsto pq \qquad {\text{\rm and}}\qquad p\mapsto g p 
\end{equation*}
for $p,q\in E,\ g \in \Gamma$, are continuous.
Since we have assumed that $\Gamma$ is commutative, it follows from continuity of the $\Gamma$ action that each
$g \in \Gamma$ commutes with every $p \in E(X,\Gamma)$.

Notice that
$\Gamma$ acts on $E$ by  multiplication, so that $(E,\Gamma)$ is a
$\Gamma$-system (though usually non-metrizable).
%

The elements of $E$ may behave very
badly as maps of $X$ into itself; usually they are
not even Borel measurable. However our main
interest in $E$ lies in its algebraic structure
and its dynamical significance. A key lemma
in the study of this algebraic structure is the
following (see Lemma
\ref{applem01}
in Appendix
\ref{appendix-ellis}):

\begin{lem}[Ellis-Numakura]\label{idemp} Let $L$ be a compact
Hausdorff semigroup in which all maps $p\mapsto pq$ are
continuous.  Then $L$ contains an {\em idempotent \/};
i.e.,\ an element $v$ with $v^2=v$.
\end{lem}
\vspace{.5cm}

Given two $\Gamma$ dynamical systems, say $(X,\Gamma)$ and $(Y,\Gamma)$,
a continuous  map $\pi : X \to Y$ is a {\em homomorphism} or an {\em action map}
if it intertwines the $\Gamma$ actions, i.e. $\gamma \pi(x) = \pi(\gamma x)$ for every
$x \in X$ and $\gamma \in \Gamma$. If $\pi$ is injective it expresses $X$ as a \emph{subsystem} of $Y$.
If $\pi$ is surjective it expresses $Y$ as a \emph{factor} of $X$.

If $X_0 \subset X$ is closed and invariant,
 then
 the inclusion of the subsystem $(X_0,\Gamma)$ into $X$ is an injective action map. Furthermore, $X_0$ is
 invariant with respect to the $E(X,\Gamma)$ action.

An injective action map induces, by restriction, a surjective, continuous monoid homomorphism
(and an action map)
$\pi^* : E(Y,\Gamma) \to E(X,\Gamma)$.
A surjective action map $\pi : X \to Y$ induces a surjective, continuous monoid homomorphism
(and an action map)
$\pi_* : E(X,\Gamma) \to E(Y,\Gamma)$.
Observe that if $\pi(x_1) = \pi(x_2)$ then
for any $p \in E(X,\Gamma)$, $\pi(p x_1) = p \pi(x_1) =  p \pi(x_2) =  \pi(p x_2)$.

The map $E(X,\Gamma) \times X \to X$ given by $(p,x) \mapsto px$ is an action although the map is usually not jointly continuous.
However, if $x \in X$ then the evaluation map $ev_x : E(X,\Gamma) \to X$ given by $p \mapsto px$ is continuous and is an
action map. A point $x^* \in X$ is a transitive point iff $ev_{x^*}$ is surjective and so yields $X$ as a factor of
$E$.

For more details see e.g. \cite[Chapter 1, Section 4]{Gl-03} and \cite{AAG}.

\begin{lem}\label{lem01} If $p  \in E(X,\Gamma)$ is continuous at $x \in X$, then for any $q \in E(X,\Gamma)$ and $y \in X$ if
$ qy = x$ then $pqy = qpy$. \end{lem}

\proof If $\{ g^{i} \}$ is a net in $\Gamma$ which converges to $q$
pointwise then $\{ g^iy \} \to qy = x$ and so by continuity of $p$ at $x$,
$pg^iy \to pqy$.  But $pg^iy = g^ipy \to qpy$ since the elements of $\Gamma$ commute with $p$.

$\Box$ \vspace{.5cm}

\begin{prop} \label{prop02} If $x^*$ is a transitive
point for $(X,\Gamma)$ and $p \in E(X,\Gamma)$
then the following are equivalent.
\begin{itemize}
\item[(i)] $p$ is continuous on $X$.
\item[(ii)]  For all $q \in E(X,\Gamma) \ pq = qp$ on $X$.
\item[(iii)]  For all $q \in E(X,\Gamma) \ pqx^* = qpx^*$.
\end{itemize}
\end{prop}

\proof (i) $\Rightarrow$ (ii): This follows from Lemma \ref{lem01}.

(ii) $\Rightarrow$ (iii): Obvious.

(iii) $\Rightarrow$ (i): Suppose the net $\{ x^i \}$ in $X$ converges to $x$.
To show that then $p x_i \to p x$, it suffices to show that every
convergent subnet has limit $p x$. So we can assume that $\lim p
x_i$ exists. Since $x^*$ is a transitive point, $ev_{x^*}$ is surjective and
there are $q^i \in E(X,T)$ with $x^i = q^i x^*$. Let
$\{ q^{i'} \}$ be a convergent
subnet in $E(X,\Gamma)$ with limit $q$. Then, by continuity of $ev_{x^*}$, $qx^* = x$. By assumption (iii)
$\lim px^i = \lim p x^{i'}
= \lim p q^{i'} x^* = \lim q^{i'} p x^* = qp x^* = pq x^* = px$.

$\Box$ \vspace{.5cm}


\subsection{The enveloping semigroup of WAP systems}\label{subWAP}
$\qquad$

\vspace{.5cm}

A dynamical system
$(X,\Gamma)$ is called {\it WAP (weakly almost periodic)}
when the elements of $E(X,\Gamma)$ are all continuous functions on $X$.
Clearly for a WAP system $(X, \Gamma)$ the multiplication on its
enveloping semigroup is continuous in each variable separately; i.e.
$E(X, \Gamma)$ is an abelian {\it semi-topological semigroup}.
The converse however is not necessarily true;
see Proposition \ref{prop04}.

N.B. It still need not be true that the
action $E(X,\Gamma) \times X \to X$ is jointly continuous, even in the case of a WAP cascade.

\begin{cor} \label{cor02a}If $(X,\Gamma)$ is a point transitive system with
a transitive point $x^*$ then the following are equivalent.
\begin{itemize}
\item[(i)] $(X,T)$ is WAP
\item[(ii)] $E(X,T)$ is abelian.
\item[(iii)] For every $p, q \in E(X,T) \ pqx^* = qpx^*$.
\end{itemize}
When these conditions hold  $ev_{x^*} :(E(X,\Gamma),\Gamma) \to (X,\Gamma)$ is an isomorphism and
there is a unique minimal subset of $X$.
\end{cor}

\proof The equivalence of (i), (ii) and (iii) follows from Proposition \ref{prop02}

If $(X,\Gamma)$ is WAP and $px^* = qx^*$ then $p = q$ on $\Gamma \cdot x^*$ because
the elements of $\Gamma$ commute with $p$ and $q$. Since $\Gamma \cdot x^*$ is dense, $p = q$ by continuity.
That is, $ev_{x^*}$ is
injective and so is an isomorphism.

If $x_i \in M_i$ for minimal sets $M_1, M_2$ there exist $p_1, p_2
\in E(X,\Gamma) $ s.t. $p_i(x^*) = x_i$ and so $p_2(x_1) = p_2(p_1(x^*))
\in M_1$ while $p_1(x_2) = p_1(p_2(x^*)) \in M_2$. Since $E(X,T)$ is
commutative, $M_1 \cap M_2 \not= \emptyset$ and so $M_1 = M_2$.

$\Box$ \vspace{.5cm}

\begin{prop}\label{prop03} If $ev_{ x^*} : E(X,\Gamma) \to X$ is an
homeomorphism (e.g. if $(X,\Gamma)$ is WAP with transitive point $x^*$)  and $\{ q^i \}$ is a net
in $E(X,\Gamma)$ such that $\{q^i x^*\}$  converge to a point $x \in X$
then $\{ q^i (z)\}$
 converges in $ X$ for every $z \in X$. In fact, $\{ q^i  \}$ converges pointwise to the
unique $p \in E(X,\Gamma)$ such that $p(x^*) = x$.
Thus, if $p \in E(X,T)$ and  $\{ q^i (x^*) \}$ converges to $p(x^*)$ in
$X$ then $\{ q_i  \} \to p$ in $E(X,\Gamma)$.
 \end{prop}

\proof Obvious by inverting the homeomorphism $ev_{x^*}$.

$\Box$ \vspace{.5cm}

A dynamical system
$(X,\Gamma)$ is called {\it AP (almost periodic)}
or {\it equicontinuous} if  $\Gamma$ acts equicontinuously on $X$.
Clearly an AP system is WAP and its enveloping semigroup
is a (commutative) compact topological group. Note
that there are (non-transitive) systems $(X, \Gamma)$
with $E(X, \Gamma)$ a (commutative) compact topological group
which are not equicontinuous. This is the case e.g. in
Example \ref{ex2}(a) below
(see also \cite[Example 6.5]{GMU}).
However, for a transitive system,
$E(X, \Gamma)$ is a commutative compact group iff
$(X,\Gamma)$ is equicontinuous.


\vspace{.5cm}

Even in the cascade case, $ev_{x^*}$ a homeomorphism need not imply WAP.

\begin{ex}\label{ex1}
Let $X$ be a compact, connected metric space and $T= id_X$, the identity map.  So
$E(X,T) = \{ id_{X} \}$.
Let $\{ x_i : i \in \Z \}$ be a  sequence of distinct points  in $X$ such that
$\lim_{|i| \to \infty} d(x_i,x_{i+1}) = 0$
and
so that the positive and negative tails are dense in $X$. Thus, $\{ x_i \}$ is a dense asymptotic chain for
$(X,T)$.  Following Example \ref{chainconstruct} we embed $(X,T)$ as a subsystem of $(X^*,T)$ with $X^* = X \cup O_T(x^*)$ and
$\{ T^ix^* \}$ asymptotic to $\{ x_i \}$.

 Now assume that
$y \in X$ and $T^{n_k}x^*$ converges to $y$ ( and so with $|n_k| \to \infty$). Then
$ T^{n_k + N} x^*$ converges to $T^N y = y$. Furthermore, for any $z \in X_0$,
$T^{n_k}z = z$ converges to $z$. Thus, $T^{n_k}$ converges pointwise to the function which is the identity
on $X$ and which is constantly $y$ on the orbit of $x^*$. In particular,
$ev_{x^*} : (E( X^*,T),T) \to (X^*,T)$
is an isomorphism.  On the other hand,
$E(X^*,T)$
is not commutative and none of the elements of
$E(X^*,T)$ are continuous
except for the iterates $T^n$.
\end{ex}

$\Box$  \vspace{.5cm}

 \begin{ex}\label{ex2}An abelian
 enveloping semigroup
 does not imply WAP in general.

(a)
Let the circle be $\R/\Z$. Let $X = \R/\Z \times \Z^*$ where $\Z^*$ is the
one-point compactification of $Z$. Define $T$ to be the identity on $\R/\Z \times \{ \infty \}$ and by
$(t,n) \mapsto (t + 3^{-(|n|+ 1)},n)$. On each circle the map is just a rotation,
hence equicontinuous, and so is WAP.
Therefore the enveloping semigroup is commutative.
Consider the sequence $\{ T^{\Sigma_{i=0}^{k} \ 3^i } \}$. On $\R/\Z \times \{ n \}$ this is
eventually constant at the rotation $ t \mapsto t + \Sigma_{i=0}^{|n|} \ 3^{ -(|n| - i + 1)} $. As $|n| \to \infty$
this approaches the rotation $t \mapsto t + \frac{1}{2}$.  But on the circle $\R/\Z \times \{ \infty \}$
every element of the enveloping semigroup restricts to the identity.
\vspace{.25cm}

(b)
A countable example is the \emph{spin} $(Z,T)$ of the identity on $\Z^*$. $Z$ is a subset of $ \Z^* \times \Z^*$
\begin{equation}
\begin{split}
Z \quad = \quad (\Z^* \times \{ \infty \}) \ \cup \ \bigcup_{n \in \Z} \{ [-|n|,+|n| \ ] \times \{ n \} \},\hspace{1cm} \\
T(x) \quad = \quad \begin{cases}  \ (\infty, \infty) \qquad  \ \ \  \mbox{for} \quad x \ = \ (\infty, \infty) \\
(t + 1, \infty) \qquad  \  \mbox{for} \quad x \ = \ (t, \infty) \\
(t + 1, n)    \qquad \mbox{for} \quad x \ = \ (t, n) \ \mbox{with} \ -n \leq t < n, \\
(-n, n)  \qquad  \ \  \ \mbox{for} \quad x \ = \ (n,n). \end{cases}
\end{split}
\end{equation}
\end{ex}

$\Box$ \vspace{.5cm}

The following proposition describes the equivalences for what might be called
{\it local WAP}.

\begin{prop}\label{prop04} For a system $(X,\Gamma)$ the following are equivalent.
\begin{itemize}
\item[(i)] Multiplication for the enveloping semigroup is continuous in each variable;
i.e. $E(X, \Gamma)$ is a semi-topological semigroup.

\item[(ii)] Every element of the enveloping semigroup has a continuous restriction on the orbit closure of each element of
$X$.

\item[(iii)] The enveloping semigroup is commutative.

\item[(iv)] Each orbit closure in $X$ is a WAP system.

\end{itemize}
\end{prop}

\proof  Since each orbit closure is invariant for the enveloping semigroup and since the topology of the latter is
pointwise convergence, each of these conditions holds
for
$(X,\Gamma)$ iff it holds for the restriction to each orbit closure.
This restricts to the topologically transitive case for which (iii) implies (ii) by Proposition \ref{prop02}.
Because of pointwise convergence, (ii) implies (i) is obvious.  If $p,q$ are in the enveloping semigroup and $g^j$ is a
net converging to $q$ then $p g^{j} = g^{j}p$ and
the separate
continuity of multiplication at p implies $pq = qp$.

$\Box$ \vspace{.5cm}

In the case of a cascade $(X,T)$, we
write
$(E(X,T),
\allowbreak
A(X,T))$
for the
enveloping semigroup and the ideal which is the
 {\it adherence semigroup} (= the limit points
of $\{ T^n \}$ as $ n \to \pm \infty$) (see Definition \ref{adherencedef} in Appendix
\ref{appendix-StoneCech}).
Let $T_*$ on $E(X,T)$ be the homeomorphism given by $T_*(p) =  Tp = pT$. Thus,
$y \in R_T(x)$ iff $y = px$ for some $p \in A(X,T)$ and $A(X,T) = R_{T_*}(id_X)$.
Note that $(E(X,T),T_*)$ is a - usually
non-metrizable - cascade.
\vspace{.5cm}

\subsection{The Birkhoff Center, CP and CT systems}\label{subsec,rec}
$\qquad$
 \vspace{.5cm}

We now restrict our attention to cascades.

A point $x \in X$ was defined above to be
{\em recurrent} when it is in the closure of $\{ T^n(x) : n \not= 0 \}$. It follows that $x$ is
recurrent iff
 $x \in R_T(x)$ and so iff there exists $p \in A(X,T)$ such that $px = x$.
That is, $Iso_x = \{ p \in E(X,T) : px = x \}$ intersects $A(X,T)$ in a nonempty, closed subsemigroup.
Hence, there exists an idempotent $u \in A(X,T)$ such that $ux = x$.
On the other hand, $x$ is non-recurrent iff
the orbit $\mathcal{O}_T(x)$ is disjoint from
the limit set $R_T(x)$. The closure of the set of recurrent points
is the \emph{Birkhoff Center}. We will denote by $Cent_T$ the Birkhoff center for $(X,T)$.

A  open set $A \subset X$ was defined to be \emph{wandering} if $\{ T^n(A) : n \in \Z \}$ is a
bi-infinite sequence of pairwise disjoint sets.
Otherwise, there exists
$n > 0$ such that $A \cap T^n(A) \not= \emptyset$ and $A$ is called \emph{non-wandering}.  A point $x$ is
\emph{non-wandering} if every open neighborhood of $x$ is non-wandering.  A recurrent point is clearly non-wandering.
Recall that $(X,T)$ is called \emph{central} if every point is non-wandering.
In that case, with $X$ metrizable, the set of recurrent points
forms
a dense
$G_{\delta}$ subset, by  Proposition \ref{transprop}. Even in the non-metrizable case, Ellis proved that the set of recurrent
points is dense if $(X,T)$ is central, see, e.g.
\cite[ Proposition 5.18]{A-97}.


Let $(X,T)$ be a cascade with $X$ metrizable and hence separable.
Define $z_{CAN}(X)$ to be the complement of the set of isolated points in $X$. Let $z_{NW}(X)$ be the complement of the union
of all wandering open sets, i.e. the set of non-wandering points.
Note that if a point is isolated and non-periodic then it is wandering.
Thus, if there are no isolated periodic points, then $z_{NW}(X) \subset z_{CAN}(X)$.
Since $T$ is a homeomorphism,
the set of isolated points is invariant.  Hence,  both $z_{NW}(X)$ and $ z_{CAN}(X)$ are closed invariant sets.

Let
$z_{LIM}(X) = \ol{R_T(X)}$.
If an open set $U$ meets $\ol{R_T(X)}$  then there exist $x, y \in X$ such that
$y \in U \cap R_T(x)$ and so for infinitely many $n \in \Z, \ T^n(x) \in U$. In particular, $U$ is non-wandering.
 It follows that
 $z_{LIM}(X) \subset z_{NW}(X)$.

 For each of these operators $z$
we define the descending transfinite sequence of closed sets by
\begin{equation}\label{z}
z_0(X) \ = \ X, \quad z_{\a + 1}(X) \ = \ z(z_{\a}(X)), \quad
z_{\b}(X) \ = \ \bigcap_{\a < \b} z_{\a}(X),
\end{equation}
for $\b$ a limit ordinal.

We say that the sequence stabilizes at $\b$ when $z_{\b}(X) = z_{\b + 1}(X)$ in which case it
is constant from then on. The first $\b$ at which stabilization occurs for the $CAN/ NW/ LIM$ sequence is called
the $CAN/ NW/ LIM$
{\em level}. Since $X$ is a separable metric space, each level is a countable ordinal
(because for any limit ordinal $\b $ less than or equal to the level
$\{ X \setminus z_{\a}(X) : \a < \b \}$ is a strictly increasing open cover of the Lindel\"{o}f space $X \setminus z_{\b}(X)$,
and so there is a sequence of ordinals $\{ \a_n < \b \}$ with limit $\b$. Hence, $\b$ is smaller than the first uncountable
ordinal
and so the level is a countable ordinal).

We let
$z_{\infty}(X) = z_{\b}(X)$ when the sequence stabilizes at $\b$.
Clearly $z_{LIM, \a}(X) \allowbreak \subset z_{NW, \a}(X)$ for all
$\a$.
Recall that when $(X,T)$ is non-wandering, i.e. $z_{NW}(X) \allowbreak = X$, then the recurrent points are dense.
Since all recurrent points are contained in $z_{LIM,\infty}(X)$ it follows that $ z_{LIM,\infty}(X) = z_{NW,\infty}(X)$ is
the closure of the set of recurrent points, i.e. the Birkhoff Center. While the two transfinite sequences are eventually
equal, they can stabilize at different levels.

At each stage of the $z_{CAN}$ sequence, the set of isolated points is countable. Hence,  $X \setminus z_{CAN,\infty}(X)$ is
a countable union of countable sets and so is itself countable. On the other hand, $z_{CAN,\infty}(X)$ has no
isolated points. So  every non-empty subset which is open in $z_{CAN,\infty}$ contains a Cantor
set. Thus, $X \allowbreak \setminus z_{CAN,\infty}$ is the maximum countable open subset of $X$ and
$$z_{CAN,\infty} = \{ x \in X : \ {\text
{every neighborhood of $x$ is uncountable}}\}.$$

In particular, if $X$ is compact and countable then $z_{CAN,\infty} = \emptyset$.
 Since the intersection of the decreasing family of
nonempty closed sets has a nonempty intersection, $\b_{CAN}(X)$ is not a limit ordinal if $X$ is compact and countable.

\begin{df}\label{df,height}
We will call the ordinal $\b_{LIM}(X)$ at which the $z_{LIM}$ sequence stabilizes, the \emph{height} of $(X,T)$.
\end{df}

Call $(X,T)$ \emph{semi-trivial} (hereafter ST) if $R_T = X \times \{ e \}$ for a point, a fixed point,
$e \in X$. That is, for every $x \in X, \ \ R_T(x) = \{ e \}$.  Call $(X,T)$ \emph{center periodic} (hereafter CP)
if the only recurrent points are periodic. Call $(X,T)$ \emph{center trivial}
(hereafter CT) if there is a unique recurrent point $e$, necessarily a fixed point, and so the Birkhoff center
is $\{ e \}$. Call $(X,T)$ \emph{min center trivial} (hereafter minCT) if there is a unique minimal point $e$,
necessarily a fixed point.

Clearly, ST implies CT and CT implies CP and minCT. A nontrivial system is ST iff it is a CT system of height $1$.
For a minCT system we will denote by $u$ the retraction to the fixed point $e$. In general, if $p \in E(X,T)$ is minimal, then
$p x$ is a minimal point of $X$ for all $x \in X$. So if $(X,T)$ is minCT, then $u$ is the
only minimal element of $E(X,T)$. Thus, $(E(X,T),T_*)$ is minCT.

For any cascade $(X,T)$, if a point $x \in X$ is periodic, then its orbit closure is finite.
If $x$ is non-recurrent then its orbit is disjoint
from $R_T(x)$.  In exactly these cases, $x$ is isolated in its orbit closure. In particular,
an isolated point is recurrent iff it is periodic. It follows that a cascade is CP iff every
point is isolated in its orbit closure.

If $(X,T)$ is central,
then the recurrent points are dense and so every isolated point is recurrent and so is periodic.

 If a metrizable $(X,T)$ is CP and
central
then the recurrent points - which are all periodic - form a dense
$G_{\delta}$ and so by the Baire Category Theorem, $\{ x : T^n(x) = x \}$ has nonempty interior for some positive $n$.
In fact, the union of such interiors
is dense in $X$. If there are only countably many periodic points then this open dense set is countable and Polish and so
the isolated points are dense in $X$. Also, every isolated point is  periodic because $X$ is central

The identity on any compact space defines a CP system and the finite product of CP's is CP
(not the infinite product since the product of periodic orbits can contain
an adding machine). Any subsystem or factor of a CP system is CP
(since any recurrent point in the factor lifts to some recurrent point in
the top). Inverse limit does not work. Again,
an adding machine is the inverse limit of periodic orbits.

\begin{remark}\label{cp,weakmix} A nontrivial CP system $(X,T)$ can never be weak mixing, i.e. $(X \times X, T \times T)$ is
never topologically transitive. If $x^*$ is a transitive point for a CP system $(X,T)$ then it is isolated in $X$. If $x^*$ is an
isolated, transitive point for a nontrivial  system $(X,T)$ then $T(x^*) \not= x^*$ and so $U = \{ x^* \}$ and $V = \{ T(x^*) \}$
are nonempty open subsets of $X$, but $N(U \times U, U \times V) = \emptyset$. \end{remark}

$\Box$ \vspace{.5cm}

The CT condition  is closed under arbitrary products and subsystems.
In particular, the enveloping semigroup of a CT system is CT.
The retraction $u$ to the fixed point $e \in X$ is the unique fixed point
in $E(X,T)$ (the system is minCT).
Also, it is the unique idempotent in
$A(X,T)$.
If $\pi : (X,T) \to (Y,S)$ is an action map and $X$ is CT then $Y$ is.
In general, $(X,T)$ is CT iff $(Y,S)$ is CT and $\pi^{-1}(e)$ is a CT subsystem of $X$.

Mapping $(X,T)$ to the factor system on $X/Cent_T$,
where $Cent_T$ is collapsed to a single point,
defines a functor from
compact systems to CT systems. An action map $X \to Y$ with $Y$ CT factors through the projection from $X$ to $X/Cent_T$
and so the functor is adjoint to the inclusion functor.

If $(X,T)$ is a countable CT system, then $e$ is not an isolated point in any invariant closed subset of $X$ except $\{ e \}$ itself.
Thus, if the Cantor sequence stabilizes at $\b + 1$ then $z_{CAN,\b} = \{ e \}$ and conversely.  In that case, for any $\a \leq \b$
$z_{LIM,\a} \subset z_{NW,\a} \subset z_{CAN,\a}$. That is, up to $\a = \b$ the isolated points are all
wandering.

A CT system $(X,T)$ has height $0$ iff $X = \{e\}$, i.e. the system is trivial, and has height  $1$ iff it is
ST.

\begin{prop}\label{prop05} (a) If $(X,T)$ is an ST system then it is WAP.

(b) If $(X,T)$ is a CT system with height at most 2 then
$E(X,T)$ is commutative.
If, in addition, $(X,T)$ is topologically transitive then
it is WAP. \end{prop}

\proof (a) If $(X,T)$ is ST, then  $E(X,T) = \{ T^n : n \in \Z \} \cup \{ u \}$  where $u$ is the retraction onto $e$.
So every element of $E(X,T)$ is continuous.

(b) If $p, q \in A(X,T)$ then $pq = u = qp$. Hence, the semigroup is abelian.
So if $(X,T)$ is topologically transitive, then it is WAP by Proposition \ref{prop02}.

$\Box$ \vspace{.5cm}

Let $S$ denote the shift homeomorphism on
$\{ 0, 1 \}^{\Z}$, i.e. $S(x)_i = x_{i+1}$. We will be focusing much of our attention on nonempty subsystems of the full shift
$(\{ 0, 1 \}^{\Z},S)$.  These are called \emph{subshifts}.

\begin{ex}
In his work \cite{Sh} Shapovalov shows that within the class of countable subshifts
one can find, for any countable ordinal $\a$, a subshift $X_\a \subset \{0,1\}^\Z$
whose Birkhoff degree, i.e. its NW
level, is $\a + 1$.
Now it is easy to verify that all of these subshifts $X_\a$ constructed by Shapovalov are in fact
ST and therefore also WAP. One can make them topologically transitive by attaching
a single orbit. Thus we conclude that
{\em the class of WAP, topologically transitive subshifts is rich enough to present
every countable Birkhoff degree}.
Note however that being semi-trivial Shapovalov's original examples  all are of height 2 and they become of height 3 when
an orbit is attached to make them topologically transitive.
As we will show later (Theorem \ref{towtheo28}) the class of WAP, topologically transitive
subshifts is also rich enough to present every countable height.
\end{ex}


%
%

\begin{ex}\label{ex3}  Various non-WAP examples.

Let $e = \bar 0, x[0] = 0^{\infty} \dot{1} 0^{\infty}$, i.e. $x[0]_0 = 1$ and $x[0]_i = 0$ otherwise.
Here $\bar \alpha$ for a finite binary word $\alpha$
is the periodic orbit $\cdots\alpha\alpha\alpha\cdots \in \{ 0, 1 \}^{\Z}$.
Thus
$\bar 0$ is the constant sequence $(\dots,0,0,0,\dots)$.
Let $X(0)$ be the ST subshift generated by $x[0]$. Thus, $(X(0),S)$ is
isomorphic to translation on the one point compactification
$\Z^*$ of $\Z$.

(a) For $k = 1,2,\dots$, let $b^k_j = 1 $ for $j = 10^{nk}, n \in \Z$
and $= 0 $ otherwise. Let $(X,S)$ be the generated subshift.
$R_S(X) = X(0)$ and so $(X,S)$ has height 2.
We have
$b^k \to x[0]$ as $k \to \infty$. The sequence $S^{10^{n!}} \to p$ in $A(X,T)$ with
$p(b^k) = x[0]$ for all $k$ and $p(x[0]) = e$. So $p$ is not continuous at $x[0]$, despite the fact that all of the points of
$X \setminus X(0)$ are isolated. That is, the assumption of topological transitivity in Proposition \ref{prop05} (b) is necessary.
\vspace{.25cm}

(b) A topologically transitive system of height 1
with minimal set not a fixed point need not be WAP. Let $c$ be given by $c_i = 1$ for
$i =  2n, -1 - 2n$ for $n \in \N$ and $= 0$ otherwise, i.e. $c = (01)^{\infty}(10)^{\infty}$. The orbit closure of
$c$ consists of $O_S(c)$ together with the periodic orbit
$\{ \ol{01}, \ol{10} \}$.
$S^{-2k}(c) \to \ol{01}$,
$S^{2k}(c) \to \ol{10}$. $S^{-2k} \to p$ and $S^{2k} \to q$ both $p, q$ are
identity on the periodic orbit. Hence, $q = pq \not= qp = p$
on $c$.
\vspace{.25cm}

%
%

(c) For $(Y,S)$  any compact metric system, let $(X,T)$ be the one point compactification of
$(Y \times \Z, S \times t)$ with $t$ the translation on $\Z$. This is an ST system.
If $S = id_Y$ it is easy to build a countable sequence of
periodic orbits with limit set $(X,T)$.  The expanded system is CP with an uncountable center although there are only countably
many periodic orbits.
\vspace{.25cm}

(d) Call $x \in \{ 0, 1 \}^{\Z}$
\emph{selective}
if for any
$n \geq 0$ the word $1 0^n 1$ occurs at most once in $x$.
Let $X$ be the set of all
selective
$x$. Clearly, if $x$ is
selective
then $R_S(x) \subset X(0)$. Note that if
$A \subset \Z$
is
such that all the nonzero differences $a_i - a_j$ are distinct, then
the characteristic function
$\chi(A)$
of $A$ in $\{ 0, 1,\}^{\Z}$, is a
selective
element. $(X,S)$ is an uncountable
CT subshift with height 2. Hence every orbit closure
of a point in $X$ is WAP. It is not hard to see that $X$ itself is not WAP.
Let $x \in X$ such that $x_{n_k} = 1$ with $n_k \to \infty$. Let $z^N_i$ agree with $x_i$
for $|i| \leq N$ and equal $0$ otherwise.  Then $\{ z^N \}$ is a sequence in $X$ which
converges to $x$. If a subnet of $T^{n_k}$ converges to $p \in E(X,S)$ then $p(x) = x[0]$ but
$p(z^N) = \bar 0$ for all $N$.  Hence, $p$ is not continuous.
\vspace{.25cm}

(f) Let $(Y,S)$ be any CP subshift. Let $\{ w_i \}$ count the finite words in $Y$. Then
$ Y \cup \bigcup \{ \ol{w_i} \}$ is a CP subshift with dense periodic points.\end{ex}

$\Box$ \vspace{.5cm}

If $(X,T)$ is a metrizable CT then  it is
chain transitive, because the fixed point $e$ lies in $\omega T(x) \cap \alpha T(x)$ for every $x \in X$. So if $x_1, x_2 \in X$,
there is an $\epsilon$ chain from $x_1$ to $e$ and from $e$ to $x_2$ for every $\epsilon > 0$.
Hence, we can attach a single orbit of isolated points and obtain a metrizable CT, $(X^*,T)$ which is
topologically transitive so that
of $X^* $ the disjoint union of $X$ and the dense orbit of isolated points $O_T(x^*)$. See Example \ref{chainconstruct}.

\begin{ex}\label{ex4} It may happen that we cannot choose the extension
so that $(X^*,T)$ is WAP.

Let $(X,T)$ be a CT WAP with fixed point $e$ and which is not semi-trivial.
That is, there exists $p$ in the
$A(X,T)$ with $p \not= u$ and so $p(X) \setminus \{e \}$ is nonempty. Let $\bar X = X_1 \vee X_2$, two copies of
$X$ with the fixed points identified.  For any map $g$ on $X$ which fixes $e$, let $\bar g$ on $\bar X$ be copies of
$g$ on each term. The system $(\bar X, \bar T)$ is clearly a CT WAP and $p \mapsto \bar p$ is an isomorphism from
 $E(X,T)$ onto  $E(\bar X, \bar T)$. Notice the $\bar p(\bar X) \setminus X_i$
is nonempty for $i = 1,2$.

Now let $(X^*,  T^*)$ contain
$(\bar X, \bar T)$ and with $X^* \setminus \bar X$ consisting of a single dense orbit $O_{T^*}(x^*)$.

Let $q$ be an element of the
enveloping semigroup of $E(X^*,T^*)$ with $q(x^*) \in X_1$. Then $q$ maps the whole orbit of $x^*$ into
$X_1$ and if $q$ is continuous then $q(X^*) \subset X_1$. Thus, every continuous element of the enveloping
semigroup $E(X^*, T^*)$ maps all of $X^*$ either into $X_1$ or into $X_2$. Every element of the
enveloping semigroup of $(\bar X, \bar T)$ extends to some element of the enveloping semigroup of $(X^*, T^*)$.
Thus, if $ p^*$ extends $\bar p$ it cannot be continuous and so $(X^*, T^*)$ is not WAP.\end{ex}

$\Box$ \vspace{.5cm}

We will say that $A(X,T)$ \emph{distinguishes points} when $p(x_1) = p(x_2)$ for all $p \in A(X,T)$ implies
$x_1 = x_2$.  It suffices that some $p \in A(X,T)$ be injective. If $X$ has any non-trivial, but semi-trivial subspace
then $A(X,T)$ does not distinguish points.

Let $T_*$ be composition with $T$ on $E(X,T)$.
Clearly $id_X$ is a transitive point for $T_*$. If $(X,T)$ is not
weakly rigid, i.e. $id_X$ is not a recurrent point for $T_*$, then
$id_X$ is an isolated transitive point for $T_*$ and $A(X,T) = R_{T_*}(id_X)$ is a proper subset of
$E(X,T)$.

Now we assume that $(X^*,T)$ is obtained from $(X,T)$ as above by attaching a single dense orbit of isolated points.
Assume that $x^*$ is a transitive point for $(X^*,T)$. Then $ev_{x^*} : (E(X^*,T), \allowbreak T_*) \to (X^*,T)$ is a factor map
sending $A(X^*,T)$ to $X = R_T(x^*)$. If $(X^*,T)$ is WAP then the map is an isomorphism by Proposition \ref{prop02}.

Now assume that the subspace $(X,T)$ is WAP.  As is true for any subsystem the restriction map $\rho : A(X^*,T) \to A(X,T)$ is
surjective.

\begin{prop}\label{miscprop1} Assume that $(X^*,T)$ is  topologically
transitive with an isolated transitive point $x^*$ such that the subsystem $(X,T)$
with $X = R_T(x^*)$ is WAP. The map $\rho$ is injective, and so is an isomorphism, iff $(X^*,T)$ is WAP
and, in addition, $A(X,T)$ distinguishes
points of $X$. \end{prop}

\proof  Since $X$ is WAP, $A(X,T)$ is abelian. If $\rho$ is injective then $A(X^*,T)$ is abelian and so $(X^*,T)$ is WAP
by Proposition \ref{prop02}.

Now assume $(X^*,T)$ is WAP.  We show that $\rho$ is injective iff $A(X,T)$ distinguishes the points of $X$.

Let $p_1, p_2 \in A(X^*,T)$. Since $A(X^*,T)$ is abelian, $p_1(q(x^*)) = p_2(q(x^*))$ for all $q \in A(X^*,T)$ iff
$q(p_1(x^*)) = q(p_2(x^*))$ for all $q \in A(X^*,T)$. The first says $\rho(p_1) = \rho(p_2)$ and the second says
$p_1(x^*), p_2(x^*) \in X_1$ are not distinguished by $A(X,T)$.  $\rho$ is injective says that the first implies
$p_1 = p_2$ while $A(X,T)$ distinguishes points says that the second implies $p_1(x^*) = p_2(x^*)$ and so by continuity
$p_1 = p_2$. This proves the equivalence.

$\Box$ \vspace{.5cm}

\begin{cor}\label{misccor2} Assume $(X,T)$ is WAP, is not weakly
rigid and is such that
$A(X,T)$ distinguishes points. If there
exists $(X^*,T)$ topologically transitive with an isolated transitive point $x^*$ such that $(X,T)$ is the subsystem
with $X = R_T(x^*)$ then $(X^*,X,T)$ is isomorphic to
$(E(X,T),A(X,T),T_*)$.
\end{cor}

\proof
Since $X$ is not weakly rigid, $\rho$ is an isomorphism from \\ $(E(X^*,T),A(X^*,T),T_*)$ onto $(E(X,T),A(X,T),T_*)$.
On the other hand, $ev_{x^*}$ is an isomorphism from $(E(X^*,T),A(X^*,T),T_*)$ onto $(X^*,X,T)$.

$\Box$ \vspace{.5cm}
%

\subsection{Coalescence,  LE, HAE and CT-WAP systems}\label{sec,coal}
$\qquad$

\vspace{.5cm}

Given a metric dynamical system $(X,T)$, a point $x \in X$ is an {\em equicontinuity point}
if for every $\epsilon > 0$ there is $\delta >0$ such that $d(x',x) < \delta$ implies $d(T^nx',T^nx)< \ep$
for every $n \in \Z$.
A system $(X,T)$ is called {\em equicontinuous} if it is a metric system and every point in $X$ is an equicontinuity point
(and then it is already {\em uniformly equicontinuous} meaning that the $\delta$
in the above definition does not depend on $x$). It is called
\emph{almost equicontinuous}, hereafter AE, when  it is a metric system and there
is a dense set of points in $X$ at which $\{ T^n : n \in \Z \}$ is equicontinuous.
Following \cite{GM-06}
we will call $(X,T)$ \emph{hereditarily almost equicontinuous}, hereafter HAE, when
every subsystem (i.e. closed invariant subset) is again an AE system.
As was shown in \cite{AAB} every metrizable WAP is HAE (see also \cite[Chapter 1, Sections 8 and 9]{Gl-03}).

An isolated point in a metric system  is an equicontinuity point and so if the isolated points are dense then the system is AE.
Hence, if $(X,T)$ is countable then every subsystem is AE
i.e. $(X,T)$ is HAE.

A system $(X,T)$ is \emph{expansive} if  it is a metric system and
there exists  $\ep > 0$ such that for every $x_1 \not= x_2$,
 $d(T^n(x_1),T^n(x_2))
> \ep$ for some $n \in \Z$. Any subshift is expansive.
The following is obvious.

\begin{lem}\label{lemexpansive} If $(X,T)$ is expansive then $x \in X$ is an equicontinuity point iff it is isolated.
\end{lem}

$\Box$ \vspace{.5cm}

From this follows the result from \cite{GM} that a subshift is HAE iff it is countable.

\begin{prop}\label{propHAE} An expansive dynamical system $(X,T)$ is HAE iff $X$ is countable.
In particular, a subshift is HAE iff it is countable. \end{prop}

\proof It was observed above that a countable system is HAE.

 Now assume that $X$ is uncountable and so contains a Cantor set $C$.
 If $X_1$ is the closure of $\bigcup_n \ \{ T^n(C) \}$, then
 the subsystem $(X_1,T)$ is expansive and contains no isolated points.  So by Lemma \ref{lemexpansive} it has no equicontinuity
 points.  Thus, $(X,T)$ is not HAE.

 $\Box$ \vspace{.5cm}

Thus:

\begin{prop}\label{WAPctbl}
A WAP subshift is countable.
\end{prop}

$\Box$ \vspace{.5cm}

Following \cite{GW-LE}
we
call $(X,T)$ \emph{locally equicontinuous} (hereafter LE) if  it is a metric system
and each point $x$ is an equicontinuity point in
its orbit closure or, equivalently, if each orbit closure is an almost equicontinuous subsystem.
The equivalence follows from the Auslander-Yorke
Dichotomy Theorem, \cite{AY}, which says that
in a topologically transitive system
 the set of equicontinuity
points either coincides with the set of transitive points or else it is empty.

\vspace{.5cm}

\begin{remark}\label{cp,ae} From the latter condition, it follows that an HAE system is LE. Any CP system
is LE since each point is isolated in its orbit closure.
From Proposition \ref{propHAE} it follows that any uncountable CP subshift is LE but not HAE.
\end{remark}

 \vspace{.5cm}

 A system $(X,T)$ is \emph{coalescent} when any surjective action map $\pi$ on $(X,T)$ is an isomorphism.

\begin{prop}\label{prop06} A topologically transitive
metric system which is WAP is coalescent. \end{prop}

\proof
Let $x^*$ be a transitive point.
There exists $p$ in the enveloping
semigroup with $p(x^*) = \pi(x^*)$. Because $p$ is continuous it is an action map and so since $p$ and $\pi$ agree on the
dense orbit of $x^*$, $p = \pi$. Since $p$ is surjective, $p(x^*)$ is a transitive point and so there exists $q$ such that
$qp(x^*) = x^*$ and so $qp = id$. Hence, $p$ is injective with inverse $q$.

$\Box$ \vspace{.5cm}

\begin{ex}\label{ex5}  In general a WAP system need not be coalescent.

If $(X,T)$ is WAP then the countable product $(X^{\N},T^{\N})$ is
WAP and the shift map is a surjective action map which is not injective. If $X$ is CT with  fixed point $e$
 then the \emph{infinite wedge} which is $\{ x \in X^{\N} : x_i \not= e $ for at most one $i \}$ is  a closed
invariant set which is shift invariant as well. This is also WAP and not coalescent. In addition, it is countable if $X$ is.
\end{ex}

$\Box$ \vspace{.5cm}

\begin{lem}\label{tt}
If a  dynamical system $(X,T)$ contains an increasing net
of topologically transitive  subsystems $\{ X^i \}$ with
$\bigcup_i \ \{ X^i \}$ dense in $X$, then $X$ is also topologically transitive.
\end{lem}

\proof
Let $U, V \subset X$ be two nonempty open subsets. For some $i$
$$
U \cap X^i \ne \emptyset \quad  {\text{and}} \quad
V \cap X^i \ne \emptyset.
$$
As $X^i$ is topologically transitive, there exists $k \in \Z$ with $T^k (U \cap X^i)
\cap (V \cap X^i) \ne \emptyset$ and, a fortiori, also
$T^k(U) \cap V \ne \emptyset$.

$\Box$ \vspace{.5cm}

\begin{prop}\label{max}
Every dynamical system is a union of maximal topologically
transitive subsystems.
\end{prop}

\proof
Let $(X,T)$ be a dynamical system and consider the family
$\T$ of topologically transitive subsystems of $X$. Using Lemma \ref{tt} it is easy to check
that this family is inductive. Hence, by Zorn's Lemma, every topologically transitive subsystem
of $X$ is contained in a maximal element of $\T$. In particular, for $x \in X$, the orbit closure
of $X$ is contained in a maximal element of $\T$.

$\Box$ \vspace{.5cm}

We use this proposition to obtain the following results on
{\em E-Coalescence} (i.e.
the property that every continuous surjective element of $E(X,T)$ is
injective).

Recall that a dynamical system $(X,T)$ is called
(i)
{\em weakly rigid} if there
is a net $\{ T^{n_i} \}$ with
$|n_i|
\to \infty$
and $\{ T^{n_i}(x) \} \to x$ for every
$x \in X$, or, equivalently, if $id_X \in A(X,T)$.
(ii)
{\em rigid} if the net can be chosen to
be a sequence,
and (iii)
{\em uniformly rigid} if the convergence can be taken to be uniform
(see \cite{GMa}).
Recall that if $X$ is a
topologically transitive AE system, and
a fortiori a metrizable, topologically transitive WAP, then
it is uniformly rigid
(see \cite{GM}, \cite{GW-Sen} and \cite{AAB1}).

\begin{theo}\label{coaleth}Let  $(X,T)$ be an AE  system.
Assume that $p \in A(X,T)$ is  continuous and surjective.
\begin{itemize}

\item[(i)] If $(X,T)$  is topologically transitive then $p$ is injective and so is an isomorphism.
If $T^{n_j}$ is a net converging pointwise to $p$ then it converges uniformly to $p$ and $T^{-n_j}$ converges uniformly to
$p^{-1}$. Thus, $p^{-1}, id_X \in E(X,T) = A(X,T)$.

\item[(ii)] If $X_1$ is a maximal topologically transitive subset of $X$ then $p(X_1) = X_1$.

\item[(iii)] If $(X,T)$ is HAE then $p$ is an isomorphism and if  $T^{n_j}$ is a net converging pointwise to $p$ then
 $T^{-n_j}$ converges pointwise to $p^{-1}$. Thus, $p^{-1}, id_X \in E(X,T) = A(X,T)$ and the system is weakly rigid.
\end{itemize}

\end{theo}

\proof (i) Let $x^*$ be a transitive point for $X$. Since $p$ is surjective and continuous, $px^*$ is a transitive point
and so there exists a sequence $T^{n_i}$ with $T^{n_i}px^*$ converging to $x^*$. Let
$\ep > 0$. By the Auslander-Yorke
Dichotomy Theorem, \cite{AY} $x^*$ is an equicontinuity
point 
and
for sufficiently large $i$, $T^{n_i}pT^kx^* = T^kT^{n_i}px^*$ is within $\ep$ of $T^kx^*$
for all $k$. Since the orbit of $x^*$ is dense it
follows that $T^{n_i}p$ converges uniformly to $id_X$. Hence, $p$ is injective and so is an isomorphism by compactness.
 If $q$ is any limit point of $T^{n_i}$ in $E(X,T)$ then continuity of $p$ implies that $pq$, which is the limit
 of $pT^{n_i} = T^{n_i}p$,
 is the identity, and so $q = p^{-1}$. 
Thus, $p^{-1} \in E(X,T)$. Hence, $id_X = p^{-1}p \in A(X,T)$ and so $(X,T)$ is weakly rigid and $E(X,T) = A(X,T)$.
Since $pT^{n_i} = T^{n_i}p$ converges uniformly to $pp^{-1}$, uniform continuity of $p^{-1}$ implies
$T^{n_i}$ converges uniformly to $p^{-1}$. Hence,
$\lim_{i,j \to \infty} T^{n_i - n_j} = id_X$ and so the system is uniformly
rigid.  Finally, if a net $T^{m_i}$ converges to $p$ pointwise then $p^{-1}T^{m_i}x^*$ is eventually close to $x^*$ and so,
as above, $p^{-1}T^{m_i}$ converges to $id_X$ uniformly. It follows that $T^{m_i}$
converges to $p$ uniformly and $T^{-m_i}$ converges
uniformly to $p^{-1}$.

(ii) Let $x^*$ be a transitive point for $X_1$. Since $p$ is surjective,
there exists $x_1 \in X$ such that $px_1 = x^*$. The orbit closure
of $x_1$ is a topologically transitive subset of $X$ and it contains
$px_1 = x^*$ and so contains $X_1$, which is the orbit closure of $x^*$.
Hence, by maximality, $X_1 = \overline{O_T(x_1)}$. Since $p$ is a
continuous action map, $p(X_1) = p(\overline{O_T(x_1)}) = \overline{O_T(x^*)} = X_1$.

(iii) 
By \ref{max}
each point is contained in a maximal topologically transitive subset of $X$, which is
necessarily closed.

Now if for some points $x_1, x_2 \in X$ we have $z = px_1 = px_2$, then
let $X_i$ be a maximal topologically transitive subset of $X$ which contains $x_i$. Since each $X_i$ is closed and
invariant, $z \in X_1 \cap X_2$.
 By (ii) $p$ is surjective on each $X_i$ and so by (i)  there exist $q_i \in A(X,T)$ such that on $X_{i} \ q_ip$ is the identity.
In particular, $q_iz = x_i$. Hence,
$q_2px_1 = x_2$ and so $x_2$ as well as $x_1$ is in $X_{1}$. Since $q_1p$ is the identity on $X_{1}$ it follows that $x_1 = x_2$.
Thus, $p$ is an isomorphism. Now let  $T^{m_i}$ be a net converging to $p$ pointwise. By part (i) $T^{-m_i}$ converges to
$p^{-1}$ uniformly on each orbit closure  and so pointwise on $X$.  Hence, $p^{-1}$ and $id_X = p^{-1}p$ are in $A(X,T)$.  Hence,
$E(X,T) = A(X,T)$ and $(X,T)$ is weakly rigid.

$\Box$ \vspace{.5cm}

\begin{cor}\label{cor,e-coa1}
Every metrizable WAP dynamical system is E-coalescent.
\end{cor}

\begin{proof}
If $(X,T)$ is a metrizable WAP dynamical system then (i) it
 is
HAE, and (ii) every $p \in E(X,T)$ is continuous.
Now apply Theorem \ref{coaleth}.
\end{proof}

$\Box$ \vspace{.5cm}

\begin{cor}\label{cor,e-coa}
Let $(X,T)$ be a WAP system and $p \in A(X, T)$.

The restriction of $p$ to the subsystem $Z = \bigcap_{n \in \N} p^n X$
is surjective on $Z$ and so is an automorphism of $Z$. In particular the system $(Z,T)$ is weakly rigid.

If, in addition, $(X,T)$ is  CP, then
every point of $Z$ is periodic.
If, moreover, $(X,T)$ is topologically transitive then  $Z$ consists of a single
periodic orbit which is the unique minimal subset of $X$ and  is independent of the choice of $p \in A(X,T)$.

\end{cor}

\proof Assume first that $X$ is metrizable. Hence, $(X,T)$ is HAE because it is WAP.

Clearly $Z$ is a (nonempty) subsystem and $pZ =Z$. Thus, by
Theorem \ref{coaleth} $p | Z$ is an automorphism of $Z$.
Since $p \in A(X,T)$ it follows that $Z$ is weakly rigid and so every point of $Z$ is recurrent.
So if the system is CP then every point of $Z$ is periodic.

If, in addition, $X$ is transitive then by Corollary \ref{cor02a} $X$ contains a unique
minimal subset and so $Z$ consists of a single periodic orbit which is that minimal subset.

Since the group $\Z$ is countable, $(X,T)$ is  an inverse limit of a net
 $\{ (X^i,T^i) \} $
 of
 metrizable
systems which are WAP since the latter property
is preserved by factors. The set $Z$ projects onto the corresponding set
$Z^i$. Thus, the restriction of $p$ to $Z$ is the inverse limit of isomorphisms and
so is an isomorphism. The inverse limit of weakly rigid systems is weakly rigid. So in the CP case
the points of $Z$ are periodic.
Having a unique minimal subset is preserved by inverse limits and so if $X$ is transitive then $Z$ is a single
periodic orbit as before.

$\Box$ \vspace{0.5cm}

\begin{Qs}\label{Qs,I}
$\qquad$
\begin{enumerate}
\item
Is there a
metric
WAP system which is central, but which is not rigid?
\item
If a homeomorphism for $X$ is in $A(X,T)$ then is its inverse
also in $A(X,T)$
(and so it is weakly rigid)?
For WAP
or even
for
HAE the answer is yes, by Theorem \ref{coaleth} above.

\end{enumerate}
\end{Qs}

\vspace{.5cm}

\begin{lem}\label{limitlem1} Assume $(X,T)$ is a CP system.
If there is an infinite sequence $\{ x_i : i \in \N \}$ in $X$ such that
$x_i \in R_T(x_{i+1})$ for all $i$, then $x_1$ is a periodic point and all $x_i$'s are in the
orbit of $x_1$. \end{lem}

\proof

First assume that $X$ is metrizable.

If $x_i$ is periodic then all $x_j$'s with $j < i$ are in the orbit of $x_i$. Hence, if infinitely many of the
$x_i$'s are periodic then they all are and all lie in the same periodic orbit.
We show that the alternative to
infinitely
many periodic points cannot happen.


If instead there are only finitely many periodic points in the sequence then
by omitting finitely many initial terms and re-numbering we can assume that none are periodic.
Let $A_i$ be the orbit closure of $x_i$.

If $x_{i+1} \in R_T(x_{i+1})$
then it is positively or negatively recurrent. Since $X$ is CP the only recurrent points are periodic. Hence,
if $x_{i+1}$ is not periodic then $x_{i+1}$ is not in the orbit closure of $x_{i}$ which we label $A_i$.
Since $x_i \in A_{i+1}$ we have $A_i \subset A_{i+1}$. Since $x_{i+1} \in A_{i+1} \setminus A_i$, each inclusion is strict.
Each $A_i$ is topologically transitive with transitive point $x_i$. Therefore the closure of the union
$A = \ol{ \bigcup{A_i}}$ is topologically transitive. Let $z$ be a transitive point for $A$.  It is isolated in
$A \ = \ \ol{ \bigcup{A_i}}$ and so must lie in some $A_j$, but since the $\{ A_i \}$ sequence is
a strictly increasing sequence of closed invariant sets, it cannot be in any of them.

For the general case, we can assume that some $x_i$ is not periodic.
There is a metrizable factor for which the image of
$x_i$ is not periodic and this contradicts the metric result above.

$\Box$ \vspace{.5cm}

\begin{prop}\label{limitprop2} Assume that
$(X,T)$ is a CP, WAP system.
  If $\{p_i : i \in \N \}$
is
a sequence in
$A(X,T)$ then $\bigcap_{n=1}^{\infty} \
p_n p_{n-1} \cdots p_1 X$
is a closed subset consisting of periodic points.
\end{prop}

\proof For $i = 1,2,\dots$ and $n \geq i$ let $X_{i,n} =
p_i p_{i+1} \cdots p_n X$.
$X_i = \bigcap_{n=i}^{\infty} \ X_{i,n}$. Since each $p_i$ is
continuous, each $X_{i,n}$ is a closed, invariant subspace
and the sequence is decreasing in $n$ and $p_i(X_{i+1,n}) = X_{i,n}$ when $n \geq i + 1$.
Hence, continuity and compactness imply that $p_i(X_{i+1}) = X_i$. Since the semigroup is abelian, $X_{i,n} =
p_n p_{n-1} \cdots p_i X$.
Let $x_1 \in X_1$. By induction we can build
a sequence $x_i \in X_i$ such that $p_i(x_{i+1}) = x_i$ and so $x_i \in R_T(x_{i+1})$. From
Lemma \ref{limitlem1} it follows that $x_1$ is periodic.

$\Box$ \vspace{.5cm}

\begin{cor}\label{limitcor3} Assume that $(X,T)$ is a CP, WAP system. If $\{ x_i : i \in \N \}$ is a sequence
in $X$ such that $x_{i+1} \in R_T(x_{i})$ for all $i$, then $\bigcap_i \ \ol{O(x_i)}$ (the
orbit closures of the $x_i$'s)
is a closed subset consisting of periodic points.\end{cor}

\proof The sequence $ \{ \ol{O(x_i)} \}$ is decreasing and so we can restrict to the case $X =  \ol{O(x_1)}$.
This is a topologically transitive  WAP system.
By assumption, there
exist
$p_i \in
A(X,T)$ such that $p_i(x_i) = x_{i+1}$. In the notation of the proof of
Proposition \ref{limitprop2}, $x_{n+1} \in X_{1,n}$ for all $n \geq 1$. Hence,
$\bigcap_n \ \ol{O(x_n)} \ \subset \bigcup_n  X_{1,n}$
and the latter consists of periodic points by  Proposition \ref{limitprop2}.

$\Box$ \vspace{.5cm}

%
%

 \vspace{0.5cm}

 Recall that for $(X,T)$ a cascade, we defined the  relation $R_T$ on $X$ with $R_T(x)$ the set of limit points of the
 bi-infinite orbit sequence $\{ T^i(x) : i \in \Z \}$. For $A \subset X$ we let $R_T^*(A) = \{ x : R_T(x) \subset A \}$.
 Let $R_T^{*n}(A) = R_T^*(R_T^{*(n-1)}(A)) = \{ x : R_T^n(x) \subset A \}$. Observe that
 $R_T^*(A) = \bigcap \{ p^{-1}(A) : p \in A(X,T) \}$.
 Hence, in the WAP case with all
 members of $A(X,T)$ continuous, $R^*_T(A)$ is closed when $A$ is.

A subset $A \subset X$ is called \emph{orbit-closed} if $x \in A$
implies $ \ol{O(x)} \ \subset \ A$.

 An orbit-closed set is clearly invariant, but need not be closed.  On the other hand, a closed, invariant set is
 orbit-closed. In general, a set $A$ is orbit-closed iff it is invariant and $x \in A$ implies $R_T(x) \subset A$.
  Thus, a subset $A$ is orbit-closed iff it is invariant and $A \subset R_T^*(A)$.

  \begin{lem}\label{isolem2a} Let $A \subset X$. The set $R_T^*(A)$ is orbit-closed, i.e.
 \begin{equation}\label{iso2}
 x \ \in \ R_T^*(A) \qquad \Rightarrow \qquad \ol{O(x)} \ \subset \ R_T^*(A)
 \end{equation}
  \end{lem}

 \proof Since $R_T(T^i(x)) = R_T(x)$ for all $i \in \Z$ it follows that $R^*_T(A)$ is invariant. Since
 $R_T \circ R_T \subset R_T$, $y \in R_T(x)$ implies $R_T(y) \subset R_T(x)$.  It follows that $x \in R_T^*(A)$
 implies $y \in R_T^*(A)$.

 $\Box$ \vspace{.5cm}

%

We will say that $A$ is \emph{limit determined}  or \emph{L-determined} if it
is invariant and $x \in A$ iff $R_T(x) \subset A$. That is, $A$ is L-determined when it is orbit closed and $R_T(x) \subset A$
only when $x \in A$. Thus, a subset $A$ is L-determined iff it is invariant and $A = R_T^*(A)$. Clearly,
\begin{equation}\label{capeq}
B \subset A \ \text{and} \ A \ \text{L-determined} \quad \Longrightarrow \quad  R_T^*(B) \subset A.
\end{equation}

The entire space $X$ is L-determined. For a collection  $\{ A_i \}$  of subsets of $X$, the intersection is
$\bigcap_i \ A_i$
invariant, or orbit-closed or L-determined, if each $A_i$ satisfies the corresponding property.
The union $\bigcup_i \ A_i$ is invariant or orbit-closed if
each $A_i$ satisfies the corresponding property.

It follows that for any set $A$ there is a minimum L-determined set which
contains it, i.e. the intersection of all L-determined sets containing it. For an orbit-closed set, we can obtain this
set via the usual transfinite construction.

Assume that $A \subset X$ is orbit-closed. Define
\begin{align}\label{zstar}
\begin{split}
&z^*_0(A) \quad = \quad A, \\
&z^*_{\a + 1}(A) \quad =\quad R_T^*(z^*_{\a}(A)), \\
&z^*_{\b}(A) \quad = \quad \bigcup_{\a < \b} z^*_{\a}(A),
\end{split}
\end{align}
for $\b$ a limit ordinal. Every $z^*_{\a}$ is orbit-closed by Lemma \ref{isolem2a} and so the
sequence is increasing until it stabilizes at the minimum L-determined set which contains $A$.

A closed invariant set $K \subset X$ is an \emph{isolated invariant set}
if there is an open set $U$ containing $K$ such that $K$ is the maximum
invariant subset of $U$, i.e. $K = \bigcap_{i \in \Z} \ T^i(U)$.  If $\{ K_n \}$
is a decreasing sequence of closed invariant sets with intersection $K$ isolated,
then $K_n \subset U$ implies $K_n = K$ and so eventually $K_n = K$. In an expansive system,
like a subshift,
any fixed point is an isolated invariant set.

\begin{lem}\label{seqlem} Let $(X,T)$ be a CT, WAP subshift
with fixed point $\{e\}$.
For any
 infinite sequence $p_1, p_2, \dots$  of elements of $A(X,T)$ there
there exists an $n$ such that $p_1p_2\cdots p_n(X) =
\{ e \}$. \end{lem}

\proof Proposition \ref{limitprop2} implies that $\{ e \}$ is the intersection of the decreasing
sequence $\{ p_1p_2\cdots p_n(X) \} $.
If $V$ is any neighborhood of $e$ then the invariant set
 $p_1p_2\cdots p_n(X)$ is contained in $V$ for sufficiently large $n$. Since $e$ is an isolated invariant set,
 we can choose $V$ so that $\{ e \}$ is the only invariant subset of $V$.

$\Box$ \vspace{.5cm}

%



Now we restrict to the case where $(X,T)$ is a metrizable CT, WAP with a single minimal set consisting of a fixed point $e$.
We will call such a system \emph{CT-WAP}.

Fix a closed neighborhood $V(e)$ of $e$.  let $A(e)$ denote the maximal invariant
subset of $X$ which is contained in  $ V(e)$,
i.e.
 \begin{equation}\label{df,Ae}
 A(e) = \bigcap_{i \in \Z} \ T^i(V(e))
 \end{equation}

If  $V(e)$ is clopen,
$A(e)$ is an isolated invariant set.
If $\{ e \}$ is an isolated invariant set, e.g. if $(X,T)$ is a subshift, then we  choose $V(e)$
so that $A(e) = \{ e \}$.

\begin{prop}\label{isoprop1} Assume $(X,T)$  is CT-WAP.
Let $A$ be any  invariant set which contains $A(e)$.  If $x \not\in A$ then there exists
$q \in E(X,T)$ such that $qx \not\in A \cup V(e)$ but for all $p \in A(X,T)$, $pqx \in A$,
i.e. $R_T(qx) = q(R_T(x)) \subset A$. That is,  $qx \in R^*_T(A) \setminus (A \cup V(e))$.
\end{prop}

\proof Clearly,
\begin{equation}\label{iso1}
A \quad \subset \quad \bigcap_{\ell \in \Z} \ T^{\ell}(A \cup V(e)) \quad \subset \quad A \cup A(e) \quad = \quad A
\end{equation}
So for some $i \in \Z, \ T^i(x) \not\in A \cup V(e)$. If $R_T(T^i(x)) \subset A \cup V(e)$
then let $q = T^i$. Otherwise,
there exists $q_1$ such that $q_1(T^i(x)) \not\in A \cup V(e)$. Continue inductively defining
$q_n$ such that
$q_n \cdots q_1(T^i(x)) \not\in A \cup V(e)$ whenever $R_T(q_{n-1} \cdots q_1(T^i(x))$ is not
contained in $A \cup V(e)$. This process must terminate.

Assume it did not. By Proposition \ref{limitprop2} $\{ e \} = \bigcap_n \{ q_n \cdots q_1(X) \}$ and so eventually
$ q_n \cdots q_1(X)$ is a closed invariant set contained in $V(e)$.
By invariance, $ q_n \cdots q_1(X) \subset A(e)$, contradicting the definition of $q_n$.

If  $q_n \cdots q_1(T^i(x)) \not\in A \cup V(e)$ but
$R_T(q_{n} \cdots q_1(T^i(x)) \subset  A \cup V(e)$, then
 let $q = q_{n} \cdots q_1T^i$.
Since $R_T(qx) \subset A \cup V(e)$ and $R_T(qx)$ is an invariant set, equation \ref{iso1} implies
$R_T(qx) \subset A$.

Since the system is WAP, $q$ is an action map and so $q(R_T(x)) = R_T(qx)$.

$\Box$ \vspace{.5cm}


%

\begin{prop}\label{isoprop2}  Assume $(X,T)$  is CT-WAP.
Let $A$ be any  orbit-closed set which contains $A(e)$ and let $\{ z^*_{\a}(A) \}$ be the sequence defined by
(\ref{zstar}). For any ordinal $\a$, if $x \not\in z^*_{\a}(A)$  then there exists $q \in E(X,T)$ such that
\begin{equation}\label{iso3}
\begin{split}
 qx \ \not\in \ z^*_{\a}(A) \cup V(e) \qquad \mbox{and} \qquad
R_T(qx) \ \subset \ z^*_{\a}(A)  \\
\mbox{i.e.} \qquad qx \in z^*_{\a + 1}(A)  \setminus ( z^*_{\a}(A)  \cup V(e) ). \hspace{2cm}
\end{split}
\end{equation}

For any such $q$,  $R_T(qx)$ meets
$z^*_{\b+1}(A)  \setminus (z^*_{\b}(A)  \cup V(e))$ for every $\b < \a$.
\end{prop}

\proof  We repeatedly apply Proposition \ref{isoprop1}. First we obtain $q \in E(X,T)$ which satisfies (\ref{iso3}).

For any $\b < \a$
there exists $q_1 \in E(X,T)$ such that $q_1qx \in z^*_{\b+1}(A)  \setminus (z^*_{\b}(A)  \cup V(e))$. Since $\b + 1 \leq \a$,
$qx \not\in z^*_{\b +1}(A)  $ and so $ T^i(qx) \not\in z^*_{\b + 1}(A) $ for any $i \in \Z$.
Hence, $q_1 \in A(X,T)$ and so $q_1(qx) \in R_T(qx)$.

$\Box$ \vspace{.5cm}

\begin{cor}\label{isocor3} Assume $(X,T)$  is CT-WAP.
Let $A$ be any  orbit-closed set which contains $A(e)$ and let $\{ z^*_{\a}(A)  \}$ be the sequence defined by
(\ref{zstar}). If for any ordinal $\a$, $z^*_{\a}(A) $ is a proper subset of $X$ then $z^*_{\a + 1}(A)  \setminus z^*_{\a}(A) $ is
nonempty. So the transfinite sequence $\{ z^*_{\a}(A)  \}$ is strictly increasing until the first ordinal $\a^*$ such that
$z^*_{\a^*}(A)  = X$. If $(X,T)$ is topologically transitive with transitive
point $x^*$ then $\a^*$ is the first ordinal such that $x^* \in z^*_{\a^*}(A) $ and  $\a^*$ is a non-limit ordinal.
\end{cor}

\proof That the sequence is strictly increasing until $z^*_{\a}(A)  = X$ is clear from
Proposition \ref{isoprop2}. It is also clear
that if $x^* $ is a transitive point then $x^* \in z^*_{\a}(A) $ implies $X = \ol{O(x^*)} \subset a_{\a}$ by (\ref{iso2}) and so
$z^*_{\a}(A)  = X$ iff $x^* \in z^*_{\a}(A) $. If $x^* \in \bigcup_{\b < \a} z^*_{\b}(A) $ then $x^* \in a_{\b}(A) $ for some $\b < \a$ and
so $\a^*$ cannot then be a limit ordinal.

$\Box$ \vspace{.5cm}

For the case when $(X,T)$ is a
CT-WAP with the fixed point $e$ isolated
(as a closed invariant subset), e.g. a CT-WAP subshift,
we have chosen $V(e)$
 so that $A(e) =  \{ e \}$ and we let $A = A(e) = z^*_0(e)  = \{ e \}$.
 We then call the ordinal at which the $z^*_{\a}(e)$ sequence stabilizes the
\emph{height}$^*$ of $(X,T)$. Because $X$ is then countable it follows that the ordinal $\a^*$ is countable.

\begin{cor}\label{isocor4}  Assume $(X,T)$  is CT-WAP.
If the fixed point $e$
is an
isolated  invariant set then the entire
space $X$ is the only L-determined subset of $X$. \end{cor}

\proof The transfinite sequence $\{ z^*_{\a}(e) \}$  stabilizes at $X$ by Corollary \ref{isocor3}.
If $A$ is any L-determined set then for any $x \in A$, $e \in R_T(x) \subset A$ since $A$ is orbit-closed. Hence, $z^*_0(e) \subset A$
and so inductively $z^*_{\a}(e) \subset A$ for all $\a$.  Hence, $A = X$.

$\Box$
\vspace{1cm}

\subsection{Discrete suspensions and spin constructions}\label{sec,const}
$\qquad$
\vspace{.5cm}

For any $(X,T)$ and positive integer $N$ we define on $X \times [0, N-1]$ the homeomorphism $\tilde T$ by
\begin{equation}\label{spin1}
\begin{split}
\tilde T(x,i)\quad = \quad \begin{cases} (x,i+1) \qquad \mbox{for} \quad i < N - 1, \\
(T(x),0) \qquad \mbox{for} \quad i = N - 1. \end{cases} \\
\mbox{so that} \qquad \tilde T^N \quad = \quad T \times id_{[0,N-1]}. \hspace{3cm}
\end{split}
\end{equation}
$(X \times [0, N-1], \tilde T)$ is called the \emph{discrete $N$ step suspension}. It is countable if $X$ is, it is CP if
$(X,T)$ is CT.  It is WAP if $(X,T)$ is. Apply the following

\begin{lem} \label{spinlem1} A point $x \in X$ is an equicontinuity point
for $T$ on $X$ iff it is an equicontinuity point for $T^N$ on $X$.
Hence, $(X,T)$ is AE or HAE iff $(X,T^N)$ is. In general, $(X,T)$ is WAP iff $(X,T^N)$ is. \end{lem}

\proof Since $\{T^{Ni} : i \in \Z \} \subset \{ T^i : i \in \Z \}$, an equicontinuity point for $T$ is one for $T^N$,
$E(X,T^N) \subset E(X,T)$ and so each of the conditions for $T$ implies the corresponding condition for $T^N$. In fact,
\begin{align}\label{spin2}
\begin{split}
\{ T^i : i \in \Z \} \quad  & = \quad \{T^{Ni}\circ T^k : i \in \Z, k = 0,\dots,N-1 \}\\
\quad & = \quad \bigcup_{k = 0}^{N-1} T^k \circ \{T^{Ni} : i \in \Z \}.
\end{split}
\end{align}
It follows that if $x$ is an equicontinuity point for $T^N$ then it is for $T$. Also, we obtain
$E(X,T) \  = \ \bigcup_{k = 0}^{N-1} T^k \circ E(X,T^N)$ and so the
elements of $E(X,T)$ are continuous when those of $E(X,T^N)$ are.

$\Box$ \vspace{.5cm}

\begin{theo} \label{spintheo1a} If $(X,T)$ is a CP WAP system with a unique minimal set, a periodic orbit of period $N$, then
$(X,T)$ is isomorphic to the  discrete $N$ step suspension of a CT WAP. \end{theo}

\proof Let $\{ x_0,\dots,x_{N-1} \}$ be the periodic orbit in $X$. By Lemma
\ref{spinlem1}
$(X,T^N)$ is a WAP with $N$
minimal fixed points $ x_0,\dots,x_{N-1}$. For any $x \in X$ the restriction of $T^N$ to the $T^N$ orbit closure of $x$ is a
point transitive WAP and so with a unique minimal set, necessarily one of the $x_i$'s.
Let $X_i = \{ x \in X : x_i \in
\ol{\mathcal{O}_{T^N}(x)} \} $
Clearly, $T(X_i) = X_{i+1}$ (addition mod $N$) and the $X_i$'s are
pairwise disjoint. Each is $T^N$ invariant.
Let $u$ be a minimal element of $E(X,T^N)$. If $x \in X_i$ then $u(x)$ is a minimal element of the $T^N$
orbit closure of $x$ and so $u(x) = x_i$.  That is, $u$ retracts $X_i$ to $x_i$.  Because $(X,T^N)$ is WAP, $u$ is continuous
and so each $X_i = u^{-1}(x_i)$ is closed.

Define
$H : X_0 \times \{ 0,\dots,N-1 \} \to X$ by $H(x,i) \ = \ T^i(x)$.
This is continuous and surjective with inverse,
$x \mapsto (T^{-i}x,i)$ for $ x \in X_i$ and so $H$ is bijective. Furthermore, $H(x,i+1) \ = \ T(H(x,i))$ for $i < N-1$ and
$H(T^N(x),0) = T^N(x) = T(T^{N-1}(x)) = T(H(x,N-1))$. Thus, $H$ is an isomorphism from the discrete suspension of height $N$
of $(X_0,T^N)$ onto $(X,T)$.

$\Box$ \vspace{.5cm}

Recall that $(X,T)$
is
minCT  when there is a fixed point which is the unique minimal subset, i.e.
the mincenter is a single point.

\begin{lem}\label{spinlem2} Let $(X,T)$ be a nontrivial, metric minCT system with fixed point $e$ and let
$\ep > 0$.
There
exists an $\ep$-dense sequence of distinct points $\{ e = x_0, \dots ,x_{N - 1} \}$ in $X$ such that
with $ x_N = e$, $\{
x_0, \dots , x_N \}$ is an $\ep$ chain for $(X,T)$, i.e. $d(T(x_{i}),x_{i+1})  < \ep$ for
$i = 0, \dots, N - 1$. \end{lem}

\proof Since $X$ is separable we can choose a finite or infinite sequence $\{ a_1, a_2, \dots \}$ of points of $X \setminus \{ e \}$
with pairwise distinct orbits and such that the union of the orbits is dense in $X \setminus \{ e \}$. Since this set is nonempty
the sequence contains at least one point. Since $e$ is the only minimal point,
$ e \in \a_T(x) \cap \o_T(x)$ for every $x \in X$.
Now truncate so that the union of the orbits of the finite sequence $\{a_1, \dots ,a_K \}$ is $\ep/2$ dense in $X \setminus \{ e \}$.
For each $a_i$ we can choose a finite piece of the orbit
$\{ y_{0,i}, \dots ,y_{K_i + 1,i}\}$
which begins and ends $\ep/2$
close to $e$ and which is $\ep/2$ dense in the orbit of $a_i$. We concatenate to
obtain the  sequence $\{ x_1, \dots ,x_{N - 1} \}$.
Then let $x_0 = e$.

$\Box$ \vspace{.5cm}

On the one-point compactification $\Z^*$ with $e$ the point at infinity and $T(t) = t+1$, define the ultra-metric $d$ by
\begin{equation}\label{spin3}
 d(i,j) \quad = \quad \begin{cases} \hspace{2cm} 0 \qquad \hspace{2.5cm} \ \ \mbox{if} \quad i = j, \\
 \max(1/(|i| + 1),1/(|j| + 1))   \qquad \mbox{if} \quad i \not= j, \end{cases}
 \end{equation}
where $1/(|i| + 1) = 0$ if $i = e$.  If $(X,T) = (\Z^*,T)$ then with
$K > 1/\ep$ and $N = 2K + 2$ we can use
the sequence $\{ e, -K, -K + 1, \dots , K \}$.

\vspace{0.5cm}

When $(X,T)$ is a nontrivial metric minCT system we define
 a \emph{ preparation} for $(X,T)$ to be a choice for each $i = 0,1, \dots$ of
 a sequence $\{ e = x^i_0, \dots ,x^i_{N_i - 1} \}$  which is an
 $ 2^{-i}$ dense sequence of distinct elements of $X$ so that $\{ e = x^i_0, \dots ,x^i_{N_i - 1}, e \}$ is
 a
 $ 2^{-i}$ chain.
 For $i = 0$ we let $N_0 = 1$ so that with $i = 0$ the sequence is $\{ e \}$.

 An
{\em ultrametric minCT system} is a minCT system $(X,T)$
 with $d \leq 1$
 a compatible
 ultra-metric on $X$ (and so $X$ is
 zero-dimensional).

 As we will dealing with different metrics at the same time we will use for $\ep > 0$
\begin{equation}\label{Vep}
V_{d, \ep} = \{ (a,b) : d(a,b) < \ep \}
\end{equation}

so that $V_{d, \ep}(a)$ is the
$\ep$ ball centered at $a$. For an ultrametric $V_{d,\ep}$ is a clopen equivalence relation
and so each ball $V_{d, \ep}(a)$ is a clopen set.

Let  $(X_1,T_1), (X_2,T_2)$ be nontrivial ultrametric
minCT systems with fixed points $e_1, e_2$. Assume that
$(X_2,T_2)$ is given a preparation. The ultrametrics on $X_1$ and $X_2$ are labeled $d_1$ and $d_2$, respectively.
From the product  $X_1 \times X_2)$ we have the projections
$\pi_1$ and $\pi_2$ to $X_1$ and $X_2$, respectively.

Fix  $1 > \ep > 0$. On $X_1 \times X_2$
 we will use the ultrametric $d = \max(\pi_1^*d_1, \ep \pi_2^*d_2)$,
 i.e. $d((x_1,x_2),(y_1,y_2)) = \max(d_1(x_1,y_1), \ep d_2(x_2,y_2)$.
 so that $\pi_1$ has Lipschitz constant $1$ and for any
 $\d > 0$
 \begin{equation}\label{spin3a}
 V_{d, \d }(e_1,e_2) \quad \subset \quad \pi_1^{-1}(V_{d_1,\d}(e_1))  \hspace{4cm}
 \end{equation}
 with equality if $\d \geq \ep$.

 We will define the \emph{$\ep$ spin of $(X_2,T_2)$ into
$(X_1,T_1)$} to be the ultrametric system $(X,T)$ where $X$ is the closed subset of $X_1 \times X_2$
and the homeomorphism $T$ on $X$ are described below.

In $X_1$ we define
the sequence of pairwise disjoint clopen sets: $A_0 = X_1 \setminus V_{d_1,\ep}(e_1)$
and for $i = 1,2,\dots, A_i = V_{d_1,\ep 2^{-i + 1}}(e_1) \setminus V_{d_1,\ep 2^{-i}}(e_1)$. So
$X_1 = \{ e_1 \} \ \cup \ \bigcup_{i = 0}^{\infty} \ A_i$.
Now let
\begin{equation}\label{spin4}
\begin{split}
X = (\{ e_1 \} \times X_2 ) \ \cup \ (\bigcup_{i = 0}^{\infty} \ A_i \times \{ \ x^i_0, \dots, x^i_{N_i - 1} \}),\hspace{3cm} \\
T(x,y) \quad = \quad \begin{cases} (e_1, T_2(y)) \qquad \mbox{when} \quad x = e_1, \hspace{2cm} \\
(x,x^i_{k + 1}) \qquad \mbox{when} \quad x \in A_i, \ \ y = x^i_k \quad \mbox{with} \ k < N_i - 1, \\
(T_1(x),e_2) \qquad \mbox{when} \quad x \in A_i, \ \ y = x^i_{N_i - 1}. \hspace{1cm} \end{cases}
\end{split}
\end{equation}
It is obvious that $X$ is a closed subset of $X_1 \times X_2$ and easy to check that $T$ is invertible with
$T^{-1}(x,e_2) = (T_1^{-1}(x), x^i_{N_i - 1})$ when $T_1^{-1}(x) \in A_i$. Continuity of $T$ is clear on each
$A_i \times \{ \ x^i_0, \dots, x^i_{N_i - 1} \}$, because each $A_i \times \{ x^i_j \}$ is a
clopen subset of $X$.  Continuity at the points of $\{ e_1 \} \times X_2$ follows from
the following estimate.

\begin{lem}\label{spinlem2a}(a) Let $\d \geq \ep$. If $(x,y) \in X$ with
$x \in V_{d_1,\d}(e_1) \cap T_1^{-1}(V_{d_1,\d}(e_1))$ then
$(x,y) \in V_{d,\d}(e_1,e_2) \cap T^{-1}(V_{d,\d}(e_1,e_2))$.

(b) For $i \geq 1$, let $\d$ with $2^{-i}
> \d > 0$ be a $2^{-i}$ modulus of uniform continuity
for $ T_2$ on $ X_2$. For $(x,y) \in X, \tilde y \in X_2$
\begin{gather}\label{spin4b}
\begin{split}
x \in V_{d_1,\ep 2^{-i+1}}(e_1) \cap T_1^{-1}(V_{d_1, 2^{-i+1}}(e_1)), \ y \in V_{d_2,\d}(\tilde y)  \\ \Longrightarrow
\quad T(x,y) \in V_{d, 2^{-i + 1}}(e_1,T_2(\tilde y)).
\end{split}
\end{gather}

\end{lem}

\proof (a) follows from (\ref{spin3a}) applied to $x$ and to $T_1(x)$.

(b) If $x = e_1$ then $T(x,y) = (e_1, T_2(y))$ and $d(T_2(y),T_2(\tilde y)) \leq 2^{-i}$ if $y \in V_{d_2,\d}(\tilde y)$.

If
$x \in V_{d_1, \ep 2^{-i+1}}(e_1) \setminus \{ e_1 \}$ then $X_2$ coordinates
of the fiber in $X$ over $x$ is a $2^{-i}$ chain for $T_2$ and so if
$y \in V_{d_2, \d}(\tilde y)$ then the second coordinate of $T(x,y)$ is within $2^{-i+1} = 2^{-i} + 2^{-i}$ of $T_2(\tilde y)$.

In particular, if $j \geq i$ and $(x,y) = (x,x^j_{N_j - 1})$ with $x \in A_j$ then
$d_2(y,\tilde y) = d_2(x^j_{N_j - 1}, \tilde y) < \d$ and by definition of the chains in a preparation,
$d_2(T_2(x^j_{N_j - 1}), e_2)
\allowbreak
< 2^{-j}$. Since $T(x,y) = (T_1(x),e_2)$ we again have that
the second coordinate of $T(x,y)$ is  within $2^{-i+1} = 2^{-i} + 2^{-i}$ of $T_2(\tilde y)$.

$\Box$ \vspace{.5cm}

 So we obtain
 the ultrametric system $(X,T)$.

 From (\ref{spin3a}) it follows that the restriction
 \begin{equation}\label{spin4a}
 \begin{split}
 \pi_1 : X \setminus V_{d,\ep}(e_1,e_2) \ = \ (\pi_1)^{-1}(X_1 \setminus V_{d_1,\ep}(e_1)) \
 \to  \ X_1 \setminus V_{d_1,\ep}(e_1) \quad \mbox{is bijective,} \\
 \mbox{ and on it} \qquad  \pi_1 \circ T \quad = \quad T_1 \circ \pi_1, \hspace{4cm}\\
x \in X_1 \setminus V_{\ep}(e_1) \quad \Longrightarrow \quad (\pi_1)^{-1}(x) \ = \ \{ (x,e_2) \}, \quad  T(x,e_2) \ = \ (T_1(x),e_2).
 \end{split}
 \end{equation}

 On the rest of the space the map $\pi_1 : X \to X_1$ does not define an action map,
 but we obviously have for $x \in X_1 \setminus \{ e_1 \}$:
 \begin{equation}\label{spin5}
\begin{split}
\pi_1^{-1}(\{ T_1^i(x) : i = 0,1,\dots \}) \quad = \quad \{ T^i(x,e_2) : i = 0,1,\dots
\}, \hspace{2.2cm}\\
\pi_1^{-1}(\{ T_1^{-i}(x) : i = 1,2,\dots \}) \quad = \quad \{ T^{-i}(x,e_2) : i = 1,2,\dots
\}. \hspace{2cm}\\
\end{split}
\end{equation}

\begin{prop}\label{spinprop3} If  $x \in X_1 \setminus \{ e_1 \}$, then
  \begin{equation}\label{spin6}
\pi_1^{-1}
(\o T_1 (x)) \quad = \quad \o T (x,e_2), \qquad \pi_1^{-1}
(\a T_1 (x)) \quad = \quad \a T (x,e_2).
\end{equation}

If $A$ is a  $T_1$ invariant subset of $X_1$ then $\pi_1^{-1}(A)$ is a  $T$ invariant subset of $X$.

If $B$ is a  $T$ invariant subset of $X$ then $\pi_1(B)$ is a  $T_1$ invariant subset of $X_1$.
\end{prop}

\proof The equations (\ref{spin5}) clearly imply that $\pi_1$ maps the limit point set
$\o T (x,e_2)$ onto
$\o T_1 (x)$. Then they imply that if $z \in X_1 \setminus \{ e_1 \}$ then all the points of $\pi_1^{-1}(z)$
lie in
the same orbit.  Finally, for $z \in V_{ d_1, 2^{-i}}(e_1) \setminus \{ e_1 \}$ the set $\pi_2(\pi_1^{-1}(z))$ is
$\ep 2^{-i}$ dense in $X_2$. This proves the result for
$\o T (x)$ and the result for
$\a T (x)$ is similar.

The invariant set results are obvious from (\ref{spin5}).

$\Box$ \vspace{.5cm}

\begin{cor}\label{spincor4} $(X,T)$ is a minCT system.

If $(X_1,T_1)$ and $(X_2,T_2)$ are CT systems then $(X,T)$ is a CT system.  Furthermore, if the Birkhoff center sequences for
$(X_1,T_1)$ and $(X_2,T_2)$ stabilize at
ordinals $\o_1$ and $\o_2$
respectively, then
the Birkhoff center sequence for $(X,T)$ stabilizes at $\o_1 + \o_2$.
\end{cor}

\proof If $M$ is a minimal subset of $X$ then by Proposition \ref{spinprop3} $\pi_1$ is a minimal subset of $X_1$ and so
$M \subset \{e_1 \} \times X_2$ where $T$ is isomorphic to $T_2$ and so the $M = \{ (e_1,e_2) \}$.

Now assume $(X_1,T_1)$ and $(X_2,T_2)$ are CT systems. If $(x,y)$ is a recurrent point for $T$ then $x_1$ is a recurrent
point for $X_1$ by (\ref{spin6}). Hence,  $x = e_1$ and $y$ is a recurrent point for $T_2$. So $y = e_2$. Hence, $(X,T)$ is
CT. If $A$ is a closed $T_1$ invariant subset of $X_1$ then the limit point set $R_{T}(\pi_1^{-1}(A))$ is the
limit point set $\pi_1^{-1}(R_{T_1}(A))$. So exactly at $\o_1$ the Birkhoff center sequence for $X$ arrives at $\{e_1 \} \times X_2$.
It then stabilizes at $(e_1,e_2)$ after $\o_2$ more steps.

$\Box$ \vspace{.5cm}

Notice that by replacing the metric $d$ on $X$ by the equivalent metric
$$\min(1,\max_{n = 0}^{\infty} 2^{-n} (T^{n})^*d)$$
we can assume that the metric is bounded by 1 and that $T$ has Lipschitz constant
at most 2. The new metric is an ultrametric if $d$ was.

Let $\{ (X_n,T_n) : n = 1,2,\dots \}$ be a sequence of ultrametric minCT systems with ultrametric $d_n \leq 1$ on $X_n$ and with each
$T_n$ having Lipschitz constant at most 2.
Assume  that for
$n > 1$ each
$(X_n,T_n)$ is given a preparation. Let $(Z_1,U_1) = (X_1,T_1)$, let $(Z_2,U_2)$ be the $2^{-1}$ spin of $(X_2,T_2)$ into
$(Z_1,U_1)$ with $\xi_2 : Z_2 \to Z_1$ be the first coordinate projection. Thus, $Z_2 \subset X_1 \times X_2$. Inductively,
let $(Z_{n+1},U_{n+1})$ be the $2^{-n}$ spin of $(X_{n+1},T_{n+1})$ into $(Z_n,U_n)$ which we can regard as a
subset of the product $\Pi_n = X_1 \times
\cdots \times X_n \times X_{n+1}$ equipped with the ultrametric
$\max(\pi_1^*d_1,2^{-1}\pi_2^*d_2,\dots,2^{-n-1}\pi_{n+1}^*d_{n+1})$. Let $\xi_{n+1} :Z_{n+1} \to Z_n$ be the restriction of
the coordinate projection from $\Pi_{n+1} \to \Pi_n$ which has Lipschitz constant 1.
Note again that  the $\xi_n$'s are not action maps, but by (\ref{spin4a}) the restriction
$\xi_n : (\xi_n)^{-1}(Z_n \setminus V_{2^{-n}}(e_1, \cdots,e_n)) \to Z_n \setminus V_{2^{-n}}(e_1,\dots,e_n)$ is injective
and on it $\xi_n \circ U_{n+1} = U_n \circ \xi_n$.

Let $Z_{\infty}$ denote the inverse limit, regarded as a closed subset of $\Pi_{\infty} = \Pi_{i = 1}^{\infty}\ X_i $
 equipped with the
ultrametric $\max\{ 2^{-i + 1}\pi_i^*d_i : i = 1,2,\dots \}$ which yields the product topology. The space $Z_{\infty}$
 consists of the points $z$
such that $\xi_n(z) = (z_1,\dots,z_n) \in Z_n$ for $ n = 1,2,\dots $. We let $e \in Z_{\infty}$ denote the point $(e_1,e_2,\dots)$.

Assume that
$z \in Z_{\infty}$ with $z \not= e$ and let $n$ be the smallest value such that $x = \xi_n(z) \not= \xi_n(e) = (e_1,\dots,e_n) $.
Let $\d  = \frac{1}{2} d(x, (e_1,\dots,e_n) )$. Let
 $i > n$ be such that $2^{-i} \leq \d$.
Since the projections have Lipschitz constant 1,
 is
disjoint from $V_{\d}(e_1,\dots.,e_{n+k})$ for every positive integer $k$.
Once $n + k \geq i$ it follows that
$(\xi_{n+k-1} \circ \cdots \circ \xi_n)^{-1}(V_{\d}(x))$
is disjoint from $V_{2^{-n-k}}(e_1,\dots,e_{n+k})$. By (\ref{spin4a})
if $\tilde x \in Z_{n+k} \setminus V_{2^{-n-k}}(e_1,\dots,e_{n+k})$ and
$\tilde z = (\tilde x,e_{n+k+1}, e_{n+k+2},\dots)  \in Z_{\infty}$
then $\xi_{n+k}^{-1}(\tilde x) = \{ \tilde z \}$ and by (\ref{spin4a})
 $U(\tilde z)$ is unambiguously defined by $U(\tilde z) = (U_{n+k}(\tilde x),e_{n+k+1}, e_{n+k+2},\dots)$. Since
each $U_{n+k}$ has Lipschitz constant at most 2, it follows that on each
$Z_{\infty} \setminus \xi_{n+k}^{-1}(V_{2^{-n-k}}(e_1,\dots,e_{n+k}))$
$U$ has Lipschitz constant at most 2. Finally,
\begin{equation}\label{spin7}
d(U(\tilde z),e) \ = \ d(U_{n+k}(\tilde x),(e_1,\dots,e_{n+k})) \ \leq \ 2d( \tilde x,(e_1,\dots,e_{n+k})) \ = \ 2d(\tilde z,e)
\end{equation}
shows that $U$ has Lipschitz constant at most 2 on all of $Z_{\infty}$.

Finally, with essentially the same proof as that of Corollary \ref{spincor4} we have

\begin{cor}\label{spincor5} $(Z_{\infty},U)$ is a minCT system.

If each $(X_n,T_n)$  is a CT system then $(Z_{\infty},U)$ is a CT system.  Furthermore, if the Birkhoff center sequences for
$(X_n,T_n)$  stabilize at the ordinals $\o_n$, then
the Birkhoff center sequence for $(Z_{\infty},U)$ stabilizes at $Lim_{n \to \infty} \ \o_1 + \o_2 +
 \cdots  + \o_n$.
\end{cor}

$\Box$ \vspace{1cm}

\section{Labels and their dynamics}\label{sec,labels}

\subsection{The space of labels}\label{ss,labels}

$\qquad$

\vspace{.5cm}

Let $\Z, \Z_+, \N$ denote the sets of integers, of non-negative integers and of positive integers, respectively.
Let $\Z_{+\infty} = \Z_+ \cup \{ \infty \} = \N \cup \{ 0, \infty \}$.
On the vector space $\R^{\N}$ we will use the  lattice structure, with  $x \geq y, x \leq y, x \vee y, x \wedge y,$
the pointwise relations and the pointwise operations of  maximum and minimum for vectors $x, y \in \R^{\N}$. As usual
$x > y$ means $x \geq y$ and $x \not= y$ so that the inequality is strict for at least one coordinate. The support
of a vector $x \in \R^{\N}$, denoted $supp \ x$, is $\{ \ell : x_{\ell} \not= 0 \}$.
We identify the number $n \in \N$ with the  function in $\Z_+^{\N}$ which is constant with value $n$.

We will call $ \mm \in \Z^{\N}$ an \emph{$\N$-vector} when it is non-negative and has finite support, that is,
when $\mm \geq 0$   and
$supp \ \mm = \{ \ell : \mm_{\ell} > 0 \}$ is finite. We call $\# supp \ \mm$ the \emph{size} of $\mm$ and call
$|\mm| = \Sigma_{\ell} \ \mm_{\ell}$ the \emph{norm} of $\mm$.

If $S \subset \N$ we let $\chi(S)$ be the characteristic function of $S$ with $\chi(\ell) = \chi(\{ \ell \})$.
Thus, $\chi(S) = \Sigma_{\ell \in S} \ \chi(\ell) $ is an $\N$-vector when $S$ is a finite set.

We denote by $FIN(\N)$ the discrete abelian monoid of all $\N$-vectors with vector addition and identity $ \00 $.
It is also a lattice via
the pointwise order relations described above.

For an $\N$-vector $\mm$ and $L$ a nonempty subset of $\N$ we
define $\mm \wedge L$   to be the
$\N$-vector with
\begin{equation}\label{label01a}
(\mm \wedge L)_{\ell} \quad = \quad
\begin{cases} \ \mm_{\ell} \quad \ \ \mbox{for} \quad \ell \ \in L, \\
\quad 0 \qquad  \mbox{for} \quad \ell \not\in L. \end{cases}
\end{equation}
In particular, for
a positive integer $\ell^*,$ $\mm \wedge [1,\ell^*]$ is given by
\begin{equation}\label{label01}
(\mm \wedge [1,\ell^*])_{\ell} \quad = \quad
\begin{cases} \ \mm_{\ell} \quad \ \ \mbox{for} \quad \ell \ \leq \ \ell^*, \\
\quad 0 \qquad  \mbox{for} \quad \ell \
> \ \ell^*. \end{cases}
\end{equation}

\vspace{0.5cm}

\begin{df}\label{labeldef01} A set $\M$ of $\N$-vectors is called a
\emph{label}  when   $\00 \ \leq \ {\mathbf m}^1 \ \leq \ {\mathbf m}$ and ${\mathbf m} \in \M$
 imply ${\mathbf m}^1 \in \M$. We call this the \emph{Heredity Condition}.
\end{df}


\vspace{0.5cm}

We use the term ``label'' because we will be using them to label certain associated subshifts.

For example, $0 = \{ \00 \}$ and $\emptyset$ are labels.

\begin{df}\label{df,f-contain}
\begin{itemize}
\item[(a)] Let  $Supp \ \M = \{ \ supp \ \mm \ : \ \mm \in \M \ \}$, the set of supports of members of $\M$. Thus,
 $Supp \ 0= \{ \emptyset \}$ and  $Supp \ \emptyset = \emptyset$.

\item[(b)] For $S \subset \Z_{+ \infty}^{\N} $,
Let
$\langle S \rangle \ = \ \{ \mm \in FIN(\N) : \mm \leq \nu \  {\text{ for some}}\ \r \in S \}$.
We call
$\langle S \rangle$ the {\emph{ label generated by}} $S$. In particular, if $\r  : \N \to \Z_{+ \infty}$ we
will write $\langle  \r \rangle$ for
$\langle  S \rangle$ when $S = \{ \r \} $ so that $\langle \r \rangle = \{\mm \in FIN(\N) : \mm \leq \r\}$.

\item[(c)]  We will say that
$Supp \ \M$ \emph{f-contains} a set $L \subset \N$ when every finite subset of $L$ is a member of $Supp \ \M$. That is,
$\P_f L \subset Supp \ \M$ where $\P_f L$ is the set of finite subsets of $L$.
Equivalently, $\M \supset \langle \chi(L) \rangle$.

\item[(d)] For  $N \in \Z_+$, let
$\B_N = \langle (N - 1) \chi([1,N]) \rangle$.  That is,
$\mm \in \B_N$ iff $\mm < N$, i.e. $\mm_{\ell} < N$ for all $\ell \in \N$, and
and
$supp \ \mm \subset [1,N]$.

In particular, $\B_0 = \emptyset$ and $\B_1 = 0$. Thus, $\{ \B_N \}$ is an increasing sequence of finite labels with union
$FIN(\N)$, the maximum label.


\end{itemize}
\end{df}

\begin{df}\label{labeldef01b}Let  $\M$ be a label.
\begin{itemize}
\item[(a)]$\M$ is \emph{bounded} if there exists $\r \in \Z_+^{\N}$ such that $0 \ \leq \ \mm \ \leq \ \r$ for all $\mm \in \M$
i.e. $\M \subset \langle \r \rangle$. Note that $\r_{\ell} < \infty$ for all $\ell \in \N$.
We call this the \emph{Bound Condition}.

\item[(b)] $\M$ is \emph{of finite type} if it
there does not exist an infinite increasing sequence in $\M$, or equivalently,
 any infinite nondecreasing sequence in $\M$ is eventually constant. We call this the \emph{Finite Chain Condition}.

\item[(c)]  $\M$ is \emph{size bounded} if there exists
$n \in \N$ such that $size(\mm) \leq n$ for all $\mm \in \M$. We call this the \emph{Size Bound Condition}
\end{itemize}
 \end{df}
  \vspace{.5cm}

  Clearly, a finite label  is of finite type.

 For example, $\emptyset$ and $0 = \{ \00 \}$ are finite labels. $\M \not= \emptyset$ iff $\00 \in \M $.
 We call $\M$ a
 \emph{positive label} when it is neither $\emptyset$ nor $0$.

For a label $\M$ we let $[[ \M ]]$ denote the set of labels which are contained in $\M$.
Clearly, if $\M$ is bounded, of finite type, size bounded or finite then all of the labels in $[[ \M ]]$ satisfy the
corresponding property.

 Define the \emph{roof} $\r(\M) : \N \to \Z_{+\infty}$ of a label $\M$ by
 $$\r(\M)_{\ell} = \sup_{\mm \in \M} \  \{ \mm_{\ell} \} = \sup \{ r \in \Z_+ \ : \  r \chi(\ell) \in \M \ \} \leq \infty$$
Thus, $\M$ is bounded iff $\r(\M)_{\ell} < \infty $ for all $\ell$ in which case the roof is the
 minimum of the functions $\r \in \Z_+^{\N}$ which bound the elements of $\M$.  Clearly, $\r(0) = 0$. 
We will use that convention that  $ \sup \ \emptyset = 0$
 so that $\r(\emptyset) = 0$ as well.

 \begin{lem}\label{labellem02} If a label is of finite type then it is bounded.
 If a label is bounded and size bounded then it
  is of finite type. \end{lem}

 \proof If $\r(\M)_{\ell} = \infty$ then $\{ i \chi(\ell) : i \in \N \}$ is an infinite increasing sequence in $\M$.

 Now assume that $\M$ is bounded by $\r \in \Z^{\N}_+$. If $\mm^1 < \mm^2 <  \cdots $ is an
 infinite sequence of $\N$-vectors then at each step either some
  entry increases or the size increases.  Since the entries in $\M$ are bounded by  $\r$ and the size
 is assumed bounded the sequence must eventually leave $\M$. Hence, the Finite Chain Condition holds.

 $\Box$ \vspace{.5cm}

 For a label $\M$ and $\ell^* \in \N$ we define
 \begin{equation}\label{label02}
 \M \wedge [1,\ell^*]  \quad = \quad \begin{cases} \qquad \qquad  \emptyset \hspace{2cm} \mbox{when} \ \ell^* = 0,\\
 \{ \ \mm \wedge [1,\ell^*] \ : \ \mm \in \M \ \} \quad \mbox{when} \ \ell^* > 0.
 \end{cases}
 \end{equation}
 Thus, for $\ell^* \in \N, \ \M \wedge [1,\ell^*] \ = \ \{ \ \mm \in \M \ : \ supp \ \mm \subset [1, \ell^*] \ \}$.

 For a label $\M$ and an $\N$-vector $\rr$ we define
 \begin{equation}\label{label03}
 \M - \rr \quad = \quad \{ \ \ww \in FIN(\N) \ :  \ \ \ww + \rr \in \M \ \}.\hspace{2cm}
 \end{equation}
 Thus, $\M - \rr$ is the set of all non-negative vectors of the form $\mm - \rr$ for ${\mathbf m} \in \M$.
 Clearly, $(\M - \rr) - \ss = \M - (\rr + \ss) = (\M - \ss) - \rr$ (and so we can omit the parentheses) since each is the
 set of $\N$-vectors $\ww $ such that $\ww + \rr + \ss \in \M$.

 Let
 $max \ \M$ denote the set of $\N$-vectors which are maximal in $\M$. That is,
 $\nn \in max \ \M$ if
 \begin{equation}\label{label04}
 \mm \geq \nn \ \mbox{and} \ \mm \in \M \qquad \Longleftrightarrow \qquad \mm = \nn.\hspace{1cm}
 \end{equation}

 \begin{prop}\label{labelprop03} Let $\M$ be a  label and let $\rr > \00 $ be an $\N$-vector.
 \begin{itemize}
 \item[(a)]If $\M$ is of finite type then for every $\mm  \in \M$
  there exists
 $\nn \ \in max\ \M$ such that $\mm \leq \nn$.
 Hence, $\M \ = \ \langle max \ \M \rangle$.
 Thus, a label $\M$ of finite type is determined by $\max \ \M$. Also,  every $A \in Supp \ \M$ is contained in
some $B \in Supp \ \M$ which is maximal with respect to inclusion in $Supp \ \M$.

 \item[(b)] If  $\M$ f-contains some infinite set $L \subset \N$ then $\M$ is not of finite type.
Conversely, if $\M$ is bounded
  but not of finite type then it f-contains some infinite set $L$.
Thus, a bounded label is of finite type iff it does not f-contain an infinite set.

 \item[(c)] If $\M$ is  bounded and $\ell^* \in \N$ then $\M \wedge [1,\ell^*]$ is a finite label.

 \item[(d)] $\M - \rr$ is a label contained in $\M$  with $\max \ \M \ \subset \ \M \setminus (\M - \rr)$.
 If $\M$ is nonempty and bounded then $\M - \rr$ is a proper subset of $\M$.

 \item[(e)] $\M - \rr  \ \not= \ \emptyset$ iff $\rr \in \M$.

  \item[(f)] If $\Phi$ is a  set of labels then $\bigcup \ \Phi $ and $\bigcap \ \Phi $ are labels.

 \item[(g)] If $\Phi$ is a finite set of labels then $\bigcup \ \Phi $ is of finite type iff all of the labels in
 $\Phi$ are.
 \end{itemize}
 \end{prop}

 \proof (a): If $\mm$ is not maximal then  there exists $\mm_1 \in \M$ with $\mm_1 > \mm$.
 Continue if $\mm_1$ is not maximal.  This sequence can continue only finitely many steps by the Finite Chain Condition.
 It terminates at a maximal vector $\nn$. Similarly, if
 $A \in Supp \ \M$ is not contained in a maximal element then there is an increasing sequence
in $\{ A_0, A_1,\dots \}$ in $Supp \ \M$  with $A = A_0$. Then $\mm^k = \chi(A_k)$ is a strictly increasing sequence in $\M$.

 (b) If $\M$ f-contains $L = \{ \ell_1, \ell_2,\dots \}$  then $\nn^k = \Sigma_{i = 1}^{k}  \ \chi(\ell_i)$ is a strictly
increasing infinite sequence in $\M$ and so $\M$ is not of finite type.  Conversely, that if $\nn^k \in \M$ is an infinite
increasing sequence in $\M$ then $\M$ f-contains the union $L$ of the increasing sequence $ \{ supp \ \nn^k \}$ of finite sets.
If $\M$ is bounded then $L$ must be an infinite set.
%

 (c):  For any $\mu \in \Z_+^{\N}$ and $\ell^* > 0$, if $N = \max_{\ell = 1}^{\ell^*} \{ 1 + \mu_{\ell} \}$, then
 $\M \wedge [1,\ell^*] \subset \B_N$.

 (d) $\M - \rr$ satisfies the Heredity Condition and so is a label contained in $\M$. If $\mm \in \M - \rr$
 then $\mm + \rr$ is an element of $\M$ with $\mm + \rr > \mm$ since $\rr > \00 $.
 Hence, $\mm$ is not a maximal element of $\M$.

 Now assume  $\M$ is nonempty and bounded. Then $\00 = 0 \rr \in \M$.
 If $\rr_{\ell} > 0$ then for $n \in \N$ such that $n > \r(\M)_{\ell}/\rr_{\ell} $,
 $n \rr \not\in \M$. So there is a maximum $N \geq 0$ such that $N \rr \in \M$. It follows that $N \rr \in \M \setminus (\M - \rr)$.

 (e) If $\rr \in \M$ then $0 \in \M - \rr$. If $\mm \in \M - \rr $ then
 $\mm + \rr \in \M$ and so $\mm + \rr \geq \rr$ implies $\rr \in \M$.


 (f) : Obvious.

 (g) If $\Phi$ is finite collection of labels of finite type
 and $\{ \mm^i \}$ is a strictly increasing sequence of
 $\N$-vectors
 then it can remain in each member of
 $\Phi$ for only finitely many terms.  As $\Phi$ is finite, the sequence must eventually leave $\bigcup \Phi$. Hence, the
 union satisfies the Finite Chain Condition.

 If the union is of finite type then each member of $\Phi$ is of finite type by (b).

 $\Box$ \vspace{.5cm}

\begin{ex}\label{ex6} A label  $\M$ which is generated by $max \ \M$ need not be of finite type.

(a) If $\M = \langle \{ \ 2\chi(k+1) + \Sigma_{\ell = 1}^k \ \chi( \ell): k \in \N \}
\rangle$
then $\M$ is not of finite type although every $\mm \in \M$ is bounded
by an element of $max \ \M$.

(b) Let $\M = \langle \{ \ \chi(2k+1) + \Sigma_{\ell = 1}^k \ \chi(2 \ell): k \in \N \} \} \rangle$.
Clearly $\M $ f-contains $L$ the infinite set of even numbers, but every
$A \in Supp
\ \M$ is contained in
some $B \in Supp \ \M$ which is maximal with respect to inclusion in $Supp \ \M$. \end{ex}

$\Box$ \vspace{.5cm}

%
%
%
%

%

We denote by $\LAB$  the space of labels.  On $\LAB$ we define an ultrametric by
\begin{equation}\label{label06}
d(\M_1, \M_2) \  = \  \inf \ \{ \ 2^{-N} \ : \ N \in
\Z_+
 \  \mbox{and} \  \M_1 \cap \B_N = \M_2 \cap \B_N \ \}.
\end{equation}

Notice that since $\B_0 = \emptyset, \ \M_1 \cap \B_0 = \M_2 \cap \B_0$ is always
true.

\begin{lem}\label{labellem03a} (a) $d(\M_1, \M_2) = 0$ iff $\M_1 = \M_2$.

(b) The label $\emptyset$ is an isolated point of $\LAB$ with $d(\emptyset,\M) = 1$ for all $\M \not= \emptyset$.

(c) If $\NN_1 \subset \NN$ are finite labels and $\mm \in FIN(\N)$ then each of the following  is a clopen subset of $\LAB$:
\begin{gather*}
\{ \M : \M \cap \NN = \NN_1 \}, \quad
\{ \M : \M \cap \NN = 0 \}, \quad
\{ \M :  \NN \subset \M \},\\
\{ \M :  \mm \in \M \}, \quad
\{ \M :  \mm \not\in \M \}.
\end{gather*}

Furthermore, the set $\{ (\M, \M_1) : \M \cap \NN = \M_1 \cap \NN \}$ is a clopen subset of
$\LAB \times \LAB$.

(d) For any label $\M$ the set $[[\M]]$ of labels contained in $\M$ is a closed subset of $\LAB$. The set
$$ INC = \{ (\M_1,\M_2) : \M_1 \subset \M_2 \} $$
is a closed subset of $\LAB \times \LAB$.

(e) The set of finite labels is a countable dense subset of $\LAB$.

(f) The set of bounded labels is a dense $G_{\d}$ subset of $\LAB$.
\end{lem}

\proof (a) Every $\mm \in \B_N$ for some $N$.

(b) If $\M \not= \emptyset$ then $\00 \in \M \cap \B_1 $.

(c) If $\NN \subset \B_N$ then the $2^{-N}$ ball around $\M$ is either contained in or
disjoint from $\{ \M : \M \cap \NN = \NN_1 \}$.
So $\{ \M : \M \cap \NN = \NN_1 \}$ is clopen. With $\NN_1 = 0$ or $\NN$ these become  $\{ \M : \M \cap \NN = 0 \}$
and $\{ \M :  \NN \subset \M \}$. Finally, let $\NN = \langle \mm \rangle$.

The set of pairs such that $\M \cap \NN = \M_1 \cap \NN$ is the union of the set of pairs such that
$\M \cap \NN = \NN_1 = \M_1 \cap \NN$ taken of the finite set of labels $\NN_1 \subset \NN$.

(d) The complement of  $[[\M]]$ is the union of $\{ \M_1 :  \mm \in \M_1 \}$ as $\mm$ varies over $FIN(\N) \setminus \M$.
The complement of $INC$ is the union of $\{ \M_1 : \mm \in \M_1 \} \times \{ \M_2 : \mm \not\in \M_2 \}$ as $\mm$ varies
over $FIN(\N)$.

(e) $\M \cap \B_N$ is a finite label in the $2^{-N}$ ball about $\M$.  The set of finite labels
is countable since $FIN(\N)$ is countable.

(f) For each $\ell, \{ \M : \r(\M)_{\ell} = \infty \} = \bigcap_k \{ \M : k \chi(\ell) \in \M \}$ is
a closed set.  So the set of bounded
labels is $G_{\d}$. It is dense because it contains the set of finite labels.

$\Box$ \vspace{.5cm}

 Let $\M^i$ be a sequence of labels, or more generally a net of labels with $i$ in a directed set.   Define the labels
 \begin{equation}\label{label07}
 \begin{split}
 LIMSUP \ \{ \M^i \} \quad = \quad \bigcap_i  \ \bigcup_{j \geq i} \ \{ \M^j \} \ , \\
  LIMINF \ \{ \M^i \} \quad = \quad \bigcup_i  \  \bigcap_{j \geq i} \ \{ \M^j \} \ .
   \end{split}
   \end{equation}

   Clearly, $ {\mathbf m} \in LIMSUP$ iff frequently ${\mathbf m}  \in \M^i$ and
  $ {\mathbf m} \in LIMINF$ iff eventually ${\mathbf m} \in \M^i$ and so $LIMINF \subset LIMSUP$.

  As usual, if we go to a subsequence $\{ \M^{i'} \}$ with $LIMSUP'$ and $LIMINF'$ then
  $LIMINF \subset LIMINF' \subset LIMSUP' \subset LIMSUP$.

 \begin{prop}\label{labelprop04} Let $\{ \M^i \}$ be a sequence of labels.

 \begin{itemize}
 \item[(a)]The following are equivalent.
 \begin{enumerate}
  \item[(1)] The sequence  $\{ \M^i \}$ is a convergent sequence.

 \item[(2)] The sequence  $\{ \M^i \}$ is a Cauchy sequence.

 \item[(3)] For every finite label $\NN $, the sequence $\{ \M^i  \cap \NN \}$ of finite labels is eventually constant.

 \item[(4)] For every $\mm \in FIN(\N)$ either eventually $\mm\in \M^i$ or eventually $\mm \not\in \M^i$.

 \item[(5)] $LIMSUP = LIMINF$.
 \end{enumerate}
 The common value $LIMSUP = LIMINF$ is then the limit,
 and is then denoted  $ LIM \ \{ \M^i \}$.

 \item[(b)] If  $\M^i \subset \M^{i+1}$ then $LIM \{ \M^i \} = \bigcup \{ \M^i \}$.
 If $\M^i \supset \M^{i+1}$ then  $LIM \{ \M^i \} \allowbreak = \bigcap \{ \M^i \}$.

 \end{itemize}
 \end{prop}

 \proof  (a)  (1) $\Rightarrow$ (2): Obvious.

 (2) $\Rightarrow$ (3): Since $\NN \subset \B_N$ for some $N$ this
 is obvious from the definition of the ultrametric.

 (3) $\Rightarrow$ (4): If $\mm \in \B_N$ then since $\M^i \cap \B_N$ is eventually constant,
    either eventually $\mm \in \M^i$ or eventually $\mm \not\in \M^i$.

    (4) $\Rightarrow$ (5): (4) says that $\mm$ frequently in $\M^i$ implies $\mm$ is eventually in $\M^i$.

 (5) $\Rightarrow$ (1): Assume that $\M = LIMSUP = LIMINF$.  Let $\NN_1 = \M \cap \NN$. If $\mm \in \NN_1$ then eventually
 $\mm \in \M^i$ and if $\mm \in \NN \setminus \NN_1$ then eventually $\mm \not\in \M^i$.  Since $\NN$ is a finite set
 it follows that eventually $\M^i \cap \NN = \NN_1$.

 (b):  For an increasing sequence the $LIMSUP = LIMINF$ is the union and for a decreasing sequence $LIMSUP = LIMINF$ is the
 intersection.

 $\Box$ \vspace{.5cm}

 The results of part (a) apply more generally to nets.  In the metric space $\LAB$ we need only consider sequences, but
 later we will be considering pointwise convergence of maps in $\LAB^{ \LAB}$ and there we will need nets.

 \begin{cor}\label{labecor04a} The maps defined from $\LAB \times \LAB$ to $\LAB$ by $(\M_1,\M_2) \mapsto \M_1 \cap \M_2$ and
$(\M_1,\M_2) \mapsto \M_1 \cup \M_2$ are continuous. \end{cor}

 \proof Let $\{ (\M_1^i, \M^i_2) \}$ be a sequence converging to $(\M_1,\M_2)$ and let $N \in \N$.
 Eventually, $\M_1^i \cap \B_N = \M_1 \cap \B_N$ and,  eventually, $\M_2^i \cap \B_N = \M_2 \cap \B_N$. So, eventually,
$ (\M_1^i \cup \M_2^i) \cap \B_N = (\M_1^i \cap \B_N) \cup (\M_1^i \cap \B_N) $ equals $ (\M_1 \cup \M_2) \cap \B_N$ and similarly
for the intersection.

 $\Box$ \vspace{.5cm}

  \begin{prop}\label{labelprop05} Let $\{ \M^i \}$ be a  sequence of labels.

   $\mm \in LIMSUP \{ \M^i \}$ iff there is a subsequence $\{ \M^{i'} \}$ which
  is convergent with $ \mm \in LIM \ \{ \M^{i'} \} $.

$\mm \not\in LIMINF \{\M^i\}$ iff there is a subsequence
  $\{ \M^{i'} \}$ which is convergent with $\mm \not\in LIM \ \{ \M^{i'} \}$.

  \end{prop}

  \proof  The $LIMSUP$ of a subsequence is contained in the $LIMSUP$ of the original sequence
  and the $LIMINF$ of a subsequence contains the $LIMINF$ of the original sequence.
Thus, sufficiency is clear in each case.

  Let $\{ \mm_1, \mm_2,\dots \}$ be a numbering of the countable set
  $FIN(\N)$ with $\mm_1 = \mm$. Since $\mm \in \M^i$  frequently,
  we can choose  $SEQ_1$ an infinite subset of $\N$ so that $\mm_1 \in \M^i$ for all $i \in SEQ_1$.
  If  eventually $\mm_2 \in \M^i$ for  $i \in SEQ_1$ let these values of $i$ define $SEQ_2 \subset SEQ_1$.
  Otherwise, let $SEQ_2$ be the $i \in SEQ_1$ such that ${\mathbf m}_2 \not\in \M^i$. Inductively
  we define a decreasing sequence $SEQ_n$ of infinite subsets of $\N$ such that $p \leq n$ implies either
  $\mm_p \in \M^i$ for all $i \in SEQ_n$
  or for no $i \in SEQ_n$.  Diagonalizing, we obtain a convergent
  subsequence whose limit contains $\mm$. That is, if $i_n$ be the $n^{th}$ element of $SEQ_n$, then
  $\{ M^{i_n} \}$ is convergent and the limit contains $\mm$.

 Alternatively, if $\mm \not\in LIMINF$ we begin by choosing $SEQ_1$ so that
 $\mm_1 \not\in \M^i$ for all $i \in SEQ_1$ and continue as before.

  $\Box$ \vspace{.5cm}

\begin{cor}\label{labelcor04a}
$\LAB$ is a compact, zero-dimensional metric space with $\emptyset$ the only
isolated point.
\end{cor}

\proof  By Propositions  \ref{labelprop05}  every sequence in $\LAB$
has a convergent subsequence, i.e. the metric space satisfies the Bolzano-Weierstrass property and so is
compact. It is zero-dimensional because it has an ultrametric. If $\NN$ is a finite, nonempty label
then $\NN \cup \{ \chi(\ell) \}$ is a sequence of finite labels which are eventually distinct. The sequence
 converges to $\NN$ as $\ell \to \infty$ and so $\NN$ is not an isolated point.
Since the set of finite labels is dense, it follows that  no nonempty
label is isolated.

$\Box$ \vspace{.5cm}

\begin{lem}\label{labellem06a} Let $\Phi$ be a compact subset of $\LAB$. If $\{ \mm^i : i \in\N \}$ is a nondecreasing sequence
in $\bigcup \ \Phi$ then there exists $\M \in \Phi$ such that $\mm^i \in \M$ for all $i$. \end{lem}

\proof Assume $\mm^i \in \M^i \in \Phi$ for all $i$. By compactness some subsequence $\{ \M^{i'} \}$ converges to
$\M \in \Phi$. For each $k$, $i' > k$ implies $\mm^k \leq \mm^{i'} \in \M^{i'}$. That is, each $\mm^k$ is eventually
in $\M^{i'}$ as $i' \to \infty$. Hence, $\mm^k \in \M$ for all $k$.

$\Box$ \vspace{.5cm}

 \begin{prop}\label{labelprop06b} Let $\L \subset \LAB$ be either the set of bounded labels or the set
 of labels of finite type. A subset $\Phi \subset \L$ is compact iff $\Phi$ is closed in the relative topology of
 $\L$ and $\bigcup \ \Phi \in \L$.  In particular, the union of a compact set of bounded labels is a bounded label.
 \end{prop}

 \proof  If $\bigcup \ \Phi \in \L$ then the compact set $[[\bigcup \ \Phi ]]$ is contained in $\L$ and
 since $\Phi \subset [[\bigcup \ \Phi  ]]$ is closed relative to
$\L$,
 it is closed
 relative to $[[ \bigcup \ \Phi ]]$ and so is itself compact.

 Now assume that $\Phi$ is compact.

 If $\bigcup \ \Phi $ is not bounded then for some $\ell \in \N$ the
 strictly  increasing sequence $\{ \mm^i = i \chi(\ell) \}$ is in
 $\bigcup \ \Phi $ and so Lemma \ref{labellem06a} implies that the sequence is contained in some $\M \in\Phi$ and
 so $\M$ is unbounded.

 Similarly, if $\bigcup \ \Phi $ is not of finite type then there exists a
 strictly  increasing sequence $\{ \mm^i \}$ in  $\bigcup \ \Phi $ and so
 again Lemma \ref{labellem06a} implies that the sequence is contained in some $\M \in\Phi$ and
 so $\M$ is not of finite type.

 In each case the contrapositive says that $\Phi \subset \L$ implies  $\bigcup \ \Phi \in \L$.

 $\Box$ \vspace{1cm}

\subsection{The action of $FIN(\N)$ on $\LAB$}\label{ssFIN}

$\qquad$

\vspace{.5cm}

  Given an $\N$-vector $\rr$ we define the map $P_{\rr}$ on $\LAB$ , by $P_{\rr}(\M) = \M - \rr$.

\begin{prop}\label{labelprop06} For each $\rr \in FIN(\N)$ the function $P_{\rr}$ on $\LAB$ is continuous.
In particular, if $\{ \M^i \}$
is a convergent sequence of labels then $\{ \M^i - \rr \}$ is convergent with $LIM \{ \M^i - \rr \} \ = \ LIM \{ \M^i \} - \rr$.
\end{prop}

\proof Let $N(\rr)$ be the minimum value such that $ \rr \in \B_{N(\rr)}$. Notice that $N(\mm + \rr) \leq N(\mm) + N(\rr)$ because
$\mm + \rr < N(\mm) + N(\rr)$ and $supp \ \mm + \rr = (supp \ \mm) \cup (supp \ \rr) \subset [1,\max(N(\mm),N(\rr))]$.

It follows that for labels $\M_1, \M_2$

\begin{align}\label{label08}
\begin{split}
\M_1 \cap \B_{N + N(\rr)} \ &= \ \M_2 \cap \B_{N + N(\rr)} \\
\Longrightarrow \quad ( \M_1 - \rr) \cap &\B_{N} \ = \ ( \M_2 - \rr)  \cap \B_{N}.
\end{split}
\end{align}
For if $\mm \in (\M_1 - \rr) \cap \B_{N}$ then $\mm + \rr \in \M_1 \cap \B_{N + N(\rr)} = \M_2 \cap \B_{N + N(\rr)}$.
Hence, $\mm \in \M_2 - \rr$. Since $\mm \in \B_N$, it follows that $\mm \in (\M_2 - \rr) \cap \B_{N}$. Symmetrically for
reverse inclusion.

From (\ref{label08}) it follows that $d(\M_1,\M_2) < 2^{-N-N(\rr)}$ implies $d(P_{\rr}(\M_1),P_{\rr}(\M_2))
\allowbreak
 < 2^{-N}$.
This shows that $P_{\rr}$
is Lipschitz with Lipschitz constant at most $2^{N(\rr)}$.

$\Box$ \vspace{.5cm}

\begin{cor}\label{labelcor06a} The map $P : FIN(\N) \times \LAB  \ \to \ \LAB$ given by
$ (\rr , \M) \mapsto \M - \rr = P_{\rr}(\M) $
is a continuous monoid action of $FIN(\N)$ on $\LAB$.

The action  is \emph{faithful} i.e.
if $P_{\rr}(\M) = P_{\ss}(\M)$ for all $\M$ then $\rr = \ss$. \end{cor}

\proof It is an action since  $(\M - \rr_1) - \rr_2 = \M - (\rr_1 + \rr_2)$ for $\N$-vectors $\rr_1, \rr_2$ and
$\M - \00 = \M$. It is a
continuous action by Proposition \ref{labelprop06}.

For a label $\rr$ let $\M$  be the finite label $\langle  \rr \rangle$.  Since $\{ \rr \} = max \ \M$,
$P_{\rr}(\M) = 0$. If $P_{\ss}(\M) = 0$ then $\ss \in \M$ and so $\ss \leq \rr$.  Clearly,
$\rr - \ss \in  P_{\ss}(\M)$ and so $\rr - \ss = \00 $.

$\Box$ \vspace{.5cm}
%

Notice that $FIN(\N)$ is the free abelian monoid generated by $\{ \ \chi(\ell) \ : \ \ell \in \N \}$ and it is
a submonoid of the free abelian group consisting of the members of $\Z^{\N}$ with finite support. Furthermore,
$\rr_1 + \ss = \rr_2 + \ss$ implies $\rr_1 = \rr_2$.  In particular, $\rr + \rr = \rr$ only when $\rr = \00 $.
Also, $\rr + \ss = \00 $ iff $\rr = \ss = \00 $. Thus, $FIN(\N)$ is a cancelation monoid without inverses and
Lemmas \ref{applem03} and \ref{applem04} apply to $FIN(\N)$.

Giving $FIN(\N)$ the discrete topology, we obtain on the Stone-\v{C}ech compactification $\b FIN(\N)$ the structure of
an Ellis semigroup with product
which extends the addition on $FIN(\N)$ and is such that $Q \mapsto QR$ is continuous for any $R \in \b FIN(\N)$.
Let $\b' FIN(\N) \ = \ \b FIN(\N) \setminus \{ 0 \}$ and $\b^* FIN(\N) \ = \ \b FIN(\N) \setminus  FIN(\N)$.
Notice that since $FIN(\N)$ is discrete, it is the set of isolated points in $\b FIN(\N)$. The elements of
$FIN(\N)$  are continuous on $\b FIN(\N)$ and commute with all elements of $\b FIN(\N)$. Thus, the submonoid $FIN(\N)$ acts
continuously on $\b FIN(\N)$.
The action of $FIN(\N)$ extends to an Ellis action of $\b FIN(\N)$ on $\LAB$. See Appendix \ref{appendix-StoneCech}

\begin{prop}\label{labeltheo06x} (
(a) If $\NN \subset \M$ then $Q(\NN) \subset Q(\M) \subset \M$ for all $Q \in \b FIN(\N)$.

(b)  The sets $\b' FIN(\N)$ and $\b^* FIN(\N)$ are closed, $FIN(\N)$ invariant subsets of $\b FIN(\N)$ and so are ideals in the
Ellis semigroup.

(c)  Every nonempty, closed sub-semigroup of $\b FIN(\N)$ contains an idempotent and all the idempotents of $\b' FIN(\N)$
lie in $\b^* FIN(\N)$.

(d) For all $Q \in \b FIN(\N)$, $Q(\M) = FIN(\N)$ iff $\M = FIN(\N)$.
\end{prop}

\proof
(a) Clearly, $\NN \subset \M$ implies $P_r(\NN) \subset P_r(\M) \subset \M$ and  the inclusions are preserved in the limit by
Lemma \ref{labellem03a}.

(b) See  Lemmas \ref{applem03} and \ref{applem04} .

(c) The existence of idempotents is the
Ellis-Numakura Lemma,
Lemma \ref{applem01}.  We saw above that $\00 $ is the only idempotent
in $FIN(\N)$ and so there are no idempotents in $\b' FIN(\N) \setminus \b^* FIN(\N).$

(d) For all $\rr, \ P_{\rr}(FIN(\N)) = FIN(\N)$ and so $FIN(\N)$ is fixed by all $Q \in \b FIN(\N)$. If $\M \not= FIN(\N)$
then $Q(\M) \subset \M$ and so does not equal $FIN(\N)$.

$\Box$ \vspace{.5cm}

Recall that
\begin{equation}\label{extra}\begin{split}
\rr \in \M \quad \Longleftrightarrow  \quad \00 \in P_{\rr}(\M). \\
\rr \not\in \M \quad \Longleftrightarrow  \quad \emptyset = P_{\rr}(\M).
\end{split}\end{equation}
For a discrete set $\Gamma$, a subset $A$ of $\Gamma$ and its complement have disjoint
closures in $\b \Gamma$ because the characteristic function of $A$ extends to a continuous function
on $\b \Gamma$. It follows that for any label $\M$, the closure in $\b FIN(\M)$ of $\M \subset FIN(\M)$
is the clopen set $\{ Q \in \b FIN(\N) : \00 \in Q(\M) \}$ with complement $\{ Q \in \b FIN(\N) :
\emptyset = Q(\M) \}$. In particular, if $\M$ is a finite label then $Q(\M) = \emptyset$ for all
$Q \in \b^* FIN(\N)$.

 Let
  $\Theta(\M) $ be the closure in the space of labels of the set
  $\{ \ P_{\rr}(\M) = \M - \rr  \ : \ \rr $ an $\N$-vector $ \}$. That is, $\Theta(\M)$ is the orbit closure of $\M$ with
  respect to the $FIN(\N)$ action or, equivalently, $\Theta(\M) = \b FIN(\N)\cdot \M $. Since $[[\M]]$ is closed
  and invariant, $\Theta(\M) \subset [[\M]]$.
  Even in the finite case, it can happen that the inclusion is proper.

\begin{ex}\label{proper}
Set $\M = \langle
\chi(1) +  \chi(2),
2\chi(2) +  \chi(3)  \rangle$, and let
$\NN = \langle
\chi(1) +  \chi(2),
\chi(2) +  \chi(3) \rangle$.
It is easy to check that $\NN \in [[\M]] \setminus \Theta(\M)$.
\end{ex}

For $FIN(\N)$,  $P_{\rr}(FIN(\N)) = FIN(\N)$ for all $\rr \in FIN(\N)$ implies $\Theta(FIN(\N))
\allowbreak = \{ FIN(\N) \}$.

The action of $FIN(\N)$ on $\LAB$ restricts to an action on the closed invariant set $\Theta(\M)$ for any label $\M$.

  \begin{lem}\label{labellem05c} (a) For any label $\M \not = FIN(\N)$, $\emptyset \in \Theta(\M)$.

  (b)If $\M$ is nonempty and
  bounded then   $0 \in \Theta(\M)$.

  (c) If $\M$ is nonempty and
  bounded then the action of $FIN(\N)$ on $\Theta(\M)$ is weakly faithful. In fact, $P_{\rr}(\M) = \M$ only
  for $\rr = 0$. \end{lem}

  \proof (a) : If $\rr \not\in \M$ then $\M - \rr = \emptyset$ and so $\emptyset \in \Theta(\M)$.

  (b): If
 $\rr \in max \ \M$ then $\M - \rr = 0$.

 Now assume that $\M$ is bounded so that each $ \M \wedge [1,i]$ is a finite label. Let $\rr^i$ be maximal element
 of $ \M \wedge [1,i]$.
Clearly $ \00 \in LIMINF \{ \M - \rr^i \}$. On the other hand if $\ww \in LIMSUP \{ \M - \rr^i \}$
then for some $j$ we have $supp \  \ww \subset [1,j]$. Frequently $\ww + \rr^i \in \M$, and so there
exists
 $i \ge j$, $\ww + \rr^i \in \M \wedge [1,i]$. Maximality implies $\ww = \00 $. That is, $0 = LIM \{ \M - \rr^i \}$.

 (c): Proposition \ref{labelprop03}(d) implies that $P_{\rr}(\M)$ is a proper
 subset of $\M$ when $\M$ is bounded and non-empty and
 $\rr > 0$.


 $\Box$ \vspace{.5cm}

 {\bfseries Remark:}  Notice that $P_{\rr}(\M) = 0 $ iff $\rr \in max \ \M$ and so if $max \ \M = \emptyset$ then
 $\M - \rr \not= 0$ for any $ \rr \in FIN(\N)$.
 \vspace{.5cm}

  Let $\Theta'(\M) $ be the closure in the space of labels of the set
  $\{ \ \M - \rr  \ : \ \rr \in FIN(\N) \ $ with $ \rr > \00  \}$. Thus, $\Theta(\M) = \Theta'(\M) \cup \{ \M \}$ and
 $\Theta'(\M) = \b' FIN(\N) \cdot \M  $.

\begin{df}\label{df,recurrent}
Call $\M$ a \emph{recurrent label} if $\M \in \Theta'(\M)$.
Equivalently, $\M$ is recurrent if there exists a sequence $\{ \rr^i > \00 \}$ such
that $\M = LIM \{ \M - \rr^i \}$ and so for all $\mm \in \M$ eventually $\mm + \rr^i \in \M$.
 \end{df}

 This is the notion of recurrence for
 the
 system $(\LAB, FIN(\N))$ as defined in Section 1.

   If $\mm \in max \ \M$
  then $\{ \NN : \mm \in \NN \}$ is a clopen subset of $\LAB$ which contains $\M$ and is disjoint from $\Theta'(\M)$.
  So if $\M$ is a recurrent label then $max \ \M = \emptyset$. In particular, if
  $\M$ is of finite type and nonempty then $\M \not\in \Theta'(\M)$ and so is not recurrent.

 For example, $\M = FIN(\N)$ is a recurrent label since then $\M = P_{\rr}(\M)$ for all $\rr \in \M$.
 By Proposition \ref{labelprop03} (d) $P_{\rr}(\M)$ is a proper subset of $\M$ when $\M$ is bounded
and
$\rr > \00 $
and so then
cannot equal $\M$. Nonetheless, there are bounded recurrent labels.

 For example, if $\M = \langle \chi(L) \rangle$ for some infinite $L \subset \N$ then
   \begin{equation}\label{labelchiL}
 \M =  LIM \{ \M - \chi(\ell^i) \}
 \end{equation}
 when $\{ \ell^i \}$ is a sequence of distinct elements of $L$.
%

  Define the \emph{isotropy set of $\M$}
  \begin{equation}\label{labeliso}
  ISO(\M) \ = \ \{ \ Q \in \b FIN(\N) \ : \ Q(\M) = \M \ \}
  \end{equation}

 \begin{prop}\label{labelprop05a1} Let $\M$ be a label.

\begin{enumerate}
 \item[(a)] $ISO(\M)$ is a closed submonoid of $\b FIN(\N) $ and for $Q_1, Q_2 \in \b FIN(\N)$
 the product  $ Q_1 Q_2$ is in $ ISO(\M)$ iff both $ Q_1$ and $Q_2$ are in $ISO(\M)$.

\item[(b)] The following are equivalent.
\begin{itemize}
\item[(i)] $\M$ is recurrent.

\item[(ii)] There exists $Q \in \b' FIN(\N)$ such that $Q(\M) = \M$.

\item[(iii)] $ISO(\M) \cap \b' FIN(\N)$ is a nonempty, closed subsemigroup of \allowbreak $\b FIN(\N)$.

\item[(iv)] There exists an idempotent $Q \in \b^* FIN(\N)$ such that $Q(\M) = \M$.
 \end{itemize}

 \item[(c)] If $Q$ is an idempotent in $\b^* FIN(\N)$ then $Q(\M)$ is a recurrent element of $\Theta'(\M)$.
 In particular, if
  $\M$ is of finite type then $Q(\M) = \emptyset$.

 \item[(d)] If $\NN$ is a recurrent label with $\NN \subset \M$,
 then there is a recurrent label $\M_{1} \in \Theta(\M)$ such that $\NN \subset \M_{1}$.

 \item[(e)] If $\rr \in FIN(\N)$, then $ISO(\M) \subset ISO(P_{\rr}(\M))$. So if $\M$ is a recurrent label then
 $P_{\rr}(\M)$ is a recurrent label.
  \end{enumerate}
  \end{prop}

  \proof (a) Since $Q \mapsto Q(\M)$ is continuous, $ISO(\M)$ is closed.  It is clearly closed under composition
  and so is a semigroup.  Since $P_{\00 } = id, \ ISO(\M)$ contains $ \00 $ and so is a monoid.

  For $Q, P \in \b FIN(\N)$ if $Q(\M) \not= \M$
 then $\NN = Q(\M)$ is a proper subset of $\M$. It follows that $PQ(\M) = P(\NN) \subset \NN $ and so it,too,
 is a proper subset of $\M$ for any $P$. Also,
 $P(\M) \subset \M$ and so $QP(\M) \subset Q(\M) = \NN$ is also a proper subset of $\M$.
 Thus, $\{ Q \in \b FIN(\N) : Q(\M) \not= \M \}$ is a two-sided ideal (though it is not closed when $\M$ is recurrent).
 It follows that $Q_1 Q_2(\M) = \M$ implies $Q_1(\M) = \M$ and $Q_2(\M) = \M$.

 (b) (i) $\Rightarrow$ (ii): If $\{ \M - \rr^i \} \to \M$ with all $\rr^i > 0$ then
 $Q(\M) = \M$ for any limit point in $\b FIN(\N)$ of
the sequence $\{ P_{\rr^i} \}$. Since, $0$ is an isolated point of $\b FIN(\N)$, it follows that $Q \in \b' FIN(\N)$.

 (ii) $\Rightarrow$ (iii): As the intersection of two closed subsemigroups, $ISO(\M) \cap \b' FIN(\N)$ is a closed
 subsemigroup.  It is nonempty by (ii).

  (iii) $\Rightarrow$ (iv): The nonempty, closed subsemigroup $ISO(\M) \cap \b' FIN(\N)$ contains an idempotent which
  cannot lie in
  $FIN(\N) \setminus \{ \00 \}$
  and so is in
  $\b^* FIN(\N)$.

   (iv) $\Rightarrow$ (i): $\Theta'(\M) =  \b' FIN(\N) \cdot \M$ and so $\M \in \Theta'(\M)$ if a nonzero idempotent
   fixes $\M$.

   (c): If $Q \in \b'FIN(\N)$ then $Q(\M) \in \Theta'(\M)$. If $Q$ is an idempotent then $Q(Q(\M)) = Q(\M)$ implies that
   $Q(\M)$ is recurrent. If $\M$ is of finite type, then every nonempty label in $[[ \M ]]$ is of finite type and so
   is not recurrent. It follows that when $\M$ is of finite type $Q(\M)$ must be empty.

   (d): If $\NN$ is recurrent then there exists an idempotent $Q \in \b^* FIN(\N)$ such that $Q(\NN) = \NN.$
   Since $\NN \subset \M$, $\NN = Q(\NN) \subset Q(\M) \subset \M$.  Let $\M_1 = Q(\M)$.

   (e): If $Q(\M) = \M$ then $Q(P_{\rr}(\M)) = P_{\rr}(Q(\M)) = P_{\rr}(\M)$ because $Q$ and $P_{\rr}$ commute.

 $\Box$ \vspace{.5cm}

\begin{df}\label{df,st-rec}
  Call $\M$ a \emph{strongly recurrent label} if $\M$ is bounded and infinite and for every $\mm \in \M$,
  there is a finite set $F(\mm) \subset \N$
  such that $\M - \mm \supset \{ \ \ww \in \M \ : \ (supp \  \ww) \ \cap \  F(\mm) = \emptyset \} $.

 Call a label $\NN$ a
 \emph{strongly recurrent set for a bounded label}
 $\M$ if $\NN$ is infinite, $\NN \subset \M$ and for every
 $\mm \in \M$,
  there is a finite set $F(\mm) \subset \N$
  such that $\M - \mm \supset \{ \ \ww \in \NN \ : \ (supp \  \ww) \ \cap \  F(\mm) = \emptyset \} $.
 Thus, $\M$ is strongly recurrent iff $\M$ itself is a strongly recurrent set for $\M$.
\end{df}
\vspace{.5cm}

Notice that $\NN \subset \M$ implies that $\NN$ is bounded and $\M$ is infinite.

 \begin{prop}\label{labelprop05a2} Let $\M$ be a  nonempty label.

\begin{enumerate}
\item[(a)] If $   \M \ = \ LIM \{ \M - \rr^i \}$ and $\M$ is bounded, then $ LIMSUP \{ supp \  \rr^i \} = \emptyset$.
If, in addition, $\rr^i > \00 $ for all $i$ then $\bigcup_i supp \  \rr^i $ is infinite.

 \item[(b)]  Assume $\NN$ is a strongly recurrent set for a bounded label $\M$. If
 $\{ \rr^i \}$ is a sequence of elements of $\NN$ such that
 $LIMSUP \{ supp \  \rr^i \} = \emptyset$, then
 $\M = LIM \ \{\M - \rr^i \}$.
 In particular, $\M$ is recurrent if it has a strongly recurrent set,
 and every strongly recurrent label is recurrent.

\item[(c)] Assume $\M$ is bounded and infinite.
$\M$ is strongly recurrent iff for every sequence
$\{ \rr^i \}$ of elements of $\M$,
 $LIMSUP \{ supp \  \rr^i \} = \emptyset$ implies
  $\M = LIM \ \{\M - \rr^i \}$.

 \item[(d)] If $\M$ is bounded and recurrent then there exists an infinite set $L \subset \N$ such that
 $\langle \chi(L) \rangle $ is a strongly recurrent set for
 $\M$.

 \item[(e)] The following conditions are equivalent.
 \begin{itemize}

  \item[(i)] If $\mm_1, \mm_2 \in \M$ with disjoint supports then $\mm_1 + \mm_2 \in \M$.

 \item[(ii)]  $\M$ is a sublattice of $FIN(\N)$, i.e. if  $\mm_1, \mm_2 \in \M$ then
 $\mm_1 \vee \mm_2 \in \M$.

  \item[(iii)] $\M \ = \ \langle \r \rangle $ for some $\r \in \Z_{+\infty}^{\N} $.

 \item[(iv)] $\M \ = \ \langle \r(\M) \rangle $.

 \end{itemize}

If these conditions hold, and, in addition, $\M$ is bounded and infinite,  then $\M$ is strongly recurrent.
\end{enumerate}

 \end{prop}

 \proof (a) If $\ell \in supp \  \rr^i $, then
  $\r(\M)_{\ell} \cdot \chi(\ell) \in \M   \setminus (\M - \rr^i)$. So it cannot happen that
  $\ell \in supp \  \rr^i $ infinitely often.  Thus, for any $\ell \in \N$, eventually $\ell \not\in supp \  \rr^i $.

If $\rr^i > \00 $ then $supp \  \rr^i$ is nonempty.  If an infinite sequence of nonempty subsets of $\N$ has a finite union, then
some $\ell \in \N$ must lie in infinitely many of them.

(b) For any $\mm \in \M$, $\mm \in \M - \rr^i$ as soon as $F(\mm) \ \cap \  supp \  \rr^i = \emptyset$ which happens eventually
since  $LIMSUP \{ supp \  \rr^i \} = \emptyset$.

(c) If  $\M$ is bounded and infinite but not strongly recurrent then there exists $\mm \in \M$ and
for every $F$ finite subset of $\N$ there exists $\nn \in \M$ with $supp
\  \nn \ \cap \  F = \emptyset$
but with $\mm + \nn \not\in \M$. Note that this implies  $\nn > \00 $.

Let $F_1 = supp \ \mm$ and choose a positive $\rr_1 \in \M$ with support disjoint
from $F_1$ and is such that $\mm + \rr^1 \not\in \M$.
Let $F_2 = F_1 \cup supp \ \rr^1$.  Inductively, we build an increasing sequence of finite sets $\{ F^i \}$ and
positive elements $\rr^i \in \M$ such that $supp \ \rr^i \subset F^{i+1} \setminus F^i$ and $\mm + \rr^i \not\in \M$.
Since the supports are disjoint,  $LIMSUP \{ supp \  \rr^i \} = \emptyset$. Because $\mm \not\in   LIM \ \{\M - \rr^i \}$,
the limit is not $\M$.

The converse follows from (b).

(d) Since $\M$ is recurrent there exists a sequence $\{ \rr^i > \00 \}$
be such that $\M = LIM \{ \M - \rr^i \}$. Since $\M$ is bounded, it follows from (a)
that $\bigcup_i supp \  \rr^i $ is infinite.

Choose $\ell_1 \in supp \  \rr^1$
and let $i_1 = 1$ and $\NN_1 = \{ \chi(\ell_1) \} \cup \M \wedge [1,1]$.

There exists $\rr^{i_2}$ with
$\rr^{i_2}_{\ell_2} > 0$
and
$\ell_2$ not in the support of a member of
$\NN_1$ and is such that $\mm +
\rr^{i_2} \in \M$ for all $\mm \in  \NN_1$. Let
$\NN_2 = \{ \mm + \chi(\ell_2) : \mm \in \NN_1 \} \cup \M \wedge [1,2]$.

Inductively, we can choose $\ell_{k+1}$ such that
 $\mm + \chi(\ell_{k+1}) \in \M$ for all $\mm \in  \NN_k$ and with $\ell_{k+1}$ not in the support of any member of $\NN_k$.
 Let
 $  \NN_{k+1}  = \{ \mm + \chi(\ell_{k+1}) : \mm \in \NN_k \} \cup \M \wedge [1,k+1]$.

 By the inductive construction
 if $\mm \in \NN_{k-1}$ then $\mm + \Sigma_{i = k}^j \chi(\ell_i) \in \M$
 for all $j \geq k$. With $\mm = 0$ this says that $\Sigma_{i = 1}^j \chi(\ell_i) \in \M$ for all $j$.

 Let $L = \{ \ell_k \}$ and
 $\NN = \langle \chi(L) \rangle$.
 For any $\mm \in \M$ there exists $k$ such that $\mm \in \NN_{k-1}$. Let $F(\mm) = [1,\ell_{k-1}]$.
If $\ww \in \NN$ with support disjoint from $F(\mm)$ then $\ww \leq \Sigma_{i = k}^j \chi(\ell_i)$ for some $j > k$.
Hence, $\ww + \mm \in \M$  and so $\NN$
is a strongly recurrent set for $\M$.

 (e)
 (iv)  $\Rightarrow$ (iii)  $\Rightarrow$ (ii): Obvious.

(ii) $\Rightarrow $ (i): If $\mm_1$ and $\mm_2$ have disjoint supports
 then $\mm_1 \vee \mm_2 = \mm_1 + \mm_2$.

(i)  $\Rightarrow $ (iv): For any  $\M$ and $n \in \N$, $n \leq\r(\M)_{\ell}$ implies $n \chi(\ell) \in \M$.
 So if $\mm \leq \r(\M)$ then
 (i) (and induction) implies that $\mm = \Sigma \ \{ \ \mm_{\ell} \chi(\ell) \ : \ \ell \in supp \ \mm \ \}$ is in $\M$.
%

Now assume that $\M$ is bounded and infinite so that $\bigcup \ Supp \ \M$. For any $\mm \in \M$, let $F(\mm) = supp \  \mm$.
By (i), $\ww \in \M$ and $supp \ \ww \cap F(\mm) = \emptyset$ imply $\mm + \ww \in \M$.

 $\Box$ \vspace{.5cm}

\begin{remark} Let $FIN^{\ell} \ = \ \{ \ \mm \in FIN(\N) \ : \ supp \ \mm \ \cap \  [1,\ell] \ = \ \emptyset \ \}$. It is
not hard to show from the above Proposition
\ref{labelprop05a2} that, if $\M$ is a strongly recurrent label $\M$, then
$\bigcap_{\ell} \ \ol{\M \ \cap \  FIN^{\ell}} \ = \ ISO(\M)$,
where the closure is taken in $\b FIN(\N)$. \end{remark}
\vspace{.5cm}

The following directly illustrates the relationship between recurrence and the failure of the finite type condition.

\begin{cor}\label{labelcor05a3}  If $\{ \mm^i \}$ is a strictly increasing infinite sequence in a label $\M$
 then $\r  \in \Z_{+ \infty}^{\N}$ defined by
 $\r_{\ell} = \sup_i \{ \mm^i_\ell \}$ satisfies $\r \ \leq \ \r(\M)$ and
 $\NN \ = \ \langle \r \rangle \ = \ \bigcup_i \{ \langle \mm^i \rangle  \}$
  is a  recurrent label with $\NN \subset \M$ and $\mm^i \in \NN$ for
 all $i$. If $\M$ is bounded then $\NN$ is strongly recurrent.
\end{cor}

 \proof  Since the sequence is strictly increasing it is clear that
 $$\langle \r \rangle \ = \ \bigcup_i \{ \langle \mm^i \rangle  \} \subset \M.$$
 If $\M$ is bounded, then $\langle \r \rangle$ is bounded and infinite and so is strongly recurrent by Proposition
 \ref{labelprop05a2} (e).

 %

 $\Box$ \vspace{.5cm}

 \begin{cor}\label{labelcor05b} (a) For a label $\M$ the following are equivalent:
 \begin{itemize}
 \item[(i)] $\M$ is of finite type.

  \item[(ii)] The only recurrent label which is a subset of $\M$ is $\ \emptyset$.

  \item[(iii)] The only recurrent label which is contained in
 $\Theta(\M)$ is $ \emptyset$.

   \item[(iv)] $Q(\M) = \emptyset$ for every idempotent $Q$ in $\b^* FIN(\N)$.
\end{itemize}

 (b) For a label $\M, \ \ max \ \M = \emptyset$ iff $\M$ is the union of the recurrent labels contained in $\M$.
 Thus:
\begin{gather*}
 \{\M : \text{finite type}\} \subset \{\M : max \ \M \not=\emptyset\} = \\
 \{\M : \text{not a union of recurrent labels}\}
 \subset \{\M : \text{not recurrent}\}.
\end{gather*}
\end{cor}

 \proof (a): (i) $\Rightarrow$ (iv): Part (c) of Proposition \ref{labelprop05a1}.

 (iv) $\Rightarrow$ (ii): If $Q(\M) = \emptyset$ then $Q(\NN) = \emptyset$ for all $\NN \subset \M$.

 (ii) $\Leftrightarrow$ (iii): Proposition \ref{labelprop05a1} (d).

 (ii) $\Rightarrow$ (i): If $\M$ is not of finite type then it contains a strictly increasing infinite sequence
 and so by Corollary \ref{labelcor05a3}
  it contains a recurrent label.

 (b): If $\mm \in \max \ \M$ and $\mm \in \NN \subset \M$ then $\mm \in max
\  \NN$ and so $\NN$ is not recurrent. That is,
 $max \ \M$ is disjoint from all nonempty recurrent labels contained in $\M$.

 If $max \ \M = \emptyset$ and $\mm \in \M$ then inductively we can define a strictly increasing
 sequence $\{ \mm^i \}$
 in $\M$ with $\mm =
 \mm^1$. By Corollary \ref{labelcor05a3} there is a recurrent label $\NN \subset \M$ with $\mm \in \NN$.

 $\Box$ \vspace{.5cm}

 \begin{cor}\label{labelcor05b1} Assume $\M$ is a nonempty, bounded label.
 If $\M$ is a recurrent label then $\Theta(\M)$ is a Cantor set.
 If $\M$ is  not of finite type, then $\Theta(\M)$ is uncountable. \end{cor}

 \proof  If $\M$ is a nonempty recurrent label then for some sequence $\{ \rr^i \}$ of positive elements of $\M$,
 $\{ \M - \rr^i \}$ converges to $\M$.  Since $\M$ is bounded, each of the $\M - \rr^i$ is a proper subset of $\M$
  and each lies in
 $\Theta(\M)$. It follows that $\M$ is not an isolated point of $\Theta(\M)$.  For each $\rr \in \M$, the label
 $\M - \rr$ is nonempty and bounded and it is recurrent by Proposition \ref{labelprop05a1}(e).  Hence, $\M - \rr$ is not isolated
 in $\Theta(\M - \rr) \subset \Theta(\M)$. It follows that $\{ \ \M - \rr \ : \ \rr \in \M \ \}$ is a dense subset of
 $\Theta(\M)$ no point of which is isolated and so no point of $\Theta(\M)$ is isolated.  On the other hand, $\Theta(\M)$
 is a nonempty, compact, ultra-metric space. It follows that it is a Cantor set.

 If $\M$ is not of finite type then there exists a nonempty recurrent $\NN \in \Theta(\M)$ by Corollary \ref{labelcor05b}(a).
 Hence, $\Theta(\M)$ contains the Cantor set $\Theta(\NN)$.

 $\Box$ \vspace{.5cm}

\begin{cor}\label{labelcor05b2} For any label $\M$, $\M$ is the only $FIN(\N)$-transitive point in
$\Theta (\M)$.
 \end{cor}

\proof
Suppose $\NN \in \Theta(\M)$ is a transitive point. Then there are $P, Q \in \beta FIN(\N)$
with $P(\M) = \NN$ and $Q(\NN) = \M$. Thus $QP \in ISO(\M)$ and it follows, by
Proposition  \ref{labelprop05a1}(a), that $P \in ISO(\M)$ and so $\NN \ = \ \M$.

$\Box$ \vspace{.5cm}

\begin{remark}
Note the sharp contrast with the case where the acting semigroup is a group.
In fact, for a dynamical system $(X,\Gamma)$, where $X$ is a compact metric space
and $\Gamma$ is a group,
the existence of one dense orbit implies that the set $X_{tr}$ of transitive
points forms a dense $G_\delta$ subset of $X$. See Proposition  \ref{transprop}.
\end{remark}

\begin{ex}\label{ex10cmoved} It can happen that
 $\Theta(\M)$ is uncountable even with $\M$ of finite type.

 Define a bijection
  $w \mapsto \ell_w$,
 from the set of finite words on the alphabet $\{ 0, 1\}$ onto $\N$. For $x \in \{0,1\}^{\N}$
 let $w_k(x)$ be the initial word of length $k$ in $x$, that is, $w_k(x) = x_1x_2\dots x_k $, for $k \in \N$.

 Let
 \begin{equation}\label{tow26}
 \begin{split}
 M_x \  = \ \{ \ \ell_{w_k(x)} \ : \ k \in \N \ \} \hspace{3cm} \\
 \M_x \ =  \ \{ \00 \} \cup \{ \ \chi(\ell) \ : \ \ell \in M_x \ \}, \hspace{3cm} \\
\M_x^{(2)} \ = \ \M_x \oplus \M_x \ = \ \langle \{ \ \chi(\ell^1) + \chi(\ell^2) \ : \ \ell^1, \ell^2 \in M_x  \ \} \rangle,
 \hspace{1cm} \\
 \M \ = \ \bigcup \ \{ \M_x^{(2)} \ : \ x \in \{0,1\}^{\N} \ \}.
  \hspace{2.5cm}
 \end{split}
 \end{equation}
 Since $\r(\M) \leq 2 $ and the size of the elements of $\M$ are bounded by $2$, it follows from Lemma \ref{labellem02} that
 $\M$ is a label of finite type. Notice that if $x \not= y $ in $\{0,1\}^{\N}$ then $M_x \ \cap \  M_y$ is finite.

 It is easy to see
 that for each $x \in \{0,1\}^{\N}, \ \ LIM \ \{ \M - \chi(\ell_{w_i(x)}) \} \ = \ \M_x$ and
 it follows that $\Theta(\M)$ is uncountable.
\end{ex}

\vspace{0.5cm}

 If $\Phi$ is a compact, invariant subset of $\LAB$ then the action of $FIN(\N)$ restricts to an action
 on $\Phi$ as does the Ellis action of $\b FIN(\N)$. The image of the continuous map $\b FIN(\N) \to \Phi^{\Phi}$
 which extends
 $\rr \mapsto P_{\rr} $ is the \emph{enveloping semigroup} of $\Phi$, denoted $\E(\Phi)$.
  A map $Q$ on $\Phi$ is an element of $\E(\Phi)$ iff for every
 finite sequence $\{ \M^i \}$ in $\Phi$ and any $N \in \N$ there exists $\rr \in FIN(\N)$ such that
 $P_{\rr}( \M^i ) \cap \B_N = Q( \M^i ) \cap \B_N$ for all $i$.

 The \emph{adherence semigroup} $\A(\Phi)$ is the closure in $\Phi^{\Phi}$ of $\{ P_{\rr} : \rr > \00 \}$, i.e. the
 image of $\b' FIN(\N)$ (see Definition \ref{adherencedef} in Appendix
 \ref{appendix-StoneCech}).
 It follows that for $\M \in \Phi$
 \begin{align}\label{label08a}
 \begin{split}
\E(\Phi) \ = &\ \A(\Phi) \cup \{ P_{ \00 } = id_{\Phi} \}, \\
\E(\Phi) \cdot \M \ = \ &\b FIN(\N) \cdot \M \ = \  \Theta(\M),\\
\A(\Phi) \cdot \M \ = \ &\b' FIN(\N) \cdot \M \ = \  \Theta'(\M) .
\end{split}
\end{align}

If $\Phi_1 \subset \Phi$ is also closed and invariant then the restriction map defines a continuous, surjective homomorphism
from $\E(\Phi)$ to $\E(\Phi_1)$. In particular, every $\E(\Phi)$ is the image of $\E(\LAB)$ via the restriction map.

\begin{prop}\label{minidprop} The enveloping semigroup $\E(\LAB)$ contains a unique minimal idempotent $U$ defined by
\begin{equation}\label{minid01}
U(\M) = \begin{cases} \emptyset \qquad \text{if} \quad \M \not= FIN(\N), \\
FIN(\N) \qquad \text{if} \quad \M = FIN(\N). \end{cases}
\end{equation}
Furthermore, $P U = U = U P$ for all $P \in \E(\LAB)$. \end{prop}

\proof  The set $I$ of labels not equal to $FIN(\N)$ is directed by inclusion for if $\rr_1 \not\in \M_1$ and
$\rr_2 \not\in \M_2$ then $\rr_1 + \rr_2 \not\in \M_1 \cup \M_2$. Choose for each $i \in I$ an element
$\rr^i$ not in the corresponding label. If $\M \not= FIN(\N)$ then $\{ \M - \rr^i \}$ converges to $\emptyset$.
Since $FIN(\N) - \rr = FIN(\N)$ for all $\rr$ it follows that $P(FIN(\N)) = FIN(\N)$ for all $P \in \E(\LAB)$.
Thus, $\{ P_{\rr^i} \}$ converges to $U$.  Clearly, $P U = U = U P$ from which it follows that $\{ U \}$ is the
unique minimal subset of $\E(\LAB)$.

$\Box$ \vspace{.5cm}

No point of $\b^* FIN(\N)$ is the limit of a sequence in $FIN(\N)$ (see Corollary \ref{appcor02}). On the other
hand, $\E(\LAB)$ has interesting points which are limits of sequences in $FIN(\N)$.

\begin{ex}\label{sequenceex} For any $\rr > 0$ the sequence $\{ P_{n \rr} : n \in \N \}$ converges
in the enveloping semigroup $\E(\LAB)$ to an idempotent $Q_{\rr}$
with $Q_{\rr}(\M) = \{ \mm : \mm + n \rr \in \M $ for all $n \in \N \}$.
\end{ex}

On the other hand, if $\{ \rr^i \}$ is a sequence of positive $\N$-vectors such that $LIMSUP \{ supp \  \rr^i \} = \emptyset$
then the sequence $\{ P_{\rr^i} \}$ is not convergent in $\E(\LAB)$.  By going to a subsequence we can assume that the
supports are  pairwise disjoint. Let $\M = \langle \{ \rr^1 + \rr^{i} : i \in 2\N \} \rangle$. Then $\rr^1 \in P_{\rr^i}(\M)$
for $i$ even but $P_{\rr^i}(\M) = \emptyset$ for odd $i > 1$.

 We consider convergence of nets of the form
   $ \M -  \rr^i =
   P_{\rr^i}(\M)$ for a net $\{ \rr^i \} $ of $\N$-vectors and a fixed label $\M$.
  Clearly, $\ \mm \in LIMSUP$ iff
 frequently $\mm + \rr^i \in \M$ and
  $ \mm \in LIMINF$ iff eventually $\mm + \rr^i \in \M$.
%
%

  \begin{lem}\label{labellem07} Assume that the net $\{ P_{\rr^i}(\M) = \M - \rr^i \}$ is convergent.

  (a) If $ max \ \M \not= \emptyset $, e.g. if $\M$ is of finite type, then either eventually $\rr^i = \00 $
  and  $LIM \ \{ \M - \rr^i \} = \M = \M - \00 $ or eventually
  $\rr^i > \00 $ and $ (max \ \M) \ \cap \  LIM \ \{ \M - \rr^i \} = \emptyset$.

  (b) Either eventually $\rr^i \in \M$ and $ \00 \in LIM \ \{ \M - \rr^i \}$ or eventually
  $ \rr^i \not\in \M$ and $LIM \ \{ \M - \rr^i \} = \emptyset$.

  (c) If $\M - \mm_1 = \M - \mm_2$ then $LIM \ \{ \M -  \rr^i \} - \mm_1 = LIM \ \{ \M - \rr^i \} - \mm_2$.

  (d) If there exists $\rr$ such that frequently, $\M - \rr^i = \M - \rr$ then the limit is $\M - \rr$.
  \end{lem}

  \proof (a) Since $ \rr > \00 $ implies $max \ \M \ \cap \  (\M - \rr) = \emptyset$, if $ \rr^i = \00 $ frequently
then convergence implies that the limit is $\M$ and so any element of $max \ \M$ is eventually in   $\M - \rr^i $
  which can only happen when $\rr^i$ is eventually $ \00 $.

  (b) Since $ \00 \in \M - \rr$ iff $ \rr \in \M$ we see that if $ \rr^i \in \M$ frequently then
  convergence implies that eventually $ \00 \in \M - \rr^i$ and so eventually $ \rr^i \in \M$.

  (c) Since $P_{\mm_1}$ and $P_{\mm_2}$ are continuous,
  \begin{equation}\label{label08b}
  \begin{split}
  LIM \ \{ \M -  \rr^i \} - \mm_1 \ = \ LIM \ \{ \M -  \mm_1 - \rr^i \}  \ =  \hspace{2cm} \\
  \ LIM \ \{ \M - \mm_2 - \rr^i \} \ = \ LIM \ \{ \M - \rr^i \} - \mm_2. \hspace{1cm}
  \end{split}
  \end{equation}

  (d) By assumption there is a subnet $\rr^{i'}$ such that $\M - \rr^{i'}$ is constant at $\M - \rr$ and so
  converges to $\M - \rr$. By the assumption of convergence the limit of the original sequence is $\M - \rr$.

$\Box$ \vspace{.5cm}

For labels $\M_1$ and $\M_2 $ define
$$  \M_1 \oplus \M_2 = \{ \mm_1 + \mm_2 : \mm_1 \in \M_1, \mm_2 \in \M_2 \}.$$
This is empty if either $\M_1$ or $\M_2$ is empty.  It follows that
\begin{equation}\label{label12}
\rho(\M_1 \oplus \M_2) \quad = \quad \rho(\M_1) + \rho(\M_2). \hspace{3cm}
\end{equation}
and so $\M_1 \oplus \M_2$ is bounded iff both $\M_1$ and $\M_2$ are bounded.

Observe that for any $\ell, N \in \N$, we have
\begin{align}\label{label12a}
\begin{split}
(\M_1 \oplus \M_2) \wedge [1,\ell] \ &= \ (\M_1 \wedge [1,\ell]) \oplus (\M_2 \wedge [1,\ell]), \\
(\M_1 \oplus \M_2) \cap \B_N = [&(\M_1\cap \B_N) \oplus  (\M_2 \cap \B_N )] \cap \B_N.
\end{split}
\end{align}
It follows that the map
$\oplus : \LAB \times \LAB  \to \LAB $ is continuous.

\begin{lem}\label{labellem12xa} (a) If $\M$ is a recurrent label, then $\M_1 \oplus \M$ is a recurrent label for any
label $\M_1$. If $\M$ is a strongly recurrent label,  then $\M_1 \oplus \M$ is a strongly recurrent label for any finite
label $\M_1$.

(b) If $\M$ is any label and $\ell \in \N$ then there exists a recurrent label $\M_1$ with $\M \subset  \M_1$ and
$\M \wedge [1,\ell] = \M_1 \wedge [1,\ell]$. There exists a strongly recurrent label $\M_2$ with
$\M \wedge [1,\ell] = \M_2 \wedge [1,\ell]$. If $\M$ is bounded then $\M_1$ and $\M_2$ can be chosen bounded.
\end{lem}

\proof (a) Clearly $\M_1 \oplus (\M - \rr) \subset (\M_1 \oplus \M) - \rr$ for any $\rr \in \M$. Hence,
if $\{ \M - \rr^i \} \to \M$ then $\{ (\M_1 \oplus \M) - \rr^i \} \to \M_1 \oplus \M$ and so $\M_1 \oplus \M$ is recurrent.

If $\mm_1 \in \M_1$ and $\mm \in \M$ we let $F(\mm_1 + \mm) =
(\bigcup Supp \ \M_1) \cup F(\mm)$
(see Definition \ref{df,st-rec}).
If $\rr_1 + \rr$ has support
disjoint from this set then $\rr_1 = 0$ and $\rr + \mm \in \M$. Hence, $\rr_1 + \rr + \mm \in \M_1 \oplus \M$.

(b) Begin with a bounded, strongly recurrent label $\NN$ with $\NN \wedge [1,\ell] = 0 $. For example,
 $\NN = \langle \chi([\ell + 1, \infty)) \rangle$ is strongly recurrent by Proposition \ref{labelprop05a2}.
  Then $\M_1 = \M \oplus \NN$ is a recurrent label
 and  $\M_2 = (\M \wedge [1,\ell])\oplus \NN$ is a strongly recurrent label
 with $\M \wedge [1,\ell] = \M_1 \wedge [1,\ell] =  \M_2 \wedge [1,\ell]$.

$\Box$ \vspace{.5cm}

\begin{ex}\label{exx}
Let $L$ be an infinite set disjoint from $\bigcup Supp \ \M$.
Let $\M_1 = \{ \emptyset \} \cup \{ \ \chi(\ell) \ : \ \ell \in L \ \}$.
If $\M$ is strongly recurrent then $\M_1 \oplus \M$
is recurrent by (a)
but it is not strongly recurrent.\end{ex}
$\Box$ \vspace{.5cm}

\begin{prop}\label{labelprop12xb} The set $RECUR$ of recurrent labels is a dense, $G_{\delta}$ subset of $\LAB$.\end{prop}

\proof For any $N \in \N$ and $\rr \in FIN(\N)$, the set  $\{ \ \M  \ : \ \M \cap \B_N  = \ P_{\rr}(\M)  \cap \B_N \ \}$ is a clopen
subset of  $ \LAB$ by Lemma \ref{labellem03a} (c) and continuity of the map $P_{\rr}$.

Clearly,
\begin{equation}\label{label12x}
RECUR \ = \ \bigcap_{N \in \N} \ \bigcup_{\rr \in FIN} \ \{ \ \M  \ : \ \M \cap \B_N  = \ P_{\rr}(\M)  \cap \B_N \ \}
\end{equation}
and so $RECUR$ is a $G_{\delta}$ subset of $\LAB$.  It is dense by Lemma \ref{labellem12xa} (b).

$\Box$ \vspace{.5cm}

Thus, the labels of finite type, upon which we focus most of our attention, comprise a subset of first category in $\LAB$.
Since the finite labels are dense, the labels of finite type are dense. Thus, the set of recurrent labels has empty interior.
As it is the intersection of two dense $G_{\d}$ sets, the set of bounded, recurrent labels is a dense $G_{\d}$ set.

\vspace{1cm}

\section{Labeled subshifts}\label{sec,exp}

\subsection{Integer expansions}

$\qquad$

\vspace{.5cm}

Let $e \in \N$ and $b = 2e +1$ so that $b$ is an odd number greater than $1$. Define $k : \N \to \N$ by
$k(n) = b^{n-1}$. The \emph{symmetric  $b$-expansion} of an integer $t$ is the sum $\Sigma_{n \in \N} \ \ep_n(t) k(n) = t$
such that:
\begin{itemize}
\item  $|\ep_n(t)| \leq e$ for all $n \in \N$.
\item  $\ep_n(t) = 0$ for all but finitely many $n$.
\end{itemize}

\begin{prop}\label{prop401}  Every integer $t \in \Z$ has a unique symmetric $b$-expansion and $\ep_n(-t) = -\ep_n(t)$
for all $t \in \Z, n \in \N$.

\end{prop}

\begin{proof}  By the Euclidean Algorithm every integer $t$ can be expressed uniquely as $\ep + bs$ with $|\ep| \leq e$.
It follows by induction that every integer $t$ with $|t| <\frac{1}{2}(b^k - 1)$ has an expansion with
$\ep_n = 0$ for $n \geq k$. There are $b^k$ such integers and the same number of expansions. So by the Pigeonhole
Principle the expansions are unique.

Since $-t = \Sigma_{n \in \N} \ -\ep_n(t) k(n)$ and the expansions are unique, it follows that $-\ep_n(t) = \ep_n(-t)$.
\end{proof}

$\Box$  \vspace{.5cm}

\begin{df}\label{def401a} For $t \in \Z \setminus \{ 0 \}$ let
\begin{align}\label{tow401b}
\begin{split}
n_*(t) = \min &\{ n \in \N : \ep_n(t) \not= 0 \ \}, \\
n^*(t) = \max &\{ n \in \N : \ep_n(t) \not= 0 \  \}.
\end{split}
\end{align}
\end{df}

Thus, $n_*(t)$    is the smallest  non-zero place  in the expansion of $t$, and  $n^*(t) $ is the largest.
From Proposition \ref{prop401} $n_*(-t) = n_*(t)$ and $n^*(-t) = n^*(t)$.  By convention, we define
$n_*(0) = \infty$ and $n^*(0) = 0$. Observe that for $k \in \N$,
\begin{equation}\label{tow01ba}
n_*(t) > k \qquad \Longleftrightarrow \qquad t \equiv 0 \ (\text{mod} \ b^k).
\end{equation}

Define  the function $sk : \Z_+ \to \Z_+$ by
\begin{equation}\label{tow02a}
sk(n) =  \Sigma_{i=1}^{n-1} \ k(i),
\end{equation}
with $sk(1) = sk(0) = 0$, the empty sum.

We extend $k$ to an odd function on $\Z$ so that $k(-j) = - k(j)$ for
all $j \in \Z$ and so, in particular, $k(0) = 0$.  Clearly, $|k(j)| = k(|j|) $ for all $j \in \Z$.

For integers $a, b$ we will denote by $[ \ a \ \pm \ b \ ]$ the interval $[ \ a - |b|, a + |b| \ ]$ in $\Z$. When $a = 0$
we will write $[ \ \pm b \ ]$ for the interval $[ \ - |b|, + |b| \ ]$.

Clearly, for $n \in \N$
\begin{align}\label{tow02b}
\begin{split}
k(n) = 1 + &(b - 1) sk(n), \hspace{2cm} \\
\text{and so} \quad sk(n) &< \frac{1}{b-1} k(n).
\end{split}
\end{align}
Also, for $n \in \N$:
\begin{equation}\label{tow03.1}
k(n + 1) - k(n) = (b - 1)k(n).\hspace{1cm}
\end{equation}

\begin{lem}\label{towlem01} (a) For $t \in \Z \setminus \{ 0\}$,
$t > 0 $ iff $\ep_{n^*}(t) > 0$ and
\begin{equation}\label{tow02ba}
\frac{b - 1 }{2} sk(n^*(t) + 1) \geq |t| \geq \frac{b - 1 }{2} sk(n^*(t)).
\end{equation}

(b)  The sequence
$\{ \ [ \ k(n) \ \pm \ \frac{b - 1 }{2} sk(|n|) \ ] \ : \ n \in \Z \ \}$
is a pairwise disjoint bi-infinite sequence of intervals in $\Z$.

(c) For $t, s \in \Z$, if $|\ep_n(t) + \ep_n(s)| \leq e$ for all $n \in \N$, then
$\ep_n(s + t) = \ep_n(t) + \ep_n(s)$ for all $n \in \N$.
\end{lem}

\begin{proof}  Recall that $\frac{b - 1 }{2} = e,$ a positive integer.

  (a) If $\ep_{n^*(t)} > 0$ then $t \geq k(n^*(t)) - e sk(n^*(t))$ which is positive by (\ref{tow02b}).
$\ep_{n^*(t)} < 0$ then $\ep(n^*(-t)) = -\ep(n^*(t)) > 0$ and so $-t > 0$.

In any case, for any non-zero $t$ we have $e(k(n^*) + sk(n^*)) \geq |t| \geq k(n^*) - e sk(n^*)$ and
the inequalities of (\ref{tow02ba})  follow from the definition of $sk(n^* + 1)$ and (\ref{tow02b}).

(b) If $|n| = 0,1$ then $sk(|n|) = 0$ and the intervals are the singletons $\{ k(n) \}$. For $n > 1$,
(\ref{tow02b}) and $e \geq 1$ imply that
\begin{equation}\label{tow03.1a}
k(n) - e sk(n) > e sk(n) \geq k(n-1) + e sk(n-1).
\end{equation}

(c) This is clear from the uniqueness of the symmetric expansion. The assumption implies that there are no ``carries''
in any place when $t$ and $s$ are added.
\end{proof}
$\Box$ \vspace{.5cm}

We now fix $e > 1$ and so  $b = 2 e + 1 \geq 5$. We will frequently use the observation:
\begin{equation}\label{tow3.1x}
n \geq 2 \qquad \Longrightarrow \qquad k(n) \geq  \frac{b}{2} \ n > 2n.
\end{equation}

Let $IP(k)$ denote the set of sums of finite subsets of the image
set $k(\Z) \subset \Z$ and let $IP_+(k)$ denote the set of sums of
finite subsets of the image set $k(\Z_+) \subset \Z_+$.  It is clear that $IP(k)$ consists of those integers
$t = \Sigma_{n \in \N} \ \ep_n k(n)$ such that every nonzero $\ep_n = \pm 1$.  That is, we exclude those
integers $t$ such that some $\ep_n$ has absolute value equal to $2,3,\dots e$. In $IP_+(k)$ every $\ep_n$ is zero or one.
 Notice that Lemma \ref{towlem01} (c) applies
 to any $t, s \in IP(k)$.

For $t \in IP(k)$ we have $t = \Sigma_{n \in \N} \  k(\ep_n n)$ and so for such $t$ we describe the expansion a
slightly different way, listing only the nonzero terms and absorbing the sign into the index.

\begin{df}\label{towdef02} (a) For $t \in IP(k)$ we write

$$ t \ = \ k(j_1) + k(j_2) +  \cdots  + k(j_r).  $$

with $j_1, \dots, j_r \in \Z $ such that $|j_i| > |j_{i+1}| > 0  $ for $i = 1,\dots,r-1$. We will call
$j_1, \dots , j_r$ the \emph{expansion} of $t \in IP(k)$, and we call $r = r(t) \geq 0$ the \emph{length } of the expansion.

$0 \in IP(k)$ has the empty expansion with length $0$.

By convention, if $i > r(t)$ then we write $j_i(t) = 0$.

(b) For $0 \leq \tilde r \leq r$ the \emph{$\tilde r$ truncation} is the element $\tilde t \in IP(k)$
with expansion
$j_1,\dots, j_{\tilde r}$ so that
$$\tilde t \ = \ k(j_1) + k(j_2) +  \cdots  + k(j_{\tilde r}).$$

The sequence
$j_{\tilde r + 1},\dots, j_r $ is the
expansion for the \emph{$\tilde r$ residual} $t - \tilde t$
with length $r - \tilde r$. We call $\tilde t$
a \emph{truncation} of $t$ and $t$ an \emph{extension} of $\tilde t$.

When $\tilde r = 0$ the truncation $\tilde t = 0$ with the empty expansion.
When $\tilde r = r$ the truncation $\tilde t = t$ with residual $0$.

\end{df} \vspace{.5cm}

Comparing the alternative notations we have for $t \in IP(k) \setminus \{ 0 \}$
\begin{equation}\label{tow3.1xa}
|j_1(t)| = n^*(t), \qquad |j_{r(t)}(t)| = n_*(t).
\end{equation}

\begin{lem}\label{towlem03a}(a) Let $t \in IP(k) \setminus \{ 0 \}$. The integer
 $t$ has a unique expansion $j_1, \dots , j_r$ with $r \geq 1$  and
$t$ has the same sign as $k(j_1)$. Furthermore,
\begin{equation}\label{tow06}
\begin{split}
|t - k(j_1(t))|  \quad \leq \quad \Sigma_{i = 2}^r \ |k(j_i(t))| \quad \leq \quad  sk(|j_1(t)|),\hspace{1cm} \\
\frac{b - 2 }{b - 1} |k(j_1(t))| \quad < \quad |t| \quad < \quad   \frac{b}{b - 1} |k(j_1(t))|, \hspace{1cm}
\end{split}
\end{equation}

(b) If $\tilde t$ is the $\tilde r$ truncation of $t$ then
 \begin{equation}\label{tow08}
\frac{b - 2}{b - 1} |k(j_{\tilde r + 1}(t))| \ \leq \ |t - \tilde t| \ \leq \ \frac{b}{b - 1} |k(j_{\tilde r + 1}(t))|.
\end{equation}

(c) If $t, s$ are distinct integers and $\hat r$ is the smallest positive integer $p$
such that $j_{p}(t) \not= j_{p}(s)$, then
 \begin{equation}\label{tow07}
| t - s| \ > \ \frac{b - 3}{b - 1} \cdot \max( \ |k(j_{\hat r}(t))|,|k(j_{\hat r}(s))| \ ).
\end{equation}
\end{lem}

\begin{proof} (a): The uniqueness of the expansion is Proposition \ref{prop401}. That
$t$ and $j_1(t)$ have the same sign follows Lemma
\ref{towlem01}.

Because the sequence $\{ \ |j_i| \ : i = 2,\dots,r \ \}$ consists of distinct positive integers
all less than $|j_1|$ the first line of (\ref{tow06}) is clear.  The second follows from (\ref{tow02b}).

(b): If $r = \tilde r$ then $t - \tilde t = 0 = j_{\tilde r + 1}(t)$ and so (\ref{tow08}) is clear.

If $r > \tilde r$ then  $t - \tilde t$ is the residual whose expansion begins with $j_{\tilde r + 1}(t)$.
So (\ref{tow08}) follows from  (\ref{tow06}).

(c): Applying Lemma \ref{towlem01} we see that $t - s$ has a symmetric $b$ expansion with $|\ep_n| \leq 2$ for all
$n$ and with $n^* = n^*(t - s) = \max( \ |j_{\hat r}(t)|,|j_{\hat r}(s)| \ )$.  It follows from (\ref{tow02b}) that
$|t - s| \geq k(n^*) - 2 sk(n^*) \geq (1 - \frac{2}{b - 1}) k(n^*)$.
\end{proof}

$\Box$ \vspace{.5cm}

\begin{cor}\label{towcor04} If $s, t \in IP(k)$ then $s$ is an extension of $t$ iff
\begin{equation}\label{tow09a}
| t - s| \ \leq \ \frac{b - 3}{b - 1}  |k(j_{r}(t))|.
\end{equation}
\end{cor}

\begin{proof}  The result is clear if $t = s$ and so we can assume $s, t$ are distinct.

If $s$ is an extension of $t$ then $t$ is the truncation of $s$ at $r$. By (\ref{tow08})
$|t - s| \leq \frac{b}{b - 1} k(|j_{r + 1}(s)|) \leq \frac{1}{b - 1} k(|j_r(t)|)$ since
$|j_{r+1}(s)| < |j_r(s)|$ and $j_r(s) = j_r(t)$. Then (\ref{tow09a}) follows because $b - 3 \geq 1$.

On the other hand, if $| t - s| \ \leq \ \frac{b - 3}{b - 1}  |k(j_{r}(t))|$ then, in the notation of
Lemma \ref{towlem03a}(c) it must be true that $ \max( \ |j_{\hat r}(t)|,|j_{\hat r}(s)| \ ) < |j_r(t)|$.
That is, $\hat r > r$ and so equals $r + 1$ with $j_{\hat r}(t) = 0$. Thus, $j_i(t) = j_i(s)$ for
$i = 1,\dots,r(t)$. That is, $t$ is a truncation of $s$.

\end{proof}

$\Box$ \vspace{.5cm}

%

For a subset $A$ of $\Z$  the \emph{upper Banach density} of $A$
is
\begin{equation}\label{Banach01}
\limsup_{\# I \to \infty}  \ \frac{ \# (I \cap A)}{\# I }
\end{equation}
as $I$ varies over finite intervals in $\Z$. When the limsup is zero, or, equivalently,
when the limit exists and equals zero, we say that $A$ has
{\em Banach density zero}.


\begin{theo}\label{towtheoBanach01} The subset $IP(k)$ of $\Z$ has Banach density zero.
\end{theo}

\proof Given an interval $I \subset \Z$ in with  $N = \# I$,
if $t, s $ are distinct points in $I$ then $0 < | t - s | < N $.
Now assume that $t,s \in IP(k)$ and apply Lemma \ref{towlem03a}. Let $\hat r$ be the smallest positive integer
such that $j_{\hat r}(t) \not= j_{\hat r}(s)$. Since $b \geq 5$, $\frac{b - 3}{b - 1} \geq \frac{1}{2}$ and so
Lemma \ref{towlem03a}(c) implies that $2N > \max( \ k(|j_{\hat r}(t)|), k(|j_{\hat r}(s)|))$. Thus, the terms in the
expansions of $t$ and $s$ agree except for terms $k(j)$ with $|k(j)| < 2N$. Call these the \emph{variable} $|j|$'s.

Since $|k(j)| = b^{|j| - 1}$ and
$b \geq 5$, we have $5^{|j| - 1} < 2N$. So the number of variable $|j|$'s is bounded by
$1 + \log_5 \ (2N) = \log_5 (10N)$. For each such $j, \ k(|j|)$ can
occur with coefficient $-1, 0$ or $+1$ in the symmetric $b$ expansion of $t \in I$.
Consequently, $\# (I \cap IP(k))$ is bounded by
$3^{\log_5\ (10N)} \ = \ (10N)^{\log_5 (3)}$. Since $0 < \log_5 \ (3) < 1$, it follows that
$  \frac{(10N)^{\log_5 (3)}}{N} \ \to \ 0$ as $N \ \to \ \infty$.

Hence,
$$ \lim_{\# I \to \infty}  \ \frac{ \# (I \cap IP(k))}{\# I } \ = \ 0.$$
$\Box$ \vspace{.5cm}

\vspace{1cm}

\subsection{Labeled integers}

$\qquad$

\vspace{.5cm}

We now fix a partition of $\N$ by an infinite sequence
$$\D \quad = \quad \{ D_{\ell} : \ell \in \N \}$$
with $\# D_\ell =\infty$ for every $\ell \in \N$,
numbered so that
$\min D_{\ell} < \min D_{\ell + 1}$. For example, if $p_1 = 2, p_2, \dots$ numbers the primes in $\N$ is ascending
order we can let $D_1 = \{ 1 \} \cup 2 \N$ and for $\ell > 1$ let $D_{\ell}$ consist of those integers $n$ with
$p_{\ell}$ the smallest prime which divides $n$.

In an case, $\min D_{1} = 1$ and $\min D_{\ell} \geq \ell$.
We then let $Q(\ell,i)$ be the $i^{th}$ smallest member
of $D_{\ell}$. Thus, $Q : \N \times \N \to \N$ is a bijection such that $\ell \mapsto Q(\ell,1)$ is increasing and
for each $\ell \in \N$, $i \mapsto Q(\ell,i)$ is increasing with image $D_{\ell}$.

The \emph{support map}  $n \mapsto \ell(n)$ associates to each $n \in \N$ the member of $\D$ which contains it, so that
$n \in D_{\ell(n)}$. The map is the composition of $Q^{-1}$ with the projection onto the first coordinate.

\begin{df}\label{towdef02a} (a) If
$j_1,\dots,j_r \in \Z $ is an expansion for $t \in IP(k)$,
so that
$$ t \ = \ k(j_1) + k(j_2) +  \cdots  + k(j_r),  $$
then the \emph{length vector} for the expansion
is the $\N$-vector
$$ \rr \ = \ \rr(t) \ = \ \chi(\ell(|j_1|)) + \chi(\ell(|j_2|)) +
\cdots + \chi(\ell(|j_r|)) $$

so that $|\rr|$ is the length $r = r(t)$.

$0 \in IP_+(k) \subset IP(k)$ has the empty expansion with length $0$ and length vector $ \00 $.

\end{df}
\vspace{.5cm}

Thus, $\rr(t)_{\ell} \ = \ \# \{ i : |j_i| \in D_{\ell} \}$ in the expansion $j_1,\dots,j_r$ for $t$.

If $0 \leq \tilde r \leq r$, then the $\tilde r$ truncation, $\tilde t$, has
length vector
$\tilde \rr \ = \ \rr(\tilde t) \ = \ \chi(\ell(j_1)) + \chi(\ell(j_2)) +
\cdots
+ \chi(\ell(j_{\tilde r}))$,
so that $|\tilde \rr| = \tilde r$.
The residual
 $t - \tilde t$ has length vector $\rr - \tilde \rr \geq \00 $.

\begin{df}\label{df,A[M]}
For a label $\M$ define $A[\M] \subset \Z$ to consist of those $t \in IP(k)$ which  have length vector
$\rr(t) \in \M$. Define $A_+[\M] = A[\M] \cap IP_+(k)$. \end{df}

 If $\M = \emptyset$ then $A_+[\M] = A[\M] = \emptyset$. Otherwise $0 \in A_+[\M] \subset A[\M]$ since $\00 \in \M$.
\vspace{.5cm}

\begin{prop}\label{towprop05} Assume $t \in IP(k)$  with expansion length $r$. For any positive integer $N$,
\begin{equation}\label{tow12}
N \
< \ |j_{r}(t)| \quad \Longrightarrow \quad [\ t \ \pm \ N] \ \ \cap \  \ A[\M] \ =
\ t \ + \ (  [ \ \pm \ N \ ] \ \ \cap \  \ A[\M - \rr(t)] ),
\end{equation}
and the elements of $ [\ t \ \pm \ N]  \ \ \cap \  \ A[\M]$ are extensions of $t$ in $A[\M]$. In particular, this
set is nonempty iff $t \in A[\M]$.

If, in addition, $t \in IP_+(k)$ then
\begin{equation}\label{tow12+}
[\ t \ \pm \ N] \ \ \cap \  \ A_+[\M] \ =
\ t \ + \ (  [ \ \pm \ N \ ] \ \ \cap \  \ A_+[\M - \rr(t)] )
\end{equation}
and the elements of $ [\ t \ \pm \ N]  \ \ \cap \  \ A_+[\M]$ are extensions of $t$ in $A_+[\M]$. In particular, this
set is nonempty iff $t \in A_+[\M]$.
\end{prop}

\proof  If $N < |j_r(t)|$ then $2 \leq |j_r(t)|$ and so (\ref{tow3.1x}) implies
$$N <  |j_r(t)| \leq \frac{2}{b} k(|j_{r}(t)|) < \frac{b - 3}{b - 1} k(|j_{r}(t)|). $$

So if   $s \in [\ t \ \pm \ N]  \ \cap \  A[\M]$,
then by Corollary \ref{towcor04}
$s$ is an extension of $t$ and
$\rr(s)  \in \M$. Because $t$ is a truncation of $s$ with residual $s - t$ we have $\rr(s) = \rr(t) + \rr(s - t)$
so $s - t  \in A[\M - \rr(t)]$.

Furthermore, if $s \in [\ t \ \pm \ N]  \ \cap \  A_+[\M]$, i.e. in addition, $s \in IP_+(k)$,
then since $t$ is a
truncation of $s$ with residual $s - t$ we have $t, s - t \in IP_+(k)$.

On the other hand, if
$u = k(j_1(u)) + \cdots + k(j_p(u)) \in  [\pm N] \cap A[\M - \rr(t)]$
with length $p$ then $\rr(u) + \rr(t) \in \M$.

Furthermore, by (\ref{tow06}) applied to $u$, $|u| \leq N$ implies
$ | k(j_1(u))|
< 2|u| \leq 2 N < 2 |j_{r}(t)|$.

If $|j_1(u)| = 1$ then $|j_1(u)| = 1 < |j_r(t)|$. Otherwise, (\ref{tow3.1x}) implies
$|j_1(u)| \leq \frac{1}{2} k(|j_1(u)|) < |j_r(t)|$.

Either way, it follows that
$$j_1(t), \dots,j_r(t), j_1(u), \dots, j_p(u)$$
is the expansion
for $t + u$ with length $r + p$ and length vector $\rr(t) + \rr(u)$. Hence, $t + u \in [\ t \ \pm \ N] \ \cap \  A[\M]$.

Furthermore if $u, t \in IP_+(k)$ then $t + u \in IP_+(k)$.

$\Box$ \vspace{.5cm}

%
%
%
%
%
%

\begin{lem}\label{towlem06} If $t \in IP(k)$
with
$|t| < \min D_{\ell}$, then $\rr(t)_{\ell} = 0$.
\end{lem}

\proof We may assume $t \not= 0$. So $1 \leq |t| < \min D_{\ell}$ and $\min D_1 = 1$ imply $1 < \ell$.

If $|j_1(t)| = 1$ then $|j_1(t)| < \ell \leq \min D_{\ell}$.

If $|j_1(t)| \geq 2$ then (\ref{tow3.1x}) and (\ref{tow06}) imply
$$ |j_1(t)| \leq \frac{1}{2}|k(j_1(t))| \leq \frac{b - 1}{2(b - 2)} |t| < \min D_{\ell},$$
since $ \frac{b - 1}{2(b - 2)} < 1$.

In either case, $|j_1(t)| < \min D_{\ell}$ implies that $|j_i(t)| \not\in D_{\ell}$ for $i = 1,\dots, r(t)$.

$\Box$ \vspace{.5cm}

\begin{prop}\label{towprop07}For any label $\M$ and  integer $N \geq 2$
\begin{equation}\label{tow13}
\begin{split}
[\ \pm \ N \ ] \ \cap \  A[\M] \quad = \quad [\ \pm \ N \ ] \ \ \cap \  \ A[\M \ \cap \B_N], \\
[\ \pm \ N \ ] \ \cap \  A_+[\M] \quad = \quad [\ \pm \ N \ ] \ \ \cap \  \ A_+[\M \ \cap \B_N]
\end{split}
\end{equation}
\end{prop}

\proof If $ N < \ell$ then $N < \min D_{\ell}$. Hence, $|t| \leq N$ implies
$\rr(t)_{\ell} = 0$ by Lemma \ref{towlem06}. Thus, $ supp \ \rr(t) \subset [1,N]$.

If $ r = |\rr(t)| \geq N$. Then $|j_1| > |j_2| > \dots > |j_r| > 0$ implies $|j_1| \geq r \geq N$. In particular, $|j_1| \geq 2$.
From (\ref{tow3.1x})
we obtain
$|t| \geq \frac{b - 2}{b - 1} k(|j_1|) > |j_1| \geq N$ since $\frac{2(b - 2)}{b - 1} > 1$.
Contrapositively, $|t| \leq N$ implies
$|\rr(t)| < N$.

It follows that $|t| \leq N$ implies $ \rr(t) \in \B_N$.

$\Box$ \vspace{.5cm}

If $\mm = \00 $ then $t = 0$ is the unique time with length vector $\mm$. For a nonzero
$\N$-vector $\mm$ there are
infinitely many such times.

\begin{lem}\label{towlem08} (a) If $\mm$ is a nonzero $\N$-vector so that $|\mm| = r > 0$
then for any positive integer $N$ there exist $t \in IP_+(k)$ such that $\rr(t) = \mm$.
and $j_r(t) > N$.

(b) For labels $\M_1, \M_2$ if $\mm \in \M_1 \setminus \M_2$ and $t \in IP(k)$
with $\rr(t) = \mm$ then $t \in A[\M_1] \setminus A[\M_2]$.
In particular, $ \M_1 \ = \ \M_2$ iff $A[\M_1] = A[\M_2]$ and iff $A_+[\M_1] = A_+[\M_2]$.

(c) For any  $N \in \N$ there exists  $M \in \N$ such that for all labels $\M_1, \M_2$
\begin{equation}\label{tow14}
\begin{split}
[\ \pm \ M \ ] \ \cap \  A[\M_1] \ = \ [\ \pm \ M \ ] \ \cap \  A[\M_2] \qquad
\Longrightarrow \qquad \\
[\ \pm \ M \ ] \ \cap \  A_+[\M_1] \ = \ [\ \pm \ M \ ] \ \cap \  A_+[\M_2] \hspace{1cm}\\
\Longrightarrow \qquad  \M_1 \cap \B_N \ = \ \M_2 \cap \B_N. \hspace{2cm}
\end{split}
\end{equation}
\end{lem}

\proof  (a) We use the fact that each $D_{\ell}$ is an infinite subset of $\N$.

Write $\mm = \chi(\ell_1) +
\cdots + \chi(\ell_r)$ where $r = |\mm|$.
Choose $j_r \in D_{\ell_r}$ with $j_r > N$.
For $i = 1,...,r-1 $ inductively choose
$j_{r - i} \in D_{\ell_{r - i}}$ with
$j_{r - i} > j_{r - i + 1}$.
Then
$j_1,\dots,j_r$ is an expansion for $t \in IP_+(k)$ with
$\rr(t) = \mm$.

(b) If $\mm \in \M_1 \setminus \M_2$, then by definition $t$ is in $A[\M_1] \setminus A[\M_2]$ when  $\rr(t) = \mm$.
If $t \in IP_+(k)$ then $t$ is in $A_+[\M_1] \setminus A[\M_2]$.
Hence, if $\M_1 \not= \M_2$, then $A[\M_1] \not= A[\M_2]$ and $A_+[\M_1] \not=  A_+[\M_2]$. The reverse implications are obvious.

(c) The first implication follows by intersecting with $IP_+(k)$.

For each $\mm \in  \B_N$ choose $t = t(\mm) \in IP_+(k)$ with $\rr(t) = \mm$ and let $M = \max \{ t(\mm) : \mm \in \B_N \}$.
Now assume $[\ \pm \ M \ ] \ \cap \  A_+[\M_1] \ = \ [\ \pm \ M \ ] \ \cap \  A_+[\M_2]$.

If $\mm \in \M_1 \cap \B_N$ then
 $0 \leq t(\mm) \leq M$. Hence $t(\mm) \in [\ \pm \ M \ ] \ \cap \  A[\M_1] \ = \ [\ \pm \ M \ ] \ \cap \  A[\M_2]$.
Hence, $t(\mm) \in A[\M_2]$ and so $\mm \in \M_2$.  That is, $\M_1 \cap \B_N \ \subset \ \M_2 \cap \B_N$.
Symmetrically for the reverse inclusion.

$\Box$

\vspace{1cm}

\subsection{Subshifts}\label{Sub}

$\qquad$

\vspace{0.5cm}

On $\{ 0, 1 \}^{\Z}$ define the ultrametric
$d$
by

\begin{equation}\label{tow15}
d(x, y) \  = \  \inf \ \{ \ 2^{-N} \ : \ N \in
 \Z_+
 \  \mbox{and} \   x_t = y_t \  \mbox{ for all } \ |t| < N  \}.
\end{equation}
We denote by $S$ the shift homeomorphism on $\{ 0, 1 \}^{\Z}$.  That is, $S(x)_t = x_{t + 1}$. Hence, for any $k \in \Z$
$S^k(x)_t = x_{t + k}$.

For $A \subset \Z$ we defined $\chi(A) \in \{ 0, 1 \}^{\Z}$ by $\chi(A)_t = 1$ iff $t \in A$,
so that $\chi(A)$ is the characteristic function
of
$A$.
Thus, $\chi(\emptyset) = e$ with $e_t = 0$ for all
$t$. If $z  \in \{ 0, 1 \}^{\Z}$ then $z = \chi(A)$ for $A = \{ t \in \Z : z_t = 1 \}$.

 \begin{df}\label{df,X(A)}
For $A \subset \Z$, let $X(A)$ be the
orbit closure of $\chi(A)$, i.e. the smallest closed invariant subspace of $\{ 0, 1 \}^{\Z}$ which contains $\chi(A)$. For
$\A$ a nonempty collection of subsets of $\Z$, let $X(\A)$ denote the smallest closed invariant subspace which contains
$\chi(A)$ for all $A \in \A$. In particular, for $A \subset \Z$ with power set $\P A$, the set $X(\P A)$  is the
smallest closed invariant subset of $ \{ 0, 1 \}^{\Z} $ which contains $\chi(A_1)$ for every $A_1 \subset A$.
\end{df}

We will say that
$x = \chi(A)$ \emph{dominates} $z = \chi(A_1)$ when $A \supset A_1$. Clearly, the set of pairs $\{ (x,z) \}$ such that
$x$ dominates $z$ is closed.

For $x = \chi(A)$ we will say that $x$ has positive or zero upper Banach density when the subset $A$ has positive
or zero
upper Banach density, as defined in (\ref{Banach01})  above. The point $x$ has positive upper Banach
density iff there exists an invariant
probability measure $\mu$ on its orbit closure such
that the cylinder set $C_0 = \{ z \in \{0,1\}^{\Z} : z_0 = 1 \}$ has
positive $\mu$ measure, see \cite[Lemma 3.17]{F}.
Observe that the union of the translates $S^k(C_0)$ is the complement of $\{ e \}$.
So for an invariant probability measure $\mu$ the set $C_0$ has positive measure iff $\mu$ is not the point measure
$\d_e$ where $e$ is the fixed point $ \chi(\emptyset)$.
Thus, for example, if  $X(A)$
contains a minimal subset other than the fixed point $e$ then $A$ has positive upper Banach density.

 \vspace{.5cm}

\begin{theo}\label{towtheoBanach02}
If $A$ has Banach density zero then every point $x \in X(\P A)$ has Banach density zero.
The dynamical system $(X(\P A),S)$ is uniquely ergodic
with $\delta_e$, the point mass at $e$, as the unique invariant probability measure and so $\{ e \}$ is the
unique minimal subset of $X(\P A)$. In particular, $e$ is the only periodic point in $X(\P A)$.
Furthermore, $(X(\P A),S)$ has entropy zero.
\end{theo}

\proof
First observe that any block of a point $x \in X(A)$ is also a block in $\chi(A)$. Thus every such $x$
has Banach density zero. Moreover, every point $y \in X(\P A)$ is dominated by a point $x \in X(A)$
and therefore has also Banach density zero.
Since, by Birkhoff ergodic theorem, every ergodic measure admits a generic point
we can apply Furstenberg characterization \cite[Lemma 3.17]{F} to deduce that the only ergodic measure on $X(\P A)$ is $\delta_e$.
The ergodic decomposition theorem implies that indeed $\delta_e$ is the only invariant probability measure on $X(\P A)$.
The topological entropy of $X(\P A)$ is zero by the Variational Principle since the entropy of the unique invariant measure is zero
(see e.g.\cite{Gl-03}).
Finally, as any minimal subset carries an invariant probability measure and distinct minimal subsets are disjoint, it follows that
$\{ e \}$ is the only minimal subset of $X(\P A)$.

$\Box$ \vspace{.5cm}

\begin{df}\label{df,x[]}
For a label $\M$ let
$x[\M] = \chi(A[\M]), x_+[\M] = \chi(A_+[\M])$.
 Let $X(\M)$ and $ X_+(\M)$ denote the orbit closures $X(A[\M])$ and $X(A_+[\M])$, respectively.
 For $\Phi \subset \LAB$, let
 $X(\Phi)$ and $X_+(\Phi)$  denote  $X(\{ A[\M] : \M \in \Phi \})$ and
 $X(\{ A_+[\M] : \M \in \Phi \})$, respectively.
\end{df}

For example, $x_+[\emptyset] = x[\emptyset] = e$ and
 $x_+[0]_t = x[0]_t = 1$ for $t = 0$ and $= 0$ otherwise. In general,
 the points $x[\M]$ are symmetric in that $x[\M]_{-t} = x[\M]_{t}$.
 On the other hand, $x_+[\M]_t = 0$ if $t < 0$. Hence, $x[\M] \not= x_+[\M]$
 if $\M$ is a positive label.

\begin{cor}\label{towcorBanach03} For the subshifts $(X(\P IP(k)),S), (X(\LAB),S)$ and \ \
\ \allowbreak $(X_+(\LAB),S)$
the point measure $\d_e$ is the unique invariant probability measure and so $\{ e \}$ is the unique
minimal subset. In particular, the systems are minCT and have entropy zero. \end{cor}

\proof By
Theorem \ref{towtheoBanach01} $IP(k)$ has Banach density zero.  So
by
Theorem \ref{towtheoBanach02}
$\d_e$ is the unique invariant probability measure for $(X(\P IP(k)),S)$ and so the same is true for
the subsystems $(X(\LAB),S)$ and \\ $(X_+(\LAB),S)$

$\Box$ \vspace{.5cm}

\begin{theo}\label{towtheo09} The maps $x[\cdot], x_+[\cdot]$ defined by  $\M \mapsto x[\M]$ and $\M \mapsto x_+[\M]$
are homeomorphisms from $\LAB$ onto
their images in $\{ 0, 1 \}^{\Z}$. \end{theo}

\proof Let $\M_1, \M_2$ be labels with  $\M_1 \ \cap \ \B_N  \ = \ \M_2 \ \cap \ \B_N $ for some $N > 1$.

By Lemma \ref{towlem08} (b) $\M_1 \ \cap \ \B_N  \ = \ \M_2 \ \cap \ \B_N $
iff $A[\M_1 \ \cap \ \B_N  ]\ = \ A[ \M_2 \ \cap \ \B_N  ] $.
By Proposition
\ref{towprop07} this implies $[\ \pm \ N \ ] \ \ \cap \  \ A[ \M_1 ] \ = \ [\ \pm \ N \ ] \ \ \cap \  \ A[ \M_2 ]$
and so, by intersecting with $IP_+(k)$, that
$[\ \pm \ N \ ] \ \ \cap \  \ A_+[ \M_1 ] \ = \ [\ \pm \ N \ ] \ \ \cap \  \ A_+[ \M_2 ]$
these are equivalent to $x[\M_1]_t = x[\M_2]_t$  and $x_+[\M_1]_t = x_+[\M_2]_t$ for all $t \in [ \ \pm \ N \ ] $.
Since $d(\emptyset, 0) = 1 =
d(x[\emptyset],x[0])$, it follows that $x[\cdot]$
has Lipschitz constant $1$.

The map $\M \mapsto A_+[\M]$ is injective by Lemma \ref{towlem08} (b).  Clearly,
$A \mapsto \chi(A)$ is injective.
So compactness implies that $ x[\cdot]$ and $ x_+[\cdot]$ are homeomorphisms onto their respective images.

$\Box$ \vspace{.5cm}

\begin{cor}\label{towcor10} If $\{ \M^i \}$ is a net of
labels then $\{ \M^i \}$ is convergent iff $\{ x[\M^i] \}$ is convergent and iff
$\{ x_+[\M^i] \}$ is convergent.
\end{cor}

$\Box$ \vspace{.5cm}

The following is the crucial result which relates the
$FIN(\N)$ dynamics on $\LAB$ with the shift dynamics on the
$\{ 0, 1 \}^{\Z}$.

\begin{lem}\label{towlem11} ({\bfseries Coherence Lemma})
Let  $\{ t^i \}$ be a
net in $IP(k)$ with $r^i$ the length of $t^i$.
 If $\{ |j_{r^i}(t^i)| \} \to \infty$ then
\begin{equation}\label{tow16}
Lim_{i \to \infty} \ \sup_{\M \in \LAB} \ d( S^{t^i}(x[ \M ]) \ , \  x[ \M - \rr(t^i) ] \ ) \quad = \quad  0.
\end{equation}

Furthermore, if $t^i \in IP_+(k)$ for $i \in I$ then
\begin{equation}\label{tow16a}
Lim_{i \to \infty} \ \sup_{\M \in \LAB} \ d( S^{t^i}(x_+[ \M ]) \ , \  x_+[ \M - \rr(t^i) ] \ ) \quad = \quad  0.
\end{equation}

That is, the pair of nets
$(\{ S^{t^i}(x[ \M ]) \},\{ x[ \M - \rr(t^i) ] \}) $ is uniformly asymptotic in $X(\LAB)$ and if
$t^i \in IP_+(k)$ for all $i \in I$ then the pair
$$(\{ S^{t^i}(x_+[ \M ]) \},\{ x_+[ \M - \rr(t^i) ] \})$$
is uniformly asymptotic in $X_+(\LAB)$.
\end{lem}

\proof For any positive integer $N$, there exists $i_N$ in the directed set $I$ which indexes the net,
so that $i > i_N$ implies $|j_{r^i}(t^i)| > N$.
 Then Proposition \ref{towprop05} implies for any label $\M$
$[\ t^i \ \pm \ N] \ \ \cap \  \ A[\M] \ =
\ t^i \ + \ (  [ \ \pm \ N \ ] \ \ \cap \  \ A[\M - \rr(t^i)] )$.

This says that for all $t$ with $|t| \leq N$, $x[ \M - \rr(t^i) ]_t \ = \ x[ \M ]_{t^i + t}$.
Thus, if $i > i_N$, then $  x[ \M - \rr(t^i) ]_t = (S^{t^i}(x[ \M ]))_t$ for all $t$ with $|t| \leq N$.
So (\ref{tow16}) follows from the definition (\ref{tow15}) of the metric on $\{ 0,1 \}^{\Z}$.

The proof for $x_+$ is the same.

$\Box$ \vspace{.5cm}

\begin{lem}\label{towlem12} For a label $\M$  assume that $I$ is a directed set
and $\{ t^i : i \in I \}$ is a net in $A[\M] \setminus \{ 0 \}$ with
$r^i$ the length of $t^i$. Assume that  $\{ |j_{r^i}(t^i)| \}$ is bounded.

(a) There exists an integer $j^* \not= 0$ such that $\{ i \in I : j_{r^{i}}(t^{i}) = j^* \}$ is a cofinal subset.

(b) If $\M$ is a label of finite type and  $|t^i|  \to \infty$, then there exists
$\bar t \in A[\M]$ with length $\bar r > 0$ (and so $\bar t \not= 0$) and a subnet $\{ t^{i'} \}$, defined by restricting
to $i' \in I'$, a cofinal subset of $I$,  such that
$t^{i'} - \tilde t^{i'} = \bar t$ for all $i'$,
where $\tilde t^{i'}$ is the $\tilde r^{i'} = r^{i'} - \bar r$ truncation of
$t^{i'}$
and, in addition,
 $ |j_{\tilde r^{i'}}(\tilde t^{i'}) = j_{\tilde r^{i'}}(t^{i'}) |  \to \infty$.
\end{lem}

\proof   The bounded set $\{ j_{r^i}(t^i) \}$ is finite. It follows that there exists
 $j^{-1} \in \Z \setminus \{ 0 \}$ such that  $j_{r^i}(t^i) =  j^{-1}$ for $i$
in a cofinal subset $SEQ_{-1}$ of I. Let $\bar t_{-1} = k(j^{-1})$
and let $\tilde t_{-1}^i$ be the $\tilde r^i_{-1} = r^i -1$ truncation of $t^i$ for all $i$ in $SEQ_{-1}$.
Since $j^{-1} \not= 0$, $\bar t_{-1} \not= 0$. Furthermore, $\rr(t^i) = \rr(\tilde t_{-1}^i) + \chi(\ell(j^{-1})) \in \M$
for all $i \in SEQ_{-1}$.

For (a) we let $j^* = j^{-1}$ and $I' = SEQ_{-1}$.

Now we assume that $\M$ is of finite type and we continue in order to prove (b).

If $\{ |j_{\tilde r_{-1}^{i}}(\tilde t_{-1}^{i})  | \} \to \infty $ as $i \to \infty$ in $SEQ_{-1}$ then  let $I' = SEQ_{-1}$
use this as the required subnet. Otherwise, it is bounded on some cofinal subset and we can choose
$j^{-2} \in \Z \setminus \{ 0 \}$ such that
 $j_{\tilde r_{-1}^{i}}(\tilde t_{-1}^{i}) = j^{-2}$,
for $i$ in a cofinal subset $SEQ_{-2}$ of $SEQ_{-1}$. Let $\bar t_{-2} = k(j^{-2}) + k(j^{-1})$
and let $\tilde t_{-2}^i$ be the $\tilde r^i_{-2} = r^i -2$ truncation of $t^i$ for all $i$ in $SEQ_{-2}$.
The truncation is proper and $\bar t_{-2} \not= 0$. Furthermore,
$\rr(t^i) = \rr(\tilde t_{-2}^i) + \chi(\ell(j^{-1})) +
\chi(\ell(j^{-2})) \in \M$ for all $i \in SEQ_{-2}$.

Because $\M$ is of finite type, the Finite Chain Condition implies that this procedure must halt after finitely
many steps. That is, for some $k \geq 1$, $\{ |j_{\tilde r_{-k}^{i'}}(\tilde t_{-k}^{i'})  | \} \to \infty $
as $i' \to \infty$ in $SEQ_{-k}$.
The subnet is obtained by restricting to $I' = SEQ_{-k}$ and $\bar t = \bar t_{-k}$ with
length $\bar r = k$.

Since $\tilde t^{i'}$ is the truncation of $t^{i'}$ to $\tilde r^{i'}$,
we have
$|j_{\tilde r^{i'}}(\tilde t^{i'}) = j_{\tilde r^{i'}}(t^{i'}) | \to \infty $.

$\Box$ \vspace{.5cm}

\begin{remark} In (b) if $t^i \in A_+[\M]$ for all $i$ then
$\bar t \in A_+[\M]$ and the truncations are in $IP_+(k)$.
 \end{remark}

\vspace{.5cm}


\begin{df}\label{xiA}
For a nonempty subset $A \subset \Z$ define the
\emph{distance from $t$ to $A$} by $\xi_{A}: \Z \to \Z_+$ by
\begin{equation}\label{edistancedef}
\xi_{A}(t) = \min \{ n : t + n  \in A \ \text{or} \
 t - n \in A \}.
\end{equation}
Thus, $\xi_{A}(t) = 0 $ iff $t \in A$.
\end{df}

For $\M$ a nonempty label, we denote $\xi_{A[\M]}$ by $\xi_{\M}$ and  $\xi_{A_+[\M]}$ by $\xi^+_{\M}$.
From the symmetry of $A[\M]$ we see that $\xi_{\M}(-t) = \xi_{\M}(t)$ for all $t$. On the other hand,
$ \xi^+_{\M}(t) = |t|$ if $t \leq 0$.

Recall that $e \in \{ 0, 1 \}^{\Z}$ is the fixed point of $S$ with $e_t = 0$ for all $t$.

\begin{lem}\label{elimit}
Let $I$ be a directed set,
$\{t^i : i  \in I \}$ be a net in $\Z$ and
$\{ A^i : i  \in I \}$ be a net of nonempty subsets of $\Z$. The net
 $\{ S^{t^i}(x[A^i]) : i \in I \}$ in $\{ 0, 1 \}^{\Z}$ converges to $e$ iff $\xi_{A^i}(t_i) \to \infty$.
\end{lem}

 \proof For any net $\{z^i \} \in \{ 0, 1 \}^{\Z}, \  S^{t^i}(z^i) \to e $ iff
 for all $n \in \N$ eventually $z^i_{t^i + k} = 0$ for all $k$ with $|k| < n$.

$\Box$ \vspace{.5cm}

For $t \in  \Z \setminus \{ 0 \}$, recall that
$n_*(t)$ is the place of the smallest nonzero entry in the $b$ expansion of $t$ so that when $t \in IP(k)$,
$n_*(t) = |j_r(t)|$ with
$r = r(t)$.

 \begin{theo}\label{towtheo13a} Let $I$ be a directed set,
 $\{t^i : i \in I \}$ be a net of nonzero integers and
$\{ \M^i : i \in I \}$ be a net of labels in $\LAB$.

 If the net
  $\{ S^{t^i}(x[\M^i]) : i \in I \}$ converges to $z \in  \{ 0, 1 \}^{\Z} $ with $z \not= e$,  then the following are equivalent.

  \begin{itemize}
  \item[(i)]
  $z_0 = 1$ and  $z_t = z_{-t}$ for all $t \in \Z$.

 \item[(ii)] $z = x[\NN]$ for some label $\NN$.

 \item[(iii)] $n_*(t^i) \to \infty$.

  \item[(iv)] There exists $i_0 \in I$ so that on the subnet obtained by restricting to $i > i_0$,
    $t^i \in A[\M^i]$ and $|j_{r_i}(t^i)| \to \infty$ with $r_i$ the length of $t^i$.
 \end{itemize}

  If the net
  $\{ S^{t^i}(x_+[\M^i]) : i \in I \}$ converges to $w \in  \{ 0, 1 \}^{\Z} $ with $w \not= e$,  then the following are equivalent.

  \begin{itemize}

   \item[(i')] $w_0 = 1$ and $w_t = 0$ for all $t < 0 \in \Z$.

  \item[(ii')] $w = x_+[\NN]$ for some label $\NN$.

 \item[(iii')] $n_*(t^i) \to \infty$.

  \item[(iv')] There exists $i_0 \in I$ so that on the subnet obtained by restricting to $i > i_0$,
    $t^i \in A_+[\M^i]$ and $|j_{r_i}(t^i)| \to \infty$ with $r_i$ the length of $t^i$.

 \end{itemize}

 When either of these sets of conditions hold the label $ \NN$ of (ii) or (ii') is $ LIM_{i > i_0} \ \{ \ \M^i - \rr(t^i) \ \} $.
In particular, if $\M = \M^i$ for all $i$ then  $\NN \in \Theta'(\M)$.
\end{theo}

 \proof Since $z$ (or $w$) is not $e$, the labels $\M^i$ are eventually nonempty.

 (iii) $\Rightarrow$ (iv):  If for some subnet $\{ t^{i'} \}$, we have $\xi_{\M^i}(t^{i'}) \to \infty$ then by Lemma \ref{elimit}
the subnet $\{ S^{t^{i'}}(x[\M^i]) \}$ converges to $e$. Since $z \not= e$ this does not happen.
By the Smith Lemma
\ref{applem05}
it follows that for some $N \in \N$ the set $\{i \in I : \xi_{\M^i}(t^i) > N \}$ is not cofinal and
so for some $i_0 \in I, \xi_{\M^i}(t^i) \leq N $ for all $i \geq i_0$. It follows that there is a nonempty subset $K$
 of the finite interval $[-N,N]$ so that for each $k \in K$,
 $$I_k = \{ i \in I : t^i + k \in A[\M^i]  \ \text{ and} \ t^i + j \not\in A[\M^i] \ \text{when} \ |j| < |k| \}$$
  is cofinal. Thus, $I_k \cup I_{-k} = \{ i : \xi_{\M^i}(t^i) = |k| \}$. We will show that
 $K = \{ 0 \}$.

 Let $k \in K$.  By assumption (iii) eventually $n_*(t^i) > n^*(k)$.  This implies that the
 places where
  $t^i $
 has a nonzero entry in its $b$ expansion are beyond  the places where $k$ has a nonzero entry.
 On the cofinal set with $i \in I_k$ and $n_*(t^i) > n^*(k)$, we have that $t^i + k \in IP(k)$
and that $t^i$ is a truncation of $t^i + k$. Hence $\rr(t^i) \leq \rr(t^i + k) \in \M^i$ and so
 $t^i \in A[\M^i]$.  Thus, $|k| = \xi_{\M}(t^i) = 0$. It thus follows that $\{ i : \xi_{\M^i}(t^i) \not= 0 \}$ is not cofinal.

 That is, eventually $t^i \in A[\M^i] \subset IP(k)$. For $t^i \in IP(k), \ n_*(t^i) = |j_{r^i}(t^i)| $ and so (iv) follows.

 (iii') $\Rightarrow$ (iv'): Use the same proof as above replacing $\xi_{\M^i}$ by $\xi^+_{\M^i}$.

   (iv) $\Rightarrow$ (iii), (iv') $\Rightarrow$ (iii'): Obvious.

   (iv) $\Rightarrow$ (ii), (iv') $\Rightarrow$ (ii') with $\NN = Lim \{ \ \M^i - \rr(t^i) \ \}$:
   Since  $\{ S^{t^i}(x_+[\M^i]) \}$ is convergent the
   Coherence Lemma \ref{towlem11} implies that $\{ \ x[\M^i - \rr(t^i)] \ \}$ is convergent.
   By Corollary \ref{towcor10} it follows that
 $\{ \M^i - \rr(t^i) \}$ is convergent in $\LAB$.  If $\NN$ is its limit then Theorem
 \ref{towtheo09} implies that $\{ \ x[\M^i - \rr(t^i)] \ \}$ and so
 $\{ S^{t^i}(x_+[\M^i]) \}$ both converge to $x[\NN]$.

 (ii) $\Rightarrow$ (i),  (ii') $\Rightarrow$ (i'): Obvious.

 (i) $\Rightarrow$ (iv): Since $z_0 = 1$ it follows that eventually $S^{t^i}(x[\M^i])_0 = x[\M^i]_{t^i} = 1$
 and so eventually $t^i \in A[\M^i]$. Assume this is true for all $i \geq i_0$.

 We have to show that $ |j_{r^i}(t^i)| \to \infty$. If this is not true then on some cofinal set $ |j_{r^i}(t^i)|$ remains
 bounded.  From Lemma \ref{towlem11} there is a nonzero $j^*$ and a cofinal set $I'$ such that $j_{r^i}(t^i) = j^*$
 for all $i \in I'$. Let
 $\tilde t^i = t^i - k(j^*)$ so that $\tilde t^i$ is a truncation of $t^i$ and so
 $\tilde t^i \in A[\M^i] \subset IP(k)$. On the other hand
 $t^i + k(j^*) \not\in IP(k)$ because it has a $2$ or a $-2$ at the $|j^*|$ place in its expansion. Thus,
 for $i \in I'$ we have $S^{t^i}(x[\M^i])_{-k(j^*)} = 1$, but $S^{t^i}(x[\M^i])_{k(j^*)} = 0$.  Since $S^{t^i}(x[\M^i]) \to z$
 we have $z_{-k(j^*)} = 1$ and $z_{k(j^*)} = 0$ which contradicts the symmetry assumption on $z$.

 (i') $\Rightarrow$ (iv'): Begin as above. Now $t^i \in A_+[\M^i]$ and so $j^* > 0$.
 As before we have $S^{t^i}(x[\M^i])_{-k(j^*)} = 1$
 and so in the limit $w_{-k(j^*)} = 1$. This contradicts the assumption that $w_t = 0$ for all $t < 0$.

$\Box$ \vspace{.5cm}

\begin{cor}\label{towcor13b} For a label $\M$ the following are equivalent.
\begin{itemize}
\item[(i)] $\M$ is recurrent for the $FIN(\N)$ action on $\LAB$.

\item[(ii)] $x[\M]$ is a recurrent point
for the shift homeomorphism on $\{ 0, 1\}^{\Z}$.

\item[(iii)] $x_+[\M]$ is a recurrent point
for the shift homeomorphism on $\{ 0, 1\}^{\Z}$.
\end{itemize}
\end{cor}

\proof We may assume $\M \not= \emptyset$ since all three properties hold when $\M = \emptyset$.

(i) $\Rightarrow$ (ii), (iii): Recall that $\M$ is recurrent when $\M = LIM \{ \M - \rr^i \}$
for some sequence $\rr^i > 0$ in $FIN(\N)$.
By  Lemma \ref{towlem08} (a) and induction, we can choose a sequence
$\{ t^i \in IP_+(k) \}$ with $\rr(t^i) = \rr^i$ and with
$j_r(t^i) \to \infty$.
By the Coherence Lemma \ref{towlem11} $\{ S^{t^i}(x[\M]) \}$ is asymptotic to
$\{ x[\M - \rr^i] \}$ which converges to $x[\M]$ by Theorem \ref{towtheo09}. Similarly,
$\{ S^{t^i}(x_+[\M]) \}$
is asymptotic to
$\{ x_+[\M - \rr^i] \}$ which converges to $x_+[\M]$.

  (ii) $\Rightarrow$ (i): $x[\M]$ is recurrent when there is a sequence $\{ t^i \in \Z \setminus \{ 0 \} \}$ such that
  $\{ S^{t^i}(x[\M]) \}$ converges to $x[\M]$. By  Theorem \ref{towtheo13a} (ii) $\Rightarrow$ (iv), we have
  $t^i \in A[\M]$ with $j_r(t^i) \to \infty$. By the Coherence Lemma \ref{towlem11} again,
  $\{ S^{t^i}(x[\M]) \}$ is asymptotic to
$\{ x[\M - \rr(t^i)] \}$ and so the latter converges to $x[\M]$. By Theorem \ref{towtheo09}, $\{ \M - \rr(t^i) \}$ converges
to $\M$ and so $\M$ is recurrent.

 (iii) $\Rightarrow$ (i): Since $x_+[\M] \not= e$, $x_+[\M]$ is recurrent when there is a sequence $\{ t^i \in \N \}$ such that
  $\{ S^{t^i}(x_+[\M]) \}$ converges to $x_+[\M]$. Apply Theorem \ref{towtheo13a} (ii') $\Rightarrow$ (iv') and the
  Coherence Lemma as before to show that $\M$ is recurrent.

$\Box$ \vspace{.5cm}

Now we interpret these results in terms of the Ellis action of $\b \Z$ on $\{ 0, 1 \}^{\Z}$,  see Appendix \ref{appendix-StoneCech}.

Since $n_*(0) = \infty$, $n_*$ is a function from $\Z$ to the compact space $\Z_{+ \infty}$. For
$A$ a nonempty subset $\Z$ we can similarly regard $\xi_A$ as a function from $\Z$ to  $\Z_{+ \infty}$.
So we obtain the continuous extensions
\begin{equation}\label{ellis01}
\b n_* : \b \Z \to  \Z_{+ \infty}, \qquad \b \xi_A : \b \Z \to  \Z_{+ \infty}.
\end{equation}

The function $m_b : \Z \to \Z$ defined by $t \mapsto b t$ is an injective group homomorphism
and so extends to
\begin{equation}\label{ellis02}
\b m_b : \b \Z \to \b \Z.
\end{equation}
a continuous, injective monoid homomorphism 
which is injective by Lemma \ref{applem02}. 
Hence, its image is a closed submonoid. Define
\begin{equation}\label{ellis03}
\b_b \Z \quad = \quad \bigcap_{k=1}^{\infty} \ (\b m_b)^k(\b \Z).
\end{equation}

\begin{prop}\label{ellisprop04} The subset  $\b_b \Z$ is an uncountable, closed submonoid of $\b \Z$ on which
$\b m_b$ restricts to an isomorphism. Furthermore,
\begin{equation}\label{ellis05}
\b_b \Z \quad = \quad (\b n_*)^{-1}(\infty).
\end{equation}
\end{prop}

\proof  
As above $\b b_m$ is injective.  

If $p \in \b_b \Z \subset \b m_b(\b \Z)$ there is a unique $q \in \b \Z$ such that $p = \b m_b(q)$.
For $k \geq 2$, $p \in  (\b m_b)^k(\b \Z)$ implies $q \in (\b m_b)^{k-1}(\b \Z)$ since $\b m_b$ is
injective. Hence, $q \in  \bigcap_{k=1}^{\infty} \ (\b m_b)^k(\b \Z) = \b_b \Z$.  That is,
 $\b m_b$ is bijective on $\b_b \Z$. 

As it is the intersection of closed submonoids, $\b_b \Z$ is a closed submonoid.
Define the injection $j : \N \to \Z$ by $j(n) = b^n$. By Lemma \ref{applem02} the restriction of $\b j$ to $\b^* \N$ is an injection
of an uncountable set into $\b_b \Z$. Hence, $\b_b \Z$ is uncountable.

From (\ref{tow01ba}) it follows that $m_b^k(\Z) = \{ n_* > k \}$. Taking the closure in $\b \Z$ we obtain
(\ref{ellis05}).

$\Box$ \vspace{.5cm}

\begin{prop}\label{ellisprop06} Let $\M$ be a nonempty label and $p \in \b \Z$.

\begin{enumerate}
\item[(a)] The subsets $(\b \xi_{\M})^{-1}(\infty)$
and $(\b \xi^+_{\M})^{-1}(\infty)$ are closed, uncountable ideals of $\b^* \Z$. Furthermore,
\begin{align}\label{ellis07}
\begin{split}
(\b \xi_{\M})^{-1}(\infty) \quad = \quad & \{ p \in \b \Z : p x[\M] = e \}, \\
(\b \xi^+_{\M})^{-1}(\infty) \quad = \quad & \{ p \in \b \Z : p x_+[\M] = e \}.
\end{split}
\end{align}

\item[(b)] If $p \in \b_b \Z$ then $p x[\M] = x[\NN]$ for some label $\NN$. Conversely,
if $p x[\M] = x[\NN]$ for some nonempty label $\NN$ then $p \in \b_b \Z$.
\end{enumerate}
\end{prop}

\proof Let $\{ t^i  \}$ be a net in $\Z$ which converges to $p$. Then $\{ S^{t^i}(x[\M])  \}$
is a net in $\{ 0, 1 \}^{\Z}$ which converges to $p x[\M]$ and $\{ S^{t^i}(x_+[\M]) \}$
converges to $p x_+[\M]$. Furthermore, $\{ \xi_{\M}(t^i) \} \to \b \xi_{\M}(p)$,  $\{ \xi^+_{\M}(t^i) \} \to \b \xi^+_{\M}(p)$
and $\{ n_*(t^i) \} \to \b n_*(p)$.

(a):  From Lemma \ref{elimit} it follows that $p x[\M] = e$ iff $\b n_*(p) = \infty$ and similarly for $x^+[\M]$.
If $q \in \b \Z$ and $p x[\M] = e$ then $q p x[\M] = e$. Hence, $ \{ p \in \b \Z : p x[\M] = e \}$ is an ideal.
Since $A[\M]$ has Banach density zero, we can choose $j : \N \to \N$ increasing and so that
$\xi_{\M}(j(n) + i) = 0$ for all $i$ with $|i| \leq n$. It is clear that $\b j(\b^* \N)$ is an uncountable
subset of $ \{ p \in \b \Z : p x[\M] = e \}$. Finally, since $\M$ is nonempty, $\xi_{\M}(t)$ is finite for all
$t \in \Z$ and so $ \{ p \in \b \Z : p x[\M] = e \} \subset \b^* \Z$.  Use similar arguments for $x^+[\M]$.

(b): If $\NN$ is nonempty, it follows from Theorem \ref{towtheo13a} that if $p x[\M] = x[\NN]$ then
$\b n_* (p) = \infty$ and, conversely, if $\b n_* (p) = \infty$ and $p x[\M] \not= e$ then
$p x[\M] = x[\NN]$ for some nonempty label $\NN$. Finally, if $p x[\M] = e$ then
$p x[\M] = x[\NN]$ with $\NN = \emptyset$. Again, use similar arguments for $x^+[\M]$.

$\Box$ \vspace{.5cm}

From the $\b FIN(\N)$ Ellis action on $\LAB$ we defined the enveloping semigroup $\E(\LAB) \subset \LAB^{\LAB}$
with the unique minimal idempotent $U$ given by $U(\M) = \emptyset$ for $\M \not= FIN(\N)$, see Proposition \ref{minidprop}.
On the other hand, $P(FIN(\N)) = FIN(\N)$ for all $P \in \E(\LAB)$, including $U$.
We define $\hat U \in \LAB^{\LAB}$ by
$\hat U (\M) = \emptyset$ for all $\M$ including $FIN(\N)$ and let $\hat \E (\LAB) = \E(\LAB) \cup \{ \hat U \}$.
This is an Ellis semigroup with minimal idempotent $\hat U$.

Define:
\begin{align}\label{ellis08}
\begin{split}
D, D_+ : \b_b \Z \to & \hat\E (\LAB) \qquad \text{by} \\
D(p)(\M) = \NN \quad &\Longleftrightarrow \quad p x[\M] = x[\NN], \\
D_+(p)(\M) = \NN \quad &\Longleftrightarrow \quad p x_+[\M] = x_+[\NN]
\end{split}
\end{align}

That is,
\begin{equation}\label{ellis09}
p x[\M] \ = \ x[D(p)(\M)] \quad \text{and} \quad p x_+[\M] \ = \ x_+[D_+(p)(\M)].
\end{equation}

\begin{theo}\label{ellistheo10} The maps $D, D_+ : \b_b \Z \to  \hat\E (\LAB)$ are continuous, surjective monoid
homomorphisms.

Let $\{ t^i : i \in I \}$
be
a net in $\Z$ converging to $p$ in $\b_b \Z$.

\begin{enumerate}
\item[(a)] If $t^i \in IP(k)$ for all $i$ then $D(p) = LIM \{ P_{\rr(t^i)} \}$ in $\E (\LAB)$.

\item[(b)] If $t^i \in IP_+(k)$ for all $i$ then $D_+(p) = LIM \{ P_{\rr(t^i)} \}$ in $\E (\LAB)$.

\item[(c)] If $\{ i : t^i \not\in IP(k) \}$ is cofinal, then $D(p) = \hat U$.

\item[(d)] If $\{ i : t^i \not\in IP_+(k) \}$ is cofinal, then $D_+(p) = \hat U$.

\end{enumerate}
\end{theo}

\proof Since the map $\M \mapsto x[\M]$ is injective, Proposition \ref{ellisprop06} the maps $D$ and $D_+$
are well-defined by (\ref{ellis09}).  For $q, p \in \b_b \Z$ and an arbitrary label $\M$
\begin{equation}\label{ellis11}
x[D(qp)(\M)] = qp x[\M]) = q x[D(p)(\M)] = x[D(q)(D(p)(\M))].
\end{equation}
Hence, $D$ is a homomorphism.

If $\{ p^i \}$ is a net in $\b_b \Z$ converging to $p$, then
\begin{equation}\label{ellis12}
x[D(p)(\M)] = p x[\M] = LIM p^i x[\M] = LIM x[D(p^i)(\M)].
\end{equation}
So by Theorem \ref{towtheo09}   $D(p)(\M) =  LIM D(p^i)(\M)$.  Since $\M$ is arbitrary, $D$ is continuous.

The limit results of (a)  follow from Theorem
\ref{towtheo13a} since $n_*(t^i) \to \infty$. Furthermore, if
the limit $p x[\M]$ is not $e$ for some $\M$ then by Theorem \ref{towtheo13a} eventually $t^i$ is in $A[\M] \subset IP(k)$. So (c) follows as well.
In particular, choosing a sequence $t^n \in \N \setminus IP_+(k)$ with $n_*(t^n) \to \infty$ we see that the
for any $p$ in the set of limit points of the sequence, $p \in \b_b \Z$ and $D(p) = \hat U$.

Note that $0 \in \b_b \Z$ and $D(0) = id_{\LAB} = P_{\00 }  $.

If $P \not= id_{\LAB} \in \E(\LAB)$ then we can choose a net $\{ \rr^i > \00 : i \in I \}$ such
that $\{ P_{\rr^i} \} \to P$. On the
product directed set $I \times \N$ apply  Lemma \ref{towlem08}     to  choose $t^{(i,n)}$ so that $\rr(t^{(i,n)}) = \rr^i$ and
$j_{r}(t^{(i,n)}) > n$. Since $\{j_{r}(t^{(i,n)})\} \to \infty$, $S^{t^{(i,n)}}(x[\M])$ is asymptotic to
$ \{ x[\M - \rr(t^{(i,n)})] =  x[\M - \rr^i] \}$ which has limit $x[P(\M)]$. Also
$\{ n_*(t^{(i,n)})= j_{r}(t^{(i,n)})\} \to \infty$ implies that any limit point $p$ of the net $ \{ t^{(i,n)} \}$ in
$\b \Z$ is in fact in  $\b_b \Z$, and $D(p) = P$.

Thus, we see that $D$ is surjective.

Use similar arguments for $D_+$.

$\Box$ \vspace{.5cm}


For any label $\M \not= FIN(\N)$, $\Theta(\M)$ is a proper,
closed $FIN(\N)$ invariant subspace of $\LAB$ and so the restriction map is a continuous, surjective monoid homomorphism from $\E(\LAB) \to \E(\Theta(\M))$. We can regard it
as a map from $\hat \E(\LAB)$ with $\hat U$ and $U$ both mapping to $U$ in $\E(\Theta(\M))$. Composing with this restriction we obtain
continuous, surjective monoid homomorphisms
\begin{equation}\label{ellis13}
\begin{split}
D, D_+ : \b_b \Z \to \E(\Theta(\M)), \hspace{3cm} \\
p x[\NN] \ = \ x[D(p)(\NN)] \quad \text{and} \quad p x_+[\NN] \ = \ x_+[D_+(p)(\NN)],
\end{split}
\end{equation}
for $\NN \in \Theta(\M)$.

On the other hand, the $\Z$ action on $X(\M)$ extends to an Ellis action of $\phi : \b \Z \times X(\M) \to X(\M)$
with $\phi^{\#}: \b Z \to X(\M)^{X(\M)}$  a continuous monoid homomorphism with image the enveloping semigroup $E(X(\M),S)$.
Similarly, the  action $\Z$ on $X_+(\M)$ extends to $\phi_+ : \b \Z \times X_+(\M) \to X_+(\M)$ with
$\phi_+^{\#}(\b \Z) = E(X_+(\M),S)$.
We let
\begin{equation}\label{ellis14}
E_b(X(\M),S) = \phi^{\#}(\b_b \Z) \quad \text{and} \quad E_b(X_+(\M),S) = \phi_+^{\#}(\b_b \Z).
\end{equation}
The maps $D$ and $D_+$ clearly factor through $\phi^{\#}$ and $\phi_+^{\#}$ to define continuous, surjective
homomorphisms
\begin{equation}\label{ellis15}
\begin{split}
D : E_b(X(\M),S)  \to \E(\Theta(\M)), \quad \text{and} \quad D_+ : E_b(X_+(\M),S)  \to \E(\Theta(\M)) \\
p x[\NN] \ = \ x[D(p)(\NN)] \quad \text{and} \quad p x_+[\NN] \ = \ x_+[D_+(p)(\NN)],
\end{split}
\end{equation}
for $\NN \in \Theta(\M)$.

\begin{df}\label{dfNuc}
For $\M \in \LAB$ let $\Nuc(X(\M)) = x[\Theta(\M)]$.
We have
\begin{equation*}\label{nucleus}
\Nuc(X(\M)) = E_b(X(\M),S)x[\M] = \{x[\NN] : \NN \in \Theta(\M)\}.
\end{equation*}
We call this closed, $E_b(X(\M),S)$ invariant subset, the {\em nucleus}
of the labeled subshift $(X(\M),S)$. $\Nuc(X_+(\M))$ is defined analogously.
\end{df}

\vspace{.5cm}

\begin{df}\label{df,Phi}
 (a) If $Y$ is a subset of $\{ 0, 1 \}^{\Z} $ we let
 $$  \Phi(Y) = \{ \NN : x[\NN] \in Y \} \quad {\text{and}} \quad
 \Phi_+(Y) = \{ \NN : x_+[\NN] \in Y \}. $$

 (b) If $\Psi$ is a nonempty subset of $\LAB$ we let $\Theta(\Psi)$ denote the smallest, closed
 $FIN(\N)$ invariant subset of $\LAB$ which contains $\Psi$.


 \end{df}

 That is, $\Phi(Y)$ and $\Phi_+(Y)$ are
 the preimages of $Y$ by the maps $x[\cdot]$ and
 $x_+[\cdot]$ respectively.
 By Theorem \ref{towtheo09}
 $\NN$ is uniquely determined by $x[\NN]$ and by $x_+[\NN]$.

Recall that for $\Psi$
a nonempty subset of $\LAB$,  $X(\Psi)$ and $X_+(\Psi)$ are the closed,
shift invariant subsets of $\{ 0,1 \}^{\Z}$ generated by the
 $x[\NN]$'s and $x_+[\NN]$'s respectively with $\NN $ varying over $\Psi$.

 \vspace{.5cm}

 \begin{theo}\label{towtheo14} (a) If $Y$ is a nonempty, closed, shift invariant subset of $X(\P IP(k)) $ then
 $\Phi(Y), \Phi_+(Y)$ are closed, $FIN(\N)$ invariant
 subsets of $\LAB$ with $\emptyset \in \Phi(Y) \cap \Phi_+(Y)$. Furthermore,
 $$ X(\Phi(Y)) \cup X_+(\Phi_+(Y)) \ \subset \ Y.$$

 (b) If $\Psi$ is a nonempty subset of $\LAB$, then
 $$ X(\Psi)  \ = \ X(\Theta(\Psi))  \quad \text{and} \quad X_+(\Psi)  \ = \ X_+(\Theta(\Psi)).$$

 (c) If $\Psi$ is a closed, $FIN(\N)$ invariant subset of $\LAB$ with $\emptyset \in \Psi$ then
 $$ \Psi \ = \ \Phi(X(\Psi)) \ = \ \Phi_+(X_+(\Psi)).$$
 \end{theo}

 \proof (a):
Because the maps $x[\cdot], x_+[\cdot]$ are continuous, $\Phi(Y)$ and $\Phi_+(Y)$ are closed when
$Y$ is. Since $IP(k)$ has Banach density zero, the fixed point $e$ is contained in $Y$ and so $x[\emptyset] = x_+[\emptyset] = e$
implies that $\emptyset $ is contained in $\Phi(Y) $ and $\Phi_+(Y)$.

The final inclusion in (a) is obvious.

 If $x[\NN] \in Y$ (or $x_+[\NN] \in Y$) and $\rr$ is a nonzero $\N$-vector then by Lemma \ref{towlem08} we can choose a sequence
 $\{ t^i : i \in \N \}$ in $IP_+(k)$ with length
vector $\rr$ and with $|j_{r}(t^i)| \to \infty$.  By Theorem
\ref{towtheo13a}
$S^{t_i}x[\NN] \to x[\NN - \rr]$ (or $S^{t_i}x_+[\NN] \to x_+[\NN - \rr]$).
Since $Y$ is closed, $\NN -  \rr \in \Phi$ (resp. $\NN -  \rr \in \Phi_+$ ). Thus, $\Phi(Y)$ and $\Phi_+(Y)$ are $FIN(\N)$
invariant.

(b): With $\NN \in \Psi$ apply the above argument  to $x[\NN] \in X(\Psi)$ we see that $x[\NN - \rr] \in X(\Psi)$ for
all $\rr \in FIN(\N)$. Since $x[\cdot]$ is continuous and $X(\Psi)$ we see that $x[\M] \in X(\Psi)$ for all $\M \in \Theta(\Psi)$.
Hence, $X(\Theta(\Psi)) \subset X(\Psi)$ and the reverse inclusion is obvious. The argument for $X_+$ is similar.

(c): Clearly, $\Psi \subset \Phi(X(\Psi))$.
If $\NN \in \Phi(X(\Psi))$ with $\NN$ nonempty, then there is a sequence $\{ S^{t_i}(x[\M^i]) \}$
converging to $x[\NN]$ with $\M^i \in \Psi$.

If $t^i = 0$ infinitely often then $x[\NN]$ is the limit
of a subsequence  $\{ x[\M^{i'}] \}$. Then $\NN $ is the limit of $\{ \M^{i'} \}$ because $x[\cdot]$ is a
homeomorphism.  Since $\Psi$ is closed, $\NN \in \Psi$.

Alternatively, we may assume that the $t^i$'s are nonzero.  Theorem \ref{towtheo13a} implies that eventually
$t^i \in A[\M^i]$ and $\NN = LIM \{ \M - \rr(t^i) \}$. Since $\Psi$ is closed and $FIN(\N)$ invariant
it follows that $\NN \in \Psi$.

Thus, $\Phi(X(\Psi)) $ is contained in and so equals $ \Psi$.

Again,the argument for $X_+$ is similar.

$\Box$ \vspace{.5cm}

\begin{cor}\label{towcor14a}
(a) If a label $\M $ is not the maximum label
$FIN(\N),$
then
\begin{equation}\label{phi1}
\Phi(X(\M))
= \Phi_+(X_+(\M)) = \Theta(\M).
\end{equation}

(b) For $FIN(\N),$ the maximum label, $  \Theta(FIN(\N)) = \{ FIN(\N) \}$ and
\begin{align}\label{phi2}
\begin{split}\Phi(X(FIN(\N))) = & \ \Phi_+(X_+(FIN(\N))) = \\
\Theta(FIN(\N)) \cup &\{ \emptyset \} = \{ FIN(\N), \emptyset \}.
\end{split}
\end{align}
\end{cor}

\proof (a): From Theorem \ref{towtheo14}(b) $X(\M) = X(\Theta(\M))$.  Since $\M \not= FIN(\N)$,
there exists $\rr \in FIN(\N) \setminus \M$ and
so $\M - \rr = \emptyset$.  Hence, $\emptyset \in \Theta(\M)$. Thus, Theorem \ref{towtheo14}(c) implies
that $ \Phi(X(\Theta(\M))) = \Theta(\M)$.

(b): For $FIN(\N)$, the maximum label,
$FIN(\N) - \rr = FIN(\N)$ for all $\rr \in FIN(\N)$ and so $\Theta(FIN(\N)) = \{ FIN(\N) \}$. In particular,
$\emptyset \not\in \Theta(FIN(\N))$.

Hence, $\Psi = \{ FIN(\N), \emptyset \}$ is closed, $FIN(\N)$ invariant and contains $\emptyset$.

Clearly, $ \Phi(X(FIN(\N))) = \Phi(X(\Psi))$ which equals $\Psi$ by   Theorem \ref{towtheo14}(c) again.

The similar arguments for $\Phi_+$ are left to the reader.

$\Box$ \vspace{.5cm}

%

 \begin{theo}\label{towtheo15} Let $\M$ be a label of finite type with $(X(\M),S)$  
 and $(X_+(\M),S)$
 the associated subshifts.
 \begin{itemize}
 \item[(a)] 
 $$
 X(\M) = \{ S^k(x[\NN]) : k \in \Z,\  \NN \in \Theta(\M) \} = \bigcup_{k \in \Z} S^k \Nuc(X(\M))
 $$ 
 and
 $$
 X_+(\M) = \{ S^k(x_+[\NN]) : k \in \Z, \ \NN \in \Theta(\M) \} = \bigcup_{k \in \Z} S^k \Nuc(X_+(\M)).
 $$

 \item[(b)] If $Y$ is a nonempty, closed, shift invariant subset of $X(\M) $ then
  $$ Y = X(\Phi(Y)) =  \{ S^k(x[\NN]) : k \in \Z, \NN \in \Phi(Y) \}.$$

 \item[(c)] If $Y$ is a nonempty, closed, shift invariant subset of $X_+(\M) $ then
  $$ Y = X_+(\Phi_+(Y)) =  \{ S^k(x_+[\NN]) : k \in \Z, \NN \in \Phi_+(Y) \}.$$

 \end{itemize}
\end{theo}

\proof (a): Let $z \in X(\M)$.

If $z = e$ then $z = x[\emptyset]$ and $\emptyset \in \Theta(\M)$ since $\M \not= FIN(\N)$.

Assume that $z \not= e$ and so $z = S^k(w)$ with $w_0 = 1$. Since $w \in X(\M)$
it is the limit of $\{ S^{t^i}(x[\M]) \}$ for some sequence $\{ t^i \}$ in $\Z$. Eventually $t^i \in A[\M]$ and so
we may assume that $t^i \in A[\M]$ for all $i$.

If some subsequence of $\{ t^i \}$ is bounded then there exists $k_1 \in \Z$ such that $t^i = k_1$ infinitely often and
so $w = S^{k_1}(x[\M])$.

Now assume  that $|t^i| \to \infty$ and $t^i \not= 0$ for all $i$.

If  $n_*(t^i) = |j_{r(t^i)}(t^i)| \to \infty$ then by  Theorem \ref{towtheo13a} $w = x[\NN]$ with $\NN \in \Theta'(\M)$.

Notice that  these results only require $\M \not= FIN(\N)$.

Now assume that $\{ |j_{r(t^i)}(t^i)| \}$ has a bounded subsequence. By Lemma \ref{towlem12}
we may go to a subsequence $\{ t^{i'}  \}$ so that
$t^{i'} = \tilde t^{i'} + \bar t$ with $\tilde t^{i'} \in A[\M]$ and $n_*(\tilde t^{i'}) \to \infty$. By  Theorem \ref{towtheo13a}
again $\{ S^{\tilde t^{i'}}(x[\M]) \}$ has limit $x[\NN]$ for some $\NN \in \Theta(\M)$. Since $\bar t$ is constant,
$w = Lim \{ S^{t^{i'}}(x[\M]) \} = S^{\bar t}(x[\NN])$.

Thus, in any of these cases $w = S^{k_2}(x[\NN])$ for some $k_2 \in \Z$ and $\NN \in \Theta(\M)$.  Finally
$z =  S^{k + k_2}(x[\NN])$ as required.

The argument for $z \in X_+(\M)$ is similar.  If $w \in X_+(\M)$ with $w_0 = 1$ then we can assume that
$t^i \in A_+[\M]$ for all $i$ and proceed as above.

(b), (c): If $Y$ is an invariant subset of $X(\M)$ then $Y$ consists of the orbits of some of these $x[\NN]$.  That is,
$\Phi(Y) \subset \Theta(\M)$ and $Y$ consists of the orbits of the points $x[\NN]$ for $\NN \in \Phi(Y)$.
Thus, $Y = X(\Phi(Y))$. The $X_+$ case in (c) is similar.

$\Box$ \vspace{.5cm}


\begin{cor}\label{lattice1-1}
Given $\M \in \LAB$:
\begin{enumerate}
\item
The map $\Psi  \mapsto X(\Psi)$ is one-to-one from the collection of closed $FIN(\N)$ invariant
subsets of $\Theta(\M)$ to collection of closed $S$-invariant subsets of $X(\M)$.
\item
If $\M$ is of finite type then this map is surjective, i.e.
every closed $S$-invariant subsets $Y$ of $X(\M)$
is of the form $X(\Psi)$ for some closed $FIN(\N)$ invariant $\Psi \subset \Theta(\M)$,
with
$$
\Phi(X(\Psi)) = \Psi \ \  {\text{and}} \  \ Y = X(\Phi(Y)).
$$
\end{enumerate}
Similar statements hold for $X_+(\M)$.
\end{cor}

%
%

\begin{proof}
Combine Theorem \ref{towtheo14} (c) and Theorem \ref{towtheo15} (b).
\end{proof}

$\Box$ \vspace{.5cm}

\begin{cor}\label{towcor17} If a label $\M$ is not of finite
type then $X(\M)$ and $X_+(\M)$ each contain  non-periodic recurrent points.

If a label $\M$ is of finite type then $e$ is the only
recurrent point of $X(\M)$ or $X_+(\M)$ and so $(X(\M),S)$ and $(X_+(\M), S)$ are  CT systems.
In that case, $(X(\M),S)$ and $(X_+(\M),S)$
 are
LE and topologically transitive but not weak mixing.
\end{cor}

\proof If $\M$ is not of finite type then  Proposition \ref{labelcor05b} implies that there is a positive recurrent label $\NN$
with $\NN \in \Theta(\M)$. Hence, $x[\NN] \not= e$ is a recurrent point of   $ X(\M)$ by Corollary \ref{towcor13b}.

If $\M$ is of finite type then by Corollary \ref{towtheo15} every point of $X(\M)$ is on the orbit of some
$x[\M_1]$ with $\M_1 \subset \M$. These are all labels of finite type and so none is  recurrent except for $\M_1 = \emptyset$.

Since $x[\M]$ is always a transitive point for $X(\M)$, $(X(\M),S)$ is always topologically transitive.  In the finite
type case, it is CT and so is LE and not weak mixing (see Remarks \ref{cp,weakmix} and \ref{cp,ae}).

$\Box$ \vspace{.5cm}

\begin{cor}\label{towcor17a} Assume that $\M_1$ and $\M_2$ are labels with $\M_1$ of finite type.  The subshifts $(X(\M_1),S)$ and
$(X(\M_1),S)$ are isomorphic only if $\M_1 = \M_2$. \end{cor}

\proof The result is clearly true if each of $\M_1$ and $\M_2$ is either $0$ or $\emptyset$.
If $\M_2$ is not of finite type then by Corollary \ref{towcor17} $X(\M_2)$ and $X_+(\M_2)$
contain non-trivial recurrent points while
$X(\M_1)$ and $X_+(\M_1)$ do not.

We may restrict to the case where both $\M_1$ and $\M_2$ are of finite type with
$\rr > \00 $ in $\M_1 \setminus \M_2$. By Lemma \ref{towlem08} we can choose
a sequence $\{ t^n \in IP_+(k) \}$ so that $\rr(t^n) = \rr$ and $j_r(t^n) > n$. Go to a subnet $\{ t^{n_i} \}$ converging to
$p \in \b \Z$. By  the Coherence Lemma \ref{towlem11} for any label $\M$, $p(x[\M]) = x[\M -\rr]$ and
$p(x_+[\M]) = x_+[\M -\rr]$. In particular, since $\M_2 - \rr = \emptyset$, $p(x[\NN]) =
p(x_+[\NN]) = e$
for every
label $\NN \subset \M_2$. Because $\M_2$ is of finite type,
it follows from Theorem \ref{towtheo15}(a) that
$p(x) = e$ for all $x \in X(\M_2)$ and all $x \in X_+(\M_2)$.
On the other hand, $\M_1 - \rr$ is nonempty and so
$p(x[\NN]) $ and
$p(x[\M_1]) $ and
$ p(x_+[\M_1]) $
are not equal to $ e$.
Since $e$ is the unique minimal point in any $X(\M)$ it follows that
$\{ p : p(x) = e\}$
for all $x \in X(\M) \}$ is an isomorphism invariant for the cascade $(X(\M),S)$.
It follows that
$(X(\M_1),S)$ is not isomorphic to $(X(\M_2),S)$.

$\Box$ \vspace{.5cm}

\begin{Qu}
For distinct recurrent labels $\M_1$ and $\M_2$ can it happen that the
subshifts $(X(\M_1),S)$ and
$(X(\M_2),S)$
are isomorphic?
 \end{Qu}

\vspace{.5cm}

\begin{ex}\label{ex,unc}
For
the finite type label $\M$  in Example \ref{ex10cmoved} $\Theta(\M)$ is uncountable and so the subshift $X = X(\M)$ is uncountable, but nonetheless CT and LE.
In fact, each point of $X$ is an isolated point in its orbit closure,  and
$e$ is the unique recurrent point.
Thus $\M$ is a finite type label with uncountable $X(\M)$.
\end{ex}

\vspace{.5cm}

Define
\begin{equation}\label{symzer2}
\begin{split}
SY\!M = \{ \ x \in \{ 0, 1\}^{\Z} \ : \ x_{0} = 1 \ \text{and}  \ x_{-t} = x_t \ \text{for all} \ t \in \Z \ \}, \\
ZER = \{ \ x \in \{ 0, 1\}^{\Z} \ : \ x_{0} = 1 \ \text{and}  \ x_t  = 0 \ \text{for all} \ t < 0 \ \}.
\end{split}
\end{equation}Ä

\begin{lem}\label{towlem17a} $SY\!M$ and $ZER$ are closed subsets of $\{ 0, 1 \}^{\Z}$.
\begin{equation}\label{eqsymzer2}
\begin{split}
SY\!M \cap X(\LAB) = x[\LAB] \setminus \{ e = x_+[\emptyset] \}; \\
ZER \cap X_+(\LAB) = x_+[\LAB] \setminus \{ e = x_+[\emptyset] \}.
\end{split}
\end{equation}

Each non-periodic $S$ orbit  in $\{ 0, 1 \}^{\Z}$ meets
$SY\!M$ at most once. Each $S$ orbit  in $\{ 0, 1 \}^{\Z}$ meets
$ZER$ at most once \end{lem}

\proof $SY\!M$ and $ZER$ are obviously closed. Equation
\ref{eqsymzer2} follows from Theorem \ref{towtheo13a}.

If $S^{k_1}x, S^{k_2}x \in SY\!M$ with $k_2 \not= k_1$ then for all $t \in \Z$,
\begin{equation}\label{tow20a}
 x_{t + k_1} \ = \ x_{-t + k_1} \ = \ x_{(- t + k_1 - k_2) + k_2} \ = \ x_{(t - k_1 + k_2) + k_2}
 \end{equation}
 Letting $s = t + k_1$ we have that $x_s = x_{s + 2(k_2 - k_1)}$ for all $s \in \Z$. Since $k_2 \not= k_1$
 it follows that $x$ is periodic. By Corollary \ref{towcorBanach03}  $e = x[\emptyset]$ is the only periodic point in
$X(\P IP(k))$.

 The result for $ZER$ is obvious.

 $\Box$ \vspace{.5cm}

 \begin{remark}
 It follows that if $\mu$ is any non-atomic, shift-invariant probability measure on $\{ 0,1 \}^{\Z}$ then
 $\mu(ZER) = \mu(SY\!M) = 0$. Observe first that  the countable set $PER$ of
 periodic points in $\{ 0, 1 \}^{\Z}$ has measure zero because
 $\mu$ is non-atomic. Since $\{ S^k(SY\!M \setminus PER) : k \in \Z \}$ and $\{ S^k(ZER) : k \in \Z \}$
 are pairwise disjoint sequences of sets with identical measure  the common value must be zero.\end{remark}

 \vspace{.5cm}

 \begin{prop}\label{towprop17b} Let $x^*$ be a non-periodic recurrent point of $\{ 0, 1 \}^{\Z}$ and let
 $X$ be its orbit closure, so that $(X,S)$ is the closed subshift generated by $x^*$. If $K$ is a closed subset of
 $X$ such that every non-periodic orbit in $X$ meets $K$ at most once, then $K$ is nowhere dense. The
 set $X \setminus \bigcup_{k \in \Z} \ \{ S^{-k}( K) \}$ is a dense $G_{\delta}$ subset of $X$.

 In particular, $SY\!M \ \cap \  X$ is nowhere dense in $X$ and
 the set of points of $X$ whose
  orbit
 does not meet $SY\!M$
 is a dense $G_{\delta}$ subset of $X$.
 Similarly, $ZER \ \cap \  X$ is nowhere dense in $X$ and
 the set of points of $X$ whose
orbit
 does not meet $ZER$
 is a dense $G_{\delta}$ subset of $X$.\end{prop}

 \proof Since the orbit of $x^*$ is dense in $X$, it meets any nonempty open subset of $X$. If the interior of $K$
 contained more than one point then there would be two disjoint open sets $U_1, U_2$ contained in $K$ and so there
 would exist  $k_1, k_2 \in \Z$ such that $S^{k_a}x^* \in U_a \subset K$ for $a = 1,2$. Since $U_1$ and $U_2$ are disjoint
 $k_2 \not= k_1$. This contradicts the assumption on $K$. Hence, if $K$ has nonempty interior then the interior consists
 of a single point which is on the orbit of $x^*$.  This
 would imply that $x^*$ is an isolated point and so cannot be recurrent
 unless it is periodic. Hence, the interior of $K$ is empty.

 Since $K$ is nowhere dense,  the $G_{\delta}$ set 
 $X \setminus \bigcup_{k \in \Z} \ \{ S^{-k}( K) \}$ is dense by the Baire
 Category Theorem. A point lies in this set exactly when its orbit does not meet $K$.

 The result applies to $K = SY\!M \ \cap \  X$ by Lemma \ref{towlem17a}.

 In the case of $ZER$ the result is clear because the $S^{-1}$ transitive points in the orbit closure of $x^*$  form a
 dense $G_{\d}$ set disjoint from $ZER$.

 $\Box$ \vspace{.5cm}

\begin{cor}\label{towcor15a} For any label $\M$, the set
$\{ e \} \cup X(\M) \ \cap \  SY\!M = x[\Theta(\M)]
= \Nuc(X(\M))$, the nucleus of $X(\M)$, is a compact subset of
$X(\M)$ which meets each orbit in at most one point. The
set $\{ e \} \cup X_+(\M) \ \cap \  ZER = x_+[\Theta(\M)]$ is a compact subset of
$X_+(\M)$ which meets each orbit in at most one point.

 If $\M$ is of finite type then $x[\Theta(\M)]$
meets each orbit in $X(\M)$ and $x_+[\Theta(\M)]$
meets each orbit in $X_+(\M)$.

If $\M$ is not of finite type then $X(\M) \setminus
\bigcup S^i(SY\!M) $
is non-empty and so
 $\{ S^kx[\NN] : k \in \Z, \  \NN \in \Phi(X(\M)) \}$
is a proper subset of $X(\M)$. Furthermore,
$X_+(\M) \setminus [\{ e \} \ \cup \
\bigcup S^i(ZER) ]$
is non-empty and so
 $\{ S^k(x_+[\NN]) : k \in \Z, \  \NN \in \Phi_+(X_+(\M)) \}$
is a proper subset of $X_+(\M)$.
\end{cor}

\proof  By Lemma \ref{towlem17a} a non-periodic orbit meets $SY\!M$ in at most one point.  $X(\M)$ is compact and $SY\!M$ is
closed and so the intersection is compact. $x[\Theta(\M)] \subset X(\M)$ and $x[\LAB] \subset SY\!M$. On the other hand,
by Theorem \ref{towtheo13a}  if $z \in X(\M) \cap SY\!M$ then $z = x[\NN]$ for some
$\NN \in \Theta(\M)$.

If $\M$ is of finite type then each orbit of $X(\M)$ meets $x[\Theta(\M)]$.

If $\M$ is not of finite type then by Corollary
\ref{towcor17}there exists a non-periodic recurrent point $x^* \in X(\M)$.
If $X^*$ is the orbit closure of $x^*$ then $X^* \subset X(\M)$ and by Proposition \ref{towprop17b}
$ X^* \setminus
\bigcup \ S^i(SY\!M) $
 is nonempty.

$\Box$ \vspace{.5cm}

\begin{remark}\label{cross}
Corollary  \ref{towcor15a} shows  that for any label $\M$ of finite type the
dynamical system $(X(\M), S)$ admits 
$\Nuc(X(\M)) = x[\Theta(\M)]$ as a closed {\em
cross-section}; i.e. the orbit of
any point $x \in X(\M)$ meets $x[\Theta(\M)]$ exactly at one point. This is in accordance with the
following general theorem (see
\cite[Section 1.2]{GW-06}).\end{remark}
\vspace{.5cm}

\begin{theo}\label{periodic}
For a system $(X,T)$, with $X$ a completely metrizable
separable space, there exists a Borel cross-section
if and only if the only recurrent points are the
periodic ones.
\end{theo}

$\Box$ \vspace{.5cm}

Notice, too, that if $\mu$ is an invariant probability measure on $(X,T)$ such that the measure of the
set of periodic points is zero, then any cross-section is a non-measurable set. The special case of translation
by rationals on $\R/\Z$ is used in the usual proof of the existence of a subset of $\R$ which is not Lebesgue measurable.

On the other hand, when $\M$ is recurrent, 
$\Nuc(X(\M))= x[\Theta(\M)]$ is a Cantor subset of $X(\M)$ which meets each orbit at most once.
Thus, $S^i(x[\Theta(\M)]) \cap S^j(x[\Theta(\M)]) = \emptyset$
whenever $i \not= j$ in $\Z$. While this explicit construction may be of
interest, in fact any system $(X,T)$ admits such wandering Cantor sets when $X$
is
perfect and  the set of
periodic points has empty interior,
see \cite[Theorem 1.4]{A-02} .

\vspace{1cm}

\section{WAP  labels and their subshifts}\label{ss,finitarysimple}
%
%
%

\subsection{Simple, semi-simple and finitary labels}\label{ss,finitarysimple1}


$\qquad$

\vspace{.5cm}

Since
for
any
label
$\M$,
$\Theta(\M)$ is a closed $FIN(\N)$ invariant set,
 the enveloping semigroup
$\E(\Theta(\M))$
was defined above as the pointwise closure of the set of maps $P_{\rr}$ on $\Theta(\M)$ and the adherence
semigroup
 $\A(\Theta(\M))$ as the closure of the subset of maps $P_{\rr}$ with $\rr > 0$. Of course,
$P_{\00 }$ is the identity on $\Theta(\M)$. We have
\begin{equation}\label{wap00a}
\E(\Theta(\M)) \cdot \M = \Theta(\M), \quad \text{and} \quad \A(\Theta(\M)) \cdot \M = \Theta'(\M).
\end{equation}

Recall that $\NN \in \Theta'(\M)$ implies $max\ \M \cap \NN = \emptyset$, so if $\rr \in max\ \M$ then
for $\NN \in \Theta(\M)$:
\begin{equation}\label{wap00b}
P_{\rr}(\NN) = \begin{cases} 0 \qquad \ \text{if} \ \NN = \M,\\
\emptyset  \qquad \ \text{if}  \ \NN \in \Theta'(\M).\end{cases}
\end{equation}

\begin{lem}\label{waplem00} If $\M$ is a bounded, nonempty label  then $\A(\Theta(\M))$ is a
closed subsemigroup which equals $\E(\Theta(\M)) \setminus \{ P_{\00 } \}$.
If $Q \in \A(\Theta(\M)$ then $Q(0) = \emptyset$. \end{lem}

\proof Since $\M$ is
nonempty and bounded, it follows from  Lemma \ref{labellem05c} that $0 \in \Theta(\M)$.
The set
of elements
$Q \in \E(\Theta(\M))$
such that $Q(0) = \emptyset$ is a closed subsemigroup which contains all $P_{\rr}$ for $\rr > 0$. So it contains
$\A(\Theta(\M)$.  The sole remaining point $P_{\00 }$ has $P_{\00 }(0) = 0$ and so is not in $\A(\Theta(\M))$.

 $\Box$ \vspace{.5cm}

In this section we will examine conditions on a label $\M$ which ensure that the action of $FIN(\N)$ on $\Theta(\M)$ is
WAP, that is, every element of the enveloping semigroup $\E(\Theta(\M))$ is continuous on
$\Theta(\M)$.  By Corollary \ref{cor02a}
this is equivalent to commutativity of the monoid $\E(\Theta(\M))$.
In that case, we will say that the label $\M$ is WAP.


 \begin{df}\label{waplabel}
 A label $\M$ is WAP if
 every element of the enveloping semigroup $\E(\Theta(\M))$ is continuous on
$\Theta(\M)$
 \end{df}

We will denote by $\E_0(\Theta(\M))$ the elements which come from $FIN(\N)$.  That is,
\begin{equation}\label{waplabel01}
\E_0(\Theta(\M)) \ = \ \{ P_{\rr}|\Theta(\M) : \rr \in FIN(\N) \}.
\end{equation}
We will write $P_{\rr}$ in $\E(\Theta(\M))$ for the restriction $P_{\rr}|\Theta(\M)$.

We call the members of the set
\begin{equation}\label{ex-el}
\E_*(\Theta(\M)) = \E(\Theta(\M)) \setminus \E_0(\Theta(\M))
\end{equation}
the  \emph{external elements} of
$\E(\Theta(\M))$.

\begin{lem}\label{waplem01}  Let $\M$ be a label.
\begin{itemize}
\item[(a)] In $\E_0(\Theta(\M))$ the map $P_{\00 }$ is the identity map on $\Theta(\M)$ and so is the identity element of the monoid
$\E(\Theta(\M))$.  If $\M = FIN(\N)$ then $\Theta(\M) = \{ \M \}$ and $\E(\Theta(\M))$ is the trivial monoid $\{ P_{\00 } \}$.
If $\rr \not\in \M$
then $P_{\rr} = U$ the map taking all of $\Theta(\M)$ to $\emptyset \in \Theta(\M)$.

\item[(b)] If $\M - \rr_1 = \M - \rr_2$ for some $\rr_1, \rr_2 \in FIN(\N)$ then $P_{\rr_1}|\Theta(\M) = P_{\rr_2}|\Theta(\M) $,
i.e. as elements of $\E_0(\Theta(\M))$, $P_{\rr_1} = P_{\rr_2}$.

\item[(c)] If $Q_1 \in \E(\Theta(\M))$ satisfies $Q_1(\M) = \M - \rr_1$ for some $\rr_1 \in FIN(\N)$ then
$Q_1(\M - \rr) = \M - \rr_1 - \rr$ for all $\rr \in FIN(\M)$. If, in addition,  $Q_2 \in \E(\Theta(\M))$
satisfies $Q_2(\M) = \M - \rr_2$ for some $\rr_2 \in FIN(\N)$ then $Q_1 Q_2(\M) = Q_2 Q_1(\M) = \M - \rr_1 - \rr_2$.
\end{itemize}
\end{lem}

\proof (a): Obvious.

(b): If $P_{\rr_1}(\M) = P_{\rr_2}(\M)$ then because $FIN(\N)$ is commutative, it follows that
$P_{\rr_1} = P_{\rr_2}$ on $FIN(\N)\cdot \M = \{ \M - \rr : \rr \in FIN(\N) \}$. Because $P_{\rr_1}$ and $P_{\rr_2}$ are
continuous and $FIN(\N)\cdot \M$ is dense in $\Theta(\M)$ it follows that they agree on $\Theta(\M)$.

(c): $Q_1 P_{\rr}(\M) = P_{\rr} Q_1(\M)$ implies the first result.  The second is then obvious.

 $\Box$ \vspace{.5cm}

 Let $\LAB_f$ denote the set of finite labels, which comprise a countable dense subset of $\LAB$.

 \begin{df}\label{wapdef02} Let $\M$ be a label.

 (a) The label $\M$ is of \emph{strong finite type} if $\M$ is bounded and $\Theta(\M) \subset \LAB_f \cup FIN(\N) \cdot \M$.
  That is,
 if $\NN \in \Theta(\M)$ then  either $\NN = \M - \rr$ for some $\rr \in FIN(\N)$ or $\NN$ is a finite label.

 (b) The label $\M$ is \emph{semi-simple} if
 $Q \in \E_*(\Theta(\M))$ implies $Q(\M)$ is finite.

 (c) The label $\M$ is \emph{simple} if
 $\Theta(\M) =  FIN(\N) \cdot \M$.
 \end{df}
  \vspace{.5cm}

  Since $\LAB_f$ and $FIN(\N)$ are countable, $\Theta(\M)$ is countable when $\M$ is strong finite type. It follows from
  Corollary \ref{labelcor05b1} that if $\M$ is strong finite type then, as expected, it is of finite type. Since every
  element of $\Theta(\M)$ can be obtained as $Q(\M)$ for some $Q \in \E(\Theta(\M))$ it is clear that a bounded, semi-simple label is   of strong finite type.  As we will later see, there exist strong finite type labels which do not
  satisfy  monoid condition of semi-simplicity. However, in the simple case the a priori stronger monoid condition is implied.

\begin{lem}\label{waplem03} If $\M$ is a simple label then $\E(\Theta(\M)) = \E_0(\Theta(\M))$ and so $\E_*(\Theta(\M)) = \emptyset$.
\end{lem}

\proof Let $Q \in \E(\Theta(\M))$.  By simplicity, there exists $\rr_1 \in FIN(\N)$ such that $Q(\M) = P_{\rr}(\M)$.
By Lemma \ref{waplem01} (c), $Q = P_{\rr}$ on $FIN(\N)\cdot \M$. By simplicity this is all of $\Theta(\M)$. That is,
$Q = P_{\rr}$ as elements of $\E(\Theta(\M))$.

 $\Box$ \vspace{.5cm}

 In the strong finite type case, it may happen that $Q(\M) = P_{\rr}$ on $\M$ and so on $FIN(\N)\cdot \M$ but not on all
 of $\Theta(\M)$.

We now show that a semi-simple label is WAP.

We say that a net in $FIN(\N)$,  $\{ \rr^i : i \in I \}$, is \emph{cofinal constant} if there exists $\rr \in FIN(\N)$
such that $\{ i \in I : \rr^i = \rr \}$ is a cofinal subset of the directed set $I$.

\begin{lem}\label{waplem04} Assume that $\{ \rr^i : i \in I \}$ is a net of elements of $FIN(\N)$.

\begin{itemize}
\item[(a)]  If $\{ \rr^i : i \in I \}$ is not cofinal constant and $\NN$ is a finite label then eventually
$\NN - \rr^i = \emptyset$.

\item[(b)] Assume $\M$ is a bounded label and $\rr^i \in \M$ for all $i$. If there exists a
cofinal $I' \subset I$ such that $\bigcup \ \{ supp \ \rr^i : i \in I' \}$ is finite then
the net $\{ \rr^i : i \in I \}$ is cofinal constant.
\end{itemize}

\end{lem}

\proof (a): Either $\{ i : \rr^i \in \NN \}$ is cofinal or else eventually $\rr^i \not\in \NN$.  Let $\NN = \{ \ss_1,\dots, \ss_k \}$.
Let $I_j = \{ i : \rr^i = \ss^j \}$. If $\{ i : \rr^i \in \NN \} = \bigcup_j I_j$ is cofinal then some $I_j$ is cofinal and so
the net is cofinal constant.

(b): Since $\M$ is bounded, the set $A = \{ \rr \in \M : supp \ \rr \subset \bigcup \ \{ supp \ \rr^i : i \in I' \}$
 is finite.  Apply (a) to $\NN$, the label generated by $A$. If the net were not cofinal constant then eventually
 $\rr^i \not\in \NN$ contradicting cofinality of $I'$.

 $\Box$ \vspace{.5cm}

 \begin{theo}\label{waptheo05} If $\M$ is a semi-simple label then $\M$ is WAP.  If $Q_1, Q_2 \in \E_*(\Theta(\M))$ then
 $Q_1 Q_2 = Q_2 Q_1 = U$. \end{theo}

\proof Let $\{ \rr^i : i \in I \}$ be a net in $FIN(\N)$ such that $\{ P_{\rr^i}: i \in I \}$ converges pointwise to $Q$.

Case 1: $I' = \{ i : \rr^i \not\in \M \}$ is cofinal. For any $\NN \in \Theta(\M)$, $\NN \subset \M$ and so
$\NN - \rr^i = \emptyset$ for all $i \in I'$. Hence, the subnet $\{ P_{\rr^i}: i \in I' \}$ converges $U$ on
$\Theta(\M)$. Thus, $Q = U$ and so $Q P = U$ for all $P \in \E(\Theta(\M))$.

Case 2:  $\{ \rr^i : i \in I \}$ is cofinal constant.  So there exists $\rr$ such that $\rr^i = \rr$ on a cofinal subset.
Hence, $Q = P_{\rr} \in \E_0(\Theta(\M))$.

Case 3:  $\{ \rr^i : i \in I \}$ is a net in $\M$ which is not cofinal constant. If $\NN$ is in the image of
an external element of $\E(\Theta(\M))$ then $\NN$ is finite since $\M$ is semi-simple,
and so $\NN - \rr^i $ is eventually empty.  So
$Q(\NN) = \emptyset$.

So if $Q_1, Q_2 \in \E_*(\Theta(\M))$ then Case 3 applies with $Q = Q_1$ and so $Q_1 Q_2 = U$.

It follows that all of the elements of   $\E_*(\Theta(\M))$ commute with one another. Since the elements of $\E_0(\Theta(\M))$
commute with every element of $\E(\Theta(\M))$, it follows that $\E(\Theta(\M))$ is commutative and so $\M$ is WAP.

 $\Box$ \vspace{.5cm}

 \begin{cor}\label{wapcor06}  A bounded label $\M$ is semi-simple iff it is strong finite type and WAP.\end{cor}

 \proof  Assume $\M$ is strong finite type and WAP. If $Q(\M) = \M - \rr$ then
 by Lemma \ref{waplem01} (c), $Q = P_{\rr}$ on $FIN(\N)\cdot \M$ which is dense in $\Theta(\M)$. Since $\M$ is WAP
 $Q$, as well as $P_{\rr}$, is continuous. Hence, $Q = P_{\rr}$. Since $\M$ is strong finite type, the remaining case is
 $Q(\M)$ finite.  Hence, $\M$ is semi-simple. The converse follows from Theorem \ref{waptheo05}.

 $\Box$ \vspace{.5cm}

  \begin{prop}\label{wapprop07} If $\M$ is a WAP, semi-simple or simple label, then
  every $\NN \in \Theta(\M)$ satisfies the corresponding property. \end{prop}

  \proof For $\NN \in \Theta(\M)$, $ \E(\Theta(\M))$ maps onto $\E(\Theta(\NN))$ by restriction.
  If $ \E(\Theta(\M))$ is commutative then $\E(\Theta(\NN))$ is and so $\NN$ is WAP.  The restriction
  maps $\E_0(\Theta(\M))$ onto $\E_0(\Theta(\NN))$ and so the image of $\E_*(\Theta(\M))$ contains $\E_*(\Theta(\NN))$.
  Thus, if every $Q \in \E_*(\Theta(\M))$ has range in $\LAB_f$ then the same is true for every element
  $\E_*(\Theta(\NN))$. If $\E_*(\Theta(\M))$ is empty then the same is true for $\E_*(\Theta(\NN))$.

 $\Box$ \vspace{.5cm}

We now describe a condition which is easy to check and which implies semi-simplicity.

\begin{df}\label{labeldef01bxx}A label $\M$ is \emph{finitary} if it is bounded and whenever
$  \{ \rr^i \}$ is a sequence of $\N$-vectors with $\bigcup_i \ supp \ \rr^i$ infinite
 the $LIMINF \ \{ \M - \rr^i \}$ is finite. We call this the \emph{Finitary Condition}.
\end{df}

 \vspace{.5cm}

%


 Clearly, if $\M_1 \subset \M$ is a label then $\M_1$ is finitary if $\M$ is.
%

%

 For a label $\M$  and $\NN$  a nonempty set of $\N$-vectors with $\NN$, we define
 \begin{equation}\label{label10}
 \M - \NN \ = \ \{ \ \mm : \mm + \rr \in \M \quad \mbox{for all} \ \rr \in \NN \ \}
 \ = \ \bigcap_{\rr \in \NN} \ \M - \rr.
 \end{equation}

The set $\M - \NN$ is a label for any  nonempty set $\NN \subset FIN(\N)$, since $\M$ is a label.

%
%
%
%
%
%

 \begin{prop}\label{labelprop09} Let $\M$ be a bounded label.

 (a) The following conditions are equivalent.
 \begin{itemize}
 \item[(i)] If $\{ S_i \}$ is a sequence of finite subsets of $\N$ with
 $\bigcup_i \ S_i$ infinite,   there are only finitely many subsets $S$ of $\N$ such that eventually
 $S \cup S_i \in Supp \ \M$.

 \item[(ii)] If $\{ \ell^i \}$ is a sequence of distinct positive integers then, there are only finitely
 many $\ell \in \N$ such that eventually $\{ \ell, \ell^i \} \in Supp \ \M$.

 \item[(iii)] The label is finitary, i.e. if $ \{ \rr^i \}$ is a sequence of $\N$-vectors with $\bigcup_i \ supp \ \rr^i$ infinite
 then  $LIMINF \ \{ \M - \rr^i \}$ is finite.

 \item[(iv)] If $ \rr^i $ is a sequence of $\N$-vectors with $\bigcup_i \ supp \ \rr^i$ infinite
 and $\{ \M - \rr^i \}$ convergent then
 then $LIM \ \{ \M - \rr^i \}$ is finite.

 \item[(v)] If $\NN$ is infinite then $\M - \NN$ is finite, and there is no strictly
 increasing sequence in the collection $\{ \M - \NN : \NN $ infinite $ \}$ of finite labels.

 \item[(vi)] If $ \{ \rr^i : i \in I \} $ is a net of $\N$-vectors which is not cofinal constant,
 and $\{ \M - \rr^i : i \in I \}$ is convergent then
 then $LIM \ \{ \M - \rr^i \}$ is finite.
 \end{itemize}

 (b) If $\M$ is finitary then it is semi-simple.\vspace{.25cm}

 (c) If $\M$ is finitary, then the following conditions on a finite subset $\F$ of $\M$ are equivalent.
 \begin{itemize}
 \item[(i)] There exists $ \rr^i $  a sequence of $\N$ vectors with $\bigcup_i \ supp \ {\mathbf r}^i$ infinite
 such that  $LIM \ \{ \M - \rr^i \} \ = \ \F$.

  \item[(ii)] There exists $ \rr^i $  a sequence of distinct $\N$-vectors
 such that  $LIM \
  \allowbreak
 \{ \M
 - \rr^i \} \ = \ \F$.

 \item[(iii)] There is an infinite set $\NN$ such that $\F =  \M - \NN_1$ for every infinite subset
 $\NN_1$
 of $\NN$.
%
 \end{itemize}
 \end{prop}

\proof If $\NN \subset \M$ then, since $\M$ is bounded, $\NN$ is infinite iff
 $\bigcup \ Supp \ \NN$ is infinite. Also, if $\{ \rr^i \}$ is a sequence of $\N$ vectors
 with $\bigcup_i \ \{ supp \ \rr^i \}$ infinite
  then we can choose a subsequence of distinct vectors.

(a) (i) $\Rightarrow$ (ii): Let $S_i = \{ \ell^i \}$.

(ii) $\Rightarrow$ (i): By going to a subsequence we can assume that we can choose $\ell^i \in S_i$ so that
$\{ \ell^i  \}$ is an infinite sequence of distinct integers. Let $K = \{ \ell \in \N : $ eventually $\{ \ell , \ell^i  \} \in Supp \ \M \}$.
This is a finite set by (ii). If $S \cup S_i  \in Supp \ \M $ then $S \cup \{ \ell^i  \} \in Supp \ \M $ and if $\ell  \in S$ then
$\{ \ell , \ell^i  \} \in Supp \ \M$.  Hence, if eventually $S \cup S_i  \in Supp \ \M $ then $S \subset K$. Since the power set of $K$ is finite (i) holds.

(i) $\Rightarrow$ (iii): If $\mm \in LIMINF$ then eventually $supp \ \mm \cup supp \ \rr^i \in Supp \ M$.
By (i) $\M$ there are only finitely many such sets $supp \ \mm$ and so there are only
finitely many $\mm \in \M$ with such supports, because $\M$ is bounded.

 (iii) $\Leftrightarrow$ (iv): If $\{ \rr^i \}$ is a sequence with $\bigcup_i \ \{ supp \ \rr^i \}$ infinite then
 we can choose a subsequence of distinct vectors and then go to a further subsequence $\{ {\mathbf r}^{i'} \}$ which
 is convergent. Assuming (iv) $LIM \ \{\M - \rr^{i'} \} = LIMINF \ \{\M - \rr^{i'} \} \supset
 LIMINF \ \{\M - {\mathbf r}^{i} \}$ is finite. This shows that (iv) implies (iii).  The converse is obvious.

 (iii) $\Rightarrow$ (v):  Suppose that $\{ \F^i \}$ is a nondecreasing sequence of subsets of $\M$ with
 each $\M - \F^i = \NN_i$ infinite.  So $\F^i \subset \M - \NN_i$.  Inductively, choose $\rr^{i} \in \M - \F^{i}$ distinct from the
 $\rr^{j}$'s with
  $j < i$.
 Since the sequence $\{ \F^i \}$ is a nondecreasing, $\F^j \subset \M - \rr^i$ for
 $j \leq i$. Hence, $\bigcup_i \F^i \subset LIMINF \ \{ \M - \rr^i \}$ and so it
 is finite by (iii). Thus, the sequence $\{ \F^i \}$ is eventually constant.

 In particular, if $\NN$ is infinite and  $\F = \M - \NN$ then $\M - \F \supset \NN$ is infinite and applying the above argument to
 the sequence which is constantly $\{ \F \}$ we see that $\F = \M - \NN$ is finite.

 (vi) $\Rightarrow$ (v): By going to a subsequence we can assume that  $\{ \rr^i \}$ is a sequence of distinct
 elements. Then $\{ \bigcap_{j \geq i} \{ \M - \rr^j \} \}$ is a nondecreasing sequence in $\{ \M - \NN : \NN $ infinite $ \}$.
 Hence, its union, which is the limit, is finite by (vi).

 (vi) $\Rightarrow$ (iv): Again we can assume that the sequence $\{ \rr^i \}$ consists  of distinct
 elements. It is then a convergent net which is not cofinal constant.  Hence the limit is finite by (vi).

 (iii) $\Rightarrow$ (vi):  Assume that $\NN = LIM \{ \M - \rr^i : i \in I \}$ is infinite.  Write $\NN$ as the union
 of a strictly increasing sequence of nonempty finite labels $\{ \NN^k : k \in \N \}$. Since each $\NN^k$ is finite,
 eventually $\NN^k \subset \M - \rr^i$. Inductively, we can choose $i_k \in I$ so that $i_k \geq i_{k-1}$ in $I$ and
 $\NN^k \subset \M - \rr^i$ for all $i \geq i_k$ in $I$. Furthermore, such the sequence is not cofinal constant, we may
 choose $i_k$ so that $\rr^{i_k}$ is distinct from $\rr^{i_{\ell}}$ for all $\ell < k$ in $\N$. Since, the $\NN^k$'s are
 nonempty, $\{ \rr^{i_k} : k \in \N \}$ is a sequence of distinct elements of $\M$. Furthermore,
 $LIMINF_k \ \{\M - \rr^{i_k} \} \ \supset \ \bigcup_k \ \NN^k = \NN$. This contradicts (iii).  Contrapositively, (iii) implies (vi).

(b) : Let $Q \in \E_*(\Theta(\M))$ be the pointwise limit of the net $\{ P_{\rr^i} : i \in I \}$. Since $Q$ is an external
element, the net $\{ \rr^i \}$ is not cofinal constant. It follows from (v) above that
$Q(\M) = LIM \{ \M - \rr^i \}$ is finite.

 (c) (i)  $\Leftrightarrow$ (ii): This is obvious from our initial remarks.

 (ii) $\Rightarrow$ (iii): Assume $ \rr^i $  a sequence of $\N$-vectors
with $LIM \ \{ \M - \rr^i \} \ = \ \F$. By discarding finitely many terms $\rr^i$ we can assume that the finite set $\F$
equals $\M - \NN$ with $\NN = \{  \rr^i \} $. If $\NN_1$ is any infinite subset of $\NN$ then
$\F = \M - \NN \subset \M - \NN_1$. On the other hand, if
  $ \mm \in \M - \NN_1$ then $\mm \in \M - \rr^i$ for infinitely many $i$ and so
  by convergence $\mm \in LIM = \F$.

  (iii) $\Rightarrow$ (i): If $\NN = \{ \rr^1, \rr^2, \dots \} $ is the infinite set given by (iii) then
  $\{ \rr^i \} $ is a sequence of distinct elements with $\F = LIM \{ \M - \rr^i \}$.

 $\Box$ \vspace{.5cm}

 \begin{prop}\label{labelprop09aaa} If $\M$ is a label such that $\M - \rr$ is a finite label for all $\rr > 0$,
 then $\M$ is bounded, simple and finitary. \end{prop}

 \proof If a label $\NN$ is unbounded then for some $\ell \in \N$, $\NN$ contains $\NN_1 = \{ n \chi(\ell) : n \in \Z_+ \}$
 and so $\NN - \chi(\ell)$ contains $\NN_1 - \chi(\ell) = \NN_1$, which is infinite.  Hence, $\M$ is bounded..

 If $\{ \rr^i \}$ is a sequence in $\M$ with $\bigcup_i  \ supp(\rr^i)$ infinite then for any $\rr > 0 $
 $\M - \rr$ is finite and so $ \bigcup_i  \ \{ supp(\rr^i) : \rr^i \in \M - \rr \}$ is finite. So it cannot
 be true that eventually $\rr \in \M - \rr^i$.  Hence, $LIM INF \{ \M - \rr^i \} = 0$. It follows that
 $\M$ is finitary. So $\{ \M - \rr^i \}$ is convergent with limit $0$. Hence, $\M$ is bounded and semi-simple and
 so is of finite type.

 If $\bigcup_i  \ supp(\rr^i)$ is finite then the sequence is cofinal constant by Lemma \ref{waplem04} (b). So if
 the sequence $\{ \M - \rr^i \}$ is convergent then its limit is $\M - \rr$ for some $\rr$.

 Finally, if $\{ \M - \rr^i \}$ is convergent and infinitely often $\rr^i \not\in \M$ then the limit is $\emptyset$.

 Since $\M$ is bounded, $\emptyset = \M - \rr$ for and if $\M$ is nonempty then since $\M$ is of finite type,
 $max \M$ is nonempty. If $\rr \in max \M$ then $0 = \M - \rr$.  It follows $\Theta(\M) = \{ \M - \rr : \rr \in FIN(\N) \}$
 and so $\M$ is simple.

  $\Box$
%
%
%
%

\vspace{1cm}

\subsection{WAP subshifts}\label{sub,WAP}

$\qquad$

 \vspace{.5cm}

 We now apply the above results to the associated subshifts.  We will repeatedly use the following
 \emph{fundamental fact} concerning the subshift $(X(\M),S)$ associated with a label $\M$ of finite type.

 \begin{theo}\label{towtheo17aa} Let $\M$ be a label of finite type.

 \begin{itemize}
 \item[(a)] If $z \in X(\M)$ with $z \not= e = x[\emptyset]$, then there
 exists a unique pair $(t,\NN) \in \Z \times \Theta(\M)$ such that $z = S^t(x[\NN])$ and for every
 $(t,\NN) \in \Z \times \Theta(\M)$, $S^t(x[\NN]) \in X(\M)$.

   \item[(b)]  If $w \in X_+(\M)$ with $z \not= e = x[\emptyset]$, then there
 exists a unique pair $(t,\NN) \in \Z \times \Theta(\M)$ such that $z = S^t(x_+[\NN])$ and for every
 $(t,\NN) \in \Z \times \Theta(\M)$, $S^t(x_+[\NN]) \in X_+(\M)$.
 \end{itemize}
 \end{theo}

 \proof By Theorem \ref{towtheo15} $X(\M) = \{ S^t(x[\NN]) : (t,\NN) \in \Z \times \Theta(\M) \}$ and similarly for
 $X_+(\M)$. If $S^{t_1}(x[\NN_1]) = S^{t_2}(x[\NN_2])$ with $\NN_2 \not= \emptyset$ then $S^{t_1 - t_2}(x[\NN_1]) = x[\NN_2]$,
 which is symmetric about $0$ and has a $1$ at position $0$. Hence $\NN_1 \not= \emptyset$ and $t_1 - t_2 = 0$.  Finally,
 $\NN_1 = \NN_2$ by Theorem \ref{towtheo09}. Similarly, if $S^{t_1 - t_2}(x_+[\NN_1]) = x_+[\NN_2]$ then since
 the latter has a $1$ at position $0$ and $0$'s at all negative positions it follows that $t_1 - t_2 = 0$ and again
 $\NN_1 = \NN_2$ by Theorem \ref{towtheo09}.

$\Box$ \vspace{0.5cm}

We now describe the enveloping semigroup version of this result. Recall that for $\M \not= FIN(\N)$,
and $\rr \not\in \M$, $P_{\rr} = U \in \E(\Theta(\M))$ is the idempotent with $U(\NN) = \emptyset$ for all
$\NN \in \Theta(\M)$.

We use the continuous, surjective homomorphisms $D : E_b(X(\M),S)  \to \E(\Theta(\M)), D_+ : E_b(X_+(\M),S)  \to \E(\Theta(\M))$ from (\ref{ellis15}).

\begin{lem}\label{towlem17aaa} Assume $\M$ is a label of finite type.
\begin{itemize}

\item[(a)] If $q_1, q_2 \in E(X(\M),S)$ with $q_1(x[\NN]) =  q_2(x[\NN])$ for every $\NN \in \Theta(\M)$ then
$q_1 = q_2$. If $q_1, q_2 \in E(X_+(\M),S)$ with $q_1(x_+[\NN]) =  q_2(x_+[\NN])$ for every $\NN \in \Theta(\M)$ then
$q_1 = q_2$.

\item[(b)] The maps $D$ and $D_+$ are injective and so
are homeomorphisms and monoid isomorphisms.
\end{itemize}
\end{lem}

\proof  (a) If $q_1, q_2 \in E(X(\M),S)$ with $q_1 = q_2$ on $\{ x[\NN] : \NN \in \Theta(\M) \}$ then
 since $q_1$ and $q_2$ commute with all of the $S^t$'s, it follows from Theorem \ref{towtheo17aa} that $q_1 = q_2$
 on all of $X(\M)$.

 (b) If $Q = D(q_1) = D(q_2)$ then $q_1(x[\NN]) = x[Q(\NN)] = q_2(x[\NN])$ for all $\NN \in \Theta(\M)$.  By (a)
 $q_1 = q_2$. Thus, $D$ is bijective and so is a monoid isomorphism and is a homeomorphism by compactness.

 The proofs for $X_+(\M)$ are the same.

$\Box$ \vspace{0.5cm}

We will denote the inverse isomorphisms as
\begin{equation}\label{ellisiso}
\bar D : \E(\Theta(\M))  \to  E_b(X(\M),S) \quad \text{and} \quad  \bar D_+ : \E(\Theta(\M)) \to  E_b(X_+(\M),S)
\end{equation}

 \begin{theo}\label{towtheo17ab} Let $\M$ be a label of finite type.

 \begin{itemize}
 \item[(a)] For every $Q \in \E(\Theta(\M))$, $\bar D(Q) \in E(X(\M),S)$ and
 $\bar D_+(Q) \in E(X_+(\M),S)$ are the unique elements such that
 $$ \bar D(Q)(x[\NN]) = x[Q(\NN)], \quad \bar D_+(Q)(x_+[\NN]) = x_+[Q(\NN)] \qquad \text{for all} \quad \NN \in \Theta(\M)$$
 with $\bar D(U)$, (and $\bar D_+(U)$) the idempotent $u$ mapping $X(\M)$ (resp. $X_+(\M)$) to the fixed point $e = x[\emptyset] =
 x_+[\emptyset]$.
%

 \item[(b)] For every $q \in E(X(\M),S)$ with $q \not= u$ there exists a unique pair $(t,Q) \in \Z \times \E(\Theta(\M))$
 such that $q = S^t \bar D(Q)$. For every $q \in E(X_+(\M),S)$ with $q \not= u$ there exists a
 unique pair $(t,Q) \in \Z \times \E(\Theta(\M))$
 such that $q = S^t \bar D_+(Q)$.

 \item[(c)] For every $\rr \in FIN(\N)$, $p_{\rr} = \bar D(P_{\rr}) \in E(X(\M),S)$  and $p^+_{\rr} = \bar D_+(P_{\rr}) \in E(X_+(\M),S)$
 are continuous so that $(\rr,z) \mapsto p_{\rr}(z) $ and $(\rr,z) \mapsto p^+_{\rr}(z) $ define continuous semigroup
 actions $FIN(\N) \times X(\M) \to X(\M)$ and $FIN(\N) \times X_+(\M) \to X_+(\M)$ respectively.
  \end{itemize}
 \end{theo}

 \proof  We will give the proofs just for the $(X(\M),S)$ case as the $(X_+(\M),S)$ arguments are completely analogous.

(a): This is a restatement of Lemma \ref{towlem17aaa}.

  (b):  For $q \in E(X(\M),S)$ assume that $\{ t^i \}$ is a net in $\Z$ so that $\{ S^{t^i} \}$ converges pointwise to $q$.

  If $q(x[\M]) = e$ then by Lemma \ref{elimit} $\xi_{\M}(t^i) \to \infty$ and so $\xi_{\NN}(t^i) \to \infty$ for any
  label $\NN$ contained in $\M$.  Hence, $q(x[\NN]) = e$ for all $\NN \in \Theta(\M)$ and so from
  Theorem \ref{towtheo17aa} it follows that $q(z) = e$ for all $z \in X(\M)$, i.e. $q = u = \bar D(U)$.

  If $q(x[\M])_{a} = 1$ for some $a \in Z$ we may replace $q$ by $S^{-a}q$ and so assume that $q(x[\M])_0 = 1$.
  This implies that eventually $t^i \in A[\M]$.  Since $\M$ is of finite type,
Lemma \ref{towlem12}(b)
  implies that
  there exists $b \in \Z$ so that  $t^i - b$ is eventually in $ A[\M]$ with $j_r(t^i - b) \to \infty$.
  Thus, we may assume that $t^i = k + s^i$ with $s^i \in A[\M]$ and $j_r(s^i) \to \infty$. Let $\rr^i = \rr(s^i)$.
  By going to a subnet we may assume that $\{P_{\rr^i} \}$ converges to some $Q \in \E(\Theta(\M))$.
  From the Coherence Lemma again, it follows that $S^{-k}q = Lim \{ S^{s^i} \} = \bar D(Q)$.

  Finally, if $S^{t_1}\bar D(Q_1) = q = S^{t_2}\bar D(Q_2)$, then
  $$S^{t_1 - t_2}(x[Q_1(\M)]) = S^{t_1 - t_2}\bar D(Q_1)(x[\M]) = \bar D(Q_2)(x[\M]) = x[Q_2(\M)].$$
  Since $Q_2(\M) \not=
  \emptyset$ it follows from the uniqueness in Theorem \ref{towtheo17aa} that $t_1 - t_2 = 0$.
  Hence, $\bar D(Q_1) = \bar D(Q_2)$ and so $Q_1 = Q_2$ because  $\bar D$ is an isomorphism.
  It follows that the representation is unique.

 (c): Since $P_{\rr}$ commutes with every $Q \in \E(\Theta(\M))$ it follows from (b) that $p_{\rr}$ commutes with every
 $q \in E(X(\M),S)$ and so $p_{\rr}$ is continuous by
Proposition \ref{prop02}.

 Since $\bar D$ is a semigroup isomorphism and $P_{\rr_1}P_{\rr_2} = P_{\rr_1 + \rr_2}$,
   it follows that $p_{\rr_1}p_{\rr_2} = p_{\rr_1 + \rr_2}$ and so
 $(\rr,z) \mapsto p_{\rr}(z)$ defines a continuous semigroup action.

$\Box$ \vspace{0.5cm}

\begin{remark} From the representation in (b) it is clear that $\bar D $ maps $\A(\Theta(\M))$ into $A(X(\M),S)$.
\end{remark}
\vspace{0.5cm}

Let $(\Z \times \E(\Theta(\M))/(\Z \times \{ U \})$ denote the quotient of the product semigroup obtained by identifying
all of the points of $\Z \times \{ U \}$ with the single point $U$.  Define
\begin{equation}\label{tow17ad}
\begin{split}
\hat D : (\Z \times \E(\Theta(\M))/(\Z \times \{ U \}) \to E(X(\M),S) \qquad \text{by} \quad (t,Q) \mapsto S^t \bar D(Q), \\
\hat D_+ : (\Z \times \E(\Theta(\M))/(\Z \times \{ U \}) \to E(X_+(\M),S) \qquad \text{by} \quad (t,Q) \mapsto S^t \bar D_+(Q).
\end{split}
\end{equation}

From Theorem \ref{towtheo17ab} we immediately obtain:

\begin{cor}\label{towcor17ac} The mappings $\hat D$ and $\hat D^+$ are semigroup isomorphisms. \end{cor}

$\Box$ \vspace{0.5cm}

N.B.  This is an algebraic result with no topological implications.

\vspace{0.5cm}

The maximum label $FIN(\N)$ is simple since $\Theta(FIN(\N)) = \{ FIN(\N) \}$. Hence, it is an unbounded WAP label.
Since it is recurrent, $X(FIN(\N))$ is uncountable and so $(X(FIN(\N)),S)$ is not WAP.
In fact, it is even not tame; see Proposition \ref{IP-not-tame} below.

\begin{lem}\label{towlem17ac1} If a bounded label is WAP then it is of finite type. \end{lem}

 \proof If $\M$ is not of finite type then $\Theta(\M)$ contains a recurrent label by Corollary
 \ref{labelcor05b}.
 Since a subsystem
 of a WAP system is WAP, it suffices to show that a non-trivial bounded, recurrent label $\M$ is not WAP.

 Since $\M$ is recurrent, there exists a sequence $\{ \rr^i > 0 \}$ such that $P_{\rr^i}(\M) \to \M$. Choose a subnet
 $\{ \rr^{i'} \}$ so that $ P_{\rr^{i'}} \to Q$ in $\E(\Theta(\M))$. Clearly, $Q(\M) = \M = P_{\00 }(\M)$. If $Q$
 were continuous then since $\M$ is a transitive point for $\Theta(\M)$ we would have $Q = P_{\00 } = id$ on $\Theta(\M)$.
 Because $\M$ is bounded, $0 \in \Theta(\M)$ by Lemma \ref{labellem05c}.
 But $ 0 - \rr^{i} = \emptyset $ for all $i$. Hence, $Q( 0 ) = \emptyset \not= 0 = P_{\00 }(0)$.

$\Box$ \vspace{0.5cm}

So we obtain

\begin{theo}\label{towtheo17ad} For a label $\M$ the following are equivalent.

\begin{itemize}
\item[(a)] The label $\M$ is  bounded and WAP.

\item[(b)] The cascade $(X(\M),S)$ is WAP.

\item[(c)] The cascade $(X_+(\M),S)$ is WAP.
\end{itemize}

In that case, $\Theta(\M)$ and $X(\M)$ are countable and $\M$ is of finite type.
\end{theo}

\proof  If $\M$ is bounded and WAP then by Lemma \ref{towlem17ac1} it is of finite type and $\E(\Theta(\M))$ is commutative.  Hence,
$(\Z \times \E(\Theta(\M))/(\Z \times \{ U \}) $ is commutative and so from Corollary \ref{towcor17ac}
it follows that $E(X(\M),S)$ and $E(X_+(\M),S)$ are commutative.  As these are topologically transitive cascades, they are WAP.

If either $(X(\M),S)$ and $(X_+(\M),S)$ is WAP then as WAP subshifts the spaces are countable.  Since
$\Theta(\M)$ injects into $X(\M)$ and $X_+(\M)$, it is  is countable and so $\M$ is of finite type.
From Corollary \ref{towcor17ac} again it follows that
$\E(\Theta(\M))$ is commutative and so $\M$ is a WAP label.

$\Box$ \vspace{0.5cm}

In summary, we have the implications:
\begin{quote}
simple \ $\Longrightarrow$ \ semi-simple \ $\Longrightarrow$ \ WAP, \\
finitary  $\Longrightarrow$  bounded and semi-simple  $\Longrightarrow$\\
strong finite type  $\Longrightarrow$  finite type, \\
bounded and WAP \ $\Longrightarrow$ \ finite type.
\end{quote}

\vspace{1cm}

\subsection{Constructions and examples}

$\qquad$

\vspace{0.5cm}

We turn now to the tools we will use to construct many examples of WAP labels.

Recall that
for labels $\M_1$ and $\M_2 $ we defined
$$  \M_1 \oplus \M_2 = \{ \mm_1 + \mm_2 : \mm_1 \in \M_1, \mm_2 \in \M_2 \}.$$

Recall that
 $(\M_1, \M_2) \mapsto \M_1 \oplus \M_2$ is a continuous map from $\LAB \times \LAB \to \LAB$. So
 if $\Phi_1, \Phi_2 \subset \LAB$ are compact then $\Phi_1 \oplus \Phi_2$ is compact and  is therefore closed.
 If either $\M_1$ or $\M_2$ is empty then $\M_1 \oplus \M_2 = \emptyset$. Otherwise, $\M_1 \oplus \M_2 \not= \emptyset$.


\begin{df}\label{hereditary,fcontain}
Let $M$ be a collection  of finite subsets of $\N$. Call it \emph{hereditary}
if $ A \subset B $ and $B \in M$ implies
$A \in M$.  We say that $M$ f-contains $L \subset \N$ if every finite subset of $L$ is in $M$,
i.e. $\P_f L \subset M$.
\end{df}

We will need a combinatorial lemma.

For two hereditary
collections $M_1, M_2$ define $M_1 \oplus M_2 = \{ A_1 \cup A_2 : A_1 \in M_1,
A_2 \in M_2 \}$.

\begin{lem}\label{labellem12a} Let $M_1$ and $M_2$ be hereditary collections of finite subsets of $\N$.  If
$M_1 \oplus M_2$ f-contains an infinite set then either $M_1$ or $M_2$ f-contains an infinite set. \end{lem}

\proof Let $L = \{ \ell_1, \ell_2,\dots \}$ be a counting for an infinite set f-contained in $M_1 \oplus M_2$.
Define the directed binary tree with vertices at level $n = 0, 1,\dots$ consisting of the $2^n$ ordered partitions
$(A,B)$ of $\{ \ell_1,\dots,\ell_n \}$. Connect $(A,B)$ to the $n+1$ level vertices $(A \cup \{ \ell_{n+1} \},B)$ and
$(A, B \cup \{ \ell_{n+1} \})$. The set of paths to infinity
forms
a Cantor set. Call a path
{\em{good at level}} $n$ if
for the partition $(A_n,B_n)$ at level $n$, $A_n \in M_1$ and $B_n \in M_2$.  Since $M_1 \oplus M_2$
f-contains $L$, the
set  $G_n$ of paths good at level $n$ is nonempty.  Each $G_n$ is closed and $G_{n+1} \subset G_n$. So the intersection
contains a path $\{ (A^i,B^i) : i = 0,1,\dots \}$. Let $A^{\infty} = \bigcup \  A^i, B^{\infty} = \bigcup \  B^i$.  Clearly,
$\{ A^i \}$ is a nondecreasing sequence of finite sets in $M_1$ with union $A^{\infty}$ and so $A^{\infty}$ is
f-contained in $M_1$.
Similarly, $B^{\infty}$ is f-contained in $M_2$.  Since $A^{\infty} \cup B^{\infty} = L$, at least one of them is infinite.

$\Box$ \vspace{.5cm}

For a label $\M$, if $\mm \in \M$ then $ \00 \leq \chi(supp \ \mm) \leq \mm$ and so $\chi(supp
\ \mm) \in \M$.
Hence, $Supp \ \M = \{ \ A \subset \N \ : \ \chi(A) \in \M \}$. Thus, $Supp \ \M$ is a hereditary
collection of finite subsets of $\N$ with $Supp \ \M = \emptyset$ iff $\M = \emptyset$. Thus,
Proposition \ref{labelprop03}(b) says
that a bounded label $\M$ is not of finite type iff $Supp \ \M$ f-contains some infinite subset.

\begin{prop}\label{labelprop13a} Let $\M_1, \M_2 $ be bounded, nonempty labels.

 The  label $\M_1 \oplus \M_2$ is nonempty and bounded. It is of finite type iff
$\M_1$ and $\M_2$ are labels of finite type.

If $\M_1$ is finite and $\M_2$ is  finitary, then $\M_1 \oplus \M_2$  is finitary.
\end{prop}

 \proof  Clearly, $\r(\M_1 \oplus \M_2) = \r(\M_1) + \r(\M_2)$.  So $\M_1 \oplus \M_2$ is bounded and nonempty if $\M_1$ and
 $\M_2$ are.

 Now assume that $\M_1$ and $\M_2$ are nonempty labels
of finite type. If $\M_1 \oplus \M_2$ is not of finite type then
$Supp \ (\NN \oplus \M)$ f-contains an infinite set by Proposition
\ref{labelprop03}(c).

Since
\begin{equation}\label{sumsupp1}
Supp \ (\M_1 \oplus \M_2) = (Supp \ \M_1) \oplus (Supp \ \M_2),
\end{equation}
it follows from Lemma  \ref{labellem12a} that either
$Supp \ \M_1$ or $ Supp \ \M_2$ f-contains an infinite set. Thus, $\M_1$ and $\M_2$ of finite type implies
$\M_1 \oplus \M_2$ is of finite type.  The converse is obvious since $\M_1 \cup \M_2 \subset \M_1 \oplus \M_2$.

 Now assume that $\M_1$ is finite and $\M_2$ is finitary. To show that $\M_1 \oplus \M_2$ is finitary, we
 apply condition  (ii) of Proposition \ref{labelprop09} (a).
Let $\{ \ell^i \}$ be a sequence of distinct positive integers. Since $\M_1$ is finite, $K = \bigcup \ Supp \ \M_1$ is
finite and we may assume, by discarding finitely many elements that
$\ell^i \not\in K$ for all $i$. Since $K$ is finite, it suffices to show that $\{ \ell \in \N \setminus K :$
such that eventually $\{ \ell, \ell^i \} \in Supp \ \M_1 \oplus \M_2 \}$ is finite.
Since $\ell, \ell^i \not\in K$,
\begin{align}\label{finuniona}
\begin{split}
\{ \ell, \ell^i \} \in Supp \ \M_1 \oplus \M_2  \quad &\Longleftrightarrow \quad \chi(\ell) + \chi(\ell^i) \in  \M_1 \oplus \M_2 \} \\
\quad \Longleftrightarrow \quad \chi(\ell) + \chi(\ell^i) \in  \M_2  &\quad \Longleftrightarrow \quad \{ \ell, \ell^i \} \in Supp  \ \M_2.
\end{split}
\end{align}

Since $\M_2$ is finitary, the set of $\ell$ such that eventually $\{ \ell, \ell^i \} \in Supp  \ \M_2 $ is finite.

%

$\Box$ \vspace{.5cm}

\begin{remark}
 If $\M_1 \oplus \M_2$ is finitary and $\M_2$ is infinite then $\M_1$ must be finite since if $\{ \rr^i \}$ is an
infinite sequence of distinct positive vectors
in $\M_2$
then $\M_1\subset LIMINF \ \{ (\M_1\oplus \M_2) - \rr^i \} $.
\end{remark}
   \vspace{.5cm}

 A pointed space $(X,*)$ is a compact space with a chosen base point. If $(X_1,*_1)$ and $(X_2,*_2)$ are pointed
 spaces then the \emph{wedge} is $X_1 \vee X_2 = (X_1 \times \{*_2 \} \cup \{*_1 \} \times X_2)$ with base point
 $ (*_1,*_2)$ and
  the \emph{smash product} is the quotient space $X_1 \wedge X_2 = (X_1 \times X_2)/ (X_1 \vee X_2)$, i.e.
  with the wedge identified to a point, the base point of the smash.  If $A_1 \subset X_1, A_2 \subset X_2$,
  we let $A_1 \wedge A_2$ denote the image of $A_1 \times A_2$ in the smash product.

  If $\M$ is a label other than $FIN(\N)$
  we use $\emptyset$ as the base point for $[[\M]]$, the closed set of labels contained in $\M$, and for $\Theta(\M)$.
  We use the retraction $U$ to $\emptyset$ as the base point
  for the enveloping semigroup $\E(\Theta(\M))$. Since $\{ U \}$ is an ideal in $\E(\Theta(\M))$, it follows that
  for $\M_1, \M_2 \not= FIN(\N)$,
  $\E(\Theta(\M_1))\vee \E(\Theta(\M_2))$
is an ideal in the product monoid $ \E(\Theta(\M_1))\times \E(\Theta(\M_2))$ and the quotient map
$\E(\Theta(\M_1))\times \E(\Theta(\M_2)) \to \E(\Theta(\M_1))\wedge \E(\Theta(\M_2))$ induces a monoid structure on
the smash so that the quotient map is a monoid homomorphism.

\begin{df}\label{dis-labels}
For nonempty labels $\M_1, \M_2$ we have $(\bigcup \ Supp \ \M_1)  \ \cap \   (\bigcup \ Supp \ \M_2) \ = \ \emptyset$ iff
 $\M_1 \ \cap \  \M_2 \ = \ 0 = \{ \00 \}$. In that case we will say that $\M_1$ and $\M_2$ are \emph{disjoint} labels.
\end{df}

 Assume $\M_1$ and $\M_2$ are disjoint, positive labels, i.e. neither is empty or $0$. If $\rr \in \M_1 \oplus \M_2$
 then $\rr = \rr_1 + \rr_2$
with $\rr_1 \in \M_1$ and $\rr_2 \in \M_2$ uniquely determined by $\rr$. For example, $\rr \in \M_1$ iff $\rr_2 = 0$.
If $\rr \in \M_1 \oplus \M_2$, then $(\M_1 \oplus \M_2) - \rr = (\M_1 - \rr_1) \oplus (\M_2 - \rr_2)$. Assume
$\NN_1 \in [[\M_1]]$ and $\NN_2 \in [[\M_2]]$ with neither
$\NN_1$
nor
$\NN_2$ empty then $\rr \in \NN_1 \oplus \NN_2$ iff $\rr_1 \in \NN_1$ and 
$\rr_2 \in \NN_2$.
It follows that the restriction
$$\oplus : ([[\M_1]] \setminus \{ \emptyset \}) \times
([[\M_2]] \setminus \{ \emptyset \}) \to [[\M_1 \oplus \M_2]] \setminus \{ \emptyset \}$$
is injective. Since $\NN_1 \oplus \NN_2 = \emptyset$ iff $\NN_1$ or $\NN_2$ is empty, we see that
there is an induced injection between  compact spaces:
$$  \widehat{\oplus} : [[\M_1]] \wedge [[\M_2]]  \to [[\M_1 \oplus \M_2]].$$

   \begin{theo}\label{towtheo27a}  Assume $\M_1$ and $\M_2$ are positive disjoint
labels with $\M_1 \oplus \M_2 \not= FIN(\N)$. The map $ \widehat{\oplus}$ restricts to a homeomorphism
$$  \widehat{\oplus} : \Theta(\M_1) \wedge \Theta(\M_2)  \to \Theta(\M_1 \oplus \M_2).$$

For $Q_1 \in \E(\Theta(\M_1)), Q_2 \in \E(\Theta(\M_2))$, $Q_1 \widehat{\oplus} Q_2$ is well-defined by
$(Q_1 \widehat{\oplus} Q_2)(\NN_1 \oplus \NN_2) = (Q_1(\NN_1))\oplus (Q_2(\NN_2))$. This defines
a monoid isomorphism
$$  \widehat{\oplus} : \E(\Theta(\M_1)) \wedge \E(\Theta(\M_2))  \to \E(\Theta(\M_1 \oplus \M_2)).$$
\end{theo}

\proof Since the labels are positive and disjoint, neither equals $FIN(\N)$.

As described above, $\rr = \rr_1 + \rr_2$ with $\rr_1 \in \M_1, \rr_2 \in \M_2$ implies
$P_{\rr}(\M) = P_{\rr_1}(\M_1) \oplus P_{\rr_2}(\M_2)$. Since $\Theta(\M_1 \oplus \M_2)$ is closed
it follows that  $ \widehat{\oplus}$ maps $\Theta(\M_1) \wedge \Theta(\M_2)$ into $\Theta(\M_1 \oplus \M_2)$.

Now let $\{ Q_1^i : i \in I \}$ and $\{ Q_2^j : j \in J \}$ be nets in $\E(\Theta(\M_1))$ and $ \E(\Theta(\M_2))$, respectively,
converging to $Q_1$ and $Q_2$. By continuity of $\oplus$, for any $\NN_1 \in \Theta(\M_1)$, $\NN_2 \in \Theta(\M_2)$
the net $\{  Q_1^i(\NN_1) \oplus Q_2^j(\NN_j) (i,j) \in I \times J \}$ converges to
$Q_1(\NN_1) \oplus Q_2(\NN_2)$, which is empty iff either $Q_1(\NN_1)$ or $ Q_2(\NN_2)$ is empty.

If $\NN \in \Theta(\M_1 \oplus \M_2)$ with $\NN \not= \emptyset$ then there exists a net
$\{ \rr^i \in \M_1 \oplus \M_2 \}$ such that
$\{ P_{\rr^i}(\M) \}$ converges to $\NN$. By going to a subnet, we may assume that $\{ P_{\rr_1^i} \}$ and
$\{ P_{\rr_2^i} \}$ converge to $Q_1 \in \E(\Theta(\M_1))$ and $Q_2 \in \E(\Theta(\M_1))$ respectively.
Then $\NN = Q_1(\M_1) \oplus Q_2(\M_2)$. It follows that $ \widehat{\oplus}$ maps
$\Theta(\M_1) \wedge \Theta(\M_2)$ onto $\Theta(\M_1 \oplus \M_2)$. Since it is injective, it is a homeomorphism.

Similarly, if $Q \in \E(\Theta(\M_1 \oplus \M_2))$ with $Q \not= U$ and so $Q(\M) \not= \emptyset$,
we obtain $Q_1 \in \E(\Theta(\M_1))$ and $Q_2 \in \E(\Theta(\M_1))$ so that $Q = Q_1 \oplus Q_2$. It follows that
$\widehat{\oplus}$ is surjective between the enveloping semigroups. Furthermore, $Q_1 \oplus Q_2 = U$ iff
either $Q_1 = U$ or $Q_2 = U$, because $Q_1(\M_1) \oplus Q_2(\M_2) = \emptyset$ iff $Q_1(\M_1)$ or $Q(\M_2)$ is empty
in which case $(\Q_1 \oplus \Q_2)(\M_1 \oplus \M_2) = \emptyset$.

Finally, if $(Q_1,Q_2) \not= (Q_1',Q_2') \in \E(\Theta(\M_1)) \times \E(\Theta(\M_2))$ with none of the four elements equal to
$U$. We may assume that $\NN_1 \in \Theta(\M_1)$ with $Q_1(\NN_1) \not= Q_1'(\NN_1)$ and $Q_1(\NN_1) \not= \emptyset$.
Since $\widehat{\oplus}$ is injective on $\Theta(\M_1) \wedge \Theta(\M_2)$ it follows that
$Q_1 \oplus Q_2(\NN_1 \oplus \M_2) \not=  Q'_1 \oplus Q'_2(\NN_1 \oplus \M_2)$. Hence,
$\widehat{\oplus}(Q_1, Q_2) \not= \widehat{\oplus}(Q'_1, Q'_2)$ and so $\widehat{\oplus}$ is
bijective on the enveloping semigroups. We showed above that it is continuous and so is a homeomorphism. It clearly
preserves composition and so is a monoid isomorphism.

$\Box$ \vspace{.5cm}

 \begin{cor}\label{towcor27c}  Assume $\M_1$ and $\M_2$ are positive disjoint
labels. If both $\M_1$ and $\M_2$ are simple or WAP then $\M_1 \oplus \M_2$ satisfies the corresponding property.
If $\M_1$ is finite and $\M_2$ is semi-simple then $\M_1 \oplus \M_2$ is semi-simple.
\end{cor}

\proof $\M_1 \oplus \M_2 = FIN(\N)$ then it is simple and WAP. So we may assume $\M_1 \oplus \M_2 \not= FIN(\N)$.

If $\E(\Theta(\M_1))$ and $\E(\Theta(\M_2))$ are abelian then $\E(\Theta(\M_1)) \wedge \E(\Theta(\M_2)) $ is abelian.
So $\M_1 \oplus \M_2$ is WAP when $\M_1$ and $\M_2$ are by Theorem
\ref{towtheo27a}.

If every element of $\E(\Theta(\M_1))$ other than $U$ is of the form $P_{\rr_1}$ for some $\rr_1 \in \M_1$ and the
analogous condition holds for $\M_2$ then every element of $\E(\Theta(\M_1 \oplus \M_2))$ is of the form
$P_{\rr}$ for some $\rr \in  \M_1 \oplus \M_2$.  On the other hand, $U = P_{\rr}$ for any $\rr \not\in \M_1 \oplus \M_2$.
Thus, if $\M_1$ and $\M_2$ are simple then $\M_1 \oplus \M_2$ is.

Assume $\M_1$ is finite, and so is simple, and assume $\M_2$ is semi-simple.
If $Q_1 \oplus Q_2 \not= P_{\rr}$ then $Q_2$ is an external element of $\E(\Theta(\M_2))$.
Then $Q_2(\M_2)$ is finite and so $(Q_1 \oplus Q_2)(\M_1 \oplus \M_2) \subset \M_1 \oplus Q_2(\M_2)$ which is finite.
Thus, $\M_1 \oplus \M_2$ is semi-simple.

$\Box$ \vspace{.5cm}

\begin{remark}
 If  $\M_2$ is not simple and $\M_1$ is infinite then $\M_1 \oplus \M_2$ is not semi-simple. If $Q \in \E(\Theta(\M_2)$ is
 not equal to some $P_{\rr_2}$ then$ P_{0} \oplus Q$ is not of the form $P_{\rr}$ but
 $(P_{0} \oplus Q)(\M_1 \oplus \M_2) = \M_1 \oplus Q(\M_2)$ which contains $\M_1$ and so is infinite.
\end{remark}
   \vspace{.5cm}

\begin{ex}\label{wapnotsemi}   Assume $\M_1$ and $\M_2$ are positive disjoint
labels with $\M_1 \oplus \M_2 \not= FIN(\N)$. If $\M_1$ is an infinite simple label and $\M_2$ is an
infinite finitary label which is
not simple, then $\M_1 \oplus \M_2$ is WAP but not semi-simple.
\end{ex}

\begin{prop}\label{labelprop13b} Let $\NN $ be a nonempty label and
 $\{ \M_a \}$ be a finite or infinite sequence of positive
labels  such that $\M_a \ \cap \  \M_b \subset \NN$
when $a \not= b$.
The label $\M =  \bigcup_a \ \M_a$  is bounded if  all the $\M_a$'s and $\NN$ are bounded, and
 is of finite type if all the $\M_a$'s and $\NN$ are of finite type.

 If $\NN$ is finite and
the $\M_a$'s are all finitary  then $ \M $ is  finitary.
For example, if  all the $\M_a$'s and $\NN$ are finite then $\M$ is finitary.
\end{prop}

\proof Let $a \not= b$. If $\rho(\M_{a})_{\ell} > \rho(\M_{b})_{\ell}$ then
$ \rho(\M_{b})_{\ell} \chi(\ell) \in \M_{a} \ \cap \  \M_{b} \subset \NN$.
Thus, for each $\ell$ there is at most one index $a$ with $ \rho(\M_{a})_{\ell} > \rho(\NN)_{\ell}$. Hence,
$\rho(\M)_{\ell} \ = \ \max_a \ \rho(\M_{a})_{\ell} \
<
\ \infty$ for all $\ell$ if  all the $\M_a$'s and $\NN$ are bounded. Hence, in that case $\M =  \bigcup_a \ \M_a$ is bounded.

  If
$  \mm_a \in \M_{a}, \ \mm_b \in \M_{b}$ with $\mm_a \geq \mm_b$ then
$\mm_b \in
\M_{a} \ \cap \  \M_{b} \subset \NN$. Thus, an increasing sequence in $\M$ which is not contained in $\NN$
is contained in some $\M_{a}$. Thus, the Finite Chain Condition for $\M$ follows from the condition for $\NN$ and each $\M_a$.

%
%

Now assume that
$\NN$ is finite and that
the $\M_a$'s are finitary. To show that $\M$ is finitary we again apply condition  (ii) of Proposition \ref{labelprop09} (a).
Let $\{ \ell^i \}$ be a sequence of distinct positive integers. Since $\NN$ is finite, $K = \bigcup \ Supp \ \NN$ is
finite and we may assume, by discarding finitely many elements that
$\ell^i \not\in K$ for all $i$. Since $K$ is finite, it suffices to show that $\{ \ell \in \N \setminus K :$
such that eventually $\{ \ell, \ell^i \} \in Supp \ \M \}$ is finite.

If $\ell^i \not\in \bigcup \ Supp \ \M_a$ for any $a$ we let $a_i = *$.
Otherwise, $\ell^i \in \bigcup \ Supp \ \M_{a_i}$ for a unique $a_i$ since $\ell^i \not\in K$. If $a_i = *$ then
$\{ \ell, \ell^i \} \not\in Supp \ \M$ for any $\ell \in \N$. Otherwise,

\begin{align}\label{finunionb}
\begin{split}
 \{ \ell, \ell^i \} \in Supp \ \M \quad &\Longleftrightarrow\quad \{ \ell, \ell^i \} \in Supp \ \M_{a_i}, \\
\{ \ell, \ell^i \} \in Supp \ \M_{a_i} \quad &\Longrightarrow \quad \{ \ell, \ell^j \} \not\in Supp \ \M \ \text{if} \ a_j \not= a_i,
\end{split}
\end{align}
because if $a_j \not= a_i$ then $\ell \not\in \bigcup \ Supp \ \M_{a_j}$.

If there exists $a$ such that eventually $a_i = a$, then eventually $\{ \ell, \ell^i \} \in Supp \ \M $ iff
eventually $\{ \ell, \ell^i \} \in Supp \ \M_a$ and the set of such $\ell$ is finite because $\M_a$ is finitary.

Otherwise, for no $\ell \in \N$ is it true that eventually $\{ \ell, \ell^i \} \in Supp \ \M $.

%
%
%
%

 $\Box$ \vspace{.5cm}

\begin{remark}
Since any  label $\M$ is the union of
the finite labels $\M \cap \B_N$ it follows that some condition
 is needed to get the finite type or finitary conditions for a union.
 \end{remark}
\vspace{.5cm}

 From the definition (\ref{label06}) of the metric on $\LAB$ we see that if $\M$ is a nonempty label then
 \begin{equation}\label{limtozero}
 [0,N] \cap \bigcup Supp\ \M = \emptyset \qquad \Longrightarrow \qquad d(0,\M) \leq 2^{-N}.
 \end{equation}

 \begin{theo}\label{towtheo26} Let $\{ \M_a \} $
 be
 a finite or infinite pairwise disjoint collection of positive labels
 and let $\M  = \bigcup \{ \M_a \}$. Assume $\M \not= FIN(\N)$.
  \begin{equation}\label{tow35}
  \Theta'(\M) \ = \ \bigcup_a \ \{ \Theta'(\M_a) \}.
 \end{equation}

 The map sending $P_{\rr}$ on $\Theta(\M_a)$ to $P_{\rr}$ on $\Theta(\M)$
 for  $\00 < \rr  \in \M_a$,
 extends to a
 continuous injective homomorphism  $j_a : \A(\Theta(\M_a)) \to \A(\Theta(\M))$. If $Q \in \A(\Theta(\M_a))$ then
 \begin{equation}\label{tow35a}
 j_a(Q)(\NN)  = \begin{cases}   Q(\M_a) \qquad \text{if} \quad \NN = \M, \\
 Q(\NN) \qquad \text{if} \quad \NN \in \Theta'(\M_a), \\
 \emptyset \qquad \text{if} \quad \NN \in \Theta'(\M_b) \quad \text{with} \ b \not= a.\end{cases}
 \end{equation}

 If $Q_1 \in \A(\Theta(\M_a))$ and $Q_2 \in \A(\Theta(\M_b))$ with $a \not= b$ then $j_a(Q_1)j_b(Q_2) = j_b(Q_2)j_a(Q_1) = U$.
 Furthermore,
 $\A(\Theta(\M)) \setminus \bigcup \{ j_a(\A(\Theta(\M_a))) \}$
 contains at most one point $Q^*$ in which case $Q^*(\M) = 0$ and $Q^* = \emptyset$ on $\Theta'(\M)$ and
 $Q^*Q = Q Q^* = U$ for all $Q \in \A(\Theta(\M))$.
 If some $\M_a$ is of finite type then $\A(\Theta(\M)) =\bigcup \{ j_a(\A(\Theta(\M_a))) \}$.
  \end{theo}

 \proof
 If $\{ \rr^i > \00 \}$ is a sequence in
$\M$ and $\mm > \00 $ is an $\N$-vector in $\M$ then
$\mm  \in \M_a$ for some $a$.  If $\{ \M - \rr^i \}$ converges with $\mm$ in the limit then
eventually $\rr^i \in \M_a$ in which case $\M - \rr^i = \M_a - \rr^i$ and
$LIM \{ \M - \rr^i \} = LIM \{ \M_a - \rr^i \}  \in \Theta'(\M_a)$.  Hence,
 $\Theta'(\M) \ \subset \ \bigcup \ \{ \Theta'(\M_a) \}$. The reverse inclusion is obvious.

 Now suppose that $\{ P_{\rr^i} \}$ is a net with
 $\rr^i > \00 $ in $\M$ converging to $Q \in \A(\Theta(\M))$.

 Assume  first that for some $\NN \in \A(\Theta(\M))$ that
 $\00 < \mm \in Q(\NN)$.
 Let $M_b$ be the unique member to the sequence which contains $\mm$.  Eventually $\rr^i \in \M_b$ so that
 eventually $\M_1 - \rr^i = (\M_1 \ \cap \  \M_b) - \rr^i$ for all $\M_1 \in [[M]]$. In particular,
 $Q(\M) = LIM \ \{ \M - \rr^i \} = LIM \ \{ \M_b - \rr^i \}$.

 Notice that
 $\M_2 = \bigcup \{ \M_2 \cap \M_a \}$ and by (\ref{tow35})
 $ \M_2 \in \Theta'(\M)$ iff $\M_2  = \M_2 \ \cap \  \M_a \in \Theta'(M_a)$
 for some $a$. Clearly $Q(\M_2) = \emptyset $ if $a \not= b$ and $Q(\M_2) = LIM \{ \M_1 - \rr_b \}$ for $\M_2 \in \Theta(M_a)$.
 Hence, $Q = j_b(\tilde Q)$ where $\tilde Q $ is the pointwise limit of $\{ P_{\rr^i} \}$ in $\A(\Theta(\M_b)$.

 On the
 other hand if $\{ P_{\rr^i} \}$ in
 $\A(\Theta(\M_b))$ converges to $\tilde Q$ with $\rr^i > \00 $ in $\M_b$ then
 $P_{\rr^i}$ in $\A(\Theta(\M))$ converges pointwise to $Q(\NN) = LIM \{ \NN \ \cap \  \M_b - \rr^i \}$ for all $\NN \in \Theta(\M)$.
 This defines the injection $j_b : \A(\Theta(\M_b)) \to \A(\Theta(\M))$ defined by (\ref{tow35a}. It is clear
 that $j_b$ is continuous and is a homomorphism and is obviously injective. Notice that if $\rr > 0 \in \M_b$ then
 $j_b(P_{\rr})$, the extension of $P_{\rr}$ in $\A(\Theta(\M_b))$ is just $P_{\rr}$ acting on $\Theta(\M)$.

 Now assume that $Q(\M) = 0$. If for some cofinal set $\{ \rr^{i'} \in \M_b \}$ then $Q = j_b(\tilde Q)$ with
 $\tilde Q $ the limit in $A(\Theta(\M_b))$ of some subnet $\{ P_{\rr_{i'}}\}$.
  Otherwise, eventually $\rr^{i'} \not\in M_b$ for every $b$ and
 so $Q(\NN) = \emptyset $ for all $\NN \in \Theta'(\M)$. Notice that this requires that the sequence $\{ \M_a \}$ be infinite.
 This is $Q^* \in \A(\Theta(\M)$ with $Q^*(\M) = 0$ and $Q^*(\NN) = \emptyset$ for $\NN \in \Theta'(\M)$. When
 $\{ \M_a \}$ is infinite and $\{ \rr^i > 0\}$ is a sequence in $\M$ with $\rr^i \in \M_{a(i)}$ and $a(i) \not= a(j)$ for
 $i \not= j$ then the sequence $\{ P_{\rr^i} \}$ in $\A(\Theta(\M))$ converges to $Q^*$.

 It may happen that $Q^* \not\in \bigcup_a  j_a(\A(\Theta(\M_a)))$, but if $\M_a$ is of finite type and
 $\rr \in max\ \M_a$ then $j_a(P_{\rr}) = Q^*$ by  (\ref{wap00b}).

 $\Box$ \vspace{.5cm}

 \begin{cor}\label{towcor26c}  Let $\{ \M_a \} $
 be
 a finite or infinite pairwise disjoint collection of positive labels
 and let $\M  = \bigcup \{ \M_a \}$. If each $\M_a$ is  semi-simple, finitary or WAP then $\M$ satisfies the
 corresponding property. If each $\M_a$ is simple and at least one $\M_a$ is of finite type then $\M$ is simple.\end{cor}

 \proof $FIN(\N)$ is simple and WAP.  By Proposition \ref{labelprop13b} if the $\M_a$'s are
 of finite type then $\M$ is of finite type and so is not
 $FIN(\N)$ and if the $\M_a$'s are finitary then $\M$ is finitary. So we may assume $\M \not= FIN(\N)$ and apply
 Theorem \ref{towtheo26}.

 In any case, $Q^*$ commutes with every element of $\E(\Theta(\M))$ and for $a \not= b$ the elements of $j_a(\A(\Theta(\M_a))$
 commute with the elements of $j_b(\A(\Theta(\M_b))$ with $a \not= b$. If each $\M_a$ is WAP then each of
 the $j_a(\A(\Theta(\M_a))$'s is abelian and so $\E(\Theta(\M)$ is abelian, i.e. $\M$ is WAP.

 If each $\M_a$ is simple then each $Q \in \bigcup_a j_a(\A(\Theta(\M_a))$ is of the form $P_{\rr}$. If at least one
 $\M_a$ is of finite type then this includes $Q^*$.

 An external element of $\E(\Theta(\M))$ is either $Q^*$ or is $j_a(Q)$ for some external $Q \in \A(\Theta(\M_a))$.
 So if all the $\M_a$'s are semi-simple, these all map into the set of finite labels and so $\M$ is semi-simple.

 $\Box$ \vspace{.5cm}

\begin{ex}\label{ex7aa} If $\M = \langle \{ \chi(1) + 2\chi(k) : k > 1 \} \rangle$ then $\M$ is simple and so
$\emptyset$ is the only recurrent point of $\Theta(\M)$. $\M_k = \langle \{ \chi(1) + \chi(k) \} \rangle = P_{\chi(k)} \M$
and $\M_1 = \langle \{ \chi(1)  \} \rangle = P_{2\chi(k)} \M = P_{\chi(k)} \M_k$ for each $k > 1$. Since
$\M_1 \cap \B_N = \M_k \cap \B_N$ for $N < k$, it follows that $\M_1$ is a non-wandering
point of $\Theta(\M)$.
\end{ex}

$\Box$ \vspace{0.5cm}

\begin{ex}\label{ex8} (a)  If $\M$ is defined by
$\M  = \langle \{ \chi(1) + \chi(\ell) : \ell > 1 \} \ \cup \ \{ \chi(\ell) + \chi(\ell + 1) : \ell > 1 \} \rangle$, then $\M$ is
size bounded and
{\bf finitary, but not simple}.
\begin{equation}\label{label13}
\begin{split}
\M - \chi(1)  \ = \ \{ 0 \} \cup \{ \chi(\ell) : \ell > 1 \}, \hspace{3cm}\\
 \M - \chi(\ell) \ = \ \{ 0 \} \cup\{ \chi(1), \chi(\ell - 1), \chi(\ell + 1) \},
\quad \mbox{for} \ \ell > 1, \\
\F \ = \ \{ 0, \chi(1) \} \ = \ LIM_{\ell \to \infty}  \{ \M - \chi(\ell) \}. \hspace{2cm}
\end{split}
\end{equation}
Notice that $\F \not= \M - {\mathbf r}$ for any ${\mathbf r} \in \M$.

(b) If $\M$ be defined by
$\ \M =
\langle  \{ \chi(2 a - 1) + \chi(2b)  : a , b \geq 1 \} \rangle$, then $\M$ is size bounded and is simple
but is not finitary. In general, if $\M = \M_1 \oplus \M_2$ with $\M_1, \M_2$ infinite simple labels with
$\M_1 \ \cap \  \M_2 = 0$ then $\M$ is
{\bf simple but not finitary}. \end{ex}

$\Box$ \vspace{0.5cm}

\begin{ex}\label{ex10}
It can happen that $\M$ is
{\bf strong finite type but not WAP} and so is not semi-simple.

Let $\M$ be defined by
$ \M =  \langle \{ \chi(3) + \chi(2 a + 1) + \chi(2b)  : a \geq b \geq 2 \}  \cup
\{ \chi( 1) + \chi(3) + \chi(2b) : b  \geq 2 \} \rangle$.
\begin{equation}\label{tow23}
\begin{split}
\M - \chi(1) \ = \  \langle \{ \chi(3) + \chi(2b)  : b \geq 2 \} \rangle, \hspace{4cm} \\
\M - \chi(3)  \ = \ \langle \{ \chi(2 a + 1) + \chi(2b)  : a \geq b \geq 2 \} \cup \{ \chi( 1) +  \chi(2b) : b  \geq 2 \} \rangle \\
\M - \chi(2\ell + 1)  \ = \ \langle \{ \chi(1) + \chi(2b)  : \ell \geq b \geq 2 \} \rangle \hspace{3cm}\\
 \M - \chi(2\ell)  \ = \ \langle \{ \chi(1) + \chi(2a + 1)  : a \geq \ell \geq 2 \}\cup \{ \chi( 1) + \chi(3) \} \rangle.
  \end{split}
  \end{equation}
  It follows that
  \begin{equation}\label{tow24}
  \begin{split}
LIM_{a \to \infty }  \ \{ \M - \chi(2 a + 1) \} \ = \ \langle \{ \chi(3) + \chi(2b)  : b \geq 2 \} \rangle, \hspace{2cm}\\
LIM_{b \to \infty} \   \{ \M  - \chi(2b) \} \ = \  \langle \{  \chi(1) + \chi(3) \} \rangle , \hspace{4cm} \\
LIM_{b \to \infty} LIM_{a \to \infty}   \{ \M - \chi(2 a + 1) - \chi(2b) \} \ = \ \{ \chi(3) , 0 \}, \hspace{1.5cm}\\
LIM_{a \to \infty} LIM_{b \to \infty}   \{ \M - \chi(2 a + 1) - \chi(2b) \} \ = \ \emptyset. \ \hspace{2cm}
  \end{split}
  \end{equation}

The sequences $\{ P_{ \chi(2 a + 1)} \}$ and $\{ P_{ \chi(2 b)} \}$  converge in $\E(\Theta(\M))$ to elements which we will
denote by
$Q_o, Q_e$, respectively. $Q_e(\M)$ is the finite label  $\langle \{  \chi(1) + \chi(3) \} \rangle$.
$Q_o(\M) = P_{\chi(1)}(\M)$, and so $Q_e Q_o(\M) = P_{\chi(1)}(Q_e(\M)) = \langle \{   \chi(3) \} \rangle$.
On the other hand, $Q_o Q_e(\M) = \emptyset \not= P_{\chi(1)}(Q_e(\M))$.  In particular, $Q_o \not= P_{\chi(1)}$ on
$\Theta(\M)$ and so is not continuous.

The enveloping semigroup $\E(\Theta(\M))$ is not abelian and so $ \M$ is not WAP. Notice that if $\NN \subset \M$ is  infinite
then $\M - \NN$ is finite.  Thus, this condition alone does not suffice to yield $\M$ finitary.
%
%
%
 \end{ex}

$\Box$ \vspace{.5cm}


We conclude with an example which computes $\E(\Theta(\M))$ for a label $\M$ which is recurrent and so is not WAP.


\begin{ex}\label{exrecur} For $L \subset \N$ we define the label $\langle L \rangle$ to consist of the
characteristic functions of finite subsets of $L$:
\begin{equation}\label{exrecur01}
\langle L \rangle = \{ \chi(F) : F \in \P_f L \}
\end{equation}
Since $\chi_{\emptyset} = \00 , \ \langle \emptyset \rangle = 0$. We compute the orbit space
$\Theta(\M)$ and the enveloping semigroup $\E(\Theta(\M))$ for the recurrent label
$\M = \langle \N \rangle$. Since $\M$ is recurrent, it is not WAP and so the semigroup cannot be abelian.

We will show that
\begin{equation}\label{exrecur02}
\Theta(\M) = \{ \emptyset \} \cup \{ \langle L \rangle : L \subset \N \}.
\end{equation}

Let $2^{\b \N} \cup \{ * \}$ be the compact space of all compact subsets of $\b \N$ together with an
additional isolated point, labeled $*$. We will apply the results about $2^{\b \N}$ which are reviewed
in Appendix
\ref{appendix-StoneCech}.
In particular, for $A \in 2^{\b \N}$, $\F_A = \{ L \subset \N : A \subset \ol{L} \}$ and
$A_0 = A \cap \N$. Thus, $\F_{\emptyset}$ is the power set of $\N$.
If $A$ is nonempty then $\F_A$ is a filter of subsets of $\N$ by Proposition \ref{appprop02a} (a).

For $A \in 2^{\b \N} \cup \{ * \}$ we define $Q_A : \Theta(\M) \to \Theta(\M)$
by $Q_A(\emptyset) = \emptyset$ for all $A$ and $Q_{*}(\langle L \rangle) = \emptyset$ for all $L \subset \N$ and
for $A \in 2^{\b \N}, L \subset \N$:
\begin{equation}\label{exrecur03}
Q_A(\langle L \rangle) = \begin{cases} \langle L \setminus A_0 \rangle \quad \text{for} \ L \in \F_A, \\
\emptyset \qquad \text{otherwise}. \end{cases}
\end{equation}
Thus, $Q_{*} = U,$ the retraction to the fixed point $\emptyset$.

We will show that $A \mapsto Q_A$ defines a homeomorphism from $2^{\b \N} \cup \{ * \}$ onto $\E(\Theta(\M))$.
Composition is described by
\begin{equation}\label{exrecur04}
Q_B Q_A = \begin{cases} Q_{A \cup B} \quad \text{if} \ \N \setminus A_0 \in \F_B,\\
U \qquad \text{otherwise}. \end{cases}
\end{equation}
for $A, B \in 2^{\b \N}$. Notice that $\N \setminus A_0 \in \F_B$ requires $A_0 \cap B_0 = \emptyset$.
If a finite set $F \subset \N$ is disjoint from $B_0$ then $\N \setminus F \in \F_B$. So if $A_0$ is
finite and $A_0 \cap B_0 = \emptyset$ then $\N \setminus A_0 \in \F_B$. In general, $\N \setminus A_0 \in \F_B$
iff the clopen set $\ol{A_0}$ is disjoint from $B$.

From this it is easy to check that for $A, B \in 2^{\b \N}$
\begin{equation}\label{exrecur05}
\begin{split}
Q_A Q_A = \begin{cases} Q_A \quad \text{if} \  A_0 = \emptyset, \\
U \qquad \text{otherwise}. \end{cases} \hspace{4cm} \\   \\
Q_B Q_A = Q_{A \cup B} = Q_A Q_B \quad \text{if}  \ \N \setminus A_0 \in \F_B, \ \text{and} \ \N \setminus B_0 \in \F_A. \\ \\
Q_B Q_A = U = Q_A Q_B \quad \text{if}  \ \N \setminus A_0 \not\in \F_B, \ \text{and} \ \N \setminus B_0 \not\in \F_A.\hspace{.5cm}
\end{split}
\end{equation}

Thus, if $A_0 = \emptyset$, $Q_A$ is an idempotent which fixes $ \langle L \rangle $ for all $L \in \F_A$.  In particular,
$Q_A(\M) = \M$ and thus $\M$ is recurrent. The failure of commutativity occurs only when
$\ol{A_0}$ is disjoint from $B$, and so $B_0$ is disjoint from $A$, but $\ol{B_0}$ meets $A$. In that case,
$Q_B Q_A = Q_{A \cup B},$ but $Q_A Q_B = U$.

Now we prove all this.

\proof Let $\Theta = \{ \emptyset \} \cup \{ \langle L \rangle : L \subset \N \}$. Notice that each
$\langle L \rangle$
is a lattice and
$\Theta$ is the collection of all sublattices of $\langle \N \rangle$,
see Proposition \ref{labelprop05a2} (e). The set
$\{ \NN : \rr_1, \rr_2 \in \NN,$ and $\rr_1 \vee \rr_2 \not\in \NN \}$ is clopen for every pair $\rr_1, \rr_2 \in FIN(\N)$.
The union over all pairs is the -open- complement of the set of sublattices of $FIN(\N)$.

If $F \subset \N$ is finite then $P_{\chi(F)}(\M)
= \langle \N \setminus F \rangle$ which is a sublattice. Since the set of
sublattices is closed, it follows that $\Theta(\M) \subset \Theta$.

Next we show that if $\{ A^i \}$ is a net in $ 2^{\b \N} \cup \{ * \}$ converging to $A$ then $\{ Q_{A^i} \}$
converges pointwise to $Q_A$ as a function on $\Theta$. Since $*$ and $\emptyset$ are isolated points we may assume
that $A$ and all the $A^i$'s are nonempty elements of $2^{\b \N}$.

If $L \not\in \F_A$ then by Proposition \ref{appprop02d} eventually $L \not\in \F_{A^i}$ and so
eventually $Q_{A^i}(\langle L \rangle) = \emptyset = Q_A(\langle L \rangle)$. If $L \in \F_A$ then  by Proposition \ref{appprop02d}
again eventually $L \not\in \F_{A^i}$ and so eventually
$Q_{A^i}(\langle L \rangle) = \langle L \setminus  (A^i)_0 \rangle$ and $Q_A(\langle L \rangle) =
\langle L \setminus  A_0 \rangle$.
Furthermore, for the finite set $[1, \ell]$ eventually $[1, \ell] \cap (A^i)_0 = [1, \ell] \cap A_0 $.
This implies that $\{ \langle L \setminus  (A^i)_0 \rangle \}$ converges to
$\langle L \setminus A_0  \rangle$ in $\LAB$.

For any $A \in 2^{\b \N}$ there is a net $\{ F^i \}$ of finite subsets of $\N$ with limit $A$, see Corollary \ref{appcor02c}
(b). So $\{ Q_{F^i} \} $ converges to $ Q_A$ pointwise.  In particular, $ P_{\chi(F^i)}(\M) = Q_{F^i}(\M)  \to
Q_A(\langle \N \rangle) =
\langle \N \setminus A_0 \rangle$. If $L \subset \N$ and $A = \ol{\N \setminus L}$ with $A_0 = \N \setminus L$ it
follows that $Q_A(\langle \N \rangle) = \langle L \rangle \in \Theta(\N)$.
Hence, $\Theta(\M) = \Theta$ and every $Q_A \in \E(\Theta(\M))$.
If $F \subset \N$ is finite then clearly $Q_F = P_{\chi(F)}$ on $\Theta(\M)$. In particular,
$Q_{\emptyset} = P_{\00 } = id_{\Theta(\M)}$. If $\rr \in FIN(\N) \setminus
\langle \N \rangle$ then $P_{\rr} = U = Q_{*}$.

Now suppose that $\{ P_{\rr^i} \}$ converges to $Q \in \E(\Theta(\M))$.
If $\rr^i \not\in \langle \N \rangle$ for some
cofinal set then $Q = U$. So we may assume that
$\rr^i \in \langle \N \rangle$ for all $i$.  Then for each $i$ there is a
finite set $F^i$ such that $\rr^i = \chi(F^i)$.  By going to a subnet, we may assume that $\{ F^i \}$ converges
to $A \in \b \N$. Then $\{ Q_{F^i} = P_{\rr^i} \}$ converges to $Q_A$ and so $Q = Q_A$. Thus, the continuous map
$A \to Q_A$ from $2^{\b \N} \cup \{ * \}$ to $\E(\Theta(\M))$ is surjective. If $A, B \in 2^{\b \N}$ are unequal then
we may assume that $B \setminus A \not= \emptyset$. Then there exists $L \subset A$ such that the clopen set
$\ol{L}$ contains $A$ but $B$ meets its complement. Hence, $L \in \F_A$ which $L \not\in \F_B$. So
$Q_B(\langle L \rangle) = \emptyset \not=
\langle L \setminus A_0 \rangle = Q_A( \langle L \rangle)$.
It follows that the map $A \mapsto Q_A$
is injective and so is a homeomorphism from the compact space $2^{\b \N} \cup \{ * \}$ onto $\E(\Theta(\M))$.

Now let $A, B \in 2^{\b \N }$. If $L \in \F_A$ then $Q_A(\langle L \rangle) =
\langle L \setminus A_0 \rangle$. So if
$L \setminus A_0 \in \F_B$, which requires $A_0 \cap B_0 = \emptyset$, then we have $Q_B Q_A(\langle L \rangle)
= \langle (L \setminus A_0) \setminus B_0 \rangle =
\langle L \setminus (A \cup B)_0 \rangle$. Since $\F_B$ is a filter
$L \setminus A_0 \in \F_B$ iff $L,  \N \setminus A_0 \in \F_B$.
On the other hand, $L \in \F_A$ and
$L \in \F_B$ iff $L \in \F_{A \cup B}$ in which case $Q_{A \cup B}(\langle L \rangle) =
\langle L \setminus (A \cup B)_0 \rangle$.
If $L \not\in \F_{A \cup B}$ then either $L \not\in \F_A$ and so
$Q_B Q_A(\langle L \rangle) = Q_B(\emptyset) = \emptyset = Q_{A \cup B}(\langle L \rangle)$, or $L \in \F_A$ but
$L \not\in \F_B$ so the $Q_B Q_A( \langle L \rangle)
= Q_B( \langle L \setminus A_0 \rangle) = \emptyset$.
It follows that if $\N \setminus A_0 \in \F_B$ then $Q_B Q_A = Q_{A \cup B} $.
On the other hand, if
$\N \setminus A_0 \not\in \F_B$ then $Q_A( \langle L \rangle) = \emptyset$ and so $Q_B Q_A( \langle L \rangle) = \emptyset$
or $Q_A(\langle L \rangle) =
\langle L \setminus A_0 \rangle$ and so $Q_B Q_A(\langle L \rangle) = \emptyset$
since $L \setminus A_0 \not\in \F_B$.
Thus, if $\N \setminus A_0 \not\in \F_B$ then $Q_B Q_A = U$.

The (\ref{exrecur05}) results follow easily from (\ref{exrecur04}).

\end{ex}

$\Box$ \vspace{1cm}

\section{Dynamical properties of $X(\M)$}\label{sec,dyn-prop}


\subsection{Translation finite subsets of $\Z$}

$\qquad$

\vspace{0.5cm}

We recall the following combinatorial characterization of WAP subsets of $\Z$
(\cite{Ru}).

\begin{theo}[Ruppert]\label{theorupp01}
For a subset $A \subset \Z$ the following conditions are equivalent:
\begin{enumerate}
\item
The subshift $\overline{\mathcal{O}}(\chi(A)) \subset \{0,1\}^\Z$ is WAP.
\item
For every infinite subset $B \subset \Z$ either:

(i) there exists $N \ge 1$ such that
\begin{equation}\label{rupp1}
\bigcap_{b \in B \cap [-N,N]} A - b
\quad \text{is finite},
\end{equation}
or:

(ii) there exists $N \ge 1$ and $n \in \Z$ such that
\begin{equation}\label{rupp2}
A -n \supset B \cap \bigl(\Z \setminus [-N,N]\bigr).
\end{equation}
\end{enumerate}
\end{theo}

\begin{df}[Ruppert]\label{defrupp02}
We say that a subset $A \subset \Z$ is {\em translation finite} (TF hereafter) if
for every infinite subset $B \subset \Z$ there
exists an $N \ge 1$ such that
\begin{equation}\label{rupp3}
\bigcap_{b \in B \cap [-N,N]} A - b \ = \ \{ \ n \in \N \ : \ A - n \supset B \cap [-N, N] \ \} \quad \text{is finite}.
\end{equation}
\end{df}

\begin{ex}
It is easy to check that the set
$A =2\Z_+ \cup - (2\Z_+1)$

(with $c = \chi(A) = (\dots, 1,0,1,0,1, \dot{1},0,1,0, \dots)$)
does not satisfy Ruppert's condition (and a fortiori
is not translation finite), hence $\overline{\mathcal{O}}(\chi(A))$
is not WAP.

(See Example \ref{ex3}.(b).)
\end{ex}

$\Box$

\vspace{.5cm}

\begin{prop}\label{proprupp03}
Let $A$ be a subset of $\Z$. The following conditions are equivalent.
\begin{enumerate}
\item
The subset $A$ is TF.
\item
Every point in
$R_S(\chi(A)) = (\o_S \cup \a_S)(\chi(A))$ has finite support.
\item
The subshift $\ol{\mathcal{O}}(\chi(A))$ is CT of height at most $2$.
\end{enumerate}
\end{prop}

\begin{proof}
$(1) \Lra (2)$:
Suppose first that $A$ is TF and suppose that for some sequence $\{n_i\}_{i =1}^\infty$, with
$|n_i| \to \infty$, we have
$S^{n_i} \chi(A) = x$
with $supp \ x$ an infinite set.
Let $B = supp \ x$ and observe that
for every $N \ge 1$, eventually,
\begin{equation}\label{rupp4}
S^{n_i} \chi(A) \wedge [-N, N] = (x_{-N}, \dots, x_N),
\end{equation}
whence $A - n_i \supset B \cap [-N, N]$.
But this contradicts our assumption that $A$ is TF.

$(2)  \Rightarrow (1)$:
Conversely, suppose $A$ is not TF. Then there exists an infinite $B \subset \Z$
such that for every $N \ge 1$ the intersection
\begin{equation}\label{rupp5}
 \{ n \in \Z : A -n \supset  B \cap [-N,N]\}
\quad \text{is infinite}.
\end{equation}

We can construct a strictly increasing sequence $\{n_i\}_{i=1}^\infty$
with $ A - n_i \supset B
\cap
[-i,i]$ and so for any limit point
$x \in \{0,1\}^\Z$ of
the
sequence
$ \{ S^{n_i}\chi(A) = \chi(A - n_i) \}$ the support
$supp \ x \supset B$ and so is infinite.

$(2) \Rightarrow (3)$:  is obvious.

$(3)  \Rightarrow (2)$:
Suppose finally that that $\ol{\mathcal{O}}(\chi(A))$ is CT of height at most 2.
Suppose to the contrary that $x \in
(\o_T \cup \a_T)(\chi(A))$ has infinite support,
say $supp \ x = B$. By compactness there exists a sequence $\{n_i\}_{i=1}^\infty \subset B$
such that the sequence $S^{n_i}x$ converges. Let $y =\lim_{i \to \infty} S^{n_i}x$.
Then $y \in   R_T(x)$ and $y_0 =1$. Thus $y \ne \mathbf{0}$ and this contradicts our assumption
that $\ol{\mathcal{O}}(\chi(A))$ is of height at most 2.

\end{proof}

$\Box$
\vspace{0.5cm}

We next address the question
`when is $A[\M]$ TF ?'.
This turns out to be a rather restrictive condition, because
$\emptyset$ and $ 0 $ are the only labels $\NN$ such that $A[\NN]$ has finite support. For $\M = \emptyset$ or $ 0 $,
$R_S(A[\M]) = \{ e \}$ where $e$ is the fixed point $\bar 0 = A[\emptyset] $. Thus, in these cases $x[\M]$ is TF.

\begin{prop}\label{ruppprop02} For a positive label $\M$ the following conditions are equivalent.
\begin{itemize}
\item[(i)] $\Theta(\M) = \{ \M, 0, \emptyset \}$.

\item[(ii)] For all $\rr > \00 \in \M, \ \M - \rr = 0 $.

\item[(iii)] There exists $L$ a nonempty subset of $\N$ such that
$\M = \langle \{ \chi(\ell) \ : \ \ell \in L \} \rangle$.

\item[(iv)]  $A[\M]$ is TF.

\item[(v)] $(X(\M),S)$ has height 2.

\item[(vi)] $(X_+(\M),S)$ has height 2.
\end{itemize}

When these conditions hold, $\M$ is finitary and simple.
\end{prop}

\proof (iii) $\Rightarrow$ (ii) : Obvious.

(ii) $\Rightarrow$ (i) : From (ii) it is clear that
$\Theta(\M)$
consists of $P_{\rr}(\M) = \emptyset$ for
$\rr \not\in \M$ and $P_{\rr}(\M) = 0 $ for $\rr > \00 \in \M$ and finally, $P_{\00 }(\M) = \M$.
Hence, the only limit labels possible in $\Theta(\M)$ are $\emptyset, 0$ and $\M$. In passing, we see that
$\A(\Theta(\M))$ contains one nontrivial element which maps $\M$ to $ 0 $ and maps $0$ and $\emptyset$ to
$\emptyset$.

(i) $\Rightarrow$ (iii) : If $\rr \in \M$ with $|\rr| \geq 2$ then there exists $\ell \in \N$ such that
$\rr - \chi(\ell) > \00 $. Hence, $\M - \chi(\ell) \in \Theta(\M)$ is neither $0$ nor $\emptyset$. Hence,
$\M - \chi(\ell) = \M$.  That is, $\rr + \chi(\ell) \in \M$ for all $\rr \in \M$.  So $\rr \in \M$ implies
$\chi(\ell) \in \M - \rr$ and thus, $\M - \rr = \M$ for all $\rr \in \M$. This implies that $0 \not\in \Theta(\M)$.
Contrapositively, (i) implies that $|\rr| = 1$ for any nonzero $\rr$ in $\M$. That is, each nonzero $\rr$ in $\M$
is some $\chi(\ell)$.

(iii) $\Rightarrow$ (v) :  $\M$ is bounded and size-bounded and so is of finite type.
By Corollary
\ref{towcor15a}
the limit points
of $x[\M]$ lie on the orbits of $x[\NN]$ for some $\NN \in \Theta(\M)$. Since, $\M$ is not recurrent, the set $R_S(x[\M])$ of
limit points consists of the orbits of $x[ 0 ]$ and $x[ \emptyset ]$.  These in turn map to the fixed point $x[ \emptyset ]$ and
so $(X(\M),S)$ has height 2.

(v) $\Rightarrow$ (iv) :  This follows from Proposition  \ref{proprupp03}.

(iv) $\Rightarrow$ (iii) : We prove the contrapositive,
assuming, as above, that there exist
$\rr \in \M$ and $\ell \in \N$ such that
$\rr - \chi(\ell) > \00 $.  Choose an increasing sequence
$\{ t^i \in IP(k) \}$  with length vectors $\rr(t^i) = \chi(\ell)$
and with $|j_r(t^i)| \to \infty$.  By     $\{ S^{t^i}(x[\M]) \}$ converges to $x[\M - \chi(\ell)]$ which does not have
finite support since  $\rr - \chi(\ell) \in \M - \chi(\ell)$. By Proposition  \ref{proprupp03} again $A[\M]$ is not TF.

Finally, it is clear that the labels described in (iii) are finitary and simple.

$\Box$ \vspace{1cm}

\subsection{Non-null and non-tame labels}\label{non}
$\qquad$

\vspace{0.5cm}

\begin{df}\label{towdf20b} (a) For a subshift $(X,S)$ a subset $K \subset \Z$ is called
an \emph{independent set} if the restriction to $X$ of the
 projection $\pi_K : \{ 0, 1 \}^{\Z} \to \{ 0, 1 \}^K$ is surjective. The subshift is called \emph{null} if there is a
 finite bound on the size of the independent sets for $(X,S)$.  It is called \emph{tame} if there is no infinite
 independent set for $(X,S)$.

 (b) For a label $\M$ a subset $L \subset \M$ is called an \emph{independent set} if for every  $L_1 \subset L$ there
 exists $\NN \in \Theta(\M)$ such that $L \ \cap \  \NN = L_1$.

  (c) A label $\M$ is called \emph{non-null} if for every $n \in \N$ there is a finite independent
 subset $F \subset \M$ with $\# F \geq n$.  It is \emph{non-tame}   if there is an infinite set $L \subset \M$ such that
 every finite $F \subset L$ is an independent set.
 \end{df}

 \vspace{.5cm}

 Notice that an independent set $L$ for a label $\M$ is certainly not a label.  In fact, if $\mm_1 < \mm $ and $\mm \in L$ then
 $\mm_1 \not\in L$ because if $\mm \in \NN$ for a label $\NN$ then $\mm_1 \in \NN$. Since there exists
 a label $\NN$ such that $\NN \cap L = \{ \mm \}$ it follows that $\mm_1 \not\in L$.

 \begin{remark}\label{remark,KL}
 The concepts
 `null' and `tame' are defined for any dynamical system.
 The first is defined in terms of sequential topological entropy (see e.g. \cite{Goo}
 and the review \cite{G-Y})
 and the latter in terms of the dynamical Bourgain-Fremlin-Talagrand dichotomy for
 enveloping semigroups (\cite{Gl-tame}).
 The convenient criteria which we use here for subshifts to be non-null
 and non-tame, are due basically to
Kerr and Li \cite{KL}  (see \cite[Theorem 6.1.(3)]{GM}).
 \end{remark}

 \vspace{.5cm}


 \begin{prop}\label{IP-not-tame}
$(X(FIN(\N)),S)$ is not tame.
\end{prop}

\begin{proof}
We have $X(FIN(\N)) = X(IP(k))$.
Consider the sequence $\{a_n =k(2n) + k(2n-1) : n \ge 1\}$.
Let $u = \{u_j\}_{j \ge 1}$ be a sequence of $0$'s and $1$'s and let
$s(u,n) = \sum \{k(2i -1) :  u_i =1, i \le n\}$.
Then $s(u,n) + a_j$, for $1 \le j \le n$, is in $IP(k)$ if and only if
$u_j = 0$ and thus we can interpolate on any finite initial segment of the sequence
$\{a_n\}_{n \ge 1}$. It follows that this infinite sequence is an independent set for
$(X(FIN(\N)),S)$.
\end{proof}

 \vspace{.5cm}

 \begin{lem}\label{towprop20c} (a) If $\M$ is a label and $F$ is a finite subset of $\M$ then for any $\NN \in \Theta(\M)$
 there exists an $\N$-vector $\rr$ such that $\NN \ \cap \  F$ = $\M - \rr \ \cap \  F$. In particular, if $F$ is a finite
 independent subset of $\M$ then  for every  $A \subset F$ there exists $\rr$
 such that $F \ \cap \  \M - \rr = A$.

 (b) If  every finite subset $F \subset L $ is an independent set for a label $\M$ then   $L$ is an independent set for $\M$.

 (c) If $L$ is an independent set for a label $\M$, and if for every
 $\mm \in L, \ t(\mm) \in IP_+(k)$ such that
 $\rr(t(\mm)) = \mm$ then $K = \{ \ t(\mm)  \ : \mm \in L \}$ is an independent set for the subshifts $(X(\M),S)$ and $(X_+(\M),S)$.
  \end{lem}

 \proof  (a) If $\B_N$ contains all the supports of elements of $F$ then $\mm \in F$ is in $\M_1 \in \LAB$ iff it is
 in $\M_1 \cap \B_N$. Hence, for any $A \subset F$ the set $\{ \ \M_1 \ : \ \M_1 \ \cap \  F = A \ \}$ is clopen in
 $\LAB$. Since $\{ \M - \rr \}$ is dense in $\Theta(\M)$, the result follows.

 (b) Let $L_1 \subset L$. Let $\{ F^i \}$ be an increasing sequence of finite subsets of $L$ with union $L$.
 Because $F^i$ is an independent set, part (a)
 implies there exists $\rr^i$ such that $F^i \ \cap \  \M - \rr^i  = L_1 \ \cap \  F^i $.
 It follows that if $\mm \in L_1$
 then eventually $\mm \in \M - \rr^i$.
 If $\mm \in L \setminus L_1$ then eventually $\mm \not\in \M - \rr^i$. By going to
 a subsequence, we can assume that $\{ \M - \rr^i \}$ converges to some $\NN \in \Theta(\M)$. Clearly,
 $L \ \cap \  \NN = L_1$.

 (c) For any label $\NN \in \Theta(\M), \ \ t \in
 A[\NN]$ iff
 $t \in IP(k)$  with $\rr(t) \in \NN$ and so
 $x[\NN]_t = 1$ iff $\rr(t) \in \NN$ and similarly $t \in
 A_+[\NN]$ iff
 $t \in IP_+(k)$  with $\rr(t) \in \NN$ and so
 $x_+[\NN]_t = 1$ iff $\rr(t) \in \NN$.

 For $K_1 \subset K$, let $L_1 = \{ \rr(t) : t \in K_1 \} = \{ \mm : t(\mm) \in K_1 \} \subset L$. Since $L$ is an
 independent set for $\M$, there exists $\NN \in \Theta(\M)$ such that $L \cap \NN = L_1$. Hence,
 $x[\NN]_t = 1$ for $t \in K_1$ and $ = 0$ for $t \in K \setminus K_1$. Thus, $K$ is an independent set for
 $X(\M)$. Since $K \subset IP_+(k)$ the same argument works for $X_+(\M)$.

 $\Box$ \vspace{.5cm}

 From this we obviously have

 \begin{prop}\label{towprop20d} (a) A label $\M$ is non-null iff for every $n \in \N$ there is a finite
 subset $F \subset \M$ with $\# F \geq n$ such that for every $A \subset F$ there exists $\rr$
 such that $F \ \cap \  \M - \rr = A$.

 (b) A label $\M$ is  non-tame  if there is an infinite set $L \subset \M$ such that
 for any finite $A \subset F \subset L$ there exists $\rr$ such that $F \ \cap \  \M - \rr = A$. In that
 case if $L_1$ is any subset of $L$ then there exists $\NN \in \Theta(\M)$ such that $L \ \cap \  \NN = L_1$.
 In particular, $\Theta(\M)$ is uncountable.
 \end{prop}

 $\Box$ \vspace{.5cm}

 Recall that $Supp \ \M$ f-contains $L$ when $\P_f L \subset Supp \ \M$, i.e. every finite subset of $L$ is the
support of an element of $\M$, or, equivalently, if $\langle \chi(L) \rangle \subset \M$.  By Proposition \ref{labelprop03}
if a label f-contains an infinite set then it is not of finite type and the converse holds if the label is bounded and
not of finite type then it f-contains an infinite set.

 \begin{remark}
 \label{nt-nwap} It follows that if $\M$ is a  non-tame label then $X(\M)$ and $X_+(\M)$ are uncountable and
 so neither $(X(\M),S)$ nor $(X_+(\M),S)$ can be WAP.
\end{remark}
 \vspace{.5cm}

 \begin{cor}\label{towcor20e} Given a label $\M$, if any
 label $\NN \in \Theta(\M)$ is non-null (or non-tame) then the subshifts $(X(\M),S), (X_+(\M),S)$ are
 not null (resp. not tame). \end{cor}

 \proof  If $\NN \in \Theta(\M)$ then $X(\NN) \subset X(\M)$ and so if $X(\NN)$ projects onto $\{ 0, 1\}^L$ then
 $X(\M)$ does.

 $\Box$ \vspace{.5cm}

There are some simple conditions which allow us to find non-tame labels.

\begin{df}\label{towdef20f} A
bounded label $\M$ with roof $\r(\M)$ is called \emph{flat} over a set $L \subset \N$ if
for all $F \in
Supp\ \M$ with $F \subset L$, $\r(\M)|F \in \M$. Equivalently, if $\mm \in \M$ with $supp \ \mm \subset L$ then
$\r(\M)|(supp \ \mm)  \in \M$. The label is called \emph{flat} when it is flat over $\N$. So a flat label is
bounded. \end{df}
\vspace{.5cm}

\begin{lem}\label{towlem20g}  Let $L \subset \N$.

(a) If the label $\M$ is flat over $L$ and $\rr$ is an $\N$-vector with $supp
\ \rr \subset L$ then $\M - \rr$ is flat over $L$.

(b) If $\{ \M^i \}$ is a collection of  labels each flat over $L$ then $\bigcap \{ \M^i \}$ is a  label which is flat over $L$.

(c) If $\{ \M^i \}$ is an increasing sequence of labels flat over $L$ and $\bigcup \{ \M^i \}$ is bounded,
then it is a  label flat over $L$.

(d) If $\{ \M^i \}$ is a sequence of  labels flat over $L$  and $\bigcup \{ \M^i \}$ is bounded, then $LIMINF
\allowbreak  \  \{ \M^i \}$ is a  label flat over $L$.

(e) The set of labels which are flat  is a closed in the subset of bounded labels.

(f) If $\M$ is a flat label then the elements of $\Theta(\M)$ are all flat labels.
\end{lem}

\proof (a) If $\mm \in \M - \rr$ with $supp \ \mm \subset L$  and $F = supp\ (\mm + \rr)$ then $F \subset L$ and so
$\r(\M)|F \in \M$. Hence,
$\r(\M)|F - \rr \in \M - \rr $ and  $\r(\M)|F - \rr = (\r(\M) - \rr)|F$.   Clearly, $\r(\M) - \rr \geq \r(\M - \rr)$.

(b) If $\M$ is the intersection then $\r(\M) = min \{ \r(\M^i) \}$. If $\mm \in \M$ with $supp \ \mm \subset L$ then
$\r(\M)|(supp \ \mm)  \in \M^i$ for all $i$ and so is in $\M$.

(c) If $\M$ is the union then $\r(\M) = max \{ \r(\M^i) \}$ and this is a non-decreasing sequence of functions. On any
finite set $F$, eventually $\r(\M) = \r(\M^i)$.
If $\mm \in \M$ then eventually $\mm \in \M^i$ and so eventually $\r(\M^i)| (supp \ \mm) \in \M$ and eventually these equal
$\r(\M)|(supp \ \mm)$.

(d) Obvious from (b) and (c).

(e) A convergent sequence of bounded labels with a bounded limit has a flat limit by (d).

(f) If $\M$ is flat then it is bounded and contains every element of the set $\Theta(\M)$.
The result then follows from (a), with $L = \N$, together with (e).

$\Box$ \vspace{.5cm}

 Recall that $Supp \ \M$ f-contains $L$ when $\P_f L \subset Supp \ \M$, i.e. every finite subset of $L$ is the
support of an element of $\M$, or, equivalently, if $\langle \chi(L) \rangle \subset \M$.  By Proposition \ref{labelprop03}
if a label f-contains an infinite set then it is not of finite type and the converse holds if the label is bounded and
not of finite type then it f-contains an infinite set.

\begin{lem}\label{towlem20h} Let $L \subset \N$.
A label $\M$  is flat over $L$ and $Supp \ \M$  f-contains $L$ exactly when
for any finite subset $F$ of $L$, $\r(\M)|F \in \M$. In that case, $\{ \ \chi(\ell) \ : \ \ell \in L \ \}$ is an independent
set in $\M$.
\end{lem}

\proof  The first sentence is clear from the definitions.
If $F$ is a finite subset of $L$ and $A \subset F$, then $\r(\M)|F \ = \ \r(\M)|A \ + \ \r(\M)|(F \setminus A) \in \M$.
So $\r(\M)|A \in \M - \r(\M)|(F \setminus A)$ and so $\{ \ \chi(\ell) \ : \ \ell \in A \ \} \subset \M - \r(\M)|(F \setminus A)$.
But if $\ell \in F \setminus A$ then $\chi(\ell) \not\in \M - \r(\M)|(F \setminus A)$. This means that
$\{ \ \chi(\ell) \ : \ \ell \in L \ \}$ is an independent
set.

$\Box$ \vspace{.5cm}

\begin{remark} It clearly suffices that $\r(\M)|F_k \in \M$for some increasing sequence $\{ F_k \}$ of finite subsets
with union $L$.
\end{remark}
\vspace{.5cm}

\begin{prop}\label{towprop20i} Let $\M$ be a bounded label with $L = supp
\ \r(\M)$.

(a) If $\M$ is flat and f-contains $L$, then $\M$ is a strongly recurrent label.

(b) If $\M$ is a strongly recurrent label, then there exists an infinite set $L_1 \subset L$ such that $\M$ is flat over $L_1$ and
$Supp \ \M$ f-contains $L_1$.
\end{prop}

\proof (a) In this case, $\M$ is a sublattice of $FIN(\N)$ and so it is a strongly recurrent label by Proposition
\ref{labelprop05a2} (e).

(b) Assume that inductively that we have defined $F_k = \{ \ell_1,\dots,\ell_k \}$ of distinct points of
$supp \ \r(\M)$ such that $\r(\M)|F_k \in \M$.  Because $\M$ is strongly
recurrent we can add, in $\M$ any element with support
outside of the finite set $F(\r(\M)|F_k)$. So for sufficiently large $\ell_{k+1}$ we have
that $\r(\M)|F_{k+1} \ = \ \r(\M)|F_k \ + \
\rr(\M)_{\ell_{k+1}} \chi(\ell_{k+1}) \in \M$. Let $F_{k+1} = F_k \cup \{ \ell_{k+1} \}$. Then
let $L_1 = \bigcup_k \{ F_k \}$.

$\Box$ \vspace{.5cm}

\begin{cor}\label{towcor20j} Assume that $\M$ is a bounded label not of finite type.
If $\M$ is flat or strongly recurrent then it is non-tame.
\end{cor}

\proof If $\M$ is strongly recurrent then by Proposition \ref{towprop20i} there exists an infinite subset $L_1$ of $supp
\ \r(\M)$ such that
$\M$ is flat over $L_1$ and $Supp \ \M$ f-contains $L_1$.  By Proposition
\ref{labelprop03}(b)
any label not of finite type
f-contains some infinite set $L_1$ and if $\M$ is flat then it is flat over $L_1$.  By Lemma \ref{towlem20h}
$\{ \ \chi(\ell) \ : \ \ell \in L_1 \ \}$ is an independent set in $\M$.

$\Box$ \vspace{.5cm}

If we define for a bounded label $\M$
\begin{equation}\label{tow26x1}
\F(\M,L) \ = \ \{ \ F \ : F \quad \mbox{ is a finite subset of} \ \ L \ \ \mbox{and} \ \ \r(\M)|F \in \M \ \},
\end{equation}
then $\M$  is flat over $L$ and $Supp \ \M$  f-contains $L$ exactly when $\F(\M,L) = \P_f L$.

We conjecture that for every bounded $\M$ not of finite type there exists $\NN \in \Theta(\M)$ which
is non-tame and so that $(X(\M),S)$ is not tame. Beyond the above corollary the best we can do is the following.

\begin{prop}\label{prop,bounded}
 Let $\M$ be a label not of finite type.  If there exists $N \in \N$ such that  $\r(\M) \leq N$
 then there exists $\NN \in \Theta(\M)$ which
is non-tame. \end{prop}

\proof If $K \in \N$ and $\r(\M) \leq K$ then $\r(\NN) \leq K$ for all $\NN \in [[\M]]$. By Proposition
\ref{labelprop05a1} (d)
there is
a positive recurrent label in $\Theta(\M)$ and so we can assume that $\M$ itself is recurrent. By Proposition
\ref {labelprop05a2} (d)
we can choose an infinite set $L \subset supp
\ \r(\M)$ such that $\{ \ \chi(F) \ : \ F \in \P_f(L) \ \}$ is a strongly recurrent
set for $\M$.  Consider $\F(\M,L)$.

Case (i): If there exists $\{ F^i \}$  a strictly increasing sequence of elements
of $\F(\M,L)$ with $L_1 = \bigcup \{ F^i \}$ then
$\F(\M,L_1) = \P_f L_1$ and so $\M$ itself is non-tame by Lemma \ref{towlem20h}.

Case (ii):
If $F$ is a maximal element of $\F(\M,L)$ then $\r(\M - \r(\M)|F)_{\ell} = 0$ for $\ell \in F$ and
for $\ell \in L $ with
\begin{equation}\label{tow26x2}
\r(\M)_{\ell} \ > \ 0 \qquad \Longrightarrow \qquad  \r(\M)_{\ell} \ > \r(\M - \r(\M)|F)_{\ell}
\end{equation}
because for $\ell \in L \setminus F, \ \ \r(\M)_{\ell} \not\in \M - \r(\M)|F$ by maximality of $F$. Let $\M_1 = \M - \r(\M)|F$
and $L_1 = L \setminus F(\r(\M)|F)$. We see that $\M_1$ is a recurrent element of $\Theta(\M)$ with
$\{ \ \chi(F) \ : \ F \in \P_f(L_1) \ \}$ a strongly recurrent set for $\M_1$.
Furthermore, $\r(\M_1) \leq K - 1$ by (\ref{tow26x2}).
In particular, this cannot happen if $K = 1$.

If Case (ii) occurs then we repeat the procedure with $\M$ and $L$ replaced by $\M_1$ and $L_1$. Eventually, we must terminate
in a Case i situation and so at some $\M_k \in \Theta(\M)$ which is non-tame.

$\Box$ \vspace{.5cm}

\begin{remark}
Notice that Case (i) did not require that $\r(\M)$ is bounded by a constant. Furthermore, once we have
the set $L$ associated with the strongly recurrent subset of $\M$ we can replace it by any infinite subset.  In particular, if
there is any infinite subset of $L$ on which $\r(\M)$ is bounded by a constant then the above argument will apply.
 Thus, the obstruction
to proving the conjecture in general arises when
$Lim \, \r(\M)_{\ell} = \infty$ as $\ell  \to \infty$ in $L$ and every
element of $\F(\M,L)$ is contained in a maximal element of $\F(\M,L)$ and these conditions continue to hold as
we replace $\M$ and $L$ by $\M - \r(\M)|F, L_1$ for $F$ any maximal element of $\F(\M,L)$.  Finally, we notice that
if $\F(\M,L)$ contains sets of arbitrarily large cardinality then $\M$ is at least non-null by Lemma \ref{towlem20h}.
\end{remark}
\vspace{.5cm}

\begin{cor}\label{cor,non-tame}
Let $\M$ be a bounded label not of finite type.
\begin{enumerate}
\item
There is a label $\NN \in [[\M]]$ (i.e. $\NN \subset \M$) which is not of finite type and
 with $\r(\NN)$ bounded by a constant. In particular, $\NN$ is not tame.
\item
There is a label $\NN \supset \M$ and
with $ Supp \ \NN = Supp \ \M$ which is flat, hence not tame.
\end{enumerate}
\end{cor}

\begin{proof}
(1). As $\M$ is not of finite type there is a strictly increasing sequence $\{\mm_i\}_{i=1}^\infty$
of elements of $\M$. Let
$\NN = \langle \{ \chi(supp \ \mm_i) : i =1,2, \dots\} \rangle$.
Then clearly $\NN$ is not of finite type and $\rho(\NN) \le 1$.
The non-tameness follows from
Proposition
\ref{prop,bounded}.

(2). Let $\NN = \langle \{ \rho(\M) : (supp \ \mm) : \mm \in \M\}\rangle$.
Clearly $\NN \supset \M$ and is flat, hence not tame by
Corollary
\ref{towcor20j}.
\end{proof}

$\Box$
\vspace{0.5cm}

\begin{Qs}\label{Qs,III}
$\qquad$
\begin{enumerate}
\item
Is it true that for every label $\M$ not of finite type there exists $\NN \in \Theta(\M)$ which
is non-tame (hence also so that $(X(\M),S)$ is not tame) ?
\item
Is there a label $\M$ not of finite type such that $X(\M)$ is tame or even null ?
\end{enumerate}
Since a subsystem of a tame (or null) system is tame (resp.  null), it follows that we could choose $\M$
in the latter case to be a recurrent label.
A positive answer to this second question (in the null case)
would thus yield an example of a null dynamical
system with a recurrent transitive point which is not minimal.  The question whether such a system exists
is a long standing open question.
\end{Qs}

\vspace{0.5cm}

In a private conversation Tomasz Downarowicz asked us whether it is the case that
every WAP system is null. Our next example shows that there are
(a) non-null simple labels, hence topologically transitive WAP subshifts which are non-null;
(b) non-tame labels of finite type, hence subshifts arising from finite type labels which are not tame.

\begin{ex}\label{ex11} There are simple, finitary labels which are non-null.
Accordingly, by Corollary \ref{towcor20e}, the corresponding subshifts are
topologically transitive WAP subshifts which are non-null.
Also, there are labels of finite type which are  non-tame.
Again the corresponding subshifts are topologically transitive and non-tame. Note that by
Remark \ref{nt-nwap} these latter subshifts are not WAP.

(a) Partition $\N$ into
disjoint
sets $\{ A_n : n \in \N \}$ with $\# A_n = n$.
Define $\M_n$ by $ \M_n = \langle  \ \chi(A_n)  \rangle$. By Lemma \ref{towlem20h}
$\{ \chi(i) : i \in A_n \}$ is an independent set for $\M_n$.
Since $\{ \M_n \}$ is a pairwise disjoint sequence of
finite labels, $\M = \bigcup_n \{ \M_n \}$ is a simple, finitary label which is clearly non-null.

(b) Partition $\N$ into two disjoint infinite sets $L, B$ and define a bijection  $A \mapsto \ell_A $ from the
set of finite subsets of $L$ onto $B$. Define $\M = \langle \{ \ \chi(\ell_A) + \chi(i) \ : \ i \in A, A $ a finite
subset of $L \ \} \rangle$.  Because it is size bounded, the label $\M$ is of finite type. Since
$\M - \chi(\ell_A) \ = \ \langle \{ \chi(i) : i \in A \} \rangle $ it follows that
$\{ \chi(i) : i \in L \}$ is an independent set for $\M$. \end{ex}

$\Box$ \vspace{.5cm}

\begin{remark}\label{re,LEnotHAE}
 By Proposition \ref{towprop20d} the label $\M$ of example (b) has $\Theta(\M)$ uncountable.
So this and Example \ref{ex10cmoved} are labels $\M$ of finite type with $\Theta(\M)$ uncountable. It follows
that
the corresponding subshifts
$(X(\M),S)$  are LE but not HAE (see Remark \ref{cp,ae}). \end{remark}

$\Box$ \vspace{1cm}

\subsection{Gamow transformations}\label{Gamow}
$\qquad$

\vspace{0.5cm}

For $L \subset \N$ we let $FIN(L) = \{ \ \mm \in FIN(\N) \ : \ supp \ \mm \subset L \ \}$ and
 $\LAB(L) = \{ \ \M \in \LAB \ : \ \bigcup \ Supp \ \M \ \subset L \ \}$. Clearly,
$\M \in \LAB(L)$ implies $[[\M]] \subset \LAB(L)$. If $\M \not\in \LAB(L)$ then for some $N \in \N$
$\M \cap \B_N \not\in \LAB(L)$ and so $d(\M, \M_1)
<  2^{-N}$ implies $\M_1 \not\in \LAB(L)$. Thus,
$\LAB(L)$ is a closed subset of $\LAB$. For example, $\LAB(\emptyset) = \{ 0, \emptyset \}$.

$FIN(L)$ is a submonoid of $FIN(\N)$ and it acts on $\LAB(L)$.  Furthermore, if $\rr \not\in FIN(L)$ then
$P_{\rr}(\M) = \emptyset$ for all $\M \in  \LAB(L)$. Hence, we can restrict attention to this action and
for $\Phi$ any closed, invariant subset of $\LAB(L)$, the enveloping semigroup $\E(\Phi)$ is the closure
of $FIN(L)$ in $\Phi^{\Phi}$.

Let $\tau : L_1 \to L_2$ be a bijection with $L_1, L_2 \subset \N$.  In honor of the book \emph{One, Two, Three,\dots
Infinity} we will refer to the following as the \emph{Gamow transformation} induced by $\tau$.  For an $\N$-vector
$\mm$ with $supp \ \mm \subset L_2$ we let
$\tau^* \mm = \mm \circ \tau$ so that $supp \ \tau^* \mm  =
\tau^{-1}( supp \ \mm) \subset L_1$.
Thus, $\tau^* :
FIN(L_2) \to FIN(L_1)$ is a monoid isomorphism which also preserves the lattice properties.

For $\M \in \LAB(L_2)$  we
let $\tau^* \M = \{ \ \tau^* \mm \ : \mm \in \M \}$ and for $\Phi \subset \LAB(L_2)$
 we will let $\tau^* \Phi = \{ \tau^* \NN : \NN \in \Phi \}$. Thus, $\tau^*$ is a bijection from $\LAB(L_2)$
 to $\LAB(L_1)$ with inverse $(\tau^{-1})^*$.

Given $\ell \in \N$, let $\ell' = \max
\ \tau([1,\ell] \ \cap \  L_1)$. It follows from the definition (\ref{label06})
of the metric on $\LAB$ that $d(\M_1,\M_2) \leq 2^{-\ell'}$ implies $d(\tau^*(\M_1),\tau^*(\M_2)) \leq 2^{-\ell}$.
   Thus, the  $\tau^*$ is  uniformly
continuous on $\LAB(L_2) $ and so is a homeomorphism from $\LAB(L_2)$ onto $\LAB(L_1)$.

 Clearly, $\tau^*$ preserves all label operations, e.g. $\tau^*( \M - \rr) = \tau^* \M - \tau^* \rr$ for
 $\M \in \LAB(L_2)$ and $\rr \in FIN(L_2)$. Thus, $\tau^*$ is an action isomorphism relating the
 $FIN(L_2)$ action on $\LAB(L_2)$ to the $FIN(L_1)$ action on $\LAB(L_1)$. Hence, it induces an
 Ellis
 semigroup isomorphism from
 $\E(\Phi)$ to $\E(\tau^*(\Phi))$
 where $\Phi$ is a compact, invariant subset of
 $\LAB(L_2)$.
Also  $\Theta(\tau^* \M)  = \tau^* \Theta(\M)$ for $\M \in \LAB(L_2)$.

Thus, all of the label properties are preserved by $\tau^*$.  For example,
 $\tau^* \M$ is bounded, of finite type, finitary, simple, semi-simple, WAP,
 recurrent or strongly recurrent
iff $\M$ satisfies
the corresponding property. A label $\F$ is an external element for $\E(\Theta(\M))$ iff
$\tau^* \F$ is an external element for $\E(\Theta(\tau^* \M)$.

On the other hand, the subshifts $(X(\tau^*{\M}),S)$ and $(X(\M),S)$  are only analogous. That is, they have similar
properties but are not usually isomorphic.  For the finite type case, see Corollary \ref{towcor17a}.

Let $S_{\infty}$ denote the group of all permutations on $\N$.
On $S_{\infty}$ we define an ultrametric by
\begin{equation}\label{perm1}
d(\t_1, \t_2) \  = \  \inf \ \{ \ 2^{-\ell} \ : \ \ell \in
\Z_+
\  \mbox{and} \  \t_1| [1,\ell] = \t_2|[1,\ell] \ \}.
\end{equation}

Clearly, for any $\g \in S_{\infty}, \ d(\g \circ \t_1, \g \circ \t_2) \  = \ d(\t_1, \t_2)  $. If $\g([1,\ell]) \subset
[1,\ell_{\g}]$ then $\t_1| [1,\ell_{\g}] = \t_2|[1,\ell_{\g}] $ implies $\t_1 \circ \g | [1,\ell] = \t_2 \circ \g |[1,\ell] $ and
$\g_1| [1,\ell_{\g}] = \g|[1,\ell_{\g}]$ then $\g_1^{-1}| [1,\ell] = \g^{-1}|[1,\ell]$.  It follows that $S_{\infty}$ is a
topological group with left invariant ultrametric $d$. Furthermore, the equivalent metric $\bar d$ given by
$\bar d(\t_1,\t_2) = \max(d(\t_1, \t_2),d(\t_1^{-1}, \t_2^{-1}))$  is complete.
Finally, the set of permutations $S_{fin}$ consisting
of permutations are the identity on the complement of a finite set, is a countable dense subgroup
of $S_{\infty}$. Thus, $S_{\infty}$ is
a Polish group, which is clearly perfect.

Furthermore, if $\M$ and $\M_1$ are labels with $\M \cap \B_{\ell_{\g}} = \M_1 \cap \B_{\ell_{\g}}$ and
$\g_1| [1,\ell_{\g}] = \g|[1,\ell_{\g}]$ then $\g^*\M \cap \NN_{\ell} = \g_1^*\M_1 \cap \NN_{\ell}$.  This implies
that the action $S_{\infty} \times \LAB \to \LAB$ given by $(\t,\M) \to (\t^{-1})^*\M$ is a continuous action.
The empty label $\emptyset$ is an isolated fixed point for the action.  Let $\LAB_+$ denote the perfect set of
nonempty labels.
We show that this action is topologically transitive on $\LAB_+$ by constructing explicitly a transitive point.

\begin{ex}\label{permex} Let $\Xi$ be the countable set of all pairs $(\NN_{\xi},\ell_{\xi})$ with $\NN_{\xi}$ a finite label
such that $\bigcup \ Supp \  \NN_{\xi} \subset [1, \ell_{\xi}]$. Partition $\N$ by disjoint intervals indexed by $\Xi$
such that $I_{\xi}$ has length $\ell_{\xi}$. Let $\t_{\xi} : I_{\xi} \to
[1,\ell_{\xi}]$
be the increasing linear bijection
and let $\M_{\xi} = \t_{\xi}^*\NN_{\xi}$ so that $\bigcup \ Supp \  \M_{\xi} \subset I_{\xi}$.
Let $\M_{trans} = \bigcup_{\xi} \ \M_{\xi}$.
By Corollary \ref{towcor26c} $\M_{trans}$ is finitary and simple and so is of finite type. On the other hand, given
any nonempty label $\M$ and any $N \in \N$ there exists $\xi \in \Xi$ such that
$(\NN_{\xi},\ell_{\xi}) = (\M \cap B_{N}, N)$.
It follows that if $\g \in S_{fin}$ with $\g = \t_{\xi}$ on $I_{\xi}$ then
$(\g^{-1})^*\M_{trans} \cap B_{N}= \M \cap B_{N}$. Thus, $\M_{trans}$ is a transitive point for the
action of $S_{\infty}$ on $\LAB_+$. \end{ex}
\vspace{.5cm}

Because $\LAB_+$ is a Cantor set  and $S_{\infty}$ is a group, the set $TRANS$ of $S_{\infty}$
transitive points is a dense $G_{\delta}$ subset of
$\LAB_+$.  By Proposition \ref{labelprop12xb} the set $RECUR$ of recurrent labels is a dense $G_{\delta}$ subset of $\LAB$.
Hence, \\
$TRANS \ \cap \ RECUR$ is a dense $G_{\delta}$ subset of $\LAB_+$. The transitive point $\M_{trans}$ is of finite
type and so is not recurrent.  On the other hand, the set of flat labels  is a proper, closed $S_{\infty}$ invariant
subset which contains recurrent labels (see Proposition \ref{towprop20i} (a)) which are thus not transitive with respect to
the $S_{\infty}$ action.

For background regarding our next question we refer the reader to the works \cite{KR} and \cite{GW-Ro}.

\begin{Qu}\label{permq} Does there exist a label $\M$ such that its $S_{\infty}$ orbit is residual,
or are all the orbits meager?
If such a residual orbit exists then it would be unique
and would meet and so be contained in $TRANS \ \cap \ RECUR$.
It is well known that the adjoint action of $S_{\infty}$ on itself does have a dense $G_{\delta}$ orbit
(see e.g. \cite{GW-Ro}).
\end{Qu}

\vspace{1cm}

\subsection{Ordinal constructions}\label{ssec,OC}
$\qquad$

\vspace{0.5cm}

Assume that  $(X,T)$ is  a cascade with $X$ metrizable.
For $A \subset X$ we defined
$z_{LIM}(A) = \ol{R_T(A)}$, the closure of the set of all positive or negative limit points of orbits of points of $A$, and
 $R_T^*(A) = \{ x : R_T(x) \subset A \}$. We
 defined in
 (\ref{z})
the LIM
descending transfinite sequence of closed sets by
\begin{equation}\label{zagain}
z_0(X) \ = \ X, \quad z_{\a + 1}(X) \ = \ z(z_{\a}(X)), \quad z_{\b} \ = \ \bigcap_{\a < \b} z_{\a}(X),
\end{equation}
for $\b$ a limit ordinal. This sequence stabilizes at the \emph{height}  being the first ordinal $\b$ such that
$z_{\b}$ is the Birkhoff center, i.e. the closure of the set of recurrent points.

Assume that  $(X,T)$ is a cascade with
$Y$ an orbit-closed subset of $X$.
We
defined in (\ref{zstar})
the increasing
transfinite sequence of orbit-closed sets, by
\begin{equation}\label{aagain}
z^*_0(Y) = Y, \quad z^*_{\a + 1}(Y) = R_T^*(z^*_{\a}(Y)), \quad z^*_{\b}(Y) = \bigcup_{\a < \b} z^*_{\a}(Y)
\end{equation}
for $\b$ a limit ordinal. This sequence stabilizes at the minimum L-determined set in $X$ which contains $Y$. If $(X,T)$ is
a metrizable,
CT-WAP
system with the fixed point
an
isolated
invariant set,
then $X$ is the only L-determined set by Corollary \ref{isocor4} and the \emph{height$^*$} is
the first ordinal $\b$ such that with $Y = \{ e \}$
 $z^*_{\b}(e) = X$. If the system is topologically transitive with a transitive point $x^*$ then
the height$^*$ is the first ordinal $\b$  such that $x^* \in z^*_{\b}(e)$.

Now we define the label versions of these constructions.

 For a label $\M$  define the label
 \begin{equation}\label{tow29aa}
 z_{LAB}(\M) =  \M \setminus max \ \M = \{ \ \mm \ :  \ \mm + \rr \in \M \ \text{for some} \ \rr > \00 \ \}.
 \end{equation}
 Thus $\mm \in  z_{LAB}(\M) $ iff there exists $\mm_1 \in \M$ with $\mm < \mm_1$.

 \begin{lem}\label{towem22aa} Let $\M, \M_1$
 be labels.
 \begin{itemize}

 \item[(a)] $z_{LAB}(\M) \subset \M$ and the inclusion is proper if $\M$ is
 nonempty and of finite type.

  \item[(b)] $\M_1 \subset \M$ implies $z_{LAB}(\M_1) \subset z_{LAB}(\M)$.

   \item[(c)] If $\{ \M_i \}$ is a collection of labels then $\bigcup_i \ z_{LAB}(\M_i) \ = \ z_{LAB}(\bigcup_i \ \M_i)$.

   \item[(d)]  $z_{LAB}(\M) = \emptyset$ iff $\M = 0 $ or $\emptyset$.
   \end{itemize}
   \end{lem}

   \proof (a) and (b) are obvious and so $\bigcup_i \ z_{LAB}(\M_i) \subset z_{LAB}(\bigcup_i \ \M_i)$.
   If $\mm \in z_{LAB}(\bigcup_i \ \M_i)$ then there exists $\mm_1 > \mm$ with $\mm_1 \in \bigcup_i \ \M_i$.
   So $\mm_1 \in \M_i$ for some $i$ and this implies $\mm \in z_{LAB}(\M_i)$, proving (c).

   (d): If $\M$ is positive then $\00 \in z_{LAB}(\M)$. On the other hand, if  $\M = 0 $ or $\emptyset$ then
    $z_{LAB}(\M) = \emptyset$.

 $\Box$ \vspace{.5cm}

Define
a descending transfinite sequence of labels by
\begin{equation}\label{tow29}
\begin{split}
z_{LAB,0}(\M) \quad = \quad \M, \hspace{4cm} \\
z_{LAB,\a + 1}(\M) \quad = \quad z_{LAB}(z_{LAB,\a}(\M)), \hspace{3cm}\\
z_{LAB, \b}(\M) \quad = \quad
\bigcap_{\a < \b}
 \ \{ z_{LAB,\a}(X) \} \quad \mbox{ for } \ \b \ \mbox{ a limit ordinal}.
\end{split}
\end{equation}
 The sequence stabilizes at $\b$ when $z_{LAB,\b}(\M) = z_{LAB,\b + 1}(\M)$ in which case
$z_{LAB,\a}(\M) = z_{LAB,\b}(\M)$ for all $\a \geq \b$.  So $\emptyset$ stabilizes at $0$ and
if $\M$ is nonempty and of finite type then the sequence stabilizes at $\b + 1$ where
$\b$ is the first ordinal for which $z_{LAB,\b} = 0$.

 If  $\Phi$ is a closed, bounded, invariant set of labels  define $z_{LAB}(\Phi)$
to be the closure of $\bigcup \ \{ \ P_{\rr}(\Phi) \ : \ \rr > \00 \}$, a closed, invariant subset of $\Phi$. Thus,
$z_{LAB}(\Theta(\M)) = \Theta'(\M)$.

Define the nonincreasing transfinite sequence of closed, bounded subsets of $\LAB$ by
\begin{equation}\label{tow30}
\begin{split}
z_{LAB,0}(\Phi) \quad = \quad \Phi, \hspace{4cm} \\
z_{LAB,\a + 1}(\Phi) \quad = \quad z_{LAB}(z_{LAB,\a}(\Phi)), \hspace{3cm}\\
z_{LAB, \b}(\Phi) \quad =
\quad \bigcap_{\a < \b} \ \{ z_{LAB,\a}(\Phi) \} \quad \mbox{ for } \ \b \ \mbox{ a limit ordinal}.
\end{split}
\end{equation}

\begin{theo}\label{towtheo23} Let $\Phi$ be a closed, $FIN(\N)$ invariant subset of $\LAB$ and $\M = \bigcup \Phi$, e.g.
$\Phi = \Theta(\NN)$ for any label $\NN$.
Then
$z_{LAB,\a}(\M) = \bigcup z_{LAB,\a}(\Phi)$ for every countable ordinal
$\a$.  That is, $\mm \in z_{LAB,\a}(\M)$
iff there exists $\NN \in z_{LAB,\a}(\Phi)$ such that $\mm \in \NN$. \end{theo}

\proof  We use transfinite induction. Both procedures stabilize at a countable ordinal and so
we need only consider countable ordinals.

Since $\M = \bigcup \Phi$ the result is true for $\a = 0$. Notice that for any label $\M = \bigcup \Theta(\M)$.

If $\mm \in z_{LAB,\a + 1}(\M) = z_{LAB}(z_{LAB, \a}(\M))$, then there exists
$\rr > \00 $ such that $\mm + \rr \in z_{LAB, \a}(\M)$.
By induction hypothesis, there exists $\NN \in z_{LAB, \a}(\Phi)$ such that $\mm + \rr \in \NN$ and so
$\mm \in \NN - \rr \in  z_{LAB,\a + 1}(\Phi)$.

Conversely, if $\NN \in z_{LAB,\a + 1}(\Theta(\M)) = z_{LAB}(z_{LAB, \a}(\Phi))$ then
there exists  sequences $\NN^i \in z_{LAB,\a}(\Phi)$
and $\rr^i > 0$ such that $\NN = LIM \{ \NN^i - \rr^i \}$. By induction hypothesis $\bigcup \{ \NN^i \} \subset z_{LAB,\a}(\M)$.
If $\mm \in \NN$ then eventually $\mm \in \NN^i - \rr^i$ and so $\mm \in z_{LAB,\a + 1}(\M)$.

Now let $\b$ be a limit ordinal. If $\mm \in \NN \in z_{LAB,\b}(\Phi)$ then by definition $\NN \in z_{LAB,\a}(\Phi)$ for all
$\a < \b$. So by induction hypothesis, $\mm \in z_{LAB, \a}(\M)$ for all
$\a < \b$ and hence $\mm \in z_{LAB, \b}(\M)$.

Conversely, if $\mm \in z_{LAB, \b}(\M) $ and so in $z_{LAB, \a}(\M)$ for all
$\a <  \b$ we let $\{ \a^i \}$ be an increasing
sequence of ordinals converging to $\b$. By induction hypothesis there exists $\NN^i \in z_{LAB, \a^i}(\Phi)$ such that
$\mm \in \NN^i$. Since $\{ \a^i \}$ is increasing $\NN^i \in z_{LAB, \a^j}(\M)$ for all
$j <  i$.  Let
$\{ \NN^{i'} \}$ be a convergent subsequence with limit $\NN$. Since $\mm \in \NN^i$ for all $i$, $\mm \in \NN$.
For every
$\a < \b$ there exists
$\a^j > \a$. For $i \geq j$ the sequence
 $ \NN^i \in z_{LAB, \a^j}(\Phi) \subset z_{LAB, \a}(\Phi)$ and
 so the
limit of the
subsequence
$\NN $ is in the closed set $z_{LAB, \a}(\Phi)$.
 Since this is true for all
 $\a < \b$, $\NN \in z_{LAB, \b}(\Phi)$.

 $\Box$ \vspace{.5cm}

  It follows that the sequences stabilize at the same countable ordinal.  If $\M = \emptyset$ then $\Theta(\M) = \{ \emptyset \}$
  and the sequence stabilizes at $0$. If $\M$ is nonempty and of finite type then the sequence stabilizes at $\b + 1$ where
  $\b$ is the ordinal with $z_{LAB, \b}(\M) = 0$ and
  $z_{LAB, \b}(\Theta(\M)) = \{ 0, \emptyset \}$. In this nonempty finite type case, we call $\b+1$ the \emph{height} of
  $\Theta(\M)$.

\begin{cor}\label{towcor23a} Assume that $\M$ is a label of finite type and that $\Phi$ is a nonempty, closed, $FIN(\N)$ subset of
$[[\M]]$. If $\Phi = z_{LAB}(\Phi)$ then $\Phi = \{ \emptyset \}$.\end{cor}

\proof If $\Phi = z_{LAB}(\Phi)$ then the transfinite sequence $ \{ z_{LAB,\a}(\Phi) \}$
stabilizes at $\a = 0$.
If $\NN  = \bigcup \Phi$ then since $\NN \subset \M$ it is of finite type and so the sequence $ \{ z_{LAB\a}(\NN) \}$
stabilizes at $\a = 0$ only when $\NN = \emptyset$. But by Theorem
\ref{towtheo23}
these
two sequences stabilize at
the same level.

 $\Box$ \vspace{.5cm}

For $Y$ a  subset of $\{ 0,1 \}^{\Z}$  Definition \ref{df,Phi} says
 $$  \Phi(Y) = \{ \NN : x[\NN] \in Y \} \quad {\text{and}} \quad
 \Phi_+(Y) = \{ \NN : x_+[\NN] \in Y \}. $$

\begin{theo}\label{towtheo24} Assume that $\M$ is a  label of finite type.
For every countable ordinal $\a$, every closed, invariant $Y \subset X(\M)$
and every closed, invariant $Y_+ \subset X_+(\M)$ ,
 \begin{equation}\label{tow31}
 \begin{split}
 \Phi(z_{LIM,\a}(Y)) \quad = \quad z_{LAB,\a}(\Phi(Y)), \hspace{2cm}\\
  \Phi_+(z_{LIM,\a}(Y_+)) \quad = \quad z_{LAB,\a}(\Phi_+(Y_+)). \hspace{2cm}
\end{split}
 \end{equation}
\end{theo}

\proof  The equation is clear for $\a = 0$.

Since $\Phi(Y)$ is the preimage of $Y$ with respect to the continuous map $x[\cdot]$ it follows that
$\Phi(z_{LIM}(Y))$ is a closed invariant set containing $P_{\rr}\NN$ whenever $x[\NN] \in Y$ and $\rr > \00 $ by Theorem
\ref{towtheo14}(a) and the Remark after Theorem \ref{towtheo17ab} which implies that $p_{\rr} =  q(P_{\rr}) \in A(Y,S)$.
Hence, $z_{LAB}(\Phi(Y)) \subset \Phi(z_{LIM}(Y))$.

On the other hand, Theorem \ref{towtheo15} (b) implies that $z_{LAB}(\Phi(Y)) = \Phi(\tilde Y)$ for a closed, invariant
subspace $\tilde Y$ of $X(\M)$. If $\rr > 0$ and $y \in Y$ then $y = S^k(x[\NN])$ with $k \in \Z$ and $\NN \in \Phi(Y)$.
So $p_{\rr}y = S^k(x[P_{\rr}\NN])$.  Since $P_{\rr}\NN \in z_{LAB}(\Phi(Y))$, it follows that $p_{\rr}y \in \tilde Y$.
Hence, $z_{LIM}(Y) \subset \tilde Y$ and so $\Phi(z_{LIM}(Y)) \subset \Phi(\tilde Y) =  z_{LAB}(\Phi(Y))$.

This proves equation (\ref{tow31}) with $\a = 1$ and so assuming the result for an ordinal $\a$ it follows for $\a + 1$.

For a limit ordinal, $\b$ we use
the fact
that
$\Phi( \cdot )$ commutes with intersection and so by
the induction hypothesis
\begin{equation}\label{tow32}
\begin{split}
\Phi(z_{LIM,\b}(Y)) \ = \ \Phi(\bigcap_{\a < \b}  \ z_{LIM,\a}(Y)) \ = \ \bigcap_{\a < \b} \Phi(z_{LIM,\a}(Y)) \hspace{1cm} \\
 = \ \bigcap_{\a < \b} z_{LAB,\a}(\Phi(Y)) \ = \  z_{LAB,\b}(\Phi(Y)). \hspace{2cm}
 \end{split}
 \end{equation}

 This completes the induction.

 $\Box$ \vspace{.5cm}

 We will say that $\Phi \subset \LAB$ is \emph{$\Theta$ invariant} when it is nonempty orbit-closed with respect
 to the $FIN(\N)$ action.  That is, $\M \in \Phi$ implies
 $\Theta(\M) \subset \Phi$. If $\Phi$ is closed then it is invariant iff it is $\Theta$ invariant.  $\Theta$ invariance
 always implies invariance but is usually a stronger condition since $\{ P_{\rr}(\M) \}$ is usually a proper subset of
 $\Theta(\M)$.

 Let $\M$ be a label of finite type. For a
  $\Theta$ invariant  $\Phi \subset \Theta(\M)$, we define
 $z^*_{\M}(\Phi) = \{ \ \NN \in \Theta(\M) \ : \ \Theta'(\NN) \subset \Phi \}$. Equivalently, $\NN \in z^*_{\M}(\Phi)$ iff
 $Q(\NN) \in \Phi$ for all $Q \in
  \A(\Theta(\M)) = \E(\Theta(\M)) \setminus
 \{ id_{\Theta(\M)} \}$.
 For example, $z^*_{\M}(\{ \emptyset \}) =
 \{ \emptyset, 0 \} = [[0]] = \Theta(0)$.

 Starting with a
  $\Theta$ invariant  $\Phi \subset \Theta(\M)$,
 define the nondecreasing transfinite sequence of $\Theta$ invariant subsets of $\Theta(\M)$ by
\begin{equation}\label{tow33}
\begin{split}
z^*_{\M,0}(\Phi) \quad = \quad \Phi, \hspace{4cm} \\
z^*_{\M,\a + 1}(\Phi) \quad = \quad z^*_{\M}(z^*_{\M,\a}(\Phi)), \hspace{3cm}\\
z^*_{\M, \b}(\Phi) \quad = \quad
\bigcup_{\a < \b} \ \{z^*_{\M,\a}(\Phi) \} \quad \mbox{ for } \ \b \ \mbox{ a limit ordinal}.
\end{split}
\end{equation}

Recall that for a dynamical system $(X,T)$, $Y \subset X$ is called orbit-closed when $x \in Y$ implies
$\ol{O_T(x)} \subset Y$. For $\M$ of finite
type it is easy to
adjust the proof of
Corollary \ref{towcor15a} to show that
$Y \subset X(\M)$ is orbit-closed iff $\Phi(Y)$ is
$\Theta$ closed and $Y_+ \subset X_+(\M)$ is orbit-closed iff $\Phi_+(Y_+)$ is
$\Theta$ closed.

\begin{theo}\label{towtheo25} Assume that $\M$ is a  label of finite type, that $Y$ is an orbit closed subset of $X(\M)$
and that $Y_+$ is an orbit closed subset of $X_+(\M)$.
For every countable ordinal $\a$,
 \begin{equation}\label{tow34}
 \begin{split}
 \Phi(z^*_{\a}(Y)) \quad = \quad z^*_{\M,\a}(\Phi(Y)), \hspace{2cm}\\
 \Phi_+(z^*{\a}(Y_+)) \quad = \quad z^*_{\M,\a}(\Phi_+(Y_+)), \hspace{2cm}
 \end{split}
 \end{equation}
\end{theo}

\proof This is obvious for $\a = 0$.

A point $x \in R^*_S(Y)$ iff $qx \in Y$ for all $q \in A(X(\M),S)$. Since $Y$ is $S$-invariant, Theorem
\ref{towtheo17ab}, and the remark thereafter,
imply that this is true iff $\bar D(Q)x \in Y$ for all
$Q \in \A(\Theta(\M))$. Hence,  $x[\NN] \in R^*_S(Y)$ iff
$x[Q(\NN)] \in Y$ for all $Q \in \A(\Theta(\M))$, i.e. iff $Q(\NN) \in
\Phi(Y)$ for  all $Q \in \A(\Theta(\M))$ and so
iff $\NN \in z_{\M}^*(\Phi(Y))$. This proves  equation (\ref{tow34}) for $\a = 1$ and so inductively for any
$\a + 1$.

Since $\Phi$ is the preimage operator with respect to the map $x[\cdot]$, it commutes with union. So the equation
for a limit ordinal $\b$ follows because it is assumed, inductively, to hold for all
$\a < \b$.

The proof for $Y_+$ is completely similar.

$\Box$ \vspace{.5cm}

The constructions of (\ref{tow29}),   (\ref{tow30}) and (\ref{tow33}) are label constructions and so they commute
with Gamow
transformations.  We can use Gamow transformations
to assure that a countable number of labels all occur with supports on
disjoint sets.  Let $\tau_0 : \N \to \N \times \N$ be a bijection and let $L_i = \tau_0^{-1}(\N \times \{i \})$ for $i \in \N$.
Define $\tau_i : L_i \to \N$ to be the bijection $\tau_i = \pi_1 \circ \tau_0$ where $\pi_1 : \N \times \N \to \N$ is the
first coordinate projection. Given a sequence $\{ \M^i \}$ of labels, the  label
$\{ \tau_i^* \M^i \}$  is Gamow equivalent to $\M^i$ and $\bigcup \ Supp \ \M^i \ \subset \ L_i $.

For the following two propositions we apply Theorems \ref{towtheo27a} and \ref{towtheo26} together with their proofs.
Note that  $[[0]]=  \{ \emptyset, 0 \} $.


\begin{prop}\label{towprop27again}  Assume $\NN$ and $\M$ are positive disjoint
labels with $\NN$ finite and $\M$ of finite type.

\begin{itemize}
\item[(a)]If $\Psi \subset [[\NN]]$ and $\Phi \subset [[\M]]$ are closed and invariant then
 \begin{equation}\label{tow40}
 z_{LAB}(\Psi \oplus \Phi) \ = \ z_{LAB}(\Psi) \oplus \Phi \cup \Psi \oplus  z_{LAB}(\Phi).
 \end{equation}
And for every limit ordinal $\a$
 \begin{equation}\label{tow41}
 z_{LAB,\a}(\Theta(\NN \oplus \M)) \ = \ \{ \NN \} \oplus z_{LAB,\a}(\Theta(\M)). \hspace{1.5cm}
 \end{equation}
 For $\a = 0$ or a limit ordinal and $k \in \N$
 \begin{equation}\label{tow42}
 z_{LAB,\a + k}(\Theta(\M_1 \oplus \M_2)) \ =
 \ \bigcup_{r = 0}^k \ z_{LAB,\a + r}(\Theta(\NN))  \oplus  z_{LAB,\a + k - r}(\Theta(\M)).
 \end{equation}

\item[(b)]
 For all $k \in \N$
  \begin{equation}\label{tow43}
  z^*_{\NN \oplus \M, k}([[0]]) \ = \ \bigcup_{r = 0}^k \ z^*_{\NN,r}([[0]])
  \oplus  z^*_{\M,k - r}([[0]]).
 \end{equation}
 For every limit ordinal $\a$ and $k \in \Z_+$
  \begin{equation}\label{tow44}
  z^*_{\NN \oplus \M,\a + k}([[0]]) \ = \ \Theta(\NN) \oplus  z^*_{\M,\a + r}([[0]])). \\
\end{equation}
\end{itemize}
\end{prop}

\proof If $\rr > 0$ with $\rr \in \M_1 \oplus \M_2$ then $\rr = \rr_1 + \rr_2$ with either $\rr_1 > 0$ or $\rr_2 > 0$.

It therefore follows that for $\Phi_1 \subset \Theta(\M_1)$ and $\Phi_2 \subset \Theta(\M_2)$ closed invariant subspaces,
$$z_{LAB}(\Phi_1 \oplus \Phi_2) =  z_{LAB}(\Psi) \oplus \Phi \cup \Psi \oplus  z_{LAB}(\Phi).$$
Then (\ref{tow42}) follows by induction on $k$ since the operator $z_{LAB}$ commutes with union. For (\ref{tow41}) we use
induction on the limit ordinals together with $0$ starting with $\a = 0$. Assume the result for all
$\b <  \a$.
Since $\NN$ is positive and finite, there exists a unique $r_0 \in \N$ such that
where $z_{LAB, r_0}(\Theta(\NN)) = [[ 0 ]]$. From (\ref{tow42}) we have
\begin{gather}\label{tow45}
\begin{split}
 \{ \NN \}\oplus  z_{LAB,\b + k}(\Theta(\M))  \ \subset \ z_{LAB,\b + k}(\Theta(\NN \oplus \M)) \ \\
 \subset \ \{ \NN \}\oplus  z_{LAB,\b + k - r_0 }(\Theta(\M)) \qquad \qquad \ \
 \end{split}
\end{gather}
 Intersecting we obtain  (\ref{tow41}) for $\a = \b + \o$. Otherwise, $\a$ is
an increasing limit of limit ordinals and the result follow from the induction hypothesis by intersecting.

For (\ref{tow43}) assume that $\Phi \subset \Theta(\NN)$ and $\Psi \subset \Theta(\M)$ are $\Theta$ invariant.
By Theorem \ref{towtheo27a} the labels $\Theta(\NN \oplus \M)$ are of
 the form $\NN_1 \oplus \M_1$ with $\NN_1 \in \Theta(\NN)$ and $\M_1 \in \Theta(\M)$, and the maps in
 $\E(\Theta(\NN \oplus \M))$ are of the form $P \widehat \oplus Q$ with $P \in \E(\Theta(\NN), Q \in \E(\Theta(\M)$.
So $P \widehat \oplus Q \in \A(\Theta(\NN \oplus \M))$ if either $P \not= id$ or $Q \not= id$.

We show that
for $k \in \Z_+$,
  \begin{equation}\label{tow43rev2}
  z^*_{\NN \oplus \M, k}(\Phi \oplus \Psi) \ = \ \bigcup_{r = 0}^k \ z^*_{\NN,r}(\Phi)
  \oplus  z^*_{\M,k - r}(\Psi).
 \end{equation}
 This is obvious for $k = 0$ and the inclusion $\subset$ is clear for $k = 1$ and then follows for all $k \in \N$  by induction.

 For the reverse inclusion, proceed by induction, with $k \geq 1$. Suppose that
 $\NN_1 \oplus \M_1 \in \Theta(\NN \oplus \M)$ and $r_1, s_1 \geq 0$ are the smallest integers such that
 $\NN_1 \in z^*_{\NN,r_1}(\Phi)$ and $\M_1 \in z^*_{\M,s_1}(\Psi)$ with $r_1 + s_1 > k$. We use the
 inductive hypothesis to show that $\NN_1 \oplus \M_1 \not\in  z^*_{\NN \oplus \M, k}(\Phi \oplus \Psi)$. Either $r_1$ or $s_1$ is
 at least $1$. Suppose $r_1 \geq 1$. There exists $P \in \A(\Theta(\NN))$ such that $r_2 = r_1-1$ is the smallest
 value such that $P(\NN_1) \in z^*_{\NN,r_2}(\Phi)$. Since $r_2 + s_1 > k - 1$ it follows from the induction
 hypothesis that $(P \widehat \oplus id)(\NN_1 \oplus \M_1) \not\in  z^*_{\NN \oplus \M, k-1}(\Phi \oplus \Psi)$ and so
 $\NN_1 \oplus \M_1 \not\in  z^*_{\NN \oplus \M, k}(\Phi \oplus \Psi)$.

 With $\Phi = \Psi = [[0]]$ this implies (\ref{tow43}). With $\Phi = \Theta(\NN)$ we obtain
for $k \in \Z_+$,
  \begin{equation}\label{tow43rev3}
  z^*_{\NN \oplus \M, k}(\Theta(\NN) \oplus \Psi) \ = \ \Theta(\NN)   \oplus  z^*_{\M,k }(\Psi).
 \end{equation}

 Notice that these results did not require that $\NN$ be finite.

Now we assume that $\NN$ is finite and for (\ref{tow44}) we proceed by transfinite induction. Since $\NN$ is positive and finite
there exists $r_0^*$ such that  $z^*_{\NN,r^*_0}([[0]]) = \Theta(\NN)$. That is, the $z^*$ sequence
stabilizes at $r^*_0$.

Taking the union in (\ref{tow43}) over $k \in \Z$ we obtain (\ref{tow44}) with $\a = \o$. Assuming the result for
a limit ordinal $\a$ and $k = 0$ it then follows for $k \in \N$ by (\ref{tow43rev3}) with $\Psi = z^*_{\M, \a}([[0]])$.
Again taking the
union over $k \in \N$ we obtain the result for $\a + \o$. Finally, if the ordinal $\bar \a$ is the limit
of an increasing sequence of limit ordinals $\b^i$ we take the union of the (\ref{tow44}) with $\a = \b^i$
and $k = 0$ to obtain the result with $\a = \bar \a$.

 $\Box$ \vspace{.5cm}

 \begin{prop}\label{towprop26} Assume that  $\{ \M_a \} $ is a finite or infinite pairwise disjoint
 collection of at least two
nonempty labels of finite  type
and let $\M  = \bigcup \{ \M_a \}$.
 \begin{enumerate}

 \item[(a)]
 For every $\a \geq 1$
 \begin{equation}\label{tow37}
 z_{LAB,\a}(\Theta(\M)) \ = \ \bigcup \ \{ z_{LAB,\a}(\Theta(\M_a)) \}. \hspace{1.5cm}
 \end{equation}

 \item[(b)]
 If for all $a, \ \ \Phi_a$ is a $\Theta$ invariant subspace of
  $\Theta'(\M_a)$ then $\bigcup \ \{ \Phi_a \}$ is a
 $\Theta$ invariant subset of $\Theta'(\M)$. If $\bigcup \ \{ \Phi_a \}$ is a proper subset
 of  $\Theta'(\M)$ 
 then
 \begin{equation}\label{tow38}
 z^*_{\M}(\bigcup \ \{ \Phi_a \}) \ = \ (\bigcup \ \{ z^*_{\M_a}(\Phi_a) \}) \setminus \{ \M_a \}.
 \end{equation}
 \end{enumerate}
 \end{prop}

 \proof
 If $\Phi$ is a closed subset of
 $\Theta(\M)$ then $z_{LAB}(\Phi)$ is the closure of
 $\A(\Theta(\M))\Phi$.  Hence, (\ref{tow37}) follows
 by transfinite induction starting with  \allowbreak $z_{LAB}(\Theta(\M)) = \Theta'(\M)$.

The $\Theta$ invariance of $\bigcup \ \{ \Phi_a \}$ and (\ref{tow38}) follow from (\ref{tow35}) and
 (\ref{tow35a}). 
 
 Notice that, in general, for a label $\NN$ and $\Psi$ a $\Theta$ invariant subset of $\Theta'(\NN)$
 we have $\NN \in z^*_{\NN}(\Psi)$ if and only iff $\Psi = \Theta'(\NN)$.  In particular,
 $\M_a \in z^*_{\M_a}(\Theta'(\M_a))$ but $\M_a \not\in \Theta(\M)$.  The proper subset condition
 is needed because $\M \in z^*_{\M}(\Theta'(\M))$ while $\M \not\in  \bigcup \ z^*_{\M_a}(\Theta'(\M_a))$.

 $\Box$ \vspace{.5cm}

\begin{df}\label{df,heightL}
For $\M$ a nonzero
WAP label, the
{\em height} of $\Theta(\M)$ is $\a + 1$ where $\a$ is the ordinal with
$z_{LAB,\a}(\Theta(\M)) = [[ 0 ]]$. The
{\em height$^*$} is $\a + 1$ where $\a$ is the ordinal where
$z^*_{\M,\a}(\{ \emptyset \}) = \Theta'(\M)$. Notice that
$z^*_{\M,1 + \a}(\{ \emptyset \}) = z^*_{\M,\a}([[ 0 ]])$ and
so if $\a \geq \o$ 
then $1 + \a = \a$, 
hence 
$z^*_{\M, \a}(\{ \emptyset \}) = z^*_{\M,1 + \a}(\{ \emptyset \}) = z^*_{\M,\a}([[ 0 ]])$.

\end{df}

\vspace{.5cm}

\begin{theo}\label{towtheo28} For any countable limit ordinal $\a$ there exists a  label $\M$ which is both simple and finitary
with height = height$^* \  = \a + 1$. For any countable ordinal $\a$ there exists a  label $\M$ which is both simple and finitary
with height $ = \a + 1$.
Hence, $(X(\M),S)$ and $(X_+(\M),S)$ are  topologically transitive WAP subshifts with height
$ = \a + 1$
and, if $\a$ is a limit ordinal, with height$^* \  = \a + 1$. \end{theo}

\proof First, let $\NN_n = \{ k \chi(\ell_1) : 0 \leq k \leq n \}$. It is easy to see
that $\Theta(\NN_n) = \{ \NN_k : k \leq n \} \cup \{ \emptyset \}$ has height and height$^*$ equal to $n + 1$. These are finite
labels and so are both simple and finitary.

Now suppose that $\a$ is a countable limit ordinal, the limit of an increasing sequence $\b^i$. By inductive hypothesis, we can
choose for each $i$ a finitary and simple label
$\M^i$ so that $\Theta(M_i)$ has height $\b^i + 1$, and height$^*$ $\b^i + 1$.  By using a Gamow transformation we can assume that
$\{ \M^i \}$ is a sequence with disjoint supports. By Proposition \ref{towprop26} $\M = \bigcup \{ \M^i \}$ is finitary and simple
and
by (\ref{tow37}) $\z_{LAB,\a}(\Theta(\M)) = \{ [[ 0 ]] \}$ and so $\Theta(\M)$ has height $\a + 1$.  By (\ref{tow38})
it follows that $\z^*_{\M, \a}(\{ \emptyset \}) = \z^*_{\M, \a}([[ 0 ]]) = \Theta'(\M)$ and so $\Theta(\M)$ has
height$^*$ equal to $\a + 1$.

Now for a countable limit ordinal $\a$ assume that $\M$ is a finitary and simple label with height and height$^*$ equal to $\a + 1$.
By using a Gamow transformation, we can assume that $\ell_1$ is not in the support of $\M$.
For $n \geq 1, \ \ \NN_n$ is
a positive, finite label disjoint from $\M$. By Proposition \ref{towprop26} $\NN_n \oplus \M$ is finitary and simple
and by (\ref{tow41})
$z_{LAB,\a}(\Theta(\NN_n \oplus \M)) = \Theta(\NN_n \oplus 0)$. Hence, $\Theta(\NN_n \oplus \M)$ has height $\a + n + 1$.

The results for $(X(\M),S)$ follow from Theorem \ref{towtheo24} and Theorem \ref{towtheo25}.

$\Box$ \vspace{1cm}

\section{Scrambled sets}\label{sec,scrambled}

\vspace{.5cm}

Following Li and Yorke \cite{LY} a subset $S \subset X$ is called \emph{scrambled} for a dynamical system
 $(X,T)$ when every pair of distinct points
 of $S$ is proximal but not asymptotic.

 Recall that
 the adherence semigroup
 $A(X,T)$ is the ideal of the enveloping semigroup $E(X,T)$ consisting of the limit points of $\{ T^n \}$ as
 $|n| \to \infty$.  Let $A_+(X,T)$ be the set of limit points of $\{ T^n \}$ as
 $n \to \infty$, that is, we move only in the positive direction. Thus, $\omega T(x) = A_+(X,T)x$ for every $x \in X$.

 \begin{df}\label{scramdef01} For a metric dynamical system $(X,T)$ let $(x, y)$ be  a pair in $X \times X$.

 (a) We call the
 pair $(x,y)$ \emph{proximal} when it satisfies the following equivalent conditions:
 \begin{itemize}

 \item[(i)]
$\liminf_{n > 0} \ d(T^n(x),T^n(y)) \ = \ 0$.

 \item[(ii)] There exists a sequence $n_i \to \infty$ such that $\lim  \ d(T^{n_i}(x),T^{n_i}(y)) \ = \ 0$.

 \item[(iii)] There exists $p \in A_+(X,T)$ such that $p x \ = \ p y$.

 \item[(iv)] There exists $u$ a minimal idempotent in $A_+(X,T)$ such that $u x \ = \ u y$.
 \end{itemize}

We denote by $PROX(X,T)$ (or just $PROX$ when the system is clear) the set of all proximal pairs.

(b)  We call the
 pair $(x,y)$ \emph{asymptotic} when it satisfies the following equivalent conditions:
 \begin{itemize}

 \item[(i)]   $\limsup_{n > 0}  \ d(T^n(x),T^n(y)) \ = \ 0$.

 \item[(ii)] $\lim_{n > 0}  \ d(T^n(x),T^n(y)) \ = \ 0$.

 \item[(iii)] For all $p \in A_+(X,T) \ \ p x \ = \ p y$ .
 \end{itemize}

We denote by
$ASY\!MP(X,T)$ (or just
$ASY\!MP$ when the system is clear) the set of all asymptotic pairs.

(c)  We call the
 pair $(x,y)$ a \emph{Li-Yorke pair} when it is proximal but not asymptotic.

(d) The system $(X,T)$ is called \emph{proximal}  when all pairs are proximal, i.e. $PROX = X \times X$.
It is called \emph{completely scrambled}
when all non diagonal pairs are Li-Yorke. That is, the system is proximal, but
$ASY\!MP = \Delta_X$.
\end{df}
\vspace{.5cm}

Observe that the set $\{ p \in A_+(X,T) : p x = p y \}$ is a closed left ideal if it is nonempty and so it then contains minimal
idempotents.  This shows that $(iii) \Leftrightarrow (iv)$ in (a).  The remaining equivalences are obvious.

\begin{remark}\label{rmk,prox}
This notion of proximality actually refers to the action of the semigroup
$\{T^n : n \in \Z_+\}$. The usual definition of proximality would be:
$x$ and $y$ are proximal if there exists a sequence $n_i \in \Z$ with
$|n_i| \ \to \infty$ such that $\ lim \ d(T^{n_i}(x),T^{n_i}(y)) \ = \ 0$.
\end{remark}

\begin{lem}\label{scramlem02} For $x \in X$ and $n \in \N$ the pair $(x,T^n(x))$ is proximal iff
 $T^n(y) = y$ for some  $y \in \omega T(x) $. The pair $(x,T^n(x))$ is asymptotic iff $T^n(y) = y$
for every $y \in \omega T(x)$. \end{lem}

\proof $p T^n(x) = T^n( px)$ and so $p x = p T^n(x)$ iff $p x = T^n(p x)$. The pair is proximal (or asymptotic)
iff  $p x = T^n(p x)$ for some $p \in A_+(X,T)$ (resp. for all $p \in A_+(X,T)$).

$\Box$ \vspace{.5cm}

We recall the following, see, e.g. \cite[Proposition 2.2]{AK}.

\begin{prop}\label{scramprop03} A  metric dynamical system $(X,T)$ is proximal iff there exists a fixed point
$e \in X$ which is the unique minimal subset of $X$,
i.e. $(X,T)$ is a minCT system. Consequently, $(X,T^{-1})$ is
proximal if $(X,T)$ is.
\end{prop}

\proof If $u$ is a minimal idempotent then $ux$ is a minimal point for every $x \in X$.  So if $e$ is the
unique minimal point of $X$, then $ux = e$ for every $x \in X$ and every minimal idempotent $u $. Hence, every pair is
proximal.

Assume now that $(X,T)$ is proximal.  For any $x \in X$, the pair
$(x,T(x))$ is proximal and so there exists  $p  \in A_+(X,T)$
such that $px = p T(x) = T(px)$ and so $e = px$ is a fixed point. A pair of
distinct fixed points is not proximal and so $e$ is the unique fixed point.  Hence, $e$ is in the orbit closure of
every point and so $\{ e \}$ is the only minimal set.

Since the minimal subsets for $T$ and $T^{-1}$ are the same, it follows that $(X,T^{-1})$ is proximal.

$\Box$ \vspace{.5cm}

Thus, we obtain the following obvious corollary. Compare Proposition \ref{miscprop1}

\begin{cor}\label{scramcor03a} A  metric dynamical system $(X,T)$ is completely scrambled iff
it is a minCT system and $A_+(X,T)$ distinguishes points of $X$. \end{cor}

$\Box$ \vspace{.5cm}

Completely scrambled systems were introduced by Huang and Ye \cite{HY} who provided a rich supply of examples, but all appear to
be of height the first countable ordinal.

In contrast with proximal systems there exist  completely scrambled systems $(X,T)$ whose inverse $(X,T^{-1})$ is not
completely scrambled.

\begin{ex}\label{scramex01} Begin with $(Y,F)$ a completely scrambled system with fixed point $e$. Let
$(X,T)$ be the quotient space of the product system $(Y \times \{ 0, 1 \}, F \times id_{\{ 0,1 \}})$ obtained
by identifying $(e,0)$ with $(e,1)$ to obtain the fixed point denoted $e$ in $X$. Let $X_0$ and $X_1$ be the
images of $Y \times \{ 0 \}$ and $Y \times \{ 1 \}$ in $X$. Since $(X,T)$ has a unique fixed point $e$ we can
construct a sequence $\{ x^n : n \in \Z_+ \}$ so that
\begin{itemize}
\item $ x^0 \ = \ e$.

\item $\{ d(x^n,T(x^n)) \} \to 0$ as $n \to \infty$.

\item For every $N \in\N$ the set $\{ x^i : i \geq N \}$ is dense in $X$.
\end{itemize}

 Now for $n \in \Z$ let $z^n \in X \times [0,1]$ be defined by
 $$  z^n \ = \ \begin{cases} (x^n, 1/(n+1)) \qquad \mbox{for} \ n \geq 0, \\
 (e, 1/(-n + 1)) \qquad \mbox{for} \ n < 0. \end{cases} $$

 Let
 $$\hat X \ = \ X \times \{ 0 \} \cup \{ \ z^n \ : \ n \in \Z \ \}.$$
 Let $\hat T(x,0) = (T(x),0)$ and $\hat T(z^n) = z^{n+1}$.

 The system $(\hat X, \hat T)$ is topologically transitive with  fixed point
 $(e,0)$ the unique minimal set. Hence, the system is proximal.  Since every orbit
 in $X$ is confined to either $X_0$ or $X_1$ it follows that no point $z^n$ is asymptotic
 to a point in $X \times \{ 0 \}$. By Lemma \ref{scramlem02} no two distinct points on the $z^n$ orbit are
 asymptotic.  Hence, $(\hat X, \hat T)$ is completely scrambled.  However, the inverse
 $(\hat X, \hat T^{-1})$ is not since $\{ z^n \} \to (e,0)$ as $n \to - \infty$.
 \end{ex}

 $\Box$ \vspace{.5cm}

By a result of Schwartzman (see Gottschalk and Hedlund \cite[Theorem
10.36]{GH}) a nontrivial, expansive system admits non-diagonal asymptotic
pairs.
%
%
%
%
It follows that no nontrivial subshift can be completely scrambled. However, we
note that the subshifts which arise from labels of finite type are pretty close.

 \begin{prop}\label{scramprop04} If $\M_1, \M_2$ are two
 different
 labels then the following are equivalent.
 \begin{itemize}
 \item[(i)] For some $n_1, n_2 \in \Z$ the pair $(S^{n_1}x[\M_1], S^{n_2}x[\M_2]) $ is  asymptotic  for $S$ or $S^{-1}$.

 \item[(i+)] For some $n_1, n_2 \in \Z$ the pair $(S^{n_1}x_+[\M_1], S^{n_2}x_+[\M_2]) $ is  asymptotic  for $S$ or $S^{-1}$.

 \item[(ii)] For all $n_1, n_2 \in \Z$ the pairs $(S^{n_1}x[\M_1], S^{n_2}x[\M_2]) $ are  asymptotic for both $S$ and $S^{-1}$.

 \item[(ii+)] For all $n_1, n_2 \in \Z$ the pairs $(S^{n_1}x_+[\M_1], S^{n_2}x_+[\M_2]) $ are  asymptotic for both $S$ and $S^{-1}$.

 \item[(iii)] $\{ \M_1, \M_2 \} \ = \ \{ \emptyset, 0 \}$.
 \end{itemize}
 \end{prop}

 \proof Since $R_S( x[0] ) = e = x[\emptyset]$, it is clear that (iii) implies (ii), (ii+). That (ii) implies (i)
 and (ii+) implies (i+) are  obvious.

 Now assume that $\{ \M_1, \M_2 \} \ \not= \ \{ \emptyset, 0 \}$. By renumbering we can assume that there exists
 $\rr > \00 $ such that $\rr \in \M_1 \setminus \M_2$. Let $\{ t^i \in IP_+(k)\}$ be a sequence  with length
 vector $\rr$ and such that $\{ j_r(t^i) \to \infty$. Then by Theorem
 \ref{towtheo13a}
  $Lim \, S^{t^i}(x[\M_1]) \ = \ x[\M_1 - \rr]$ and $Lim \, S^{t^i}(x_+[\M_1]) \ = \ x_+[\M_1 - \rr]$. Neither limit point
 is  the fixed point $e = x[\emptyset]$ since $\rr \in \M_1$. Hence, for any $n_1 \in \Z$,
 $Lim \, S^{t^i}(S^{n_1}(x[\M_1])) \ = \ S^{n_1}(x[\M_1 - \rr]) \ \not= \ e$.  On the other hand, since, $\rr \not\in \M_2$,
 $$Lim \, S^{t^i}(S^{n_2}(x[\M_2])) \ = \ S^{n_2}(x[\M_2 - \rr]) \ = \ e$$ and
 $$Lim \, S^{t^i}(S^{n_2}(x_+[\M_2])) \ = \ S^{n_2}(x_+[\M_2 - \rr]) \ = \ e.$$
 Thus, the pairs $(S^{n_1}x[\M_1], S^{n_2}x[\M_2]) $
 $(S^{n_1}x_+[\M_1], S^{n_2}x_+[\M_2]) $ are not asymptotic
 for $S$ or for $S^{-1}$. This prove the contrapositive of $(i), (i+) \Rightarrow (iii)$.

 $\Box$ \vspace{.5cm}

 \begin{cor}\label{scramcor05} For  any positive label $\M$ the set
 $\{ \ S^n(x[\NN]) \ : \ 0 \not= \NN   \in \Theta(\M), n \in \Z \ \}$
 is a scrambled subset for $(X(\M),S)$ and for $(X(\M),S^{-1})$.
  The set
 $\{ \ S^n(x_+[\NN]) \ : \ 0 \not= \NN  \in \Theta(\M), \ n \in \Z \ \}$
 is a scrambled subset for $(X_+(\M),S)$ and for $(X_+(\M),S^{-1})$.
 If
 $\M$ is a label of finite type then these sets are the complement of
 the orbit of $x[0]$ in $X(\M)$ and $X_+(\M)$, respectively. \end{cor}

 \proof  That the set is scrambled is clear from Proposition \ref{scramprop04}.
 In the finite type case, Corollary \ref{towcor15a} (a)
 implies that we are excluding only the orbit of $x[0]$ from the set.

 $\Box$ \vspace{.5cm}

 \begin{df}\label{scramdf06} An \emph{inverse sequence} in $\LAB $ is a sequence
 $\{ \M^i, \rr^i : i \in \Z_+ \}$ with $\rr^i > \00 $ in $\M^i$
 and such that $\M^i = \M^{i+1} - \rr^{i+1}$ for $i > 0$. For the associated inverse sequences
 $p_{\rr^{i+1}} : (X(\M^{i+1}),S) \to
 (X(\M^i),S)$ and $p_{\rr^{i+1}} : (X_+(\M^{i+1}),S) \to
 (X_+(\M^i),S)$  we let $(X( \{ \M^i, \rr^i \}),S)$  and $(X_+( \{ \M^i, \rr^i \}),S)$ denote the respective inverse limits.
 \end{df}
 \vspace{.5cm}

\begin{theo}\label{scramtheo07} Let $\{ \M^i, \rr^i \}$ be an inverse sequence in $\LAB $. The inverse limit
system
$(X( \{ \M^i, \rr^i \}),S)$ and $(X_+( \{ \M^i, \rr^i \}),S)$ are topologically transitive, compact metrizable systems.  If each
$\M^i$ is of finite type then the limit systems
and
their inverses
 are completely scrambled. If each $\M^i$ is either finitary or
simple then the limit systems are WAP. \end{theo}

\proof Topologically transitive and WAP systems are closed under inverse limits. For the latter, notice that
  the inverse limit is a subsystem of the product which is clearly WAP. In this case, each map $p_{\rr^{i+1}}$ is
surjective as required because it maps the transitive point $x[\M^{i+1}]$ of $X(\M^{i+1})$ onto the transitive point
$x[\M^i]$ of $X(\M^i)$ since $\M^i = \M^{i+1} - \rr^{i+1}$. In particular, the sequence
$x^* = \{ x[\M^i] \}$ is a transitive point for
the inverse limit.

The point $e$ associated with the sequence $\{ x[\emptyset ] \}$ is a fixed point in $X( \{ \M^i, \rr^i \})$. A minimal subset
of the limit space projects to a minimal subset of $X(\M^i)$ for each $i$.  If $\M^i$ is of finite type then this minimal
subset
is
$\{e \} \subset  X(\M^i)$. Thus, if all are of finite type, the fixed point is the only minimal subset of the limit and
so $(X( \{ \M^i, \rr^i \}), S)$ is proximal by Proposition \ref{scramprop03}.

Notice that if $x \in X(\M^i)$ is not equal to the fixed point $e$ then $x[0] \not\in p_{\rr^{i+1}}^{-1}(x)$.
If $x, y$ are distinct points of $X( \{ \M^i, \rr^i \})$ then for sufficiently large $i$ they project to distinct
points of  $X(\M^i)$ with neither projecting to $x[0]$  in $X(\M^i)$.  In the finite type case
it then follows from Corollary \ref{scramcor05}
that for sufficiently large $i$, $x$ and $y$ project to a non-asymptotic pair.  Consequently, the pair $(x,y)$ is not
asymptotic in  $X( \{ \M^i, \rr^i \})$.

$\Box$ \vspace{.5cm}

\begin{remark} Since a transitive point for $X( \{ \M^i, \rr^i \})$ projects to a transitive point on each
$X(\M^i)$
it follows that the transitive points for $X( \{ \M^i, \rr^i \})$ are all on the orbit of $x^*$ described above
and so $x^*$ is isolated when the labels $\M^i$ are of finite type. Similarly, for $X_+( \{ \M^i, \rr^i \})$.
\end{remark}
\vspace{.5cm}

For the construction of our examples, we need the following.  Recall that
\ref{towprop27again}
implies that
if $\M_1$ is a finite label and $\M_2$ is a label with supports disjoint from those of $\M_1$, then

$$\Theta'(\M_1 \oplus \M_2) \ = \ \Theta'(\M_1) \oplus \Theta(\M_2) \cup \Theta(\M_1) \oplus \Theta'(\M_2).$$

\begin{lem}\label{scramlem08} Let $\rr$ be a positive finite vector with support
disjoint from those in $Supp \ \M$ for some nonempty label $\M$.
Then
 $P_{\rr}( \Theta'(\langle \rr \rangle) \oplus \Theta(\M)) = \{ \emptyset \}$ and on
$\{ \langle \rr \rangle \} \oplus \Theta(\M) \ \ P_{\rr}$ is a bijection onto $\Theta(\M)$.

\end{lem}

\proof
Since $\rr$ is not an element of any label in $\Theta'(\langle \rr \rangle) \oplus \Theta(\M)$ it follows
that all of these labels are mapped to $\emptyset$.

By
\ref{towprop27again}
every label of $\Theta(\langle \rr \rangle  \oplus \M)$ is of the form $\NN_1 \oplus \NN_2$ with
$\NN_1 \in \Theta(\langle \rr \rangle)$ and $\NN_2 \in \Theta(\M)$. If $\NN_1 \not= \langle \rr \rangle$ then
$P_{\rr}$ maps $\NN_1 \oplus \NN_2$ to $\emptyset$.  If $\NN_1 = \langle \rr \rangle$ then $P_{\rr}$ maps
$\NN_1 \oplus \NN_2$ to $\NN_2$. Hence, for any $\NN_2 \in \Theta(\M)$ the unique label of the form
$\langle \rr \rangle \oplus \NN$ which is mapped to $\NN_2$ has $\NN = \NN_2$.

$\Box$ \vspace{.5cm}

\begin{ex}\label{ex11a} Let $\{ \rr^i \}$
be
a sequence of positive $\N$-vectors all with disjoint supports
and let $\M$ be a finitary label with the sets in
$Supp \ \M$ disjoint from the supports of the  sequence.

 Let $\NN^0 = \{ 0 \}$ and $\NN^{i+1} = \langle \rr^{i+1} \rangle \oplus
 \NN^i$ define an increasing sequence of
finite labels. Define $\{ \M^i = \NN^i \oplus \M, \rr^i \}$, an inverse sequence of finitary labels. For each $i$
Lemma \ref{scramlem08}
implies that the
preimage of $\emptyset$ by $P_{\rr^{i+1}} : \Theta(\M^{i+1}) \to \Theta(\M^i)$ is countable and the preimage of every other
point is a singleton.   It follows that the  limit system
$(X( \{ \M^i, \rr^i \}),S)$ and its inverse
$(X( \{ \M^i, \rr^i \}),S^{-1})$ are completely scrambled, topologically transitive, countable WAPs.

Notice that $X(\M)$ and $X_+(\M)$ are
factors of $X( \{ \M^i, \rr^i \})$ and $X_+( \{ \M^i, \rr^i \})$ respectively.  Hence,
if we choose $\M$ with height greater than some countable ordinal
$\a$ then $X( \{ \M^i, \rr^i \})$ and  $X_+( \{ \M^i, \rr^i \})$ have height greater than $\a$.
\end{ex}

$\Box$


\begin{cor}
For every countable ordinal $\alpha$
there exist a topologically transitive
 completely scrambled, countable WAP system of height greater than $\a$.
 \end{cor}

\vspace{.5cm}

Following Huang and Ye we can take countable products of copies of these examples to get completely scrambled WAP systems
on the Cantor set with arbitrarily large heights.
However, these examples will not be topologically transitive.
\vspace{1cm}


\appendix


\section{Directed sets and nets}\label{appendix-nets}
\vspace{.5cm}

We review the theory of nets, following \cite[Chapter 2]{K}.

A \emph{directed set} is a set $I$ equipped with a reflexive, transitive relation $\prec$
such that if $i_1, i_2 \in I$ then there exists $j \in I$ such that $i_1, i_2 \prec j$.

For $i \in I$ let $\prec_i = \{ j : i \prec j \}$. A set $F \subset I$ is  called \emph{terminal}
if $F \supset \prec_i $ for some $i \in I$.   $F$ is called \emph{cofinal} if $F \cap \prec_i \not= \emptyset $ for all
$i \in I$. In the family language of \cite{A-97} these are dual families of subsets of $I$. Because the
set $I$ is directed by $\prec$ it follows that the family of terminal sets is a filter.  That is, a finite intersection of
terminal sets is terminal. The cofinal sets satisfy the dual, \emph{Ramsey Property}: If a finite union of subsets
of $I$ is cofinal then at least one of them is cofinal.

For example, if $x$ is a point of a space $X$ then the set $\NN_x$ of neighborhoods of $x$ is
directed by $\supset$ and a subset of $\NN_x$ is cofinal iff it is a neighborhood base.
 The sets $\Z_+$ and $\N$ are directed by $\leq$ and a subset is terminal iff it is cofinite. A subset is
 cofinal iff it is infinite.

A \emph{net} in a set $Q$ is a function from a directed set $I$ to $Q$, denoted $\{ x^i : i \in I \}$. If $A \subset Q$ we
say that the net is \emph{eventually} (or \emph{frequently}) in $A$ if $\{ i : x^i \in A \}$ is terminal (resp. is cofinal).
If $x $ is a point of a space $X$ then a net in $X$ converges to $x$ (or has $x$ is a limit point) if for every $U \in \NN_x$
the net is eventually in $U$ (resp. is frequently in $U$). Thus, if a net in $A$ has $x$ as a limit point then
$x$ is in the closure of $A$. Conversely, if $x \in \ol{A}$ then we can use $I = \NN_x$ and choose $x^U \in A \cap U$. We thus
obtain a net in $A$ converging to $x$.

If $x$ does not have a countable neighborhood base there may be no sequence in $A$ which converges to $x$.  For example,
if $X = \b Q$ for some infinite set $Q$ then no infinite subset $Q_0$ of $Q$ has a unique limit point because, by
Lemma \ref{applem02} the inclusion of $Q_0$ into $Q$ extends to a homeomorphism of $\b Q_0$ onto the closure of $Q_0$.

A map $k : I' \to I$ between directed sets is a \emph{directed set morphism} if $k^{-1}(F)$ is terminal in $I'$ whenever
$F$ is terminal in $I$. If $i_1' \prec i_2' $ implies $k(i_1') \prec k(i_2') $ and the image, $k(I'),$ is cofinal in
$I$ then $k$ is a morphism.

A map $k : I' \to I$ is a morphism iff whenever $F $ is cofinal in $I'$, then $k(F)$ is cofinal in $I$. This follows because
$$k(F) \cap A \not= \emptyset \quad \Longleftrightarrow \quad F \cap k^{-1}(A) \not= \emptyset$$
and a set is  cofinal iff it meets every terminal set and vice-versa.

With this definition of morphism, the class of directed sets becomes a category.

If $i \mapsto x^i$ is a net, then the composite $i' \mapsto x^{k(i')}$ is the \emph{subnet}
induced by $k$. We will usually suppress the mention of $k$ and just write $\{ x^{i'} : i' \in I' \}$ for the subnet.
Their use is illustrated by the following result.

\begin{lem}\label{applem05} {\bfseries (The Smith Lemma)} Let $\A$ be a collection of nonempty subsets of
of a set $Q$ which is closed under finite intersection and let $\{ x^i : i \in I \}$ be a net in $Q$. If
 $  x^i \in A$ frequently for every $A \in \A$, then there is a subnet $\{ x^{i'} : i' \in I' \}$ such that
 $  x^{i'} \in A$ eventually for every $A \in \A$. Conversely, if such a subnet exists
 then $  x^i \in A$ frequently for every $A \in \A$.
 \end{lem}

 \proof By assumption, $\A$ is directed by $\supset$. On the product $\A \times I$ we use the product ordering
 $(A_1,i_i) \supset \times \prec (A_2,i_2)$ if $A_1 \supset A_2$ and $i_1 \prec i_2$.  Let
 $$I' = \{ (A,i) \in \A \times I : x^i \in A \}.$$
 From the assumptions it is easy to check that
 \begin{itemize}
 \item $I'$ is directed by $\supset \times \prec$.
 \item The coordinate projection $(A,i) \mapsto i$ is a morphism from $I'$ to $I$.
 \item For the induced subnet $\{ x^{i'} : i' \in I' \} \ $
 $  x^{i'} \in A$ eventually for every $A \in \A$.
\end{itemize}

For details, see \cite[Lemma 2.5]{K}.

The converse follows because a terminal subset of $I'$ maps to a cofinal subset of $I$.

$\Box$ \vspace{.5cm}

\begin{cor}\label{appcor06} If $\{ t^i \}$ is a net in $\Z_+$ then exactly one of the following is true.

\begin{itemize}
\item[(i)] There is a subnet $\{ t^{i'} \}$ with limit $\infty$.

\item[(ii)] There is a finite $F \subset \Z_+$ such that eventually $t^i \in F$ and for each $s \in F$,
$t^i = s$ frequently.
\end{itemize}
\end{cor}

\proof Let $\A$ be the collection of cofinite subsets of $\Z_+$. By Lemma \ref{applem05}, case (i)
occurs when eventually $t^i \in A$ for every $A \in \A$. Otherwise, there exists $A \in \A$ such that
$\{ i : t^i \in A \}$ is not cofinite and so $\{ i : t^i \in F_0 \}$ is terminal for $F_0$ the complement of $A$.
Let $F_1 = \{ s \in F_1 : \{ i : t^i = s \}$ is not cofinal $ \}$. So for each $s$ in the finite set $F_1$,
eventually $t^i \not= s$. Let $F = F_0 \setminus F_1$. By definition $ \{ i : t^i = s \}$ is  cofinal for each
$s \in F$. Finally, $\{ i : t^i \in F \}$ is the intersection of a finite collection of terminal sets and so is terminal.

$\Box$ \vspace{.5cm}

\vspace{1cm}

\section{Ellis semigroups and Ellis actions}\label{appendix-ellis}
\vspace{.5cm}

We will follow the notation of \cite{A-97}.

We write for a map $\phi : S \times X \rightarrow X $
\begin{equation}\label{app01}
\begin{split}
px \quad = \quad \phi(p,x) \quad = \quad \phi^{p}(x) \quad =
\phi_{x}(p) \qquad \mbox{for} \ (p,x) \in S \times X.\\
AB \quad = \quad \{ px : p \in A \ \mbox{and} \ x \in B \}
\qquad \mbox{for}\  A \times B \subset S \times X. \\
\phi^{\#} : S \rightarrow X^{X}
\qquad \mbox{is defined by}\  p \mapsto \phi^{p}. \hspace{3cm}\\
\phi_{\#} : X \rightarrow X^{S} \qquad \mbox{is defined by}  \ x
\mapsto \phi_{x}. \hspace{3cm}
\end{split}
\end{equation}

A semigroup $S$ is a nonempty set equipped with $M :S \times S
\rightarrow S$ which is an associative multiplication, i.e.
\begin{equation}\label{app02}
M^{p} \circ M^{q} \quad = \quad M^{pq} \qquad \mbox{for all} \ p,q
\in S.
\end{equation}

An action of a semigroup $S$ on a nonempty set $X$ is a map $\phi
:S \times X \rightarrow X$ which is an action, i.e.
\begin{equation}\label{app03}
\phi^{p} \circ \phi^{q} \quad = \quad \phi^{pq} \qquad \mbox{for
all} \ p,q \in S.
\end{equation}

If $S$ is a semigroup then the multiplication map $M$ is an  action
of $S$ on itself, called the \emph{translation action}. If $X$ is a
singleton set then the unique map from $S \times X$ to $X$ is an
action called the \emph{trivial action}.

A semigroup is called a \emph{monoid} when it contains a (necessarily unique) two-sided
identity element, $u$, i.e. $up = p = pu$ for all $p \in S$.  A \emph{monoid action} of a monoid $S$ on a
set $X$ is a semi-group action such that the identity element acts
as the identity map on $X$, i.e. $\phi^{u} = id_{X}$. The
translation action of a monoid is a monoid action as is the trivial
action of any monoid.

If $S,T$ are semigroups then $g : S \rightarrow T$ is a
\emph{semigroup homomorphism} when $g(pq) = g(p)g(q)$ for all $p,q
\in S$.  A nonempty subset $H \subset S$ is a \emph{subsemigroup}
when it is closed under multiplication, i.e. $HH \subset H$, in
which case the inclusion of $H$ into $S$ is a homomorphism. $H$ is an \emph{ideal}
when it satisfies the stronger condition $SH \subset H$. The
image $g(H) \subset T$ is a subsemigroup and the preimage of a
subsemigroup of $T$ is a subsemigroup of $S$ when it is nonempty. A
singleton $\{u\}$ is a subsemigroup iff $u$ is an
\emph{idempotent}, i.e. $uu = u$.
If $S,T$ are monoids then $g : S \rightarrow T$ is a \emph{monoid
homomorphism} when it is a semigroup homomorphism which maps the
identity of $S$ to that of $T$.  If $S$ is a monoid with identity
$u$ and $g : S \rightarrow T$ is a surjective semigroup
homomorphism then $g(u)$ is an identity in $T$.  That is, $T$ is a
monoid and $g$ is a monoid homomorphism.  A subsemigroup of a
monoid $S$ is a \emph{submonoid} when it contains the identity
element of $S$.

If  $\phi : S \times X \rightarrow X$ and $\psi : S \times Y
\rightarrow Y$ are semigroup actions, then $\pi : X \rightarrow Y$
is an \emph{action map} when $\pi(px) = p \pi(x)$ for all $ p \in
S, x \in X$.  A subset $K \subset X$ is called \emph{ invariant}
if $K \not= \emptyset$ and  $SK \subset K$. In that case, the
restricted map $\phi|K : S \times K \rightarrow K$ is a semigroup
action and the inclusion of $K$ into $X$ is an action map. We call
$\phi|K$ or just $K$ a \emph{subsystem} of $X$. The image $\pi(K)
\subset Y $ is then  invariant and the preimage of an invariant
subset of $Y$ is an  invariant subset of $X$ if it is nonempty.
The \emph{orbit} of a point $x \in X$ is the invariant set $Sx =
\phi_{x}(S)$.  Any union of  invariant sets is invariant and any
nonempty intersection of  invariant sets is invariant. A subset $H \subset S$
is invariant under the translation action exactly when it is an ideal.

For any nonempty set $X$ map composition gives $X^{X}$ the
structure of a monoid and the evaluation map $Ev :X^{X} \times X
\rightarrow X$ is a monoid action.  If $\phi : S \times X
\rightarrow X$ is a semigroup action then $\phi^{\#} : S
\rightarrow X^{X}$ is a homomorphism and for each $x \in X$,
$\phi_{x} = Ev_{x} \circ \phi^{\#} : S \rightarrow X$ is an action
map (with the translation action on $S$) whose image is the orbit of $x$.

If a semigroup $S$ acts on each member of an indexed family
$\{X_{i} : i \in I \}$ then the \emph{product action} on the
product $\Pi \{X_{i} : i \in I \}$ is uniquely defined by the
condition that for every $j \in I$ the projection map $\pi_{j}
:\Pi \{X_{i} : i \in I \} \rightarrow X_{j}$ is an action map. In
particular, if $\phi : S \times X \rightarrow X$ is an action and
$I$ is any set then we obtain the product action $\phi^{I} : S
\times X^{I} \rightarrow X^{I}$. We denote by $\phi^{2}$  the
product action on $X \times X$, i.e. the special case with $I =\{
0,1 \}$.  When $I = X$ we have $(\phi^{X})_{id_{X}} = \phi^{\#}$.

Recall that our spaces are all assumed to be nonempty and Hausdorff.
 For a compact space $X $ and a set $S, \ X^{S}$ denotes the set of all maps from $S$
to $X$, equipped with the compact, product topology.

An \emph{Ellis semigroup} is a compact space equipped with a semigroup
multiplication $M$ such that the adjoint map $M^{\#}: S
\rightarrow S^{S} $ is continuous, or, equivalently,  $M_{q} : S
\rightarrow S$ is continuous for each $q \in S$.  For any compact space
$X$ the space of maps $X^{X}$ is an Ellis semigroup.

An \emph{Ellis action} is a semigroup action $\phi$ of an Ellis
semigroup on a compact space $X$ such that the semigroup homomorphism
$\phi^{\#}: S \rightarrow X^{X} $ is continuous, i.e. for each $x
\in X$, the action map $\phi_{x} : S \rightarrow X$ is continuous.
The translation action and trivial actions of an Ellis semigroup
are Ellis actions as is the evaluation action of $X^{X}$ on a compact
space $X$.

Thus, $\phi$ is an Ellis action exactly when $\phi^{\#}: S \rightarrow X^{X} $ is continuous
homomorphism of Ellis semigroups.  The image of $\phi^{\#}$ is called the
\emph{enveloping semigroup} of $\phi$ and is denoted $E(\phi)$.

The idempotents play a special role in the theory because of the following
crucial property of an Ellis semigroup (see
\cite[Lemma 2.9]{Ellis}  and \cite[Lemma 6.6]{Aus}).

\begin{lem}\label{applem01}{\bfseries Ellis-Numakura}
A closed subsemigroup $A$ of an Ellis semigroup
$S$ contains idempotents.
\end{lem}

\proof We recall the quick proof.  By Zorn's Lemma we may assume that $A$ is a minimal nonempty, closed subsemigroup.
Let $p \in A$. $Ap \subset A$ is a closed subsemigroup and so by minimality $Ap = A$. Hence, the closed subsemigroup
$\{ q \in A : qp = p \}$ is nonempty and so equals $A$. Hence, $pp = p$ (and $A = \{ p \} $ by minimality).

$\Box$ \vspace{1cm}

\section{The Stone-\v{C}ech compactification}\label{appendix-StoneCech}
\vspace{.5cm}
If $Q$ is
an infinite set, which we regard as a space with the discrete topology,
then the Stone-\v{C}ech Compactification is a compact space $\b Q$ and a
bijection $j$ from $Q$ onto a discrete, dense subset of $\b Q$ which satisfies the
following:

\begin{itemize}
\item {\bfseries Extension Condition} If $h : Q \to X$ is a map with $X$ a compact space,
then there exists a -necessarily unique- continuous map $\b h : \b Q \to X$ such that $h = \b h \circ j$.
\end{itemize}

We will usually regard $j$ as an identification and so regard $Q$ as a discrete, dense subset of the
compact space $\b Q$.

The construction is a functor from the category of sets to that of compact spaces.  If $q : Q_1 \to Q_2$ is an
map then composing with the inclusion $j_2$ and extending, we obtain $\b q : \b Q_1 \to \b Q_2$ such that
$\b q \circ j_1 = j_2 \circ q$. The functor properties: $\b \ id_Q = id_{\b Q}$ and
$\b (q \circ \tilde q) = \b q \circ \b \tilde q$ follow from the uniqueness of the extension of a continuous
map from a dense subset when the spaces are Hausdorff.

\begin{lem}\label{applem02} If $q : Q_1 \to Q_2$ is injective, then $\b q : \b Q_1 \to \b Q_2$ restricts to
a homeomorphism onto its image, the closure $\ol{q(Q_1)} \subset \b Q_2$. \end{lem}

\proof First extend the map $ q^{-1} : q(Q_1) \to \b Q_1$ arbitrarily to a map on $Q_2$.  Then apply
the Extension Condition to obtain a continuous map from $\b Q_2 \to \b Q_1$. If $r$ is the restriction of
this continuous map to $\ol{q(Q_1)}$, then $r \circ q$ is the identity on $Q_1$ and $\b q \circ r$ is the
identity on $q(Q_1)$. Hence, $r : \ol{q(Q_1)} \subset \b Q_2$ is the inverse of $\b q$.

$\Box$ \vspace{.5cm}

\begin{cor}\label{appcor02} A sequence $\{ q^n : n \in \N \}$  in $Q$ is convergent in $\b Q$ iff it is eventually constant and
so converges to an element of $Q$. \end{cor}

\proof If a sequence in a Hausdorff space takes on only finitely many values then it is convergent iff it is eventually constant.
If $Q_0$ is a countably infinite subset of $Q$ then the inclusion of $Q_0$ into $Q$ extends to a homeomorphism
of $\b Q_0$ into $\b Q$. Since $\b Q_0 $ is uncountable, it follows that a sequence in $Q$ which takes on infinitely
many distinct values has uncountably many limit points in $\b Q$ and so is not convergent.

$\Box$ \vspace{.5cm}

If $A$ is a  subset of $\b Q$ we let $A_0 = A \cap Q$. We denote by $\ol{A}$ the closure of $A$ in $\b Q$.
Since $Q$ is discrete and dense in $\b Q$, it is the set of isolated points of $Q$.

\begin{prop}\label{appprop02a} (a) If $L_1, L_2 \subset Q$ then $\ol{L_1} \cap \ol{L_2} = \ol{(L_1 \cap L_2)}$.

(b) If $A \subset \b Q$ the following are equivalent:
\begin{itemize}
\item[(i)] $A$ is a clopen subset of $\b Q$.

\item[(ii)] There exists an open set $U$ such that $A = \ol{U}$.

\item[(iii)] There exists $L \subset Q$ such that  $A = \ol{L}$.

\item[(iv)]  $A = \ol{A_0}$.
\end{itemize}

If $L \subset Q$ then $L = (\ol{L})_0$.

(c) If $A \subset \b Q$ is a closed subset, we let $\F_{A} = \{ L \subset Q : A \subset \ol{L} \}$. If $A$ is nonempty
then $\F_A$ is a filter of subsets of $Q$. The collection $\{ \ol{L} : L \in \F_A \}$ is the set of clopen neighborhoods
of $A$ and so is a base for the neighborhoods of $A$ in $\b Q$. $\F_{\emptyset}$ is the power set of $Q$ and
the collection $\{ \ol{L} : L \in \F_{\emptyset} \}$ is the set of all clopen subsets of $\b Q$.

(d) If $E \subset \b Q \times \b Q$ is an open equivalence relation then it is a clopen neighborhood of the diagonal
$id_{\b Q}$. The set of equivalence classes $\{ E(p) : p \in \b Q \}$ is a finite clopen partition of $\b Q$
and $\{ (E(p))_0 : p \in \b Q \}$ is a finite  partition of $ Q$. Conversely, if $\{ L_1, \dots, L_k \}$ is a
partition of $Q$ then $\{ \ol{L_1}, \dots, \ol{L_k} \}$ is a
clopen partition of $\b Q$ with associated clopen equivalence relation
$\bigcup_i \ \ol{L_i} \times \ol{L_i}$.
The set of such open equivalence relations forms a neighborhood basis for the diagonal.
\end{prop}

\proof: (a): If $L \subset Q$ then the characteristic function $\chi_L :Q \to \{ 0, 1 \}$ extends to a continuous
function $\b \chi_L $ on $\b Q$. Hence, $(\b \chi_L)^{-1}(1)$ and $(\b \chi_L)^{-1}(0)$ are disjoint clopen sets
containing $L$ and $Q \setminus L$ respectively. Hence, $\ol{L}$ is disjoint from $\ol{Q \setminus L}$,
with union $\b Q$. It follows that  $\ol{L}$ is clopen and disjoint
subsets of $Q$ have disjoint closures.

For $L_1, L_2$ the trio  $ \ol{L_1 \setminus (L_1 \cap L_2)},
\ol{ (L_1 \cap L_2)}, \ol{L_2 \setminus (L_1 \cap L_2)} $ are pairwise disjoint. Since
$\ol{L_i} = \ol{L_i \setminus (L_1 \cap L_2)} \cup \ol{ (L_1 \cap L_2)}$ for $i = 1,2$ it follows that
$\ol{L_1} \cap \ol{L_2} = \ol{(L_1 \cap L_2)}$.

(b): (iv) $\Rightarrow$ (iii): Obvious.

 (iii) $\Rightarrow$ (i): It was shown above that $L \subset Q$ implies $\ol{L}$ is clopen.

 (i) $\Rightarrow$ (ii): Let $U = A$.

  (ii) $\Rightarrow$ (iv): If  $U$ is open and $ \ol{U} = A$ then, since $U \subset A$,
  we have $U_0 \subset A_0 \subset A$. Since $Q$ is dense in $\b Q$ and $U$ is open,
  $U_0 = U \cap Q$ is dense in $U$. Hence, $\ol{U_0} = \ol{U} = A$ and so $\ol{A_0} = A$.

In general, $L \subset Q$ implies $L \subset \ol{L} \cap Q = (\ol{L})_0$. If $L_1 = (\ol{L})_0 \setminus L$
then $L_1 \subset \ol{L}$ and so $\ol{L_1} \subset \ol{L}$.  But $L_1$ is disjoint from $L$ and so
$\ol{L_1}$ is disjoint from $\ol{L}$. That is, $L_1 = \emptyset$.

(c):  The family $\F_A$ is hereditary upwards and by (a) it is closed under intersection. If $U$ is any closed
neighborhood of $A$ and $\hat U$ is the closure of the interior of $U$, then by (b) $\hat U \subset U$ is a clopen
neighborhood of $A$ and equals $\ol{L}$ with $L = (\hat U)_0 \in \F_A$. The results for $\F_{\emptyset}$ are obvious from (b).

(d): If $E$ is an open equivalence relation on any compact space $X$, then the collection of equivalence classes forms a
pairwise disjoint cover of $X$ by open, nonempty subsets. Hence, there are only finitely many
equivalence classes $\{ E_1,\dots, E_k \}$ and so each is clopen. So $E = \bigcup_i E_i \times E_i$ is clopen in $X \times X$.
Then $\{ (E_1)_0, \dots, (E_k)_0 \}$ is a finite partition of $Q$. Conversely, if $\{ L_1, \dots, L_k \}$ is a
partition of $Q$ then $\{ \ol{L_1}, \dots, \ol{L_k} \}$ is a
clopen partition of $\b Q$ by (a) and (b).

For a compact space, $X$, the neighborhoods of the diagonal form a uniformity and so for any neighborhood $V$
of the diagonal, there exists a symmetric open neighborhood $W$ of the diagonal such that $\ol{W} \circ \ol{W} \subset V$.
When $X = \b Q$ we can choose a finite list $x_1,\dots,x_k \in \b Q$ so that $\{ W(x_1),\dots, W(x_k) \}$ is a finite
open cover.  Let $E_1 = \ol{W(x_1)}$ and $E_i = \ol{W(x_i)} \setminus \bigcup_{j<i} \ol{W(x_j)} $. Each of these is
clopen by (b). Thus,
$\{E_1, \dots, E_k \}$ is a finite clopen partition of  $\b Q$ and $E = \bigcup_i E_i \times E_i \subset V$
because  $\ol{W} \circ \ol{W} \subset V$.

$\Box$ \vspace{.5cm}

We will denote by $\E$\label{eq} the set of clopen equivalence relations on $\b Q$. $\E$ is closed under intersection and
is directed by (decreasing) inclusion.

If $X$ is a compact space we let $2^X$ denote the set of closed subsets of $X$ (including $\emptyset$). For a compact Hausdorff space
the set of neighborhoods of the diagonal is a uniformity on $X$, in fact, the unique uniformity with the topology of $X$, \cite[Corollary 6.30]{K}.
For $V$ a neighborhood of the diagonal in $X \times X$ we define $\hat V \subset 2^X \times 2^X$
by $\hat V = \{ (A,B) : B \subset V(A),$ and $ A \subset V(B) \}$. The collection of $\hat V$'s generate a uniformity on $2^X$
with a compact Hausdorff topology, see e.g.
\cite[pp. 92-93]{A-97}.
We will restrict our discussion to the case of $X = \b Q$
so that we need only consider the equivalence relations $E \in \E$  as these generate the neighborhoods of the diagonal.
Observe that $B \subset E(A)$ iff $E(B) \subset E(A)$ since $E$ is an equivalence relation. Hence,
$\hat E = \{ (A,B) \in 2^{\b Q} \times 2^{\b Q}: E(A) = E(B) \}$. Since each
equivalence class of $E$ is clopen we can choose from it
a point of $Q$ and thus obtain a finite set $F^E \subset Q$ such that $F^E$ meets each equivalence class in a single point.
If $A \in 2^{\b Q}$ then $E(A) \cap F^E$ meets the same $E$ equivalence classes as does $A$. Hence,
$(E(A) \cap F^E, A) \in \hat E$. Since the power set of $F^E$ is finite, it follows that the uniformity on
$2^{\b Q}$ is totally bounded. To show that $2^{\b Q}$ is compact, it suffices to show that the uniformity is complete,
see \cite[Theorem 6.32]{K}.

For $\{ A^i \}$ a net in $2^{\b Q}$ we define
\begin{align}\label{limdefs}
\begin{split}
\ol{LIM}_i \ A^i = \bigcap_i &\ol{\bigcup_{j > i} \ A^j}, \\
\ul{LIM}_i \ A^i = \bigcap_{E \in \E} \ &\bigcup_i \ \bigcap_{j > i} \ E(A^j).
\end{split}
\end{align}

 \begin{lem}\label{applem02b} Let $\{ A^i: i \in I \}$ be a net in $2^{\b Q}$ and $A \in 2^{\b Q}$.
\begin{itemize}

\item[(a)] $ \ul{LIM}_i \ A^i \subset \ol{LIM}_i \ A^i. $

\item[(b)] $A \subset \ul{LIM}_i \ A^i $ iff for all $E \in \E$, eventually, $A \subset E(A^i)$, or, equivalently
$E(A) \subset E(A^i)$.

\item[(c)] $ \ol{LIM}_i \ A^i \subset A$ iff for all $E \in \E$, eventually, $A^i \subset E(A)$, or, equivalently,
$E(A^i) \subset E(A)$.

\item[(d)] The following are equivalent
\begin{enumerate}
\item[(i)]$ \ul{LIM}_i \ A^i = \ol{LIM}_i \ A^i. $

\item[(ii)] There exists $A \in 2^{\b Q}$ such that for all $E \in \E$, eventually, $E(A^i) = E(A)$.

\item[(iii)] With $A = \ol{LIM}$ for all $E \in \E$, eventually, $E(A^i) = E(A)$.

\item[(iv)] For all $E \in \E$, there exists $i \in I$ such that $j_1, j_2 > i$ imply $E(A^{j_1}) = E(A^{j_2})$.
\end{enumerate}
\end{itemize}
\end{lem}

\proof (a): Fix $E \in \E$.  Clearly,
$$ \bigcup_i \bigcap_{j > i}  E(A^j)  \subset  \bigcap_i \bigcup_{j > i}  E(A^j)  =  \bigcap_i E( \bigcup_{j > i}  A^j).$$

Now intersect over $E \in \E$.
\begin{displaymath}
\begin{split}
 \ul{LIM}_i \ A^i \subset \bigcap_E  \bigcap_i E( \bigcup_{j > i}  A^j) = \\
  \bigcap_i  \bigcap_E E( \bigcup_{j > i}  A^j) =  \bigcap_i  \ol{ \bigcup_{j > i}  A^j} = \ol{LIM}_i \ A^i.
  \end{split}
  \end{displaymath}

  (b): It is clear that if for every $E \in \E$ there exists $i$ such that $A \subset E(A^j)$ for all $j > i$ then
  $A \subset \ul{LIM}$.  On the other hand, if $A \subset \ul{LIM}$ then for every $E$
  $A \subset \bigcup_i \ \bigcap_{j > i} \ E(A^j)$.  Since there are only finitely many $E$
  equivalence classes and each is open, this is an increasing union of open sets. It follows that for
  some $i, \ A \subset  \bigcap_{j > i} \ E(A^j)$.  Thus, $A \subset E(A^j)$ for all $j > i$.

(c): If for each $E \in \E$, there exists $j$ such that  $A^j \subset E(A)$ for all $j > i$ then since $E(A)$ is
closed, $\ol{LIM} \subset E(A)$.  Intersecting over $E$ we obtain $\ol{LIM} \subset A$.  On the other hand,
if $\ol{LIM} \subset A$ and $E \in \E$ then since $E(A)$ is open and $\bigcap_i \ol{\bigcup_{j > i} \ A^j}$
is a decreasing intersection of compact sets, it follows that for some
$i, \ol{\bigcup_{j > i} \ A^j} \subset E(A)$ and so for $j > i, \ A^j \subset E(A)$.

The terminal equivalences in (b) and (c) follow from the fact that $B \subset E(A)$ iff $E(B) \subset E(A)$ since
$E$ is an equivalence relation.

(d): (i) $\Rightarrow$ (iii): If $A = \ol{LIM} = \ul{LIM}$ then by (b) and (c), for every $E \in \E$, eventually
$E(A^i) \subset E(A)$ and eventually $E(A) \subset E(A^i)$.

(iii) $\Rightarrow$ (ii): Obvious.

(ii) $\Rightarrow$ (i): This follows from (a), (b) and (c).

(ii) $\Rightarrow$ (iv): Obvious.

(iv) $\Rightarrow$ (i): If $E \in \E$ and $E(A^j)$ is constant for $j > i$ then
\begin{displaymath}
\begin{split}
\ol{LIM} \subset \ol{ \bigcup_{j > i}  A^j} \subset  E( \bigcup_{j > i}  A^j) \\
= \bigcup_{j > i}  E(A^j)  = \bigcap_{j > i}  E(A^j) \subset \bigcup_i \bigcap_{j > i}  E(A^j).
\end{split}
\end{displaymath}
Intersecting over $E \in \E$ we obtain  $\ol{LIM}  \subset \ul{LIM}$. The reverse inclusion is always true by (a).

$\Box$ \vspace{.5cm}

\begin{cor}\label{appcor02c} (a) A net $\{ A^i: i \in I \}$  in $2^{\b Q}$ is Cauchy iff
$\ol{LIM}_i \ A^i = \ul{LIM}_i \ A^i$ in which
case this common value is the limit.  In particular, $2^{\b Q}$ is complete and so is compact.

(b) If $A \in 2^{\b Q}$ then $\{ E(A) \cap F^E : E \in \E \}$ is a net, indexed by $\E$, consisting of finite subsets of $Q$,
which converges to $A$.\end{cor}

\proof (a): The net is Cauchy iff for every $E \in \E$, there exists $i \in I$ such that $j_1, j_2 > i$ imply
$(A^{j_1},A^{j_2}) \in \hat E$, i.e.$ E(A^{j_1}) = E(A^{j_2})$. By Lemma \ref{applem02b} (d) this holds iff
$\ol{LIM} = \ul{LIM}$ and in that case with $A$ this common value we have eventually $E(A^i) = E(A)$ for all $E \in \E$.
That is, eventually $(A^i,A) \in \hat E$.  This means that the net converges to $A$.

(b):  We observed above that $( E(A) \cap F^E,A) \in \hat E$. Thus, the net converges to $A$.

$\Box$ \vspace{.5cm}

\begin{prop}\label{appprop02d} Let $\{ A^i: i \in I \}$  be a net in $2^{\b Q}$ and $A \in 2^{\b Q}$.
\begin{enumerate}
\item[(a)] If $A = \emptyset$ then $\{ A^i: i \in I \}$ converges to $A$ iff eventually $A^i = \emptyset$.  That is,
$\emptyset$ is an isolated point of $2^{\b Q}$.

\item[(b)] Let $F$ be a finite subset of $Q$. If $\{ A^i: i \in I \}$ converges to $A$ then eventually $A^i \cap F = A \cap F$.
If $A = F$ then eventually $A^i = F$. That is, $F$ is an isolated point of $2^{\b Q}$. In general,
$$ A_0 =  \bigcup_i \bigcap_{j > i} (A^i)_0 = \bigcap_i \bigcup_{j > i} (A^i)_0.$$

\item[(c)]  $\{ A^i: i \in I \}$ converges to $A$ iff for every $L \in \F_A$ eventually $L \in \F_{A^i}$ and for every
$L \not\in \F_A$ eventually $L \not\in \F_A^i$.

\item[(d)] If for every $i \in I, \ A^i = \ol{S^i}$ with $S^i \subset Q$ then $\{ A^i: i \in I \}$ converges to $A$ iff
 for every $L \in \F_A$ eventually $S^i \subset L$ and for every
$L \not\in \F_A$ eventually $S^i \cap (Q \setminus L) \not= \emptyset$.
\end{enumerate}
\end{prop}

\proof (a): For any $E \in \E$ $E(\emptyset) = \emptyset$. So convergence implies that eventually $E(A^i) = \emptyset$ and
this requires $A^i = \emptyset$.  Conversely, if $A^i = \emptyset$ eventually then the net converges to $\emptyset$.

(b): We may assume that $F$ is nonempty and we let $E_F = \{ \ol{Q \setminus F} \times  \ol{Q \setminus F} \} \cup
\{ (k,k) : k \in F \} \in \E$. Convergence implies that eventually $E_F(A^i) = E_F(A)$. This implies that $A^i$ contains
the same points of $F$ as does $A$, i.e. $A^i \cap F = A \cap F$. Furthermore, if $A = F$ and $E_F(A^i) = E_F(A)$, then
$A^i \cap \ol{Q \setminus F} = \emptyset$ and so $A^i = F$. In particular, if $k \in A_0$ then eventually
$k \in A^i$ and so $k \in (A^i)_0$. If $k \in Q \setminus A_0$ then eventually $k \not\in (A^i)_0$.

(c): For $L \subset Q$ let $E_L = \ol{L} \times \ol{L} \cup \ol{Q \setminus L} \times \ol{Q \setminus L}.$
Convergence implies that eventually $E_L(A^i) = E_L(A)$. If $L \in \F_A$ then $E_L(A) = \ol{L}$ and so
$A^i \subset \ol{L}$, i.e. $L \in \F_{A^i}$. If $L \not\in \F_A$ then $\ol{Q \setminus L} \subset E_L(A)$ and
so $A^i$ meets $\ol{Q \setminus L}$.

If $A = \emptyset$ then $\emptyset \in \F_A$ and the condition implies that eventually $\emptyset \in \F_{A^i}$,
i.e. eventually $A^i = \emptyset$. We now assume that $A$ is nonempty.

Now assume that the condition holds and that the equivalence classes of $E$ are $\ol{L_1},\dots,\ol{L_k}$
with $\{ L_1,\dots L_k \} $ a partition of $Q$. Assume that they have been numbered so that for some $1 \leq k_1 \leq k$,
$E(A) = \bigcup_{n=1}^{k_1} \ol{L_n}$. Then $\bigcup_{n=1}^{k_1} L_n \in \F_A$ and so there exists $i_0 \in I$
so that $j > i_0$ implies  $\bigcup_{n=1}^{k_1} L_n \in \F_{A^j}$. For each $m = 1,\dots,k_1$, $Q \setminus L_m \not\in \F_A$
and so there exists $i_m$ such that $j > i_m$ implies that $A^j$ meets $\ol{L_m}$. Choose $i > i_0,\dots, i_{k_1}$.
If $j >i$ then $E(A^i) = E(A)$. It follows that $\{ A^i: i \in I \}$ converges to $A$.

(d): If $A^i \subset \ol{L}$ then, intersecting with $Q$ we have $S^i = (A^i)_0 \subset (\ol{L})_0 = L$.
If $A^i$ meets $\ol{Q \setminus L}$ then $S^i$ meets $Q \setminus L$ since disjoint subsets of $Q$ have
disjoint closures in $\b Q$. In each case, the converse is obvious.

$\Box$ \vspace{.5cm}

Assume $\phi : \Gamma \times Q \to Q$ is a monoid action of a countable, discrete, abelian monoid $\Gamma$ on a set $Q$.
For each $\g \in \Gamma$, we obtain the continuous map $\b \phi^{\g} : \b Q \to \b Q$.
Since $\b \phi^{\g_1} \circ \b \phi^{\g_2} = \b (\phi^{\g_1} \circ \phi^{\g_2}) = \b \phi^{\g_1 \g_2}$,
we obtain a continuous monoid action which we denote $\ol{\phi} : \Gamma \times \b Q \to \b Q$.

Apply this to the translation action of $\Gamma$ on itself, i.e. with $\phi = M : \Gamma \times \Gamma \to \Gamma$
and we obtain the action $\ol{M} : \Gamma \times \b \Gamma \to \b \Gamma$ of $\Gamma$ on the compact space $\b \Gamma$.

In general, if $\Phi : \Gamma \times X \to X$ is a continuous monoid action of $\Gamma$ on a compact space $X$ then
for each $x \in X$, we obtain the continuous extension $\b \ \Phi_x : \b \Gamma \to X$. This
defines a map $\b \Phi : \b \Gamma \times X \to X$. In general, it is not continuous, but for each
$x \in X, \ (\b \Phi)_x =  \b \ \Phi_x : \b \Gamma \to X$ is continuous and for each
$\g \in \Gamma, \ (\b \Phi)^{\g} = \Phi^{\g} : X \to X$
is continuous. That is, $p \to px$ is continuous for each $x \in X$ and $x \to \g x$ is continuous for each $\g \in \Gamma$.

For $\g,\t \in \Gamma$ and $x \in X$, the equation $(\g \t)x = \g ( \t x)$ implies that
$$\b \ \Phi_x \circ \b M^{\g} = \Phi^{\g} \circ  \b \ \Phi_x $$
holds on $\Gamma$ and so by continuity on $\b S$.

Apply this to $\ol{M} : \Gamma \times \b \Gamma \to \b \Gamma$ and we obtain
$\b M : \b \Gamma \times \b \Gamma \to \b \Gamma$ and the above
equation implies that $\b M$ is an associative multiplication and so gives $\b \Gamma$ the structure of an Ellis semigroup.
Furthermore, if $\Phi : \Gamma \times X \to X$ is a continuous action of $S$ on a compact space $X$ then
  $\b \Phi : \b \Gamma \times X \to X$ is an Ellis action.

Observe that if $1$ is the identity in $\Gamma$ then it is the identity in $\b \Gamma$. That is, $p1 = p = 1p$ because
the equations are true for $p \in S$ and since $1 \in S$, both $p \mapsto p1$ and $p \mapsto 1p$ are continuous.
Although $\b \Gamma$ is usually not abelian, it is true that $p \g = \g p$ is true for all
$p \in \b \Gamma$ and $\g \in \Gamma$.  If $h : \Gamma \to S$ is a homomorphism with $S$ an Ellis semigroup,
then $\b h : \b \Gamma \to S$ is a homomorphism.
Observe that $\b h (pq) = \b h(p) \b h(q)$ is true for $p, q \in \Gamma$ because $\b h$ extends the homomorphism $h$.
For $p$ fixed in $\Gamma$ continuity implies the equation holds for $q \in \b \Gamma$ and so with $q$ fixed in $\b \Gamma$
continuity implies the result for all $p \in \b \Gamma$. In particular, a homomorphism $h : \Gamma_1 \to \Gamma_2$ of
discrete monoids extends to a continuous homomorphism $\b h : \b \Gamma_1 \to \b \Gamma_2$.

If $\phi : \Gamma \times X \to X$ is a continuous action of  $\Gamma$ on a compact space $X$, then the homomorphism
$\phi^{\#} : \Gamma \to X^X$ extends to a continuous homomorphism of Ellis semigroups $\b \phi^{\#} : \b \Gamma \to X^X$.
The image is the closure of $\phi^{\#}(\Gamma) \subset X^X$, i.e. the \emph{enveloping semigroup} $E(X,\Gamma)$.

Assume  $\phi : \Gamma \times X \rightarrow X$ and $\psi : \Gamma \times Y
\rightarrow Y$ are continuous  actions of $\Gamma$ on the compact spaces
$X$ and $Y$ and that $\pi : X \rightarrow Y$ is a continuous action map. The map is an action map
between the Ellis actions because the equation $\pi(px) = p \pi(x)$ extends from $p \in \Gamma$ to
$p \in \b \Gamma$ by continuity for each fixed $x \in X$.


Recall that $\Gamma_u$ is the group of units in $\Gamma$, i.e. the set of $\g$ such that $\g \t = 1$ for
some $\t \in \Gamma$.

Let $\b^* \Gamma = \b \Gamma \setminus \Gamma$ and let $\b' \Gamma = \b \Gamma \setminus \Gamma_u$.
 We call $\Gamma$ a \emph{cancelation monoid} if for each $\g \in S$ the
translation maps $M^{\g}$ and $M_{\g}$ are injective. Of course, if $\Gamma$ is a group then it is
a cancelation monoid.

\begin{lem}\label{applem03}  Assume that $\Gamma$ is a cancelation monoid.

(a) If $p, q \in \b \Gamma$, then
$pq \in \Gamma$ iff $p, q \in \Gamma$ and $pq \in \Gamma_u$ iff $p, q \in \Gamma_u$.

\begin{equation}\label{adherence}
\begin{split}
 \b \Gamma \cdot \b^* \Gamma \cup \b^* \Gamma \cdot \b \Gamma \ \subset \ \b^* \Gamma,
 \b \Gamma \cdot \b' \Gamma \cup \b' \Gamma \cdot \b \Gamma \ \subset \ \b' \Gamma.
 \end{split}
 \end{equation}

(b) The identity $1$ is the only idempotent element of $\Gamma$.

 \end{lem}

\proof (a): If $p,q \in \Gamma$ then of course $pq \in \Gamma$. Assume $t = pq \in \Gamma$. Since $z \mapsto zq$ is
 continuous and $\Gamma$ is discrete the set $A = \{ z \in \b \Gamma : zq = t \}$ is a
clopen set and so if it meets $\b^* \Gamma$ then it meets $\Gamma$ in an infinite set. Let $\g \in A \cap \Gamma$.
If $B = \{ w \in \b \Gamma : \g w = t \}$ meets $\b^* S$ then  it meets $\Gamma$ in an infinite set.
This is impossible since $M^{\g}$ is injective on $\Gamma$. It follows that $B \subset \Gamma$ and so
$q \in B$ is in $\Gamma$. Hence, $M_q$ is injective implies that $A \cap \Gamma$ is not infinite and so
$p \in A$ is also in $S$.

Thus, if either $p$ or $q$ is in $\b^* S$ then $pq \in \b^* S$.

 If $pq \in \Gamma_u$ then, as above, $p, q \in \Gamma$. There exists $\t$ which is an inverse for $pq$ and so
 $q \t$ is an inverse for $p$ and $p \t$ is an inverse for $q$, because $\Gamma$ is abelian. Hence,
 $p, q \in \Gamma_u$.

 (b)  If $\g \g = \g = \g 1$ then $\g = 1$ by cancelation.

$\Box$ \vspace{.5cm}

 If $\Gamma$ is a cancelation monoid such that $\g \t = 1$ only when $\g = \t = 1$, i.e. if  $\Gamma_u = \{ 1 \}$,
 then we say that
 $\Gamma$ is a \emph{cancelation monoid without inverses}.

 \begin{lem}\label{applem04}  If $\Gamma$ is a group then $\b^* \Gamma = \b' \Gamma$.
 If $\Gamma$ is a cancelation monoid without inverses, then $\b' \Gamma$ is the closure in
 $\b \Gamma$ of $\Gamma' = \Gamma \setminus \{ 1 \}$.

\end{lem}

\proof $\Gamma$ is a group exactly when $\Gamma = \Gamma_u$.

Since $1$ is an isolated point in $\b \Gamma$, $\Gamma'$ is dense in $\b' \Gamma = \b \Gamma \setminus \{ 1 \}$.

$\Box$ \vspace{.5cm}

Of course, $\Z$ is a group, while $\Z_+$ and $FIN(\N)$ are cancelation monoids without inverses.

\begin{df}\label{adherencedef} If $\phi : \Gamma \times X \to X$ is a continuous action of  $\Gamma$ on $X$ with $\Gamma$ a cancelation monoid
then we define the \emph{adherence semigroup}
\begin{equation}\label{adherence2}
A(X,\Gamma) \ = \ \b \phi^{\#}(\b' \Gamma) \ \subset \ E(X,\Gamma).
\end{equation}
\end{df}

From (\ref{adherence}) it follows that $A(X, \Gamma)$ is an ideal in $E(X, \Gamma)$.

For a cascade, $(X,T)$ with $\Gamma$ the group $\Z$, $A(X,T)$ is the set of limit points as $|n| \to \infty$
of the bi-infinite
sequence $\{ T^n \}$  in $E(X,T)$. If $\Gamma$ is a cancelation monoid without inverses then
$A(X, \Gamma)$ is the closure in $E(X,\Gamma)$ of $\{ \phi^{\g} : \g \not= 1 \}$.

Recall that for a system $(X, \Gamma)$ a point $x \in X$ is called \emph{recurrent} when
$x \in \overline{\Gamma' \cdot x}$.

\begin{prop}\label{appprop05} Assume that $\Gamma$ is a cancelation monoid and that $\Gamma_u$ contains no
element of finite order other than $1$, e.g. a cancelation monoid without inverses.
For a system  $(X, \Gamma)$ and a point $x \in X$ the following are
equivalent.
\begin{itemize}

\item[(i)]  The point $x$ is recurrent.

\item[(ii)] There exists $p \in \b' \Gamma$ such that $p x = x$.

\item[(iii)] There exists an idempotent $p \in \b^* \Gamma$ such that $p x = x$.

\item[(iv)] There exists $p \in A(X, \Gamma)$ such that $p x = x$.

\item[(v)] There exists an idempotent $p \in A(X, \Gamma)$ such that $p x = x$.
\end{itemize}
\end{prop}

\proof Let  $\phi : \Gamma \times X \to X$ be the action.

(ii) $\Leftrightarrow$ (iv): For $p \in \b \Gamma$, $p x = \phi^{\#}(p) x$.

(ii) $\Leftrightarrow$ (iii), (iv) $\Leftrightarrow$ (v): The sets $\{ p \in \b' \Gamma : p x = x \}$
and $\{ p \in A(X,\Gamma) : p x = x \}$ are closed sub-semigroups. If either is non-empty then it contains
idempotents by  Lemma \ref{applem01}. Also, there is no idempotent in $\b' \Gamma \setminus \b^* \Gamma \subset \Gamma'$.

(ii) $\Rightarrow$ (i): $\b' \Gamma \cdot x  \subset \overline{\Gamma'} \cdot x$.

(i) $\Rightarrow$ (ii): If $x$ is recurrent then there exists a net $\{ \g^i \in \Gamma' \}$ such that
$\g^i x \to x$. If $q$ is a limit point of $\{ g^i \}$ in $\b \Gamma$ then $q x = x$. If $q \in \b' \Gamma$ then
let $p = q$ to obtain (ii).  Otherwise, $q \in \Gamma_u$. Since $1$ is an isolated point, $q \not= 1$.
The sequence of powers $\{ q^n : n \in \N \}$ are all distinct since no element of $\Gamma_u \setminus \{ 1 \}$ has finite order.
If $p$ is a limit point of the sequence in $\b \Gamma$ then $p \in \b^* \Gamma \subset \b' \Gamma$
since $\Gamma$ is discrete. Since $q^n x = x$ for all $n$, $p x = x$.

$\Box$

\vspace{1cm}

\bibliographystyle{amsplain}

\begin{theindex}


\vspace{.5cm}

\noindent{\textbf{{Index of terms}}}

\vspace{.5cm}


\item action map (homomorphism)   \textbf{\pageref{idemp}}

\item adherence semigroup   \textbf{\pageref{subsec,rec}, \pageref{adherencedef}}

\item almost periodic (AP)  \textbf{\pageref{intro}}

\item asymptotic pair  \textbf{\pageref{scramdef01}}



\item Banach density (upper) \textbf{\pageref{Banach01}}

\item Bernoulli shift, $S$  \textbf{\pageref{shift}}

\item Birkhoff center   \textbf{\pageref{subsec,rec}}

\item center periodic (CP)   \textbf{\pageref{cp,weakmix}}
\item center trivial (CT)   \textbf{\pageref{cp,weakmix}}

\item coalescence   \textbf{\pageref{prop06}}
\subitem E-coalescence   \textbf{\pageref{coaleth}}

\item cofinal constant     \textbf{\pageref{waplem04}}


 \item completely scrambled   \textbf{\pageref{scramdef01}}

 \item cross-section   \textbf{\pageref{cross}}


\item Ellis semigroup  \textbf{\pageref{appendix-ellis}}

\item enveloping semigroup   \textbf{\pageref{df,E}}

\item equicontinuous   \textbf{\pageref{sec,coal}}
\subitem equicontinuity point   \textbf{\pageref{sec,coal}}
\subitem almost equicontinuous (AE)  \textbf{\pageref{sec,coal}}
\subitem hereditarily almost equicontinuous (HAE)  \textbf{\pageref{sec,coal}}
\subitem locally equicontinuous (LE)   \textbf{\pageref{WAPctbl}}
\subitem uniformly equicontinuous   \textbf{\pageref{sec,coal}}

\item expansive  \textbf{\pageref{sec,coal}}



\item expansion of integers in $IP(k)$  \textbf{\pageref{towdef02}}

\subitem expansion of length $r$  \textbf{\pageref{towdef02}}
\subitem length vector    \textbf{\pageref{towdef02a}}

\subitem truncation  ($\tilde{t}$)   \textbf{\pageref{towdef02}}
\subitem extension   \textbf{\pageref{towdef02}}
\subitem residual   ($t - \tilde{t}$) \textbf{\pageref{towdef02}}



\item external element    \textbf{\pageref{ex-el}}

\item f-contains    \textbf{\pageref{df,f-contain}}

\item Gamow transformations    \textbf{\pageref{Gamow}}

\item height  (for dynamical systems)  \textbf{\pageref{df,height}}
\item height$^*$  (for CT-WAP dynamical systems)  \textbf{\pageref{isocor4}}

\item height  (for labels)  \textbf{ \pageref{towtheo24}, \pageref{df,heightL}}
\item height$^*$  (for labels)  \textbf{\pageref{df,heightL}}

\item hereditary collection (of subsets of $\N$)    \textbf{\pageref{labellem12a}}

\item idempotent   \textbf{\pageref{idemp}}

\item independent set (for a dynamical system)  \textbf{\pageref{towdf20b}}
\item independent set (for a label)   \textbf{\pageref{towdf20b}}

\item labels  \textbf{\pageref{labeldef01}}

\subitem bounded \textbf{\pageref{labeldef01b}}

\subitem disjoint  \textbf{\pageref{dis-labels}}

\subitem finitary  \textbf{\pageref{labeldef01bxx}}

\subitem finite type  \textbf{\pageref{labeldef01b}}

\subitem flat   \textbf{\pageref{towdef20f}}
\subitem null    \textbf{\pageref{towdf20b}}

\subitem positive \textbf{\pageref{label02}}

\subitem recurrent  \textbf{\pageref{df,recurrent}}

\subitem semi-simple  \textbf{\pageref{wapdef02}}

\subitem simple  \textbf{\pageref{wapdef02}}

\subitem size bounded \textbf{\pageref{labellem02}}

\subitem strong finite type  \textbf{\pageref{wapdef02}}

\subitem strongly recurrent   \textbf{\pageref{df,st-rec}}
\subitem strongly recurrent  set \textbf{\pageref{df,st-rec}}



\subitem tame  \textbf{\pageref{towdf20b}}


\subitem WAP  \textbf{\pageref{waplabel}}

\item $L$-determined     \textbf{\pageref{isolem2a}}

\item Li-Yorke pair   \textbf{\pageref{scramdef01}}

\item minimal trivial (minCT)   \textbf{\pageref{spinlem2}}

\item non-wandering  \textbf{\pageref{transprop}, \pageref{subsec,rec}}

\item nucleus (of a labeled subshift)  \textbf{\pageref{dfNuc}}

\item null (dynamical system) \textbf{\pageref{towdf20b}}

\item orbit closed     \textbf{\pageref{isolem2a}}

\item proximal  (pair)  \textbf{\pageref{scramdef01}}
\item proximal  (system)  \textbf{\pageref{scramdef01}}

\item recurrent point  \textbf{\pageref{lem01}}
\item rigid  \textbf{\pageref{max}}
\subitem weakly rigid   \textbf{\pageref{max},  \pageref{miscprop1}}
\subitem uniformly rigid   \textbf{\pageref{max}}

\item roof (of a label)   \textbf{\pageref{labeldef01b}}

\item scrambled set  \textbf{\pageref{sec,scrambled}}


\item  size of $\mm$   \textbf{\pageref{sec,labels}}
\item strongly recurrent  set \textbf{\pageref{df,st-rec}}


Stone-\v{C}ech compactification \textbf{\pageref{appendix-StoneCech}}

\item subshift   \textbf{\pageref{shift}}


\item tame  \textbf{\pageref{towdf20b}}


\item topological transitivity   \textbf{\pageref{sec,WAP}}
\subitem chain transitivity \textbf{\pageref{sec,WAP}}
\subitem transitive point  \textbf{\pageref{sec,WAP}}


\item translation finite (TF)  \textbf{\pageref{defrupp02}}

\item wandering (open set) \textbf{\pageref{transprop}, \pageref{subsec,rec}}

\item weakly almost periodic (WAP)   \textbf{\pageref{intro}, \pageref{cor02a}}

\item weak mixing   \textbf{\pageref{sec,WAP}, \pageref{cp,weakmix}}

\vspace{1cm}


\newpage

\noindent{\textbf{Index of symbols}}

\vspace{.5cm}

\item $[a \pm b], [\pm b]$  \textbf{\pageref{def401a}}

\item $A(e)$  \textbf{\pageref{df,Ae}}

\item $A[\M]$,   $A_+[\M]$  \textbf{\pageref{df,A[M]}}

\item $\A(\Phi)$ ($FIN(\N)$ adherence semigroup) \textbf{\pageref{label08a}}


\item $A(X,T)$ (adherence semigroup) \textbf{\pageref{sec,WAP}}

\item $A_+(X,T)$ \textbf{\pageref{scramdef01}}

\item ASY\!MP   \textbf{\pageref{sec,scrambled}}

\item $\B_N$ (basic labels)  \textbf{\pageref{df,f-contain}}

\item $\beta_b\Z$  \textbf{\pageref{ellis03}}


\item CT-WAP   \textbf{\pageref{df,Ae}}

\item $ D,  D_+$ \textbf{\pageref{ellis08}, \pageref{ellis15}}

\item $\bar D, \bar D_+$ \textbf{\pageref{ellisiso}}


\item $\hat D, \hat D_+$ \textbf{\pageref{tow17ad}}

\item $d(\M_1, \M_2)$ (label ultrametric)   \textbf{\pageref{label06}}

\item $e = \bar 0$  \textbf{\pageref{ex3}}

\item $\E$ (set of clopen equivalence relations)  \textbf{\pageref{eq}}

\item $\E_0(\Theta(\M))$   \textbf{\pageref{waplabel01}}

\item $\E_*(\Theta(\M))$   \textbf{\pageref{ex-el}}

\item $\E(\Phi)$  ($FIN(\N)$ enveloping semigroup) \textbf{\pageref{label08a}}

\item $\E (\LAB)$   \textbf{\pageref{label08a}}
\item $\hat \E (\LAB)$   \textbf{\pageref{ellis08}}

\item $E(X,T)$ (enveloping semigroup) \textbf{\pageref{df,E}}

\item $E_b(X(\M),S)$  \textbf{\pageref{ellis14}}

\item $F(\mm)$ \textbf{\pageref{df,st-rec}}

\item $FIN(L)$  \textbf{\pageref{Gamow}}
\item $FIN(\N)$  \textbf{\pageref{labeldef01}}

\subitem $\beta FIN(\N)$  \textbf{\pageref{labeltheo06x}}
\subitem $\beta^* FIN(\N)$  \textbf{\pageref{labeltheo06x}}

\subitem $\beta' FIN(\N)$  \textbf{\pageref{labeltheo06x}}

\item $\mathcal{F}(\M,L)$   \textbf{\pageref{tow26x1}}


\item $\Phi(Y)$, $\Phi_+(Y)$  (the preimage operation) \textbf{\pageref{df,Phi}}




\item $\Gamma'$ \textbf{\pageref{gammaprime}}

\item $\Gamma_u$ (group of units) \textbf{\pageref{gammau}}

\item INC    \textbf{\pageref{labellem03a}}

\item $IP(k)$,   $IP_+(k)$   \textbf{\pageref{towdef02}}


\item $ISO(\M)$   \textbf{\pageref{labeliso}}




\item $k(n)$  (expanding function) \textbf{\pageref{towlem01}}


\item $\ell(n)$ (support map) \textbf{\pageref{towdef02a}}

\item $\LAB$ (space of labels)  \textbf{\pageref{label06}}
\item $\LAB_+$ (space of labels)  \textbf{\pageref{permex}}

\item $\LAB(L)$ (space of labels)  \textbf{\pageref{Gamow}}

\item $LIMINF$   \textbf{\pageref{label07}}
\item $LIMSUP$    \textbf{\pageref{label07}}
\item $LIM$   \textbf{\pageref{labelprop04}}

\item $m_b$   \textbf{\pageref{ellis02}}

\item $\M, \NN$ labels   \textbf{\pageref{sec,labels}}

\item $[[\M]]$    \textbf{\pageref{labellem03a}}

\item $max \ \M$   \textbf{\pageref{labelprop03}}
\item $\M \oplus \NN$   \textbf{\pageref{label12}}

\item $\mm, \rr$ $\N$-vectors   \textbf{\pageref{sec,labels}}

\item $|\mm|$  (norm of a label)  \textbf{\pageref{sec,labels}}

\item $n_*(t), n^*(t)$  \textbf{\pageref{def401a}, \pageref{towdef02}}


\item $\Nuc(X(\M))$ (nucleus) \textbf{\pageref{dfNuc}}

\item $O_T$ (relations)  \textbf{\pageref{sec,WAP}}
\subitem $R_T$   \textbf{\pageref{sec,WAP}}
\subitem $R^*_T$   \textbf{\pageref{sec,WAP}}
\subitem $\alpha_T$   \textbf{\pageref{sec,WAP}}
\subitem $\omega_T$   \textbf{\pageref{sec,WAP}}

\item $\mathcal{O}_T(x)$  or $\mathcal{O}(x)$,
$\ol{\mathcal{O}}_T(x)$ (orbit and orbit closure of $x$) \textbf{\pageref{prop02}}


\item $p_{\rr}$    \textbf{\pageref{towtheo17ab}},
\item $P_{\rr}$   \textbf{\pageref{ssFIN}}

\item $\P_f(L)$   \textbf{\pageref{df,f-contain}}

\item PROX   \textbf{\pageref{sec,scrambled}}

\item RECUR   \textbf{\pageref{label12x}}


\item $\rr(t)$ (length vector)    \textbf{\pageref{towdef02a}}

\item $\rho(\M)$ (roof function)  \textbf{\pageref{labeldef01b}}

\item  $S$ (shift) \textbf{\pageref{shift}}

\item $\langle S \rangle$  (label generated by $S$) \textbf{\pageref{df,f-contain}}


\item $supp \ \mm$   \textbf{\pageref{sec,labels}}
\item $supp \ x$   \textbf{\pageref{sec,labels}}
\item $Supp \ \M$   \textbf{\pageref{labelprop03}}

\item SY\!M  \textbf{\pageref{towlem17a}}


\item $\Theta(\M)$  ($FIN(\N)$ orbit closure) \textbf{\pageref{proper}}
\item $\Theta'(\M)$   \textbf{\pageref{labellem05c}}

\item $\Theta(\Psi)$ \textbf{\pageref{df,Phi}}

\item $U$  (minimal idempotent) \textbf{\pageref{minid01}}
\item $\hat U$   \textbf{\pageref{ellis08}}

\item $V_{d,\ep}$   \textbf{\pageref{Vep}}

\item $x[\M]$,  $x_+[\M]$  \textbf{\pageref{df,x[]}}


\item $X(A)$   \textbf{\pageref{df,X(A)}}

\item $X(\M)$, $X_+(\M)$   \textbf{\pageref{df,x[]}}
\item $(X,T)$  (cascade dynamical system) \textbf{\pageref{shift}}

\item $(X,\Gamma)$ (semigroup dynamical system) \textbf{\pageref{gammasys}}

\item $X(\Psi)$, $X_+(\Psi)$   \textbf{\pageref{df,Phi}}

\item $(X,S)$  (subshift)  \textbf{\pageref{shift}}


\item $\chi(A)$   \textbf{\pageref{sec,labels}}


\item  $\xi_A, \xi_\M$   \textbf{\pageref{xiA}}

\item $X(\P A)$   \textbf{\pageref{df,X(A)}}


\item $\Z$, $\Z_+$, $\Z_{+\infty}$  \textbf{\pageref{ss,labels}}

\item ZER  \textbf{\pageref{towlem17a}}

\item set operators: \textbf{\pageref{z}, \pageref{ssec,OC}}
\subitem $z_{CAN}$   \textbf{\pageref{z}}
\subitem  $z_{LAB}$ \textbf{\pageref{ssec,OC}, \pageref{tow30}}
\subitem $z_{LIM}$   \textbf{\pageref{z}}
\subitem $z_{NW}$   \textbf{\pageref{z}}
\subitem $z^*_{\M}$   \textbf{\pageref{tow33}}

\end{theindex}

\end{document}